\documentclass[10pt,twoside,a4paper]{article}
\usepackage{url}
\newcommand{\liuhao}{\fontsize{8pt}{\baselineskip}\selectfont}

\usepackage{amsfonts}
\usepackage{amsmath,amssymb}
\usepackage{mathrsfs}
\usepackage{amsthm}
\usepackage{hyperref}
\usepackage{appendix}

\newtheorem{thm}{Theorem}[section]

\newtheorem{lem}[thm]{Lemma}

\newtheorem{defn}[thm]{Definition}

\numberwithin{equation}{section}\allowdisplaybreaks
\def\leq{\leqslant}
\def\geq{\geqslant}

\begin{document}
	\title{ {\bf \Large 
			Local well-posedness for the quasilinear  Schr\"odinger equations via the generalized energy method }}
	\author
	{
		{ 
			{	Jie Shao$^{1}$\footnote{Corresponding
					author;
					Email address:
					shaojiehn@foxmail.com. } }   \ \	Yi Zhou $^{1}$\footnote{
				Email address:
				yizhou@fudan.edu.cn } } \\
		{\liuhao $^{1}$  School of Mathematical Sciences, Fudan University, Shanghai 200433, China.}\\
		\date{}
	}
	\date{}
	\maketitle

	\begin{minipage}{13.5cm}
		\footnotesize \bf Abstract. \rm  \quad 
	 We study the   Cauchy problem of   quasilinear Schr\"odinger equations, for which Kenig et al.  (Invent Math, 2004;  Adv
		Math, 2006)  obtained large data local well-posedness     by  pseudo-differential techniques and viscosity methods, while   Marzuola et al. (Adv Math, 2012;
		Kyoto J Math, 2014; Arch  Ration Mech Anal, 2021) and Ben et al. (Arch  Ration Mech Anal, 2024) improved the results  by  dispersive arguments.  	In this paper, we introduce  a generalized energy method that combines momentum and energy estimates to close the bounds,  thereby obtaining our results through viscosity methods.  If the data is small, the proof relies mainly on integration by parts and Sobolev embeddings, much like the classical local existence theory for semilinear Schrödinger equations. For large   data, the framework remains applicable with the incorporation of certain pseudo-differential tools. In the case of quadratic interactions, we establish low regularity local well-posedness for both small and large   data  in the same function spaces as in works of Kenig et al. For cubic interactions with small initial data, we recover the low regularity results obtained by Marzuola et al. (Kyoto J Math, 2014).

		\vspace{10pt}
		


	\end{minipage}
	
	% !1
	\section{Introduction} \label{s1}			
We consider the initial value problem (IVP) for general quasilinear Schr\"odinger equations 
	\begin{equation}	 \label{qNLS}
	\begin{cases}
			\sqrt{-1}\phi_t+\partial_i\left( g^{ij}\left(\phi,\overline{\phi}\right) \partial_j \phi\right)=F\left(\phi,\overline{\phi},\nabla\phi, \nabla \overline{\phi}\right),~~~\phi: \mathbb{R}^d \times \mathbb{R}^d\rightarrow \mathbb{C}^m\\
			\phi(0,x)=\phi_0(x),
	\end{cases}
	\end{equation}
	where  $ \sqrt{-1} $  stands for the unit imaginary number, and
	\begin{align*}
			&g: \mathbb{C}^m \times \mathbb{C}^m   \rightarrow \mathbb{R}^{d \times d}, \qquad
			F: \mathbb{C}^m \times \mathbb{C}^m \times\left(\mathbb{C}^m\right)^d \times\left(\mathbb{C}^m\right)^d \rightarrow \mathbb{C}^m
	\end{align*}
are smooth functions  satisfying  for $y, z \in \mathbb{C}^m \times\left(\mathbb{C}^m\right)^d$, 
\begin{align} \label{qgF}
	g^{ij}(y,\bar{y})-g_0^{ij}=h^{ij}(y, \bar{y})=O\left(|y|\right), \quad|F(y, \bar{y},z,\bar{z})| = O\left(|y|^2+|z|^2\right) \text{   near   }(y, z)=(0,0)
\end{align}
in the quadratic interaction problem and
\begin{align} \label{cgF}
	g^{ij}(y,\bar{y})-g_0^{ij}=h^{ij}(y,\bar{y})=O\left(|y|^2\right), \quad |F(y,\bar{y}, z,\bar{z})|=O\left(|y|^3+|z|^3\right)
	\text{  near  }(y, z)=(0,0)
\end{align}
in the cubic interaction problem.  Assume 
% ethe $ g^{ij} $ is elliptic, in other words,
$g^{ij}  $ is symmetric,  i.e. $ 	g^{ij}(y,\bar{y})= 	g^{ji}(y,\bar{y}) $ and suppose the  $ \left(g_0^{ij}\right) $ is a real constant matrix which can be ultrahyperbolic
\begin{align} \label{g0}
	\left(g_0^{ij}\right)=  \left(\begin{array}{cc}
	I_l &0\\
		0 & -I_{d-l}	
	\end{array}\right),
\end{align}
or can be elliptic 
\begin{align} \label{ellip_g}
\left(g_0^{ij}\right) = \left(	\delta_{i j}\right)  =I_d,
\end{align}
 where $1\leq  l  \leq d-1$ is an integer,   $ I_l $ is $ l\times  l $  identity matrix, and 
\begin{align*}
	\delta_{i j}= \begin{cases}0 & \text { if } i \neq j \\ 1 & \text { if } i=j\end{cases}
\end{align*}
is Kronecker delta.
The  divergence form \eqref{qNLS} is equivalent to the  nondivergence form
	\begin{equation}	 \label{qNLS0}
	\begin{cases}
		\sqrt{-1}\phi_t+g^{ij}\left(\phi,\overline{\phi},\nabla\phi, \nabla \overline{\phi} \right) \partial_{ij} \phi=F\left(\phi,\overline{\phi},\nabla\phi, \nabla \overline{\phi}\right),~~~\phi: \mathbb{R}^d \times \mathbb{R}^d\rightarrow \mathbb{C}^m\\
		\phi(0,x)=\phi_0(x).
	\end{cases}
\end{equation}
In reality,  if $ \phi $ solves \eqref{qNLS0}, then   $ (\phi, \nabla \phi)  $ solves \eqref{qNLS}.

The equations \eqref{qNLS} are   natural generalizations of linear and semilinear Schr\"odinger equation and  arise in some physical fields.   A notable example is the equation studied in Bouard-Hayashi-Saut \cite{de1997global}
\begin{align*}
	\partial_t u=\sqrt{-1} \Delta u-2 \sqrt{-1} u h^{\prime}\left(|u|^2\right) \Delta\left(h\left(|u|^2\right)\right)+\sqrt{-1} u g\left(|u|^2\right),
\end{align*}
where $h$ and $g$ are given functions, and it  models the self-channeling of a high power, ultra-short laser pulse in matter.  Another important case is   the skew mean curvature flow   in \cite{song2021local,huang2022local,li2021global} which  describes the evolution of a codimension two submanifold along its binormal direction. Under a suitable gauge, the skew mean curvature flow in  Huang-Tataru  \cite{huang2022local} can be transformed into a \eqref{qNLS0} type equation.      For further details on the physical background of \eqref{qNLS}, we refer to \cite{kenig2004cauchy, marzuola2012quasilinear, linares2014introduction}.

One main challenge in analyzing \eqref{qNLS} stems from the ``loss of derivatives" when computing its energy estimates.
As elaborated in Linares-Ponce \cite{linares2014introduction},   the $ L^2 $ well-posedness does not always hold even for the following linear IVP
\begin{equation} \label{lmd}
	\begin{cases}
		\partial_t u=\sqrt{-1} A u+\vec{b}(x) \cdot \nabla u+d(x) u  \\
		u(x, 0)=u_0(x).
	\end{cases}
\end{equation}
  Takeuchi \cite{takeuchi1974necessary}  proved the $ \sup _{l \in \mathbb{R}}\left|\int_0^l \mathrm{Im} b(s) d s\right|<\infty $ is sufficient for $ L^2 $ well-posedness of \eqref{lmd} if $ n=1, A=\partial_{xx} $, while Mizohata \cite{mizohata3cauchy} demonstrated that
\begin{align} \label{mcon}
	\sup _{\widehat{\xi} \in \mathbb{S}^{d-1}} \sup _{\substack{x \in \mathbb{R}^d \\ l \in \mathbb{R}}}\left|\int_0^l \mathrm{Im}  b_j(x+s \widehat{\xi}) d s\right|<\infty
\end{align}
	is necessary for the $L^2$ local well-posedness  of \eqref{lmd}  when  $ n\geq 1, A=\Delta $.  When $ A $ is elliptic variable coefficient, Ichinose  \cite{wataru1987cauchy} extended the results of \cite{mizohata3cauchy} and  obtained the Mizohata type condition  which associates to bicharacteristic flow. A direct observation of the difficulty presented by the term $\vec{b}(x) \cdot \nabla u$ is that the integral $ \int_0^t \int_{\mathbb{R}^d} \mathrm{Im} (\vec{b}(x) \cdot \nabla u \overline{u})  \mathrm{d} x  \mathrm{d} s $ cannot be bounded directly by $ \| u \|_{L^2(\mathbb{R}^d)} $, thus requiring additional techniques to establish the well-posedness of \eqref{lmd}.

The research history for \eqref{qNLS} has been  comprehensively introduced by \cite{kenig2004cauchy,kenig2006general,marzuola2012quasilinear,marzuola2021quasilinear,linares2014introduction,ifrim2023global}  and  we recall here for reader's convenience.  Kenig-Ponce-Vega  \cite{kenig1993small}  obtained small data local well-posedness  of ultrahyperbolic semilinear derivative Schr\"odinger equationin by utilizing the  smoothing effects.  Hayashi-Ozawa \cite{hayashi1994remarks} removed the smallness assumption on initial data in \cite{kenig1993small}  for one dimensional case and  introduced  a change of variables which can eliminate the bad terms that cannot be handled by integration.  For elliptic  semilinear derivative Schr\"odinger of general dimensions,  Chihara \cite{Chihara} removed the restriction of  initial data size by  utilizing the ellipticity to diagonalize the system and    applied   a transformation to eliminate
the  first order terms, which involved pseudo-differential
techniques and results including those in Doi \cite{Doi1994} for dimension $ d\geq 2 $. By avoiding the diagonalization procedure in \cite{Chihara}, Kenig et al. \cite{kenig1998smoothing} extended the results of  Chihara \cite{Chihara} to the ultrahyperbolic case.

 In the celebrated work  \cite{kenig2004cauchy}, Kenig, Ponce, Vega first obtained  large data local well-posedness for  for
\begin{equation} \label{km}
 \begin{cases}
		\partial_t u=-\sqrt{-1}  a_{j k}(x) \partial_{x_k x_j} u +\vec{b}_1(x) \cdot \nabla u+\vec{b}_2(x) \cdot \nabla \bar{u} \\
		\qquad~~+c_1(x) u+c_2(x) \bar{u}+f \\
		u(x, 0)=u_0(x)
	\end{cases} 
\end{equation}
in  high regularity spaces $ H^s(\mathbb{R}^d)\cap L^2(\langle x\rangle^N \mathrm{d}x) $, when  $ a_{j k}(x)= a_{j k}(t,x,u,\overline{u},\nabla u,\nabla \overline{u}) $ is symmetric elliptic, non-trapping, and thus \eqref{km}  is quasilinear. Kenig-Ponce-Vega-Rolvung \cite{kenig2005variable,kenig2006general} extended these results to semilinear and quasilinear versions of \eqref{km}  respectively, in the case where $a_{jk}(x)$ is ultrahyperbolic, non-trapping  and non-degenerate.
 It is natural to add some decay to the function space $ H^s(\mathbb{R}^d) $, since the potential $ \vec{b}_1(x),$  $\vec{b}_2(x) $ in general do not satisfy the  Mizohata  condition. In the elliptic quasilinear case,  \cite{kenig2004cauchy} rewrited the original equation into a system and diagonalize the first order terms following the idea of  Chihara \cite{Chihara}. Then by pseudo-differential techniques including the lemma of  Doi \cite{Doi1996Remarks}
and sharp Garding's inequality etc., \cite{kenig2004cauchy} eliminated  the first order terms which caused the loss of derivatives  and obtained the key smoothing effect estimates. In the ultrahyperbolic case, where the diagonalization method of Chihara cannot be used, \cite{kenig2005variable,kenig2006general} investigated the the bicharacteristic flow associated to the symbol of  $ \partial_{x_k}( a_{j k}(x) \partial_{ x_j}) $ by introducing a  class of non-standard symbols and proved
that the bicharacteristic flow is globally defined, which allowed them to establish the estimates needed for the smoothing effect.

 In \cite{marzuola2012quasilinear}, Marzuola, Metcalfe and Tataru introduced an purely dispersive approach, distinct from the methods in \cite{kenig2004cauchy, kenig2005variable, kenig2006general},  and established  the  local well-posedness of  \eqref{qNLS} with  small initial data and  quadratic interactions in the low regularity space $ l^1H^s(\mathbb{R}^d),s>\frac d2 +2$  for $ t\in[0,1]$.   The advantage of using frequency-localized spaces such as $l^1H^s$ lies in their ability to effectively recover the Mizohata condition for initial data with a suitable large $s$, since the Schr\"odinger waves at frequency $ 2^j $ travels with speed $ 2^j $  and the evolution at frequency $ 2^j $  in $ l^1H^s $
 will not be different from the corresponding evolution in $ H^s $.   \cite{marzuola2012quasilinear} also brought in spaces $ l^1X^s,l^1Y^s $ which  are associated with the  time-adapted Morrey-Campanato space  $ X $  satisfying $ X=Y^* $, and derived the multilinear estimates  and the local smoothing estimates of the   linearized \eqref{qNLS0} on norms concerning $  X, Y$,  then the uniform bounds of  linearized \eqref{qNLS0} by the initial data in $ l^1H^s $  are obtained and   the final result follows by iteration. By a similar argument, Marzuola-Metcalfe-Tataru \cite{marzuola2014quasilinear}   obtained the  local well-posedness of  \eqref{qNLS} with small data and cubic  interactions in $ H^s(\mathbb{R}^d),s>\frac {d+3}2 $  for  $ t\in[0,1]$. The addition of decay into the function spaces could be removed, since the integrability of linearized \eqref{qNLS0} with cubic nonlinearities in the function spaces can be obtained easily  from energy estimates.

 For elliptic \eqref{qNLS} with large  initial data, non-trapping
 condition and  $ g_0^{ij}   $ as \eqref{ellip_g}, Marzuola-Metcalfe-Tataru \cite{marzuola2021quasilinear} obtained local low-regularity well-posedness for both quadratic and cubic interactions. The function space  and regularity of the solution  corresponding to the quadratic and cubic interaction problem  coincide with those in \cite{marzuola2012quasilinear} and \cite{marzuola2021quasilinear}, respectively.  To our knowledge, \cite{marzuola2021quasilinear} is the first work that obtained low-regularity well-posedness for quasilinear Schrödinger equations with large initial data. Consequently, its results represent a significant improvement upon  works by Kenig et al. \cite{kenig2004cauchy,kenig2005variable,kenig2006general} and Marzuola et al. \cite{marzuola2012quasilinear,marzuola2014quasilinear}. 
 For an outline of the proof for the main results in \cite{marzuola2021quasilinear}, we refer to Section 3 of \cite{marzuola2021quasilinear}.  More recently,  Ben-Mitchell \cite{pineau2024low} extended the results of \cite{marzuola2021quasilinear}  to the ultrahyperbolic setting   by pseudo-differential calculus for classical symbols based on a spatial truncation technique to close the energy estimates for \eqref{qNLS}, which is inspired by Jeong-Oh \cite{jeong2024well-posedness}.

There have also been recent advances in  the global well-posedness theory for \eqref{qNLS} with small initial data.
Ifrim-Tataru \cite{ifrim2025global} studied the one dimensional cubic quasilinear Schr\"odinger equations and obtained several  brilliant results, including proving the global in time wellposedenss in $ H^s (\mathbb{R}), s>1 $  for  one dimensional cubic defocusing \eqref{qNLS}, which is phase rotation symmetry and is
conservative, if the initial data is small. Here the  $ s>1 $  is the almost optimal regularity and  is lower than those of the previous works. 
The \cite{ifrim2025global} is the sequel work after Ifrim-Tataru \cite{ifrim2023global,ifrim2024long} which 
investigated  classic   and  general   cubic semilinear Schr\"odinger equations  in one dimension. Later, Ifrim-Tataru \cite{ifrim-Tataru} extended the results of  \cite{ifrim2025global} to higher dimensions. 
 In a different direction, Lai-Shao-Zhou \cite{lai2025global} obtained the small initial data global well-posedness of skew mean curvature flow in critical Besov spaces for $ d\geq 5 $, and its proof method can provide a new perspective to study quasilinear Schr\"odinger equations.

Momentum-type estimates have been widely used in the study of semilinear Schrödinger equations and continue to play a central role (see, e.g., Dodson–Murphy \cite{Dodson} and Colliander–Keel–Staffilani –Takaoka–Tao \cite{Tao}). Recently, the second author  of the present paper developed a div-curl type lemma that combines momentum and energy identities. This approach yields bounds beyond those provided by standard energy estimates and has shown significant potential for applications to quasilinear dispersive equations, such as Schrödinger map flow, wave maps, time-like extremal hypersurfaces, and skew mean curvature flow (see e.g. \cite{lai2025global,LZ24,wang2212global,WZ3,wang2023proof,zhou20221+}).

Inspired by previous works, we introduce the  generalized energy method in this paper to study \eqref{qNLS} for both  quadratic  and cubic interaction problem. For the small data case, the main idea is that   the $ |\alpha| $ order  momentum type equalities  can be used to obtain the momentum type etsimates and get some $ |\alpha|+1 $ space time norms with  weight. After  the type analysis of fractional integration by parts formula, this process will produce  some  $ |\alpha|+\frac12 $ order  terms that are to be bounded and have no time integration, the   $ |\alpha|+1 $ order good terms  can be utilized to control  the nonlinear terms $ \mathrm{Im} (\overline{\phi}_{\alpha} \partial^{\alpha} F) $  in the  $ |\alpha| $ order  energy estimates and the terms $ \mathrm{Im} (D \overline{\phi}_{\alpha}\cdot \partial^{\alpha} F) $ in the $ |\alpha|+\frac 12 $ order  energy estimates, if $ |\alpha| $ is suitable large.  Consequently, one can derive both the   $ |\alpha| $ order   and   the $ |\alpha|+\frac 12 $ order  energy estimates and this in turn  can bound the bad  $ |\alpha|+\frac12 $ order bad  terms in  the momentum type estimates. In short,  by combining the energy estimates and the momentum type estimates, we can close the bounds which combining with the artificial viscosity method will lead to the desired results.  
 
  The weight we choose  for small data quadratic interaction problem is $ \frac {x_k}{\langle x_k \rangle},$ where $ \langle x\rangle =1+\left|x\right|$. In view of  $ \frac{\mathrm{d} }{\mathrm{d} x_k} \big(\frac {x_k}{\langle x_k \rangle}\big) =\frac {1}{\langle x_k \rangle^2}$,  this weight multiplying  momentum type equalities and integrating on space and time will produce a good term as $ \int_0^t     \int_{\mathbb{R}^{d}}  |\frac{\partial_k \phi_{\alpha }  } {\langle x_k\rangle} |^2  \mathrm{d} x \mathrm{d} s $.   We remark that the weight $  \frac {x_k}{\langle x_k \rangle} $  can be replaced by
$ \int_{-\infty}^{x_k}\frac{1}{(1+|y_k|^2)^{\frac {1+\varsigma}{2}}} \mathrm{d} y_k, \varsigma>0. $
For the cubic interaction problem with small initial data, we select  the weight $ \int_{-\infty}^{x_k} \|\Lambda_k^{\frac 12}\phi_{\beta}(y_k,\cdot)\|^2_{L^2(\mathbb{R}^{d-1})} \mathrm{d} y_k$ which  can help to produce the good terms $ \int_0^t\int_{\mathbb{R}}  \|\Lambda_k^{\frac 12}\phi_{\beta}(x_k,\cdot)\|^2_{L^2(\mathbb{R}^{d-1})}   \int_{\mathbb{R}^{d-1}} \left|\partial_k \phi_{\alpha } \right|^2 \mathrm{d} \hat{x}_k  \mathrm{d} x_k \mathrm{d} s$. Here,
  the operator $\Lambda_k^{\frac{1}{2}} = \left(1 - \sum_{i=1, i \neq k}^d \partial_{ii} \right)^{\frac{1}{4}}$ will allow us to reduce the required regularity by half an order in the final results.

The main results for  small data problems are  as follows.
\begin{thm} \label{thm2}
$ 	(a) $ Suppose that    $ s>\frac {d}2+3, d \geq 1  $,  \eqref{qNLS} is  quadratic interaction problem and
\begin{align*}
	   \|\phi_0\|_{H^s(\mathbb{R}^d)} +\left\| \langle x\rangle^N \phi_0 \right\|_{L^2(\mathbb{R}^d)}\ll 1
\end{align*}  
  for   some integer $ N $ that depends only on $ d, s $.   Then  the ultrahyperbolic equation \eqref{qNLS} under condition   \eqref{g0} is locally wellposed in $   H^{s} (\mathbb{R}^d) \cap L^2(\langle x\rangle^N \mathrm{d}x) $ for $ t\in [0,1] $.\\
$ (b) $ The same results holds for \eqref{qNLS0} with $ s>\frac{d}{2}+4, d\geq 1. $
\end{thm}
\begin{thm} \label{thm3}
	$ 	(a) $   Suppose that $  s>\frac {d+3}2, d \geq 2,   $   \eqref{qNLS} is  cubic interaction problem and $ \|\phi_0\|_{H^s(\mathbb{R}^d)}\ll 1 $. Then the ultrahyperbolic equation \eqref{qNLS} under condition   \eqref{g0}  is locally wellposed in $ H^s(\mathbb{R}^d) $  for $t\in  [0,1] $.\\
	$ (b) $ The same results holds for \eqref{qNLS0} with $ s>\frac{d+5}{2}, d\geq 2. $
\end{thm}

The large  data problem presents much greater  challenges than the small   data case, primarily for two reasons.  On the one hand, it is necessary to estimate the non-trapping metric associated with  \( g(\phi, \overline{\phi}) \).  On the other hand, in the absence of small initial data, the norms and weighted norms of the unknown functions become large. For example, in \eqref{km}, the norm of \( \vec{b}_1(x) \) in the term \(\vec{b}_1(x) \cdot \nabla u \)  is not small, which leads to effects that are more adverse than in the case of small initial values. To address this, we choose a weight of the form \( e^{M \frac{x_k}{\langle x_k\rangle} } \), and the proof strategy is outlined as follows. First, the relevant nonlinear terms of \( g(\phi, \overline{\phi}) \) can be controlled using  weighted norms $ \|\langle x\rangle ^N \phi \|_{L^2(\mathbb{R}^d)} $ and together with the non-trapping metric, whose precise definition will be given later. Meanwhile, in the momentum estimates, the term \( \vec{b}_1(x) \) in nonlinear term of form  \( \vec{b}_1(x) \cdot \nabla u \) can be controlled by \( M \) in \( M e^{Mx} \). However, for energy estimates, \( M \) actually does not play a role. Inspired by Kenig-Vega-Ponce \cite{kenig2004cauchy},  we use pseudo-differential operator techniques to eliminate the \( \vec{b}_{1k}(x) \cdot \partial_k u \) in the main direction \( x_k \), while the terms \(\vec{b}_{1j}(x) \cdot \partial_j u,j \neq k \) in other directions, need to be made small by using the \( \epsilon \), where \( \epsilon \) is radius of a covering on the sphere. At the same time for  the large coverings  number of \( \epsilon \)-covering, we can use the smallness of time to control.  Finally, the bounds for the norms of \( \phi \) can be obtained via the bootstrap argument.

The analysis of the large  data problem fundamentally relies on the non-trapping condition. Indeed, if this condition fails severely (for example, if all geodesics are periodic), ill-posedness can occur  even in  the semilinear settings, as shown by Chihara \cite{chihara2002initial}. We now state the definition of the non-trapping condition, adapted from Marzuola–Metcalfe–Tataru \cite{marzuola2021quasilinear}.
\begin{defn} \label{def_nontrap}
	We say that the metric $g\left(\phi_0, \overline{\phi}_0\right)$ is non-trapping if all nontrivial bicharacteristics for $\Delta_{g\left(\phi_0, \overline{\phi}_0\right)}$ escape to spatial infinity at both ends. More precisely, for a ball $ B:=B_R $ of radius $  R \gg 1 $ which is large enough, any geodesic for the metric $g\left(\phi_0, \overline{\phi}_0\right)$ that exits the ball $B$ will escape to spatial infinity without re-entering $B / 2$.
\end{defn}
Under the non-trapping assumption on \( g(\phi_0, \overline{\phi}_0) \), we obtain the following theorem for \eqref{qNLS} with large initial data and quadratic interactions.
\begin{thm} \label{thm_bigdata}
	$ 	(a) $  Let $d \geq 1$ and $s > \frac{d}{2} + \frac{7}{2}$. Suppose $\phi_0 \in x^N L^2 \cap H^s$ is a non-trapping initial datum for equation \eqref{qNLS} with quadratic interactions, where the integer $N$ depends  on $d$ and $s$. Then there exists  $T = T(\phi_0) > 0$  sufficiently small such that the elliptic equation \eqref{qNLS} under condition \eqref{ellip_g} is locally well-posed in $x^N L^2 \cap H^s(\mathbb{R}^d)$ on the time interval $I = [0, T]$.\\
	$ 	(b)  $The same result holds for the \eqref{qNLS0} with $s>\frac{d}{2}+\frac 92$.
\end{thm}

 It is worth noting that Theorem \ref{thm3} yields the same results as Marzuola–Metcalfe–Tataru \cite{marzuola2014quasilinear} for \(d \geq 2\). Meanwhile, the function spaces in Theorem \ref{thm2} align with those in Kenig et al. \cite{kenig2004cauchy,kenig2005variable,kenig2006general}, but differ from the \(l^1H^s\) spaces used in Marzuola–Metcalfe–Tataru \cite{marzuola2012quasilinear}, although Theorem \ref{thm2} requires one order higher regularity than the one in the latter.
 In the proofs of Theorems \ref{thm2} and \ref{thm3}, one may observe no essential difference between the elliptic and ultrahyperbolic cases. However, in the case of Theorem \ref{thm_bigdata}, the result is only established for the elliptic setting, since the method in Section \ref{s_big} relies on a finite number covering, which does not readily extend to the ultrahyperbolic case.
 Note that the regularity  in Theorem \ref{thm_bigdata} is half an order higher than that in Theorem \ref{thm2}, since   the term $\partial_i\left( h^{ij}(\phi,\overline{\phi}) \partial_j P_{\lambda} v \right)$ in equation \eqref{eq_w} will bring additional loss of derivatives when we drive the smallness of the $w$.

 The paper is organized as follows.  Section \ref{s2} provides the necessary preliminaries for the main results. Subsection \ref{s2.1} introduces the basic notations, Subsection \ref{s2.2} states the necessary preliminary lemmas, and Subsection \ref{s2.3} derives the momentum balance equalities. Sections \ref{s3} and \ref{s4} are devoted to the proofs of Theorems \ref{thm2} and \ref{thm3}, respectively, both addressing the small initial data case. Section \ref{s_big} contains the proof of Theorem \ref{thm_bigdata}, which deals with large initial data and quadratic interactions. Finally, the proofs of some technical lemmas are provided in Appendix \ref{appa}.

\section{Notations and Preliminaries}  \label{s2}
\subsection{Notations} \label{s2.1}

Denote
\begin{align*}
	\phi_\alpha:=\partial^\alpha \phi=\partial_{x_1}^{\alpha_1}\ldots\partial_{x_d}^{\alpha_d} \phi
\end{align*}
and the vector $ \alpha=(\alpha_1,\ldots,\alpha_d) \in \mathbb{N}^d_0$ is a $ d $-dimensional multi-index of order $ |\alpha| $, where the component $ \alpha_i $ are non-negative integers, and $ |\alpha|=\sum_{i=1}^d \alpha_i $. 
The  notation $ \phi_\beta, \phi_\gamma $ and vector $ \beta, \gamma$    are defined similarly.  As an example,  for   $ e_i=( 0,\ldots,0,1,0\ldots,0) $,  there is $ \phi_{e_i}=\partial^{e_i} \phi=\partial_{x_i} \phi $.

Denote $\Delta^2=\sum_{i,j=1}^d \partial_{ii}\partial_{jj},  \langle x \rangle:= 1+|x|, x\in \mathbb{R}^d  $, $ \partial_k:=\partial_{x_k}, \partial_{ij}:=\partial_{x_i}\partial_{x_j} $,  $ J:= (1-\sum_{i=1}^d \partial_{ii}  )^{\frac12}=\left(1-\Delta \right)^{\frac12} $, $ \Lambda_k=\widehat{J}_k=  (1-\sum_{i=1,i\neq k}^d \partial_{ii} )^{\frac12} $, $ D=(-\Delta)^{\frac 12} $ and  $ D_k:=(-\partial_{kk})^{\frac 12} $ for $ 1\leq k\leq d. $  Define $ \hat{x}_k:= (x_1,\dots,$  $x_{k-1},x_{k+1},\dots,x_d ) $ and it is indispensable to present momentum type equalities in Subsection \ref{s2.3}.

%$ \langle u,v\rangle= \frac{1}{2} \left( u \overline{v}\right)$

$ \mathcal{F} u (\xi)= \hat{u}(\xi) $ stands for the   Fourier transform of function $ u $ and $ \mathrm{sgn} (x) $ stands for the sign function of $ x, x\in \mathbb{R} $.
$L^p(\mathbb{R}^d), H^s(\mathbb{R}^d), W^{k,p}(\mathbb{R}^d), \dot{B}^s_{p,q}(\mathbb{R}^d),$  $ \text{BMO}(\mathbb{R}^d) $ are standard Lebesgue space,  Sobolv space, homogeneous Besov space and  the  space of functions of Bounded Mean Oscillation respectively. We refer to \cite{bahouri2011fourier, chen2023well} for precise definitions and norms of these spaces.  The notation $a \lesssim b$ is used to denote the estimate $a \leq C b$ for some positive constant $C$, and $ a\sim b $ to denote $ a \lesssim b \lesssim a $.

 For  a smooth function $\varphi (\xi)$  supported in $ |\xi|\leq 2$ and equalling to 1 in $|\xi|\leq 1$, we define the Littlewood-Paley operator $P_\lambda $  for $ \lambda\in \mathbb{N}$ by
\begin{align}
	&\widehat{P_{\leq \lambda} u}(\xi):=\varphi\left(2^{-\lambda} \xi\right) \widehat{u}(\xi),\quad P_\lambda=P_{\leq \lambda}-P_{\leq \lambda-1}.\nonumber
\end{align}

\subsection{ Preliminaries}  \label{s2.2}
%This section presents  ... and derive the 
The following commutator estimates  lemma, whose  proof is given in  Appendix \ref{appa}, is an extension of Theorem 1.2 in Fefferman et al. \cite{fefferman2014higher}.
	\begin{lem}  \label{lemc1}
		Given real number $ \mathscr{S}, s $ satisfying $ \mathscr{S}>d / 2,s\leq 1,s+l= 1$, there is a constant $c=c(d, s,\mathscr{S},l)$ such that, for all $u, g$ with $u \in H^s\left(\mathbb{R}^d\right) $ and $\nabla g  \in H^{\mathscr{S}}\left(\mathbb{R}^d\right)$, the following inequalities hold:
		\begin{align} 
			\left\|D^s\left(g \partial_j u \right)-g \partial_j \left(D^s u\right)\right\|_{L^2} \leqslant c\|\nabla g\|_{H^{\mathscr{S}}}\|u\|_{H^s},\label{lemc1_1} \\
				\left\|D^s\left(g D^l u \right)-g D^{s+l} u\right\|_{L^2} \leqslant c\|\nabla g\|_{H^{\mathscr{S}}}\|u\|_{L^2}. \label{lemc1_2}
		\end{align}	
	\end{lem}
Observe that a special case of \eqref{lemc1_2} is 
\begin{align}
	\left\| \left[ g D^{\frac12}, D^{\frac12} \right] \overline{\phi}_{\alpha k }\right\|_{L^2} \lesssim  \|\nabla g\|_{H^{\mathscr{S}}} \|\phi_{\alpha k }\|_{L^{2}}^2. \label{lemc1_3}  
\end{align}	

The following lemma is a corollary of Theorem 2 of Calder{\'o}n \cite{calderon1965commutators}; see also Lemma 4.2 in \cite{chen2023well}.
\begin{lem} \label{lemc2}
	Assume $\varphi \in C^{\infty}, \nabla \varphi \in L^{\infty}$, $1<p<\infty$. Then,
	\begin{align}
		\left\|[D, \varphi] f\right\|_{L^{p}( \mathbb{R}^d)  } \lesssim\|\nabla \varphi\|_{L^{\infty}( \mathbb{R}^d)} \left\|f\right\|_{L^{p}( \mathbb{R}^d)  }. \label{lemc2_1}
	\end{align}	
\end{lem}

The two items in the following commutator estimates lemma are the  Remark 1.3 and Theorem 1.9 of \cite{li2019kato}  respectively.
\begin{lem}  \label{lemc3}
$ (i) $	If $0< s<1 $ and $ 1<p<\infty $, then
	\begin{align}  \label{lemc3_1}
		\left\|D^s(f g)-g D^s f\right\|_{L^p(\mathbb{R}^d)} \lesssim\|f\|_{L^p(\mathbb{R}^d)}\left\|D^s g\right\|_{\mathrm{BMO}(\mathbb{R}^d)} \lesssim\|f\|_{L^p(\mathbb{R}^d)}\left\|D^s g\right\|_{L^\infty(\mathbb{R}^d)},
	\end{align}
where $ D $ is defined as the one in Subsection \ref{s2.1}\\
$ (ii) $ Let $1<p<\infty$,  $1<p_1, p_2  \leq \infty$ satisfy $1 / p_1+1 / p_2   =1 / p$,  and $0<s \leq 1$,  $f, g \in \mathcal{S}\left(\mathbb{R}^d\right)$,  then
\begin{align}  \label{lemc3_2}
	\left\|J^s(f g)-f J^s g\right\|_{L^p} \lesssim_{s, p_1, p_2, p, d}\left\|J^{s-1} \nabla f\right\|_{L^{p_1}}\|g\|_{L^{p_2}} .
\end{align}

\end{lem}
The \eqref{lemc3_2} immediately gives 
\begin{align}  \label{lemc3_3}
	\left\|J^s(f g)-f J^s g\right\|_{L^p} \lesssim_{s, p_1, p_2, p, d}\left\|J^{s}  f\right\|_{L^{p_1}}\|g\|_{L^{p_2}},
\end{align}
if $ p_1 \neq \infty. $

The following lemma for  $ \mathrm{BMO} $ norm  is Theorem 1.48 of \cite{bahouri2011fourier} and will be used in Section \ref{s4}.
\begin{lem}   \label{lem4}
For all functions $u \in L_{l o c}^1\left(\mathbb{R}^d\right) \cap \dot{H}^{\frac{d}{2}}\left(\mathbb{R}^d\right)$, there exists a constant $C$ such that
\begin{align} \label{lem4_1}
		\|u\|_{\mathrm{BMO}(\mathbb{R}^d)} \leq C\|u\|_{\dot{H}^{\frac{d}{2}}(\mathbb{R}^d)}.
\end{align}
\end{lem}

The following Kato-Ponce inequalities are useful to deal with the nonlinear terms. 
\begin{lem} \cite{grafakos2014kato}
	Let $\frac{1}{2}<r<\infty, \quad 1<p_{1}, p_{2}, q_{1}, q_{2} \leq \infty$ satisfy $\frac{1}{r}=\frac{1}{p_{1}}+\frac{1}{q_{1}}=\frac{1}{p_{2}}+\frac{1}{q_{2}} .$
	Given $s>\max \left(0, \frac{d}{r}-d\right)$ or $s \in 2 \mathbb{N}$, there exists $C=C\left(d, s, r, p_{1}, q_{1}, p_{2}, q_{2}\right)<\infty$
	such that for all $f, g \in \mathcal{S}\left(\mathbb{R}^{d}\right)$, we have
	\begin{align}
		\left\|D^{s}(f g)\right\|_{L^{r}\left(\mathbb{R}^{d}\right)} \lesssim &\left\|D^{s} f\right\|_{L^{p_{1}}\left(\mathbb{R}^{d}\right)}\|g\|_{L^{q_{1}}\left(\mathbb{R}^{d}\right)}+\|f\|_{L^{p_{2}\left(\mathbb{R}^{d}\right)}}\left\|D^{s} g\right\|_{L^{q_{2}}\left(\mathbb{R}^{d}\right)}, \label{nonlinear1}\\
			\left\|J^{s}(f g)\right\|_{L^{r}\left(\mathbb{R}^{d}\right)} \lesssim &\left\|J^{s} f\right\|_{L^{p_{1}}\left(\mathbb{R}^{d}\right)}\|g\|_{L^{q_{1}}\left(\mathbb{R}^{d}\right)}+\|f\|_{L^{p_{2}\left(\mathbb{R}^{d}\right)}}\left\|J^{s} g\right\|_{L^{q_{2}}\left(\mathbb{R}^{d}\right)}. \label{nonlinear2}
	\end{align}
\end{lem}

\subsection{Momentum balance equalities}  \label{s2.3}
For simplicity, we will prove Theorems \ref{thm2} and \ref{thm3}  under the assumption that the regularity index \( s \) satisfies   \( s - \frac12 \in \mathbb{N}\), namely $s - \frac12$ is an integer. The general case for  real \( s \) can be treated via the Littlewood-Paley decomposition method as  in Section \ref{s_big}. In this subsection, we derive several momentum balance equalities, which will serve as a preparation for the proofs of Theorems \ref{thm2} and \ref{thm3}.

Take derivatives on \eqref{qNLS} for $ \alpha $ times and apply the Leibniz rule  to derive
  	 \begin{align} \label{qNLS1}
  	\sqrt{-1}\phi_{\alpha t}+\partial_i\left( g^{ij} \partial_j \phi_{\alpha }\right)=\mathcal{N},
  \end{align}
%	\begin{align} \label{qNLS1_1}
%	\sqrt{-1}\phi_{\alpha t}+\partial_i\left( g^{ij} \partial_j \phi_{\alpha }\right)=F_z\nabla \phi_\alpha+F_{\overline{z}}\nabla \overline{\phi}_\alpha+\partial h^{ij}\partial_{ij } \phi_{\alpha -1}+G,
%\end{align}
	where 
	\begin{align*}
	&	\mathcal{N}:= F_z\nabla \phi_\alpha+F_{\overline{z}}\nabla \overline{\phi}_\alpha+\partial h^{ij}\partial_{ij } \phi_{\alpha -1}+G, \\
	&	\partial_j \phi_{\alpha }=\partial_j \phi_{\alpha } ,\quad \partial h^{ij}\partial_{ij } \phi_{\alpha -1}:=\sum_{n=1}^{d} \partial^{e_n}  h^{ij}\partial_{ij } \phi_{\alpha -e_n}, 
		~~ \alpha_n -1\geq0.\nonumber
	\end{align*}
 	 $F_z$, $F_{\bar{z}}$ are the usual complex derivatives satisfying
	 \begin{align*} 
	 	&F_z=\frac{1}{2}\left(\frac{\partial F}{\partial z_1}-\sqrt{-1} \frac{\partial F}{\partial z_2}\right), \quad  F_{\bar{z}}=\frac{1}{2}\left(\frac{\partial F}{\partial z_1}+\sqrt{-1}\frac{\partial F}{\partial z_2}\right) \\
	 	& \text { and } z=z_1+\sqrt{-1} z_2, \qquad z_1,z_2 \text{ are real},
	 \end{align*}	 
 and  $ G $  are the remainder terms whose  highest order derivatives of $ \phi $   are  less than $ |\alpha|  $, including $ F_y  \phi_\alpha,F_{\bar{y}}  \overline{\phi}_\alpha.$ 
	  Here and throughout,  $  g^{ij}, h^{ij} $  are abbreviations for  $  g^{ij}\left(\phi,\overline{\phi}\right),  h^{ij}\left(\phi,\overline{\phi}\right) $  respectively.
	 
	 Multiplying \eqref{qNLS1} by $ \partial_k \overline{\phi}_\alpha $ and taking the real part infers that
\begin{align}  \label{cons1_1}
	&\mathrm{Re} \left(\sqrt{-1}\phi_{\alpha t}\partial_k \overline{\phi}_\alpha \right)+	\mathrm{Re} \left( \partial_i\left( g^{ij} \partial_j \phi_{\alpha }\right)\partial_k \overline{\phi}_\alpha \right)=	\mathrm{Re} \left(\mathcal{N} \partial_k \overline{\phi}_\alpha\right).
\end{align}
 Direct calculation shows
	\begin{align} \label{cons1_2}
	&\mathrm{Re}\left(\partial_k \overline{\phi}_\alpha	\partial_i\left( g^{ij} \partial_j \phi_{\alpha }\right)\right)\\
	=&\,- \mathrm{Re}\left(\partial_{ki} \overline{\phi}_\alpha	 g^{ij} \partial_j \phi_{\alpha }\right)+\mathrm{Re} \partial_i\left(\partial_k \overline{\phi}_\alpha	 g^{ij} \partial_j \phi_{\alpha }\right). \nonumber
\end{align}
By  the  symmetry of  $ g^{ij} $, one has
	\begin{align*}
		2\mathrm{Re}\left(\partial_{ki} \phi_{\alpha }  g^{ij} \partial_j \overline{\phi}_{\alpha }\right)=&\,\partial_k \mathrm{Re}\left(\partial_{i} \phi_{\alpha }  g^{ij} \partial_j \overline{\phi}_{\alpha }\right)-\mathrm{Re}\left(\partial_{i} \phi_{\alpha }  \partial_k h^{ij} \partial_j \overline{\phi}_{\alpha }\right)
\end{align*}
which combining the identity  $ 	\overline{\sqrt{-1}\partial_k \phi_{\alpha t}}=	-\sqrt{-1}~\partial_k\overline{\phi}_{\alpha t} $ gives
	\begin{align} \label{cons1_3}
	&\mathrm{Re}\left( \sqrt{-1}\phi_{\alpha t} \partial_k \overline{\phi}_\alpha\right)=\partial_t \mathrm{Re}\left( \sqrt{-1}\phi_{\alpha } \partial_k \overline{\phi}_\alpha \right)-\mathrm{Re}\left( \phi_{\alpha } \sqrt{-1} \partial_{kt} \overline{\phi}_\alpha\right)\nonumber\\
	=&\,-\partial_t \mathrm{Im}\left( \phi_{\alpha } \partial_k \overline{\phi}_\alpha \right)-\mathrm{Re}\left( \phi_{\alpha }  \partial_{ik}\left( g^{ij} \overline{\phi}_{\alpha j}\right)\right)+\mathrm{Re}\left( \phi_{\alpha } \partial_k\overline{\mathcal{N}} \right)\\
	=&\,-\partial_t \mathrm{Im}\left(\phi_{\alpha } \partial_k \overline{\phi}_\alpha \right)-\mathrm{Re}\partial_i\left( \phi_{\alpha }  \partial_{k}\left( g^{ij} \overline{\phi}_{\alpha j}\right)\right)+\mathrm{Re}\left(\partial_i \phi_{\alpha }  \partial_{k}\left( g^{ij} \overline{\phi}_{\alpha j}\right)\right)+\mathrm{Re}\left( \phi_{\alpha } \partial_k\overline{\mathcal{N}} \right)  \nonumber\\
	=&\,-\partial_t \mathrm{Im}\left( \phi_{\alpha } \partial_k \overline{\phi}_\alpha \right)-\mathrm{Re}\partial_{ik}\left( \phi_{\alpha }  \partial_{i}\left( g^{ij} \overline{\phi}_{\alpha j}\right)\right)+ \mathrm{Re}\partial_{i}\left(\partial_k \phi_{\alpha }  g^{ij} \overline{\phi}_{\alpha j}\right)\nonumber\\
	&+ \mathrm{Re}\partial_{k}\left(\partial_i \phi_{\alpha }  g^{ij} \overline{\phi}_{\alpha j}\right)-\mathrm{Re}\left(\partial_{ki} \phi_{\alpha }  g^{ij} \overline{\phi}_{\alpha j}\right)+\mathrm{Re}\left( \phi_{\alpha } \partial_k\overline{\mathcal{N}} \right), \nonumber\\
		=&\,-\partial_t \mathrm{Im}\left( \phi_{\alpha } \partial_k \overline{\phi}_\alpha \right)-\mathrm{Re}\partial_{ik}\left( \phi_{\alpha }  \partial_{i}\left( g^{ij} \overline{\phi}_{\alpha j}\right)\right)+ \mathrm{Re}\partial_{i}\left(\partial_k \phi_{\alpha }  g^{ij} \overline{\phi}_{\alpha j}\right)\nonumber\\
	&+\mathrm{Re}\left(\partial_{ki} \phi_{\alpha }  g^{ij} \overline{\phi}_{\alpha j}\right)+ \mathrm{Re}\left(\partial_{i} \phi_{\alpha }  \partial_k h^{ij} \partial_j \overline{\phi}_{\alpha }\right)+\mathrm{Re}\left( \phi_{\alpha } \partial_k\overline{\mathcal{N}} \right), \nonumber
\end{align}
where  \eqref{qNLS1} is used in the second line.
%Combining \eqref{cons1_1}, \eqref{cons1_2} and \eqref{cons1_3}  gives
%	\begin{align} \label{cons1}
%	&-\partial_t \mathrm{Im}\left( \phi_{\alpha } \partial_k \overline{\phi}_\alpha \right)-\mathrm{Re}\partial_k\left( \phi_{\alpha }  \partial_{i}\left( g^{ij} \overline{\phi}_{\alpha j}\right)\right)\\
%	&+ 2\mathrm{Re}\partial_{i}\left(\partial_k \phi_{\alpha }  g^{ij} \overline{\phi}_{\alpha j}\right)-2\mathrm{Re}\left(\partial_{ki} \phi_{\alpha }  g^{ij} \overline{\phi}_{\alpha j}\right)\nonumber \\
%	=&\,-\mathrm{Re}\left( \phi_{\alpha } \partial_k\overline{\mathcal{N}} \right)+\mathrm{Re} \left(\mathcal{N} \partial_k \overline{\phi}_\alpha\right). \nonumber
%\end{align}
%Noticing

Combining \eqref{cons1_1}, \eqref{cons1_2} and \eqref{cons1_3} allows us to establish the following  momentum balance equalities
	\begin{align} \label{consm1}
	&-\partial_t \mathrm{Im}\left( \phi_{\alpha } \partial_k \overline{\phi}_\alpha \right)-\mathrm{Re}\partial_k\partial_i\left( \phi_{\alpha }   g^{ij} \overline{\phi}_{\alpha j}\right)\nonumber\\
	&+ 2\mathrm{Re}\partial_{i}\left(\partial_k \phi_{\alpha }  g^{ij} \overline{\phi}_{\alpha j}\right)+\mathrm{Re}\left(\partial_{i} \phi_{\alpha }  \partial_k h^{ij} \partial_j \overline{\phi}_{\alpha }\right) \\
	=&\,-\mathrm{Re}\left( \phi_{\alpha } \partial_k\overline{\mathcal{N}} \right)+\mathrm{Re} \left(\mathcal{N} \partial_k \overline{\phi}_\alpha\right)\nonumber\\
	=&\,-\partial_k\mathrm{Re}\left( \overline{\phi }_{\alpha } \mathcal{N}\right)+2\mathrm{Re} \left(\mathcal{N} \partial_k \overline{\phi}_\alpha\right) \nonumber.
\end{align}
Denote $ \hat{x}_k:=\left(x_1,\dots,x_{k-1},x_{k+1},\dots,x_d\right) $. For \( d \geq 1 \), integrating \eqref{consm1} over \( \mathbb{R}^{d-1} \) with respect to \( \hat{x}_k \) yields
	\begin{align} \label{consm2}
	&	-\partial_t \int_{\mathbb{R}^{d-1}}  \mathrm{Im}\left( \phi_{\alpha } \partial_k \overline{\phi}_\alpha \right) \mathrm{d} \hat{x}_k-\partial_{kk} \int_{\mathbb{R}^{d-1}} \mathrm{Re}\left( \phi_{\alpha }   g^{kj} \overline{\phi}_{\alpha j}\right) \mathrm{d} \hat{x}_k\nonumber\\
	&+2 \partial_k \int_{\mathbb{R}^{d-1}} \mathrm{Re}\left(\partial_k \phi_{\alpha }  g^{kj} \overline{\phi}_{\alpha j}\right) \mathrm{d} \hat{x}_k + \int_{\mathbb{R}^{d-1}}  \mathrm{Re}\left(\partial_{i} \phi_{\alpha }  \partial_k h^{ij} \partial_j \overline{\phi}_{\alpha }\right)  \mathrm{d} \hat{x}_k \\
	=&\,- \partial_k\int_{\mathbb{R}^{d-1}}   \mathrm{Re}\left( \overline{\phi }_{\alpha } \mathcal{N}\right)\mathrm{d} \hat{x}_k+2\int_{\mathbb{R}^{d-1}} \mathrm{Re} \left(\mathcal{N} \partial_k \overline{\phi}_\alpha\right)   \mathrm{d} \hat{x}_k \nonumber,
\end{align}
where the case \( d = 1 \) is interpreted as having no integration, with the formal identification \( \int_{\mathbb{R}^{0}} f  \mathrm{d} \hat{x}_k = f \).

	In the mean time, for the   approximation equation of \eqref{qNLS}
	\begin{align*}
		\sqrt{-1}\phi_t+	\varepsilon\sqrt{-1}\Delta^2 \phi+\partial_i\left( g^{ij}\left(\phi,\overline{\phi}\right) \partial_j \phi\right)=F\left(\phi,\overline{\phi},\nabla\phi, \nabla \overline{\phi}\right),
	\end{align*}
we can use the same procedure to drive momentum balance equalities analogous to \eqref{consm2}:
\begin{align} \label{aconsm}
	&	-\partial_t \int_{\mathbb{R}^{d-1}}  \mathrm{Im}\left( \phi_{\alpha } \partial_k \overline{\phi}_\alpha \right) \mathrm{d} \hat{x}_k-\partial_{kk} \int_{\mathbb{R}^{d-1}} \mathrm{Re}\left( \phi_{\alpha }   g^{kj} \overline{\phi}_{\alpha j}\right) \mathrm{d} \hat{x}_k\nonumber\\
	&+2 \partial_k \int_{\mathbb{R}^{d-1}} \mathrm{Re}\left(\partial_k \phi_{\alpha }  g^{kj} \overline{\phi}_{\alpha j}\right) \mathrm{d} \hat{x}_k + \int_{\mathbb{R}^{d-1}}  \mathrm{Re}\left(\partial_{i} \phi_{\alpha }  \partial_k h^{ij} \partial_j \overline{\phi}_{\alpha }\right)  \mathrm{d} \hat{x}_k \\
	=&\,- \partial_k\int_{\mathbb{R}^{d-1}}   \mathrm{Re}\left( \overline{\phi }_{\alpha } \mathcal{N}\right)\mathrm{d} \hat{x}_k+2\int_{\mathbb{R}^{d-1}} \mathrm{Re} \left(\mathcal{N} \partial_k \overline{\phi}_\alpha\right)   \mathrm{d} \hat{x}_k,\nonumber\\
		&-\sum_{i \neq k} \varepsilon \partial_k \partial_k\partial_k  \int_{\mathbb{R}^{d-1}} \mathrm{Im}\left(\partial_{kk} \phi_{\alpha }  \overline{\phi}_\alpha\right)  \mathrm{d} \hat{x}_k+2 \varepsilon \partial_{k}\partial_{k}\int_{\mathbb{R}^{d-1}} \mathrm{Im}\left(\partial_{kj} \phi_{\alpha } \partial_j \overline{\phi}_\alpha\right) \mathrm{d} \hat{x}_k \nonumber\\
	&\,	+2 \varepsilon \partial_{k}\partial_{k}\int_{\mathbb{R}^{d-1}} \mathrm{Im}\left(\partial_{kk} \phi_{\alpha } \partial_k \overline{\phi}_\alpha\right) \mathrm{d} \hat{x}_k-2\varepsilon\mathrm{Im}\int_{\mathbb{R}^{d-1}} \left(\partial_{i}\partial_{jj} \phi_\alpha\partial_i\partial_k \overline{\phi}_\alpha\right)  \mathrm{d} \hat{x}_k,\nonumber
\end{align} 
since  the term $ 	\varepsilon\sqrt{-1}\Delta^2 \phi_\alpha $ can be computed just as the nonlinear term $ \mathcal{N} $:
\begin{align*}
&-	\varepsilon\mathrm{Re} \sqrt{-1}	\left(-\partial_k\left( \overline{\phi }_{\alpha }\Delta^2 \phi_\alpha \right)+2  \Delta^2 \phi_\alpha\partial_k \overline{\phi}_\alpha \right)\\
=&\,\varepsilon\mathrm{Im}	\left(-\partial_k \partial_i \left( \overline{\phi }_{\alpha }\partial_{i}\partial_{jj} \phi_\alpha \right)+\partial_k  \left( \partial_i\overline{\phi }_{\alpha }\partial_{i}\partial_{jj} \phi_\alpha \right)+2\partial_i \left(\partial_{i}\partial_{jj} \phi_\alpha\partial_k \overline{\phi}_\alpha\right)-2 \left(\partial_{i}\partial_{jj} \phi_\alpha\partial_i\partial_k \overline{\phi}_\alpha\right)\right)\\
%=&\,\varepsilon\mathrm{Im}	\left(-\partial_k \partial_i\partial_{j} \left( \overline{\phi }_{\alpha }\partial_{i}\partial_{j} \phi_\alpha \right)+\partial_k \partial_i \left( \partial_{j}\overline{\phi }_{\alpha }\partial_{i}\partial_{j} \phi_\alpha \right)+\partial_k  \partial_{j}\left( \partial_i\overline{\phi }_{\alpha }\partial_{i}\partial_{j} \phi_\alpha \right) \right)\\
%&\,+\varepsilon\mathrm{Im}	\left( 2 \partial_i \partial_{j} \left(\partial_{i}\partial_{j} \phi_\alpha\partial_k \overline{\phi}_\alpha\right)-2\partial_i \left(\partial_{i}\partial_{j} \phi_\alpha\partial_k \partial_{j}\overline{\phi}_\alpha\right)-2 \left(\partial_{i}\partial_{jj} \phi_\alpha\partial_i\partial_k \overline{\phi}_\alpha\right)\right)\\
=& \, \sum_{i \neq k \text{ or } j \neq k}\varepsilon\mathrm{Im}	\left(-\partial_k \partial_i\partial_{j} \left( \overline{\phi }_{\alpha }\partial_{i}\partial_{j} \phi_\alpha \right)+\partial_k \partial_i \left( \partial_{j}\overline{\phi }_{\alpha }\partial_{i}\partial_{j} \phi_\alpha \right)+\partial_k  \partial_{j}\left( \partial_i\overline{\phi }_{\alpha }\partial_{i}\partial_{j} \phi_\alpha \right) \right)\\
&\,+\sum_{i \neq k \text{ or } j \neq k}\varepsilon\mathrm{Im}	\left( 2 \partial_i \partial_{j} \left(\partial_{i}\partial_{j} \phi_\alpha\partial_k \overline{\phi}_\alpha\right)-2\partial_i \left(\partial_{i}\partial_{j} \phi_\alpha\partial_k \partial_{j}\overline{\phi}_\alpha\right)  \right)\\
&\,+\varepsilon\mathrm{Im}	\left(-\partial_k \partial_k\partial_{k} \left( \overline{\phi }_{\alpha }\partial_{k}\partial_{k} \phi_\alpha \right)+\partial_k \partial_k \left( \partial_{j}\overline{\phi }_{\alpha }\partial_{k}\partial_{j} \phi_\alpha \right)+\partial_k  \partial_{k}\left( \partial_i\overline{\phi }_{\alpha }\partial_{i}\partial_{k} \phi_\alpha \right) \right)\\
&\,+\varepsilon\mathrm{Im}	\left( 2 \partial_k \partial_{k} \left(\partial_{k}\partial_{k} \phi_\alpha\partial_k \overline{\phi}_\alpha\right)-2 \left(\partial_{i}\partial_{jj} \phi_\alpha\partial_i\partial_k \overline{\phi}_\alpha\right) \right)
\end{align*}
where $ \Delta^2=\sum_{i,j=1}^d \partial_{ii}\partial_{jj}. $

\section{Small data and	quadratic interactions}  \label{s3}
To make the proof more concise,  we first introduce two norms  for $\phi $:
\begin{align*}
	&X(t)=X_{s_1} (t) :=\sum_{k=1}^d \sum_{|\alpha|= s_1}  \int_0^t     \int_{\mathbb{R}^{d}} \left|\frac{\partial_k \phi_{\alpha }  } {\langle x_k\rangle}\right|^2  \mathrm{d} x \mathrm{d} s,  \nonumber\\
	&Y(t)=Y_{s_1} (t) :=	\| \phi\|_{H^{ s_1+\frac12 }(\mathbb{R}^d)}^2 + \sum_{k=1}^d\sum_{|\beta| \leq s_1-\frac32}\left\| \langle x_k\rangle^2 \phi_{\beta}\left(t,\cdot\right)\right\|_{L^2(\mathbb{R}^d)}^2,
\end{align*}	
where $ \langle x \rangle:=  1+|x| , x\in \mathbb{R}^d $, as defined in the Introduction \ref{s1}, and $ s_1 >\frac d2+\frac 52$ is a positive constant. In particular for scalar $ x_k\in \mathbb{R} $, there is 
\begin{align*}
	%		&\frac{\mathrm{d}}{ \mathrm{d} x}\left( \frac{1}{\langle x \rangle}\right)=-\frac{ \mathrm{sgn}\left(x\right)}{\langle x \rangle^2},\\
	&\frac{\mathrm{d}}{ \mathrm{d} x_k}\left( \frac{x_k}{\langle x_k \rangle}\right)=\frac{\langle x_k \rangle-x_k\cdot \mathrm{sgn}\left(x_k\right)}{\langle x_k \rangle^2}=\frac{1}{\langle x_k \rangle^2}.
\end{align*}

The following estimates concerning 
 $ g^{ij} $, $ F $  and their derivatives by  the  norms of $ \phi $, will be frequently used in the subsequent analysis. Recall that for quadratic nonlinearities \eqref{qgF}:
\begin{align*}
	g^{ij}(y,\bar{y})-g_0^{ij}=h^{ij}(y, \bar{y})=O\left(|y|\right), \quad|F(y, \bar{y},z,\bar{z})| = O\left(|y|^2+|z|^2\right) \text{   near   }(y, z)=(0,0).
\end{align*}
Using Taylor expansion and techniques analogous to those in Chapter 7 of \cite{li2017nonlinear}, we may neglect higher-order error terms to obtain 
\begin{align*}
	\left| F\left(\phi,\overline{\phi},\nabla\phi, \nabla \overline{\phi}\right)\right|&\sim |\phi|^2+ |\nabla\phi|^2+  \text{higher degree power terms} \sim |\phi|^2+|\nabla\phi|^2,\\
	\left| F_z\right|+\left| F_{\overline{z}}\right|& \sim |\phi|+|\nabla\phi|,\\
\left|	g^{ij} (\phi,\bar{\phi})\right|&\sim  g^{ij}_0+  |\phi| +\text{higher degree power terms} \\
	& \sim  g^{ij}_0+|\phi|,
\end{align*}
if $ |\phi|+|\nabla \phi| $ is small enough. 
It follows that
\begin{align} 
		\|g^{ij}\|_{L^\infty(\mathbb{R}^d)}\leq&\, 1+\|h^{ij}\|_{L^\infty(\mathbb{R}^d)}\nonumber\\
	\lesssim &\,1+\|\phi\|_{H^{s_1-\frac 52}(\mathbb{R}^d)},\nonumber\\	
		\left\|\partial_{i}g^{ij}\right\|_{L^\infty(\mathbb{R}^{d})}
		=&\,\left\|	\partial_{i} h^{ij}  \right\|_{L^\infty(\mathbb{R}^{d})}  
	\lesssim \|  \phi \|_{H^{s_1-\frac 32}(\mathbb{R}^{d})}, \label{qg}\\
	\|F_z\|_{L^\infty(\mathbb{R}^{d})} +\|F_{\overline{z}} \|_{L^\infty(\mathbb{R}^{d})}
	\lesssim &\,\|\phi \|_{H^{s_1-\frac 52}(\mathbb{R}^{d})}+\|\nabla\phi \|_{H^{s_1-\frac 52}(\mathbb{R}^{d})}\nonumber\\
	\lesssim &\,\|\phi \|_{H^{s_1-\frac 32}(\mathbb{R}^{d})} 	\label{qF}
\end{align}
for  $\frac{d}{2}+\frac 52<  s_1. $ Moreover,  $ x_i^2 F_z, x_i^2  \partial_k h^{ij} $ can be similarly  bounded  by the weighted norms of $ \phi $:
\begin{align}
	\left\|\langle x_i\rangle^2  F_z   \right\|_{L^\infty(\mathbb{R}^d)}&+\left\| \langle x_i\rangle^2 F_{\overline{z}}   \right\|_{L^\infty(\mathbb{R}^d)}+\left\| \langle x_i\rangle^2\partial_k h^{ij}   \right\|_{L^\infty(\mathbb{R}^d)} \nonumber\\
	\lesssim &\, \left\| \langle x_i\rangle^2\partial_k \phi   \right\|_{L^\infty(\mathbb{R}^d)}+ \left\| \langle x_i\rangle^2 \phi  \right\|_{L^\infty(\mathbb{R}^d)} \label{qFx2}\\
	\lesssim &\,   \left\| \langle x_i\rangle^2\langle \nabla \rangle \phi  \right\|_{H^{s_1-\frac32}(\mathbb{R}^d)} \nonumber\\
	\lesssim &\,\sum_{|\beta| \leq s_1-\frac 32}\left\| \langle x_i\rangle^2 \phi_{\beta}\right\|_{L^2(\mathbb{R}^d)} \nonumber\\
	\lesssim& \,Y(t)^{\frac 12}, \nonumber
\end{align}
where $ \langle \nabla \rangle =1+\nabla $ and $ \beta=(\beta_1,\dots,\beta_d) \in \mathbb{N}^d_0 $.

\subsection{Momentum type estimates} \label{s3.1}
In this subsection we derive the Momentum type estimates for $ X(t) $.  Multiplying the \eqref{consm2} by the weight $ \dfrac{x_k}{\langle x_k\rangle} $ and integrating it on $ [0,t]\times\mathbb{R} $ with respect to $t$ and  $ x_k $ respectively indicate
\begin{align} \label{qener_mc1}
	&-	\int_{\mathbb{R}}  \int_{\mathbb{R}^{d-1}} \frac{x_k}{\langle x_k\rangle}   \mathrm{Im}\left( \phi_{\alpha } \partial_k \overline{\phi}_\alpha \right)(t) \mathrm{d} \hat{x}_k \mathrm{d} x_k+	\left.\int_{\mathbb{R}}  \int_{\mathbb{R}^{d-1}} \frac{x_k}{\langle x_k\rangle}   \mathrm{Im}\left( \phi_{\alpha } \partial_k \overline{\phi}_\alpha \right) \mathrm{d} \hat{x}_k \mathrm{d} x_k\right|_{s=0}\nonumber\\
	&- \int_0^t \int_{\mathbb{R}}   	\frac{x_k}{\langle x_k\rangle}   \partial_{kk} \int_{\mathbb{R}^{d-1}} \mathrm{Re}\left( \phi_{\alpha }  g^{kj} \overline{\phi}_{\alpha j}\right) \mathrm{d} \hat{x}_k   \mathrm{d} x_k \mathrm{d} s\nonumber \\
	&+2  \int_0^t\int_{\mathbb{R}}    	\frac{x_k}{\langle x_k\rangle}   \partial_k \int_{\mathbb{R}^{d-1}} \mathrm{Re}\left(\partial_k \phi_{\alpha } g^{kj} \overline{\phi}_{\alpha j}\right) \mathrm{d} \hat{x}_k  \mathrm{d} x_k \mathrm{d} s \\
	&+ \int_0^t	\int_{\mathbb{R}}\frac{x_k}{\langle x_k\rangle}   \int_{\mathbb{R}^{d-1}}  \mathrm{Re}\left(\partial_{i} \phi_{\alpha }  \partial_k h^{ij} \partial_j \overline{\phi}_{\alpha }\right) \mathrm{d} \hat{x}_k  \mathrm{d} x_k \mathrm{d} s \nonumber\\
	=&\,- \int_0^t	\int_{\mathbb{R}}\frac{x_k}{\langle x_k\rangle}   \partial_k\int_{\mathbb{R}^{d-1}}   \mathrm{Re}\left( \overline{\phi }_{\alpha } \mathcal{N}\right) \mathrm{d} \hat{x}_k  \mathrm{d} x_k \mathrm{d} s+2\int_0^t	\int_{\mathbb{R}}\frac{x_k}{\langle x_k\rangle}  \int_{\mathbb{R}^{d-1}}  \mathrm{Re} \left(\mathcal{N} \partial_k \overline{\phi}_\alpha\right)   \mathrm{d} \hat{x}_k  \mathrm{d} x_k \mathrm{d} s\nonumber.
	%	=&\,\int_0^t	\int_{\mathbb{R}}\int_{\mathbb{R}^{d-1}}   \mathrm{Re}\left( \overline{\phi }_{\alpha } \mathcal{N}\right) \mathrm{d} \hat{x}  \mathrm{d} x_k \mathrm{d} s+2\int_0^t	\int_{\mathbb{R}}\frac{x_k}{\langle x_k\rangle}  \int_{\mathbb{R}^{d-1}}  \mathrm{Re} \left(\mathcal{N} \partial_k \overline{\phi}_\alpha\right)   \mathrm{d} \hat{x}  \mathrm{d} x_k \mathrm{d} s\nonumber.
\end{align}
As in the definition of $ X(t) $, we fix the absolute value of $ \alpha $ as $|\alpha| = s_1$ in this subsection.
 One can derive weighted space-time norm  $ X(t) $ from
\begin{align} \label{qener_mc2}
	&-\int_0^t\int_{\mathbb{R}}    	\frac{x_k}{\langle x_k\rangle}   \partial_k \int_{\mathbb{R}^{d-1}} \mathrm{Re}\left(\partial_k \phi_{\alpha } g^{kj} \overline{\phi}_{\alpha j}\right) \mathrm{d} \hat{x}_k  \mathrm{d} x_k \mathrm{d} s\nonumber\\
	=&\,g^{kk}_0\int_0^t\int_{\mathbb{R}}    	\frac{x_k}{\langle x_k\rangle}   \partial_k \int_{\mathbb{R}^{d-1}} \mathrm{Re}\left(\partial_k \phi_{\alpha }   \overline{\phi}_{\alpha k}\right) \mathrm{d} \hat{x}_k  \mathrm{d} x_k \mathrm{d} s\\
	&+\int_0^t\int_{\mathbb{R}}    	\frac{x_k}{\langle x_k\rangle}   \partial_k \int_{\mathbb{R}^{d-1}} \mathrm{Re}\left(\partial_k \phi_{\alpha } h^{kj} \overline{\phi}_{\alpha j}\right) \mathrm{d} \hat{x}_k  \mathrm{d} x_k \mathrm{d} s,\nonumber
\end{align}
where 
 $ g^{ij}=g^{ij}_0+h^{ij}, g^{ij}_0=0 $, if $ i\neq j $.  A direct computation yields
\begin{align} \label{qener_mc2_1}
	&-g^{kk}_0\int_0^t\int_{\mathbb{R}}  	\frac{x_k}{\langle x_k\rangle}   \partial_k \int_{\mathbb{R}^{d-1}} \mathrm{Re}\left(\partial_k \phi_{\alpha }  \overline{\phi}_{\alpha k}\right) \mathrm{d} \hat{x}_k  \mathrm{d} x_k \mathrm{d} \nonumber\\
	%	 =&\,\int_0^t\int_{\mathbb{R}}    \left(	\frac{1}{\langle x_k\rangle}-\frac{ x_k^2}{\langle x_k\rangle^3}\right)    \int_{\mathbb{R}^{d-1}} \left|\partial_k \phi_{\alpha }  \right|^2 \mathrm{d} \hat{x}_k  \mathrm{d} x_k \mathrm{d} s\\
	=&\, g^{kk}_0\int_0^t\int_{\mathbb{R}}      \frac{ 1 }{\langle x_k\rangle^2}   \int_{\mathbb{R}^{d-1}} \left|\partial_k \phi_{\alpha }  \right|^2\mathrm{d} \hat{x}_k  \mathrm{d} x_k \mathrm{d} s\\
%	=&\, g^{kk}_0\int_0^t\int_{\mathbb{R}}        \int_{\mathbb{R}^{d-1}} \frac{\left|\partial_k \phi_{\alpha }  \right|^2} {\langle x_k\rangle^2}\mathrm{d} \hat{x}_k  \mathrm{d} x_k \mathrm{d} s\nonumber\\
	=&\, g^{kk}_0\int_0^t     \int_{\mathbb{R}^{d}} \left|\frac{\partial_k \phi_{\alpha }  } {\langle x_k\rangle}\right|^2  \mathrm{d} x \mathrm{d} s,\nonumber
%	&\geq  \int_0^t     \int_{\mathbb{R}^{d}} \left|\frac{\partial_k \phi_{\alpha }  } {\langle x\rangle}\right|^2  \mathrm{d} x \mathrm{d} s= \int_0^t\int_{\mathbb{R}^d}       \frac{\left|\partial_k \phi_{\alpha }  \right|^2} {\langle x\rangle^2}  \mathrm{d} x \mathrm{d} s,
\end{align}
which is the part of $ X(t). $
Meanwhile, the second term on the right-hand side of \eqref{qener_mc2} can be bounded as follows:
\begin{align*} 
	&\int_0^t\int_{\mathbb{R}}    	\frac{x_k}{\langle x_k\rangle}   \partial_k \int_{\mathbb{R}^{d-1}} \mathrm{Re}\left(\partial_k \phi_{\alpha } h^{kj} \overline{\phi}_{\alpha j}\right) \mathrm{d} \hat{x}_k  \mathrm{d} x_k \mathrm{d} s\nonumber\\
	=&\,-\int_0^t	\int_{\mathbb{R}}\frac{1}{\langle x_k\rangle^2}   \int_{\mathbb{R}^{d-1}}  \mathrm{Re}\left(\partial_{i} \phi_{\alpha }   h^{ij} \partial_j \overline{\phi}_{\alpha }\right) \mathrm{d} \hat{x}_k  \mathrm{d} x_k \mathrm{d} s\\
	\lesssim &\,\left\|  h^{ij}  \langle x_i\rangle^2 \right\|_{L^\infty(\mathbb{R}^d)}^{\frac{1}{2}}  \left\| h^{ij}  \langle x_j\rangle^2 \right\|_{L^\infty(\mathbb{R}^d)}^{\frac{1}{2}} \left(\int_0^t     \int_{\mathbb{R}^{d}} \left|\frac{\partial_i \phi_{\alpha }  } {\langle x_i\rangle}\right|^2  \mathrm{d} x \mathrm{d} s\right)^{\frac 12}   \left(\int_0^t     \int_{\mathbb{R}^{d}} \left|\frac{\partial_j \phi_{\alpha }  } {\langle x_j\rangle}\right|^2  \mathrm{d} x \mathrm{d} s\right)^{\frac 12}\nonumber\\
	\lesssim &\,\sum_{i=1}^d\sum_{|\beta| \leq s_1-\frac 32}\left\| \langle x_i\rangle^2 \phi_{\beta}\right\|_{L^2(\mathbb{R}^d)} \cdot X(t)\nonumber\\
	\lesssim &\,Y(t)^{\frac 12} X(t)\\
	\lesssim &\,X(t)^{\frac 32}+Y(t)^{\frac 32}.
\end{align*}

For the first two terms  in \eqref{qener_mc1} that do not involve time integration,  we estimate
\begin{align}  \label{qener_mc4}
	&\left|\int_{\mathbb{R}}  \int_{\mathbb{R}^{d-1}}  \frac{x_k}{\langle x_k\rangle}  \mathrm{Im}\left( \phi_{\alpha } \partial_k \overline{\phi}_\alpha \right) \mathrm{d} \hat{x}_k \mathrm{d} x_k\right| \nonumber\\
	=&\,\left|\mathrm{Im}\int_{\mathbb{R}^d}   D_k^{\frac 12}\left( \frac{x_k}{\langle x_k\rangle}    \phi_{\alpha } \right)D_k^{-\frac 12}\left(\partial_k \overline{\phi}_\alpha \right) \mathrm{d} x\right| \\
	%	=&\, \int_{\mathbb{R}^d} \mathrm{Im}\left(  \phi_{\alpha } \partial_k  \left( x_k\overline{\phi}_\alpha  \right)\right) \mathrm{d} x \\
	\lesssim &\,\left\|  \frac{x_k}{\langle x_k\rangle}   D_k^{\frac 12}    \phi_\alpha \right\|_{L^2} \left\|   D_k^{-1} \partial_k \left(D_k^{\frac 12}\phi_{\alpha}\right)\right\|_{L^2}  + \left\|  \left[D_k^{\frac 12},\frac{x_k}{\langle x_k\rangle}   \right]   \phi_\alpha \right\|_{L^2} \left\|   D_k^{-1} \partial_k \left(D_k^{\frac 12}\phi_{\alpha}\right)\right\|_{L^2} \nonumber\\
	\lesssim &\, \left\|   D_k^{\frac 12}\phi_{\alpha}\right\|_{L^2}^2+  \left\|   \phi_{\alpha}\right\|_{L^2}\left\|   D_k^{\frac 12}\phi_{\alpha}\right\|_{L^2}\nonumber\\
	\lesssim &\, \left\|\phi\right\|_{H^{s_1+\frac 12}(\mathbb{R}^d)}^2\nonumber\\ \lesssim&\, Y(t), \nonumber
\end{align}
where the commutator term is estimated as
\begin{align} \label{comxk}
	\left\|  \left[D_k^{\frac 12},\frac{x_k}{\langle x_k\rangle}   \right]   \phi_\alpha \right\|_{L^2(\mathbb{R}^d)} \lesssim &\, \left\|  D_k^{\frac 12}\left(\frac{x_k}{\langle x_k\rangle}  \right)  \right\|_{\mathrm{BMO}(\mathbb{R}_{x_k})}  \left\|     \phi_\alpha \right\|_{L^2(\mathbb{R}^d)} \\
	\lesssim &\,\left\|  \frac{x_k}{\langle x_k\rangle}   \right\|_{\mathrm{BMO}(\mathbb{R}_{x_k})}^{\frac 12 }  \left\| \frac{\mathrm{d}}{\mathrm{d} x_k}\left( \frac{x_k}{\langle x_k\rangle}  \right) \right\|_{\mathrm{BMO}(\mathbb{R}_{x_k})}^{\frac 12 }  \left\|     \phi_\alpha \right\|_{L^2(\mathbb{R}^d)} \nonumber\\
	\lesssim &\, \left\|     \phi_\alpha \right\|_{L^2(\mathbb{R}^d)}.  \nonumber
\end{align}
This relies on the following interpolation lemma, whose proof is provided in Appendix \ref{appa} (see also Lemma A.10 in \cite{tao2006nonlinear}).
\begin{lem} \label{leminter1}.
	Let $ f, \nabla f \in \mathrm{BMO}(\mathbb{R}^d ) $, then
	\begin{align*} 
		\left\| D^{\frac12} f \right\|_{\mathrm{BMO}(\mathbb{R}^d)} \lesssim 	\left\|   f \right\|_{\mathrm{BMO}(\mathbb{R}^d)}^{\frac 12 }  	\left\| \nabla f \right\|_{\mathrm{BMO}(\mathbb{R}^d)} ^{\frac 12 }.
	\end{align*}
\end{lem}

Integration by parts shows
\begin{align*}  
	-&\,\int_0^t \int_{\mathbb{R}}   	\frac{x_k}{\langle x_k\rangle}   \partial_{kk} \int_{\mathbb{R}^{d-1}} \mathrm{Re}\left( \phi_{\alpha }  g^{kj} \overline{\phi}_{\alpha j}\right) \mathrm{d} \hat{x}_k   \mathrm{d} x_k \mathrm{d} s\nonumber\\
	=&\,\int_0^t \int_{\mathbb{R}}   	\frac{1}{\langle x_k\rangle^2}   \partial_{k} \int_{\mathbb{R}^{d-1}} \mathrm{Re}\left( \phi_{\alpha }  g^{kj} \overline{\phi}_{\alpha j}\right) \mathrm{d} \hat{x}_k   \mathrm{d} x_k \mathrm{d} s\\
	=&\,2\int_0^t \int_{\mathbb{R}}   	\frac{\mathrm{sgn}(x_k)}{\langle x_k\rangle^3}  \int_{\mathbb{R}^{d-1}} \mathrm{Re}\left( \phi_{\alpha }  g^{kj} \overline{\phi}_{\alpha j}\right) \mathrm{d} \hat{x}_k   \mathrm{d} x_k \mathrm{d} s\nonumber\\
	=&\, 2g^{kk}_0\int_0^t \int_{\mathbb{R}}   	\frac{\mathrm{sgn}(x_k)}{\langle x_k\rangle^3}   \int_{\mathbb{R}^{d-1}} \mathrm{Re}\left( \phi_{\alpha }  \overline{\phi}_{\alpha k}\right) \mathrm{d} \hat{x}_k   \mathrm{d} x_k \mathrm{d} s+2\int_0^t \int_{\mathbb{R}}   	\frac{\mathrm{sgn}(x_k)}{\langle x_k\rangle^3} \int_{\mathbb{R}^{d-1}} \mathrm{Re}\left( \phi_{\alpha }   h^{kj} \overline{\phi}_{\alpha j}\right) \mathrm{d} \hat{x}_k   \mathrm{d} x_k \mathrm{d} s\nonumber
\end{align*}
and it is straightforward to verify that
\begin{align*} 
	&\left|	g^{kk}_0\int_0^t \int_{\mathbb{R}}   		\frac{\mathrm{sgn}(x_k)}{\langle x_k\rangle^3} \int_{\mathbb{R}^{d-1}} \mathrm{Re}\left( \phi_{\alpha }   h^{kj} \overline{\phi}_{\alpha j}\right) \mathrm{d} \hat{x}_k   \mathrm{d} x_k \mathrm{d} s\right|\nonumber\\
	\lesssim &\, \left\|   h^{kj}  \langle x_j\rangle \right\|_{L^\infty(\mathbb{R}^d)}  \left(\int_0^t     \int_{\mathbb{R}^{d}} \left|   \phi_{\alpha }   \right|^2  \mathrm{d} x \mathrm{d} s\right)^{\frac 12}   \left(\int_0^t     \int_{\mathbb{R}^{d}} \left|\frac{\partial_j \phi_{\alpha }  } {\langle x_j\rangle}\right|^2  \mathrm{d} x \mathrm{d} s\right)^{\frac 12}\\
		\lesssim &\,t^{\frac12}\cdot\sum_{i=1}^d\sum_{|\beta| \leq s_1-\frac 32}\left\| \langle x_i\rangle \phi_{\beta}\right\|_{L^2(\mathbb{R}^d)} \cdot \left\|     \phi  \right\|_{H^{s_1}(\mathbb{R}^d)} \cdot X(t)^{\frac 12} \\
		\lesssim &\, t^{\frac12}Y(t) X(t)^{\frac 12} \\
		\lesssim &\, X(t)^{\frac 32}+t^{\frac32}Y(t)^{\frac 32}
\end{align*}
and
\begin{align}
	&\left|\int_0^t \int_{\mathbb{R}}   	\frac{\mathrm{sgn}(x_k)}{\langle x_k\rangle^3}   \int_{\mathbb{R}^{d-1}} \mathrm{Re}\left( \phi_{\alpha }  \overline{\phi}_{\alpha k }\right) \mathrm{d} \hat{x}_k   \mathrm{d} x_k \mathrm{d} s\right|\nonumber\\
	\leq &\,\left(\int_0^t     \int_{\mathbb{R}^{d}} \left|   \phi_{\alpha }   \right|^2  \mathrm{d} x \mathrm{d} s\right)^{\frac 12}   \left(\int_0^t     \int_{\mathbb{R}^{d}} \left|\frac{\partial_k \phi_{\alpha }  } {\langle x_k\rangle}\right|^2  \mathrm{d} x \mathrm{d} s\right)^{\frac 12}\label{X_absorbed}\\
	\leq &\, Ct^{\frac12}\cdot  \left\| \phi     \right\|_{H^{s_1}(\mathbb{R}^d)} \cdot X(t)^{\frac 12}\nonumber\\
	\leq &\, Ct^{\frac12}Y(t)^{\frac 12}  X(t)^{\frac 12} \nonumber\\
	\leq &\, \frac 1{100d} X(t)+ C t  Y(t). \nonumber
\end{align}

On the other hand, the nonlinear terms can be controlled by   \eqref{qFx2} and the norm $ X(t) $. Specifically, we have
\begin{align}  \label{qener_mc5}
	&\left|\int_0^t	\int_{\mathbb{R}}\frac{x_k}{\langle x_k\rangle}   \int_{\mathbb{R}^{d-1}}  \mathrm{Re}\left(\partial_{i} \phi_{\alpha }  \partial_k h^{ij} \partial_j \overline{\phi}_{\alpha }\right) \mathrm{d} \hat{x}_k  \mathrm{d} x_k \mathrm{d} s\right|\\
	\lesssim &\,\left\| \partial_k h^{ij}  \langle x_i\rangle^2 \right\|_{L^\infty(\mathbb{R}^d)}^{\frac{1}{2}}  \left\| \partial_k h^{ij}  \langle x_j\rangle^2 \right\|_{L^\infty(\mathbb{R}^d)}^{\frac{1}{2}} \left(\int_0^t     \int_{\mathbb{R}^{d}} \left|\frac{\partial_i \phi_{\alpha }  } {\langle x_i\rangle}\right|^2  \mathrm{d} x \mathrm{d} s\right)^{\frac 12}   \left(\int_0^t     \int_{\mathbb{R}^{d}} \left|\frac{\partial_j \phi_{\alpha }  } {\langle x_j\rangle}\right|^2  \mathrm{d} x \mathrm{d} s\right)^{\frac 12}\nonumber\\
		\lesssim &\, Y(t)^{\frac 12} X(t) \nonumber\\
	\lesssim &\,  X(t)^{\frac 32}+Y(t)^{\frac 32}.\nonumber
\end{align}
Now consider the term
\begin{align}  \label{qener_mc60}
	\int_0^t	\int_{\mathbb{R}}\frac{x_k}{\langle x_k\rangle}  \int_{\mathbb{R}^{d-1}}  \mathrm{Re} \left(\mathcal{N} \partial_k \overline{\phi}_\alpha\right)   \mathrm{d} \hat{x}_k  \mathrm{d} x_k \mathrm{d} s.
\end{align}
The part involving $ F_z\nabla \phi_\alpha+F_{\overline{z}}\nabla \overline{\phi}_\alpha+\partial h^{ij}\partial_{ij }  \phi_{\alpha -1}$ can be handled similarly to \eqref{qener_mc5}. 

The part involving $ G $ is estimated in a manner analogous to \eqref{ccons3_c4}:
\begin{align*} 
		&\left|\int_0^t	\int_{\mathbb{R}}\frac{x_k}{\langle x_k\rangle}  \int_{\mathbb{R}^{d-1}}  \mathrm{Re} \left(G \partial_k \overline{\phi}_\alpha\right)   \mathrm{d} \hat{x}_k  \mathrm{d} x_k \mathrm{d} s\right|\nonumber\\
		=&\, \left|\mathrm{Re} \int_0^t	\int_{\mathbb{R}^d}  D_k^{\frac 12} \left(\frac{x_k}{\langle x_k\rangle} G\right)   D_k^{-\frac 12}\partial_k \overline{\phi}_\alpha      \mathrm{d} x \mathrm{d} s\right|\\
		\lesssim &\,t\cdot \left\|  \frac{x_k}{\langle x_k\rangle}   D_k^{\frac 12}   G \right\|_{L^2} \left\|   D_k^{\frac 12}\phi_{\alpha} \right\|_{L^2}  + t\cdot \left\|  \left[D_k^{\frac 12},\frac{x_k}{\langle x_k\rangle}   \right]   G \right\|_{L^2} \left\|     D_k^{\frac 12}\phi_{\alpha} \right\|_{L^2}\nonumber\\
		\lesssim &\, t\cdot\left( \left\|    G \right\|_{L^2}+\left\|    D_k^{\frac 12}   G \right\|_{L^2}\right)\|\phi\|_{H^{s_1+\frac 12}(\mathbb{R}^{d})},\nonumber
\end{align*}
where 
  \eqref{comxk} is used.  In view of the asymptotic relation $ |F| \sim |\phi|^2+|\nabla \phi|^2 $,     the most challenging terms in  $ \partial^\alpha F $ are those of the form  $ \partial^\alpha |\nabla \phi|^2 $, such as $ \partial^\alpha(\partial_{l} \phi\partial_{j} \phi) $ for some  integer $1\leq  l,j \leq d  $. For simplicity, we restrict our attention to estimating  $ \partial^\alpha(\partial_{l} \phi\partial_{j} \phi ) $ in $ G $.  Applying the Leibniz rule, we write
\begin{align*} 
	\partial^\alpha(\partial_{l} \phi\partial_{j} \phi )=\sum_{\alpha^1,\alpha^2} C_{\alpha^1\alpha^2}  \partial^{\alpha^1}  \partial_{l}\phi \cdot   \partial^{\alpha^2}\partial_{j} \phi,  
\end{align*}
where $ C_{\alpha^1\alpha^2}  $ are constants and  $ \alpha^i \in \mathbb{R}^{d} $  are multi-indices satisfying $ |\alpha^1|+|\alpha^2|=|\alpha| $, $ |\alpha^i|\leq |\alpha|-1,i=1,2 $.  
Using the Kato-Ponce inequality \eqref{nonlinear1}, Hölder's inequality and the fractional Sobolev embedding, we obtain
\begin{align} \label{qG}
&\max \left\{	\left\| D_k^{\frac 12} 	\partial^\alpha(\partial_{l} \phi\partial_{j} \phi ) \right\|_{L^2(\mathbb{R}^d)},\left\|     G \right\|_{L^2}, \left\|  D_k^{\frac 12}   G \right\|_{L^2} \right \}\\
	\lesssim &\,\left\|  \phi\right\|_{W^{p_1,|\alpha^1|+\frac32}(\mathbb{R}^d)} \left\|  \phi\right\|_{W^{p_2,|\alpha^2|+1}(\mathbb{R}^d)} \nonumber\\
	\lesssim &\,\|\phi\|_{H^{|\alpha|+\frac12}}^2\nonumber\\
	\lesssim &\, Y(t) \nonumber
\end{align}
provided the exponents satisfy
\begin{align}   
	&	\sum_{i=1}^2 \frac{1}{p_i}=\frac 12, \qquad
	\frac{1}{p_1}\geq \frac 12-\frac{|\alpha|-|\alpha^1|-\frac12}{d}, \nonumber\\
	&\frac{1}{p_2}\geq \frac 12-\frac{|\alpha|-|\alpha^2|-1}{d},\nonumber
\end{align}
which in turn requires $ s_1=|\alpha|\geq \frac{d}{2}+\frac32. $

Finally, the last remaining nonlinear term can be treated similarly to  \eqref{qener_mc60}:
\begin{align*}
	&\left|\int_0^t	\int_{\mathbb{R}}\frac{x_k}{\langle x_k\rangle}   \partial_k\int_{\mathbb{R}^{d-1}}   \mathrm{Re}\left( \overline{\phi }_{\alpha } \mathcal{N}\right) \mathrm{d} \hat{x}_k  \mathrm{d} x_k \mathrm{d} s\right|\\
	=&\,\left| -\int_0^t	\int_{\mathbb{R}}\frac{1}{\langle x_k\rangle^2}    \int_{\mathbb{R}^{d-1}}   \mathrm{Re}\left( \overline{\phi }_{\alpha } \mathcal{N}\right) \mathrm{d} \hat{x}_k  \mathrm{d} x_k \mathrm{d} s\right|\nonumber\\
	\lesssim &\,\sum_{i=1}^d\sum_{|\beta| \leq s_1-\frac12}\left\| \langle x_i\rangle \phi_{\beta}\right\|_{L^2(\mathbb{R}^d)} \cdot X(t)+ t\cdot\|\phi\|^3_{H^{s_1+\frac 12}(\mathbb{R}^{d})}\nonumber\\
	\lesssim & \, X(t)^{\frac 32}+ (1+t^{\frac32}) Y(t)^{\frac 32}.
\end{align*}

As  $  g^{kk}_0  $ equal to $ 1 $ or $ -1 $, $ k=1,\dots ,d $,  we can take the sign in \eqref{qener_mc1} such that
\begin{align*}
 g^{kk}_0\int_0^t     \int_{\mathbb{R}^{d}} \left|\frac{\partial_k \phi_{\alpha }  } {\langle x_k\rangle}\right|^2  \mathrm{d} x \mathrm{d} s
\end{align*}
in \eqref{qener_mc2_1} is positive. Then 
taking sum of \eqref{qener_mc1} over $ k $ from $ 1 $ to $ d $ and collecting  all the term above, we can obtain
\begin{align}   \label{qenerM}
	X(t) &=\sum_{k=1}^d \sum_{|\alpha|= s_1}  \int_0^t     \int_{\mathbb{R}^{d}} \left|\frac{\partial_k \phi_{\alpha }  } {\langle x_k\rangle}\right|^2  \mathrm{d} x \mathrm{d} s \\
%	\lesssim &\, \|\phi_0\|^2_{H^{s_1+\frac 12}(\mathbb{R}^{d})}+  \|\phi\|^2_{H^{s_1+\frac 12}(\mathbb{R}^{d})}+t^{\frac12}\cdot  \left\|     \phi \right\|_{H^{s_1}(\mathbb{R}^d)} \cdot X(t)^{\frac 12}+t\cdot\|\phi\|^3_{H^{s_1+\frac 12}(\mathbb{R}^{d})}\nonumber\\
%	&+\sum_{i=1}^d\sum_{|\beta| \leq s_1-\frac 32}\left\| \langle x_i\rangle^2 \phi_{\beta}\right\|_{L^2(\mathbb{R}^d)} \cdot X(t)+t^{\frac12}\cdot\sum_{i=1}^d\sum_{|\beta| \leq s_1-\frac12}\left\| \langle x_i\rangle \phi_{\beta}\right\|_{L^2(\mathbb{R}^d)} \cdot \left\|     \phi \right\|_{H^{s_1}(\mathbb{R}^d)} \cdot X(t)^{\frac 12}.  \nonumber \\
&\lesssim (1+t)Y(t)+Y(0) + (1+t+t^{\frac 32}) Y(t)^{\frac 32} + X(t)^{\frac 32}, \nonumber
\end{align}
where  the $ \frac 1{100d} X(t) $ in \eqref{X_absorbed}  has been absorbed by the left hand side of \eqref{qenerM}.

\subsection{Energy  estimates and  proof of the main result} \label{s3.2}

Our task now is to  derive the energy  estimates.	
Multiplying \eqref{qNLS1}, namely 
\begin{align*}
	\sqrt{-1}\phi_{\alpha t}+\partial_i\left( g^{ij} \partial_j \phi_{\alpha }\right)=F_z\nabla \phi_\alpha+F_{\overline{z}}\nabla \overline{\phi}_\alpha+\partial g^{ij}\partial_{ij } \phi_{\alpha -1}+G
\end{align*}
by
$ \overline{\phi}_{\alpha } $ and taking the imaginary part verifies
\begin{align*}  
	&\frac 12\frac{\partial}{\partial  t}\left(  \left| \phi_{\alpha } \right|^2 \right)=\mathrm{Re}\left(   \partial_t \phi_{\alpha }  \overline{\phi}_{\alpha } \right)	\nonumber\\
	=&\,-  \mathrm{Im}\partial_i\left(  g^{ij} \partial_j \phi_{\alpha } \overline{\phi}_{\alpha }\right)+\mathrm{Im} \left(  \left(F_z\nabla \phi_{\alpha}+F_{\overline{z}}\nabla \overline{\phi}_\alpha+\partial h^{ij}\partial_{ij } \phi_{\alpha -1}+G\right)\overline{\phi}_{\alpha }\right),
\end{align*}
which integrating on $  [0,t]\times\mathbb{R}^d  $  yields
 \begin{align*}  
 	  \| \phi_\alpha\|_{L^{ 2 }(\mathbb{R}^d)}^2\lesssim &\, 	 \| \partial_\alpha\phi_0\|_{L^{ 2}(\mathbb{R}^d)}^2\nonumber\\
 	 &+ \left|\int_0^t\int_{\mathbb{R}^d} \left| \mathrm{Im} \left(  \left(F_z\nabla \phi_\alpha+F_{\overline{z}}\nabla \overline{\phi}_\alpha+\partial g^{ij}\partial_{ij } \phi_{\alpha -1}+G\right)\overline{\phi}_{\alpha }\right)\right| \mathrm{d} x\mathrm{d} s \right|.
 \end{align*}
If $ |\alpha|<s_1 $,  we may estimate the right hand side in a manner similar to tbose for \eqref{qG}, giving
\begin{align}\label{qener1}
 \sum_{|\alpha|< s_1} \| \phi_\alpha\|_{L^{ 2 }(\mathbb{R}^d)}^2 \lesssim&\, \sum_{|\alpha|< s_1} \| \partial_\alpha\phi_0\|_{L^{ 2}(\mathbb{R}^d)}^2+t\cdot\|\phi\|^3_{H^{s_1+\frac 12}(\mathbb{R}^{d})}\\
 \lesssim &\, Y(0) +t Y(t)^{\frac 32}. \nonumber
\end{align}
For the case  $ |\alpha|=s_1 $, we have
\begin{align} \label{qener2}
 \| \phi_\alpha\|_{L^{ 2 }(\mathbb{R}^d)}^2	\lesssim &\,	\| \phi_0\|_{H^{s_1 }(\mathbb{R}^d)}^2+\sum_{|\beta| \leq s_1-\frac 32}\left\| \langle x_j\rangle^2 \phi_{\beta}\right\|_{L^2(\mathbb{R}^d)} \cdot X(t)+t\cdot\|\phi\|^3_{H^{s_1+\frac 12}(\mathbb{R}^{d})}, \\
	\vspace{-25pt}
	\lesssim &\, \| \phi_0\|_{H^{s_1 }(\mathbb{R}^d)}^2+  (1+t) Y(t)^{\frac 32} +  X(t)^{\frac 32},\nonumber
\end{align}
where  the terms involving $ F_z\nabla \phi_\alpha+F_{\overline{z}}\nabla \overline{\phi}_\alpha+\partial h^{ij}\partial_{ij }  \phi_{\alpha -1}$ are estimated as
\begin{align*} 
	&	\int_0^t \int_{\mathbb{R}^d}      F_z\nabla \phi_\alpha  \overline{\phi}_{\alpha } \mathrm{d} x \mathrm{d} s\\
	\lesssim &\,\sum_{j=1}^d \left\| \langle x_j\rangle\left(|\phi|+| \nabla\phi|\right)  \right\|_{L^\infty(\mathbb{R}^d)}  \left(\int_0^t     \int_{\mathbb{R}^{d}} \left|   \phi_{\alpha }   \right|^2  \mathrm{d} x \mathrm{d} s\right)^{\frac 12}   \left(\int_0^t     \int_{\mathbb{R}^{d}} \left|\frac{\partial_j \phi_{\alpha }  } {\langle x_j\rangle}\right|^2  \mathrm{d} x \mathrm{d} s\right)^{\frac 12}\nonumber\\
	\lesssim &\,\sum_{j=1}^d\sum_{|\beta| \leq s_1-\frac 32}\left\| \langle x_j\rangle^2 \phi_{\beta}\right\|_{L^2(\mathbb{R}^d)} \cdot X(t),\nonumber
\end{align*}
while
the term   involving $ G $ can be bounded in the same way as \eqref{qG}.

Next we   estimate  $ \left\| D_k^{\frac 12} \overline{\phi}_{\alpha } \right\|_{L^2} $ for $\frac d2+\frac 52< s_1=|\alpha| $. Multiplying \eqref{qNLS1} by
$D_k \overline{\phi}_{\alpha } $ and taking the imaginary part yields	
\begin{align}  \label{qecons2_1}
	&\mathrm{Im} \left(  \left(F_z\nabla \phi_\alpha+F_{\overline{z}}\nabla \overline{\phi}_\alpha+\partial g^{ij}\partial_{ij } \phi_{\alpha -1}+G\right)D_k\overline{\phi}_{\alpha } \right)\nonumber\\
	=&\,\mathrm{Re}\left( \partial_t\phi_{\alpha } D_k\overline{\phi}_{\alpha }  \right)+ \mathrm{Im} \left(  \partial_i\left(   g^{ij} \partial_j \phi_{\alpha }\right)D_k \overline{\phi}_{\alpha }\right) \nonumber\\
	%	=&\, \frac 12\frac{d}{d t}\left(  \left| \phi_{\alpha } \right|^2 \right)+   \partial_i \mathrm{Im}\left(  g^{ij} \partial_j \phi_{\alpha } \overline{\phi}_{\alpha }\right)\nonumber
	=&\, \mathrm{Re}\left( \partial_t\phi_{\alpha } D_k\overline{\phi}_{\alpha }  \right)+ \mathrm{Im} \partial_i\left(   g^{ij} \partial_j \phi_{\alpha }D_k \overline{\phi}_{\alpha }\right)- \mathrm{Im} \left(   \partial_j \phi_{\alpha } g^{ij} D_k \overline{\phi}_{\alpha i }\right) 
	%	=&\, \mathrm{Re}\left( \partial_t\phi_{\alpha } D_k \overline{\phi}_{\alpha }  \right)+ \mathrm{Im} \partial_i\left[  \left(   g^{ij} \partial_j \phi_{\alpha }\right) D_k \overline{\phi}_{\alpha }\right]\\
	%	-&\, \mathrm{Im} \left(   \partial_j \phi_{\alpha } \left[ g^{ij}  D_k^{\frac12}, D_k^{\frac12} \right] \overline{\phi}_{\alpha i }\right)-\mathrm{Im} \left(   \partial_j \phi_{\alpha } D_k^{\frac12}( g^{ij} D_k^{\frac12}  \overline{\phi}_{\alpha i })\right)\nonumber
\end{align}	
and 
\begin{align*} 
	&\mathrm{Im} \left(   \partial_j \phi_{\alpha }  g^{ij}  D_k  \overline{\phi}_{\alpha i }\right)\nonumber\\
	=&\,g^{ii}_0\sum_{i=1}^d\mathrm{Im} \left(   \phi_{\alpha i}     D_k  \overline{\phi}_{\alpha i }\right)+\mathrm{Im} \left(   \partial_j \phi_{\alpha }  h^{ij}  D_k  \overline{\phi}_{\alpha i }\right)
\end{align*}
in view of $ g^{ij}=g^{ij}_0+h^{ij} $, where $ g^{ij}_0 $ is defined as in \eqref{g0}.
Clearly,   one has
\begin{align*} 
	&g^{ii}_0 \int_{\mathbb{R}^d}	\mathrm{Im} \left(  \phi_{\alpha i}     D_k  \overline{\phi}_{\alpha i }\right) \mathrm{d} x = g^{ii}_0 \int_{\mathbb{R}^d}\mathrm{Im} \left( D_k^{\frac 12}\overline{\phi}_{\alpha i }  D_k^{\frac 12} \phi_{\alpha i}\right) \mathrm{d} x =0	,
\end{align*}
and
\begin{align*} 
	&	\mathrm{Im} \left(   \partial_j \phi_{\alpha }  h^{ij}  D_k  \overline{\phi}_{\alpha i }\right)\nonumber\\
	=&\,\mathrm{Im} \left(   \frac{\partial_j \phi_{\alpha }  } {\langle x_j\rangle}  \langle x_j\rangle \langle x_i\rangle h^{ij}  D_k \left(\frac{\overline{\phi}_{\alpha i } } {\langle x_i\rangle}\right)\right)+\mathrm{Im} \left(   \frac{\partial_j \phi_{\alpha }  } {\langle x_j\rangle}  \langle x_j\rangle h^{kj} \left[ D_k ,\langle x_k\rangle\right]\frac{\overline{\phi}_{\alpha k } } {\langle x_k\rangle} \right). 
\end{align*}
Moreover,
\begin{align*}
	&\mathrm{Im} \left(   \frac{\partial_j \phi_{\alpha }  } {\langle x_j\rangle}  \langle x_j\rangle \langle x_i\rangle h^{ij}  D_k \left(\frac{\overline{\phi}_{\alpha i } } {\langle x_i\rangle}\right)\right)\nonumber\\
	=&\, \mathrm{Im} \left(    \frac{\partial_j \phi_{\alpha }  } {\langle x_j\rangle}  \left[ \langle x_j\rangle \langle x_i\rangle h^{ij}   D_k^{\frac12}, D_k^{\frac12} \right] \frac{\overline{\phi}_{\alpha i } } {\langle x_i\rangle}\right)\\
	&+\mathrm{Im} \left(   \frac{\partial_j \phi_{\alpha }  } {\langle x_j\rangle} D_k^{\frac12}\left ( \langle x_j\rangle\langle x_i\rangle h^{ij} D_k^{\frac12}   \left (\frac{\overline{\phi}_{\alpha i } } {\langle x_i\rangle}\right)\right)\right).\nonumber
\end{align*}
From \eqref{lemc2_1}, it follows that
\begin{align*} 
	&\left|\int_0^t\int_{\mathbb{R}^d}   \mathrm{Im} \left(   \frac{\partial_j \phi_{\alpha }  } {\langle x_j\rangle}  \langle x_j\rangle h^{kj} \left[ D_k ,\langle x_k\rangle\right]\frac{\overline{\phi}_{\alpha k } } {\langle x_k\rangle} \right)\mathrm{d} x \mathrm{d} s\right|\\
	\lesssim &\,\int_0^t\left\| \frac{\partial_j \phi_{\alpha }  } {\langle x_j\rangle}\right \|_{L^2}  \left\|    \left[ D_k ,\langle x_k\rangle\right]\frac{\overline{\phi}_{\alpha k } } {\langle x_k\rangle}\right\|_{L^2} \left\|  h^{kj}  \langle x_j\rangle \right\|_{L^\infty} \mathrm{d} s\nonumber\\
	\lesssim &\,\int_0^t \left\| \frac{\partial_j \phi_{\alpha }  } {\langle x_j\rangle}\right \|_{L^2}  \left\|     \frac{\overline{\phi}_{\alpha k } } {\langle x_k\rangle}\right\|_{L^2} \left\|   h^{kj}  \langle x_j\rangle \right\|_{L^\infty} \mathrm{d} s \nonumber\\
	\lesssim &\,\sum_{|\beta| \leq s_1-\frac 32}\left\| \langle x_j\rangle \phi_{\beta}\right\|_{L^2(\mathbb{R}^d)} \cdot X(t)\nonumber
\end{align*}
and from \eqref{lemc1_3}, we obtain
\begin{align*} 
	&\left|\int_0^t \int_{\mathbb{R}^d}   \mathrm{Im}\left(    \frac{\partial_j \phi_{\alpha }  } {\langle x_j\rangle}  \left[ \langle x_j\rangle \langle x_i\rangle h^{ij}   D_k^{\frac12}, D_k^{\frac12} \right] \frac{\overline{\phi}_{\alpha i } } {\langle x_i\rangle}\right)\mathrm{d} x \mathrm{d} s\right|\\
	\lesssim &\,\int_0^t\int_{\mathbb{R}^{d-1}}\left\| \frac{\partial_j \phi_{\alpha }  } {\langle x_j\rangle}\right \|_{L^2(\mathbb{R}_{x_k})}  \left\|    \left[ \langle x_j\rangle \langle x_i\rangle h^{ij}   D_k^{\frac12}, D_k^{\frac12} \right] \frac{\overline{\phi}_{\alpha i } } {\langle x_i\rangle} \right\|_{L^2(\mathbb{R}_{x_k})} \mathrm{d} \hat{x}_k\mathrm{d} s\nonumber\\
	\lesssim &\, \int_0^t \int_{\mathbb{R}^{d-1}}\left\| \frac{\partial_j \phi_{\alpha }  } {\langle x_j\rangle}\right \|_{L^2(\mathbb{R}_{x_k})}  \left\| \frac{\partial_i \phi_{\alpha }  } {\langle x_i\rangle}\right \|_{L^2(\mathbb{R}_{x_k})} \left\|\partial_k \left( \langle x_j\rangle \langle x_i\rangle h^{ij}  \right) \right\|_{  H^{\mathscr{S}_1} (\mathbb{R}_{x_k}) } \mathrm{d} \hat{x}_k\mathrm{d} s \nonumber\\
	\lesssim &\,  \int_0^t  \left\| \frac{\partial_j \phi_{\alpha }  } {\langle x_j\rangle}\right \|_{L^2(\mathbb{R}^d)}  \left\| \frac{\partial_i \phi_{\alpha }  } {\langle x_i\rangle}\right \|_{L^2(\mathbb{R}^d)}  \left\|\partial_k \left( \langle x_j\rangle \langle x_i\rangle h^{ij}  \right) \right\|_{L^\infty \left(\mathbb{R}^{d-1}; (H^{\mathscr{S}_1} (\mathbb{R}_{x_k}))\right)}  \mathrm{d} s \nonumber\\
	\lesssim &\, \sum_{j=1}^d\sum_{|\beta| \leq s_1-\frac 32}\left\| \langle x_j\rangle^2 \phi_{\beta}\right\|_{L^2(\mathbb{R}^d)} \cdot X(t)\nonumber
\end{align*}
for $ \frac 12<\mathscr{S}_1 $,
where
\begin{align*}
	\left\|\partial_k \left( \langle x_j\rangle \langle x_i\rangle h^{ij}  \right) \right\|_{L^\infty \left(\mathbb{R}^{d-1}; (H^{\mathscr{S}} (\mathbb{R}_{x_k}))\right)} \lesssim &\,	\left\|\partial_k \left( \langle x_j\rangle \langle x_i\rangle h^{ij}  \right) \right\|_{  H^{\mathscr{S}+\mathscr{S}_2} (\mathbb{R}^d)}\\
	\lesssim &\, \left\|   \langle x_j\rangle \langle x_i\rangle \phi  \right\|_{  H^{\mathscr{S}_1+\mathscr{S}_2+1} (\mathbb{R}^d)}\\
	\lesssim &\, \sum_{j=1}^d\sum_{|\beta| \leq s_1-\frac 32}\left\| \langle x_j\rangle^2 \phi_{\beta}\right\|_{L^2(\mathbb{R}^d)},
\end{align*}
for $ \frac{d}{2}+\frac 52<s_1 $, $ \frac{d-1} 2<\mathscr{S}_2 $, and  $s_1-\frac 32=\mathscr{S}_1+\mathscr{S}_2+1 $.
By the symmetry of $ g^{ij} $, we also have
\begin{align*}
	& \int_{\mathbb{R}^d}\mathrm{Im} \left(   \frac{\partial_j \phi_{\alpha }  } {\langle x_j\rangle} D_k^{\frac12}\left ( \langle x_j\rangle\langle x_i\rangle h^{ij} D_k^{\frac12}   \left (\frac{\overline{\phi}_{\alpha i } } {\langle x_i\rangle}\right)\right)\right) \mathrm{d} x \\
	=&\,  \int_{\mathbb{R}^d}\mathrm{Im} \left(  \langle x_j\rangle D_k^{\frac12} \left(\frac{\partial_j \phi_{\alpha }  } {\langle x_j\rangle}\right) \cdot h^{ij} \cdot  \langle x_i\rangle D_k^{\frac12}   \left (\frac{\overline{\phi}_{\alpha i } } {\langle x_i\rangle}\right)\right)  \mathrm{d} x =0.	
	\end{align*}

It is straightforward to  use \eqref{lemc2_1}  to  show that
\begin{align*}
	\left\| \frac{D_k \phi_{\alpha }  } {\langle x_k\rangle}\right\|_{L^2(\mathbb{R}^d)} \lesssim &\,	\left\| \left[D_k,\frac{1 } {\langle x_k\rangle} \right] \phi_{\alpha } \right\|_{L^2(\mathbb{R}^d)}+	\left\| D_k\bigg( \frac{ \phi_{\alpha }  } {\langle x_k\rangle}\bigg)\right\|_{L^2(\mathbb{R}^d)}\\
	\lesssim &\,\|\phi\|_{H^{|\alpha|}}+ \left\| \partial_k\bigg( \frac{ \phi_{\alpha }  } {\langle x_k\rangle}\bigg)\right\|_{L^2(\mathbb{R}^d)}\\
	\lesssim &\,\|\phi\|_{H^{s_1}}+ \left\|  \frac{ \partial_k\phi_{\alpha }  } {\langle x_k\rangle} \right\|_{L^2(\mathbb{R}^d)}.
\end{align*}
A similar analysis as in \eqref{qener2} then gives
	\begin{align*}
	& \int_0^t \int_{\mathbb{R}^d} \left|  \left(F_z\nabla \phi_\alpha+F_{\overline{z}}\nabla \overline{\phi}_\alpha+\partial h^{ij}\partial_{ij } \phi_{\alpha -1}+G\right)D_k\overline{\phi}_{\alpha } \right| \mathrm{d} x \mathrm{d} s\\
%	&\lesssim\left\| \partial_k h^{ij}  \langle x_i\rangle^2 \right\|_{L^\infty(\mathbb{R}^d)}^{\frac{1}{2}}  \left\| \partial_k h^{ij}  \langle x_j\rangle^2 \right\|_{L^\infty(\mathbb{R}^d)}^{\frac{1}{2}} \left(\int_0^t     \int_{\mathbb{R}^{d}} \left|\frac{\partial_i \phi_{\alpha }  } {\langle x_i\rangle}\right|^2  \mathrm{d} x \mathrm{d} s\right)^{\frac 12}   \left(\int_0^t     \int_{\mathbb{R}^{d}} \left|\frac{\partial_j \phi_{\alpha }  } {\langle x_j\rangle}\right|^2  \mathrm{d} x \mathrm{d} s\right)^{\frac 12}\\
%	&+t\cdot\|\phi\|^3_{H^{s_1+\frac 12}(\mathbb{R}^{d})}\\
	\lesssim &\,  \sum_{j=1}^d\sum_{|\beta| \leq s_1-\frac 32}\left\| \langle x_j\rangle^2 \phi_{\beta}\right\|_{L^2(\mathbb{R}^d)} \cdot \left(X(t)+\|\phi\|^2_{H^{s_1}(\mathbb{R}^{d})}\right)+t\cdot\|\phi\|^3_{H^{s_1+\frac 12}(\mathbb{R}^{d})}.
\end{align*}
Integrating \eqref{qecons2_1} on  $ [0,t] \times \mathbb{R}^d   $ nd combining the above estimates, we derive
\begin{align} \label{qener3_1}
	\| D_k^{\frac 12} \phi_\alpha\|_{L^2}^2 \leq&\, 	\| D_k^{\frac 12} \partial^{\alpha}\phi_0\|_{L^2}^2+\left|\int_{\mathbb{R}^d}    \mathrm{Im} \left(   \partial_j \phi_{\alpha }  g^{ij}  D_k  \overline{\phi}_{\alpha i }\right) \mathrm{d} x\right|\nonumber\\
	&+\int_0^t\int_{\mathbb{R}^d} \left| \mathrm{Im} \left[  \left(F_z\nabla \phi_\alpha+F_{\overline{z}}\nabla \overline{\phi}_\alpha+\partial g^{ij}\partial_{ij } \phi_{\alpha -1}+G\right)\overline{\phi}_{\alpha }\right]\right| \mathrm{d} x\mathrm{d} s\nonumber\\
	\lesssim &\,	\| D_k^{\frac 12}\partial^{\alpha}\phi_0\|_{L^2}+\sum_{j=1}^d \sum_{|\beta| \leq s_1-\frac12}\left\| \langle x_j\rangle^2 \phi_{\beta}\right\|_{L^2(\mathbb{R}^d)} \cdot\left(X(t)+\|\phi\|^2_{H^{s_1}(\mathbb{R}^{d})}\right)+t\cdot\|\phi\|^3_{H^{s_1+\frac 12}(\mathbb{R}^{d})} \nonumber\\
	\lesssim & \,Y(0) +  (1+t) Y(t)^{\frac 32} +  X(t)^{\frac 32}.
\end{align}
Summing \eqref{qener3_1} over $ k=1,\dots,  d $,  and combining with \eqref{qener1} and \eqref{qener2}, we conclude that
\begin{align}  \label{qener3}
		\| \phi\|_{H^{ s_1+\frac12 }(\mathbb{R}^d)}^2\lesssim &\,  Y(0)+	 (1+t) Y(t)^{\frac 32} +  X(t)^{\frac 32}.
\end{align}

Now we estimate $ \sum_{j=1}^d\sum_{|\beta| \leq s_1-\frac 32} \left\| \langle x_i\rangle^2 \phi_{\beta}\right\|_{L^2(\mathbb{R}^d)}^2  .$  Similar to  \eqref{qNLS1},  differentiating  \eqref{qNLS} for $ \beta $ times shows
\begin{align}  \label{qNLSb}
	\sqrt{-1}\phi_{\beta t}+\partial_i\left( g^{ij} \phi_{\beta j}\right)=\mathcal{N}_\beta,
\end{align}
where
\begin{align*}
	&\mathcal{N}_\beta:=F_z\nabla \phi_\beta+F_{\overline{z}}\nabla \overline{\phi}_\beta+\partial h^{ij}\partial_{ij } \phi_{\beta -1}+G_\beta.
\end{align*}
Multiplying the above by $ x_k^2 $ gives
\begin{align}   \label{qNLSx2}
	&	\sqrt{-1}x_k^2\phi_{\beta t}+\sum_{i,j}\partial_i\left( g^{ij}   \partial_j \left(x_k^2 \phi_{\beta }\right)\right)\\
	=&\,4 x_k \partial_k  \phi_\beta +2 x_k h^{kj} \phi_{\beta j}+2 x_k h^{ik} \phi_{\beta i}+2 g^{kk} \phi_{\beta } +2 x_k \partial_i h^{ik} \phi_{\beta } +x^2_k \cdot\mathcal{N}_\beta,\nonumber
\end{align}
since
\begin{align*}
	&	\partial_i\left( g^{ij} \phi_{\beta j}\right) \cdot x^2_k= \partial_i\left(  x^2_k g^{ij} \phi_{\beta j}\right)  -2 g^{kk}_0x_k \partial_k  \phi_\beta -2 x_k h^{kj} \phi_{\beta j}\\ 
%	=&\,\partial_i\left( g^{ij} \partial_j\left(\phi_{\beta }x^2_k\right)\right) -2\partial_i\left(x_k g^{ik} \phi_{\beta } \right) -2g^{kk}_0 x_k \partial_k  \phi_\beta -2 x_k h^{kj} \phi_{\beta j}\\
	=&\,\partial_i\left( g^{ij} \partial_j\left(\phi_{\beta }x^2_k\right)\right)   -4 g^{kk}_0x_k \partial_k  \phi_\beta -2 x_k h^{kj} \phi_{\beta j}-2 x_k h^{ik} \phi_{\beta i}-2 g^{kk} \phi_{\beta } -2 x_k \partial_i h^{ik} \phi_{\beta }, 
\end{align*}
owing to $ g^{ij}=g^{ij}_0+h^{ij} $. 
Multiplying the resulting equation by $ x_k^2 \overline{\phi}_\beta $ and taking the imaginary part, we obtain
\begin{align*}   
	&\frac 12\frac{d}{d t}\left(  \left| x_k^2\phi_{\beta} \right|^2 \right)+ \mathrm{Im} \left(\partial_i\left( g^{ij} \partial_j\left(\phi_{\beta }x^2_k\right)\right)x_k^2\overline{\phi}_{\beta}\right)	 \\
	=&\, \frac 12\frac{d}{d t}\left(  \left| x_k^2\phi_{\beta} \right|^2 \right)+ \mathrm{Im}  \partial_i\left( g^{ij} \partial_j\left(\phi_{\beta }x^2_k\right) x_k^2\overline{\phi}_{\beta}\right)	\nonumber\\
	=&\,\mathrm{Im} \left\{  \left(4 x_k \partial_k  \phi_\beta +2 x_k h^{kj} \phi_{\beta j}+2 x_k h^{ik} \phi_{\beta i}+2 g^{kk} \phi_{\beta } +2 x_k \partial_i h^{ik} \phi_{\beta } +x^2_k \cdot\mathcal{N}_\beta  \right) x_k^2\overline{\phi}_{\beta}\right\}.\nonumber
\end{align*}	
Integrating over $ [0,t]\times \mathbb{R}^d $ leads to
\begin{align} \label{qenxc2}
		&\left\|  x_k^2 \phi_{\beta}\left(t,\cdot\right)\right\|_{L^2(\mathbb{R}^d)}^2  \lesssim\left. \left\|  x_k^2 \phi_{\beta}\right\|_{L^2(\mathbb{R}^d)}^2 \right|_{s=0}\\ 
		&+  t\cdot\left\|  x_k^2 \phi_{\beta}\right\|_{L^2(\mathbb{R}^d)}  \left\|  x_k \partial_k  \phi_{\beta}\right\|_{L^2(\mathbb{R}^d)}+t\cdot\left\|  x_k^2 \phi_{\beta}\right\|_{L^2(\mathbb{R}^d)}  \left\|   x^2_k \cdot\mathcal{N}_\beta    \right\|_{L^2(\mathbb{R}^d)} \nonumber\\
	&+t\cdot\left\|  x_k^2 \phi_{\beta}\right\|_{L^2(\mathbb{R}^d)} \left\|   2 x_k h^{kj} \phi_{\beta j}+2 x_k h^{ik} \phi_{\beta i}+2 g^{kk} \phi_{\beta } +2 x_k \partial_i h^{ik} \phi_{\beta }   \right\|_{L^2(\mathbb{R}^d)}. \nonumber 
%	\lesssim &\,\left. \left\|  x_k^2 \phi_{\beta}\right\|_{L^2(\mathbb{R}^d)}^2 \right|_{s=0}+ t\cdot\left\|  x_k^2 \phi_{\beta}\right\|_{L^2(\mathbb{R}^d)}^2+ t\cdot   \left\|  x_k \partial_k  \phi_{\beta}\right\|_{L^2(\mathbb{R}^d)}^2\nonumber\\
%	&+t\cdot  \left\|   x^2_k \cdot\mathcal{N}_\beta    \right\|_{L^2(\mathbb{R}^d)}^2+ t\cdot  \left\|   2 x_k h^{kj} \phi_{\beta j}+2 x_k h^{ik} \phi_{\beta i}+2 g^{kk} \phi_{\beta } +2 x_k \partial_i h^{ik} \phi_{\beta }   \right\|_{L^2(\mathbb{R}^d)}^2.\nonumber
\end{align}
Integration by parts and H\"older's inequality give
\begin{align}\label{qenxc20}
	\left\| \langle x_k\rangle \partial_k  \phi_{\beta}\right\|_{L^2(\mathbb{R}^d)}^2
	=&\,- \int_{\mathbb{R}^{d}} \left(2  \langle x_k\rangle   \mathrm{sgn}(x_k) \partial_k  \phi_{\beta} +\langle x_k\rangle^{2 } \partial_{kk } \phi_{\beta} \right) \overline{\phi}_\beta\mathrm{d} x\\
	\lesssim &\, \left\|  \langle x_k\rangle^{2 }  \phi_{\beta}\right\|_{L^2(\mathbb{R}^d)} \left( \left\|  \partial_k  \phi_{\beta}\right\|_{L^2(\mathbb{R}^d)} +\left\|  \partial_k^2  \phi_{\beta}\right\|_{L^2(\mathbb{R}^d)} \right).\nonumber
\end{align}
By Sobolev   and H\"older, 
%Direct calculation shows
%\begin{align} 
%	 \left\|  x_k \partial_k  \phi_{\beta}\right\|_{L^2(\mathbb{R}^d)}\lesssim &\,\left\|  \langle x_k\rangle^2   \phi_{\beta}\right\|_{L^2(\mathbb{R}^d)}^{\frac 12}\left( \left\|  \partial_k  \phi_{\beta}\right\|_{L^2(\mathbb{R}^d)}^{\frac 12}+\left\|  \partial_k^2  \phi_{\beta}\right\|_{L^2(\mathbb{R}^d)}^{\frac 12}\right)
%\end{align}
%and
\begin{align} \label{qenxc3}
	&\left\|  x^2_k\cdot\left( F_z\nabla \phi_\beta+F_{\overline{z}}\nabla \overline{\phi}_\beta+\partial h^{ij}\partial_{ij } \phi_{\beta -1} \right) \right\|_{L^2(\mathbb{R}^d)}\\
	 \lesssim &\,\left\|\langle x_k\rangle^2\left(\phi+ \nabla \phi\right)\right\|_{L^\infty (\mathbb{R}^d)} \left\| \nabla\phi_\beta\right\|_{L^{2}}\nonumber \\
	\lesssim &\, \sum_{i=1}^d\sum_{|\beta| \leq s_1-\frac 32}\left\| \langle x_i\rangle^2 \phi_{\beta}\right\|_{L^2(\mathbb{R}^d)} \cdot \|\phi\|_{H^{s_1+\frac 12}(\mathbb{R}^{d})},\nonumber
\end{align}
and
\begin{align}
		\left\| x^2_k\cdot G_\beta \right\|_{L^2(\mathbb{R}^d)}\lesssim &\,\left\|  x^2_k \phi \right\|_{W^{p^\prime_1, |\beta^1|+1}} \left\|  \phi \right\|_{W^{p^\prime_2,|\beta^2|+1}}\\
		\lesssim &\,\sum_{i=1}^d\sum_{|\beta| \leq s_1-\frac 32}\left\| \langle x_i\rangle^2 \phi_{\beta}\right\|_{L^2(\mathbb{R}^d)} \cdot \|\phi\|_{H^{s_1+\frac 12}(\mathbb{R}^{d})},\nonumber
\end{align}
provided 
\begin{align*}
		& \left|\beta^i\right|\leq \left|\beta\right|-1,~~i=1,2, \quad  \left|\beta^1\right|+\left|\beta^2\right|=\left|\beta^1+\beta^2\right|= \left|\beta\right|, \\
	&	\sum_{i=1}^2 \frac{1}{p^\prime_i}=\frac 12,\qquad
	\frac{1}{p_1^\prime}\geq \frac 12-\frac{|\beta|-|\beta^1|-\frac12}{d}, \nonumber\\
	&\frac{1}{p_2^\prime}\geq \frac 12-\frac{|\alpha|-|\beta^2|-\frac 12}{d},\nonumber
\end{align*}
which requires $  \frac d2+ \frac 32\leq s_1 =|\alpha|. $
The 
\begin{align*}
	\left\|   2 x_k h^{kj} \phi_{\beta j}+2 x_k h^{ik} \phi_{\beta i}+2 g^{kk} \phi_{\beta } +2 x_k \partial_i h^{ik} \phi_{\beta }   \right\|_{L^2(\mathbb{R}^d)}
\end{align*}
in the remainder terms  are bounded directly:
\begin{align}
	\left\|  g^{kk} \phi_{\beta }  \right\|_{L^2(\mathbb{R}^d)}\lesssim &\, \left(1+\left\|   h^{kk}    \right\|_{L^\infty(\mathbb{R}^d)}\right) 	\left\|    \phi_{\beta }  \right\|_{L^2(\mathbb{R}^d)} , \label{qenxc5}\\
	\left\|    2 x_k h^{kj} \phi_{\beta j}+2 x_k h^{ik} \phi_{\beta i}    \right\|_{L^2(\mathbb{R}^d)}
	\lesssim &\,\sum_{j=1}^d \left\|   h^{kj}  \langle x_k\rangle \right\|_{L^\infty(\mathbb{R}^d)} \left\|    \phi_{\beta j}    \right\|_{L^2(\mathbb{R}^d)} \label{qenxc6}\\
	\lesssim &\,\sum_{i=1}^d\sum_{|\beta| \leq s_1-\frac 32}\left\| \langle x_i\rangle \phi_{\beta}\right\|_{L^2(\mathbb{R}^d)} \left\| \phi\right\|_{H^{s_1+\frac 12}},\nonumber\\
	\left\|     x_k \partial_i h^{ik} \phi_{\beta }   \right\|_{L^2(\mathbb{R}^d)} \lesssim &\,\left\|     x_k   \phi_{\beta }   \right\|_{L^2(\mathbb{R}^d)}\left\|      \partial_i h^{ik}    \right\|_{L^\infty(\mathbb{R}^d)}\label{qenxc7}\\
	\lesssim &\,    \left\|    \langle x_k \rangle  \phi_{\beta}\left(t,\cdot\right)\right\|_{L^2(\mathbb{R}^d)}  \left\| \phi\right\|_{H^{s_1+\frac 12}}.\nonumber
\end{align}
Summing \eqref{qenxc2} over $ |\beta| $ from $ 0 $ to $ s_1-\frac 32 $ and over $ k =1,\dots,  d $ respectively and combining  it with \eqref{qenxc20}-\eqref{qenxc7}, we obtain
\begin{align} \label{qenx2}
		& \sum_{k=1}^d\sum_{|\beta| \leq s_1-\frac32}\left\|  x_k^2 \phi_{\beta}\left(t,\cdot\right)\right\|_{L^2(\mathbb{R}^d)}^2  \lesssim \sum_{k=1}^d \sum_{|\beta| \leq s_1-\frac32}  \left\|  x_k^2 \partial^{\beta}\phi_0\right\|_{L^2(\mathbb{R}^d)}^2  \nonumber\\
		&+t \cdot  \sum_{k=1}^d \sum_{|\beta| \leq s_1-\frac32}  \left\|  \langle x_k\rangle^2 \phi_{\beta}\right\|_{L^2(\mathbb{R}^d)}^{\frac 32}  \left\|  \phi\right\|_{H^{|\beta|+2}(\mathbb{R}^d)}^{\frac 12}\nonumber\\
		&+t \cdot  	\sum_{i,k=1}^d\sum_{|\beta| \leq s_1-\frac32}\left\| \langle x_k\rangle^2 \phi_{\beta}\right\|_{L^2(\mathbb{R}^d)}\left\| \langle x_i\rangle^2 \phi_{\beta}\right\|_{L^2(\mathbb{R}^d)} \cdot \|\phi\|_{H^{s_1+\frac 12}(\mathbb{R}^{d})}  \nonumber\\
	&+t\cdot \sum_{k=1}^d  \sum_{|\beta| \leq s_1-\frac32}\left\| \langle x_k\rangle^2 \phi_{\beta}\right\|_{L^2(\mathbb{R}^d)}\left(1+ \left\| \phi\right\|_{H^{s_1+\frac 12}} \right) 	\left\|    \phi_{\beta }  \right\|_{L^2(\mathbb{R}^d)}  	\nonumber\\
	\leq &\, C Y(0)+ \frac {1}{100} \sum_{k=1}^d \sum_{|\beta| \leq s_1-\frac32}  \left\|  \langle x_k\rangle^2 \phi_{\beta}\right\|_{L^2(\mathbb{R}^d)}^{2}   +Ct^4 \left\|  \phi\right\|_{H^{|\beta|+2}(\mathbb{R}^d)} ^2 \\
	&+ tY(t)^{\frac 32}+Ct^2 \left\| \phi_{\beta}\right\|_{L^2(\mathbb{R}^d)}^2. \nonumber
\end{align}
If $ \frac d2+\frac 52\leq s_1 $ and $ |\beta| \leq s_1-\frac32 $,  then  \eqref{qener3} implies
\begin{align}\label{qenbkk}
&	\max \left\{\left\|    \phi_{\beta}\right\|_{L^2(\mathbb{R}^d)}, \left\|   \phi\right\|_{H^{|\beta|+2}(\mathbb{R}^d)}\right\}  	\lesssim	 \|\phi\|_{H^{s_1+\frac 12}(\mathbb{R}^{d})}\nonumber\\
&\lesssim	Y(0)+	 (1+t) Y(t)^{\frac 32} +  X(t)^{\frac 32}.
\end{align}
This combining with \eqref{qenx2} and \eqref{qener1} gives
\begin{align} \label{qenx2_2}
&	\sum_{k=1}^d  \sum_{|\beta| \leq s_1-\frac32}\left\| \langle x_k\rangle^2 \phi_{\beta}\right\|_{L^2(\mathbb{R}^d)} ^2\lesssim  (1+t^2+t^4) Y(0) +(t+t^3+t^5)  Y(t)^{\frac 32} +(t^2+t^4)  X(t)^{\frac 32}, 
\end{align}
where the $  \frac {1}{100} \sum_{k=1}^d \sum_{|\beta| \leq s_1-\frac32}  \left\|  \langle x_k\rangle^2 \phi_{\beta}\right\|_{L^2(\mathbb{R}^d)}^{2}  $ in \eqref{qenx2} is absorbed by the left hand side of  \eqref{qenx2} and \eqref{qener1}.

Finally,  combining \eqref{qenerM}, \eqref{qener3},  \eqref{qenx2_2}, we establish
\begin{align} \label{qener}
		Y(t) +X(t)\lesssim &\, \left(1+ t^4\right)	Y(0)  +	\left(1+ t^5 \right)\left(X(t)^{\frac 32} +Y(t)^{\frac 32}\right)  
	\end{align} 
for $ \frac d2+\frac 52< s_1 .  $    
 Corresponding to $ Y(0) $,   the initial data $ \phi_0 $ can be  restricted in $ H^{s_1+\frac12} (\mathbb{R}^d) \cap  L^2(\langle x\rangle^N \mathrm{d}x)$ for some positive integer $ N $ thanks to the following lemma, whose proof is  \ref{lemx} is evident and provided in Appendix \ref{appa} for reader's convenience.
\begin{lem} \label{lemx} If $ N_1 $ is a  positive integer and the  components of vector $ \gamma \in \mathbb{N}^d_0 $, then
	\begin{align} 
		\int_{\mathbb{R}^d} |x|^{2N_1}|\phi_\gamma|^2 \mathrm{d} x&\lesssim_{N,\gamma} \left\|\phi\right\|_{H^{|\gamma|+1}(\mathbb{R}^d)}^{\mathscr{S}_1} \left\||x|^N \phi\right\|_{L^2(\mathbb{R}^d)}^{\mathscr{S}_2}, \label{lemx_1}  \\
%		&\lesssim_{N,\gamma} \left\|\phi\right\|_{H^{|\gamma|+1}(\mathbb{R}^d)}^{2} +\left\||x|^M \phi\right\|_{L^2(\mathbb{R}^d)}^{2}, \label{lemx_2} \\
		\int_{\mathbb{R}^d} |x|^{2N_1}|J^s\phi|^2 \mathrm{d} x&\lesssim_{N,s}\left\|\phi\right\|_{H^{s+1}(\mathbb{R}^d)}^{\mathscr{S}_3} \left\||x|^{N} \phi\right\|_{L^2(\mathbb{R}^d)}^{\mathscr{S}_4} \label{lemx_2}
	\end{align}
	for constants  $ \mathscr{S}_i>0, i=1,2  $ and  large enough integer $N$, where $ \sum_{i=3}^4\mathscr{S}_i=\sum_{i=1}^2\mathscr{S}_i=2 $ and  $N, \mathscr{S}_i  $   depend on $ N_1 $ and $ \gamma $. 
\end{lem}

If $ t\leq 1 $ and $ 	Y(0)   $	 is small enough, or $ \phi_0 \in H^{s_1+\frac12} (\mathbb{R}^d) \cap  L^2(\langle x\rangle^N \mathrm{d}x) $ is small, where $  \frac d2 +\frac 52<s_1 $, $ N $ depends only on $ d, s_1 $, then by \eqref{qener}, Lemma \ref{lemx} and a standard bootstrap argument,  we arrive at
\begin{align} \label{qbound1}
	Y(t) +X(t) \leq C\left(	Y(0)\right)
\end{align}
for $ t\in[0,1] $.

\begin{proof}[\textbf{Proof of Theorem \ref{thm2}}]
	We may apply \eqref{qbound1}  to obtain the desired results through the standard artificial viscosity method; see Section 9.4 in \cite{linares2014introduction}. 
	Consider the approximate equation	
	\begin{equation}	 \label{aqNLS}
		\begin{cases}
			\sqrt{-1}\phi_t+	\varepsilon\sqrt{-1}\Delta^2 \phi+\partial_i\left( g^{ij}\left(\phi,\overline{\phi}\right) \partial_j \phi\right)=F\left(\phi,\overline{\phi},\nabla\phi, \nabla \overline{\phi}\right)	,~~~\phi: t \times \mathbb{R}^d\rightarrow \mathbb{C}^m\\
			\phi(0,x)=\phi_0(x).
		\end{cases}
	\end{equation}
	Differentiating \eqref{aqNLS} with respect to the multi-index $\alpha$ yields
	\begin{align} \label{aqNLS1}
		\sqrt{-1}\phi_{\alpha t}+	\varepsilon\sqrt{-1}\Delta^2 \phi_\alpha+\partial_i\left( g^{ij} \partial_j \phi_{\alpha }\right)=F_z\nabla \phi_\alpha+F_{\overline{z}}\nabla \overline{\phi}_\alpha+\partial h^{ij}\partial_{ij } \phi_{\alpha -1}+G.
	\end{align}
Multiplying \eqref{aqNLS1} by $\overline{\phi}{\alpha}$ and $D_k \overline{\phi}{\alpha}$, respectively, and taking the imaginary part, the same procedure as in the derivation of \eqref{qener1}, \eqref{qener2}, and \eqref{qener3} gives
\begin{align*}
	&	\| \phi\|_{H^{s_1 }(\mathbb{R}^d)}^2+\varepsilon\sum_{i,j=1}^d\int_0^t\left\| \partial_{ij}\phi\right\|_{H^{s_1}(\mathbb{R}^d)}^2 \mathrm{d} s\\
	\lesssim &\, 		\| \phi_0\|_{H^{s_1 }(\mathbb{R}^d)}^2+\sum_{|\alpha|=0}^{s_1}\int_0^t\int_{\mathbb{R}^d} \left| \mathrm{Im} \left(  \left(F_z\nabla \phi_\alpha+F_{\overline{z}}\nabla \overline{\phi}_\alpha+\partial h^{ij}\partial_{ij } \phi_{\alpha -1}+G\right)\overline{\phi}_{\alpha }\right)\right| \mathrm{d} x\mathrm{d} s\\
	\lesssim &\,\| \phi_0\|_{H^{s_1 }(\mathbb{R}^d)}^2+ (1+t)\cdot(Y(t)^{\frac 32}+X(t)^{\frac 32} ) \nonumber
\end{align*}
and 
\begin{align} \label{aqener1}
	\|\phi\|_{H^{s_1+\frac12}(\mathbb{R}^{d})}^2+\varepsilon\sum_{i,j=1}^d\int_0^t\left\| \partial_{ij}\phi\right\|_{H^{s_1+\frac12}(\mathbb{R}^d)}^2 \mathrm{d} s\\
	\lesssim 	\|\phi_0\|_{H^{s_1+\frac12}(\mathbb{R}^{d})}^2+  \left(1+t\right) \cdot\left(Y(t)^{\frac 32}+ X(t)^{\frac 32} \right),\nonumber
\end{align}
where $ X(t), Y(t)  $ are defined as before. 

 Following the approach in \eqref{qNLSx2}, we differentiate \eqref{aqNLS} with respect to $\beta$ and multiply by $x_k^2$ to obtain
\begin{align*}   
	&	\sqrt{-1}x_k^2\phi_{\beta t}+\sum_{i,j}\partial_i\left( g^{ij}   \partial_j \left(x_k^2 \phi_{\beta }\right)\right)+	\varepsilon\sqrt{-1}\Delta^2 \left(x_k^2\phi_\beta\right)\\
	-&\,8\varepsilon \sqrt{-1} x_k \Delta\partial_k \phi_\beta -8\varepsilon \sqrt{-1} \partial_{kk} \phi_\beta-4\varepsilon \sqrt{-1}   \Delta\phi_\beta \nonumber\\
	=&\,4 x_k \partial_k  \phi_\beta +2 x_k h^{kj} \phi_{\beta j}+2 x_k h^{ik} \phi_{\beta i}+2 g^{kk} \phi_{\beta } +2 x_k \partial_i h^{ik} \phi_{\beta } +x^2_k \cdot\mathcal{N}_\beta,\nonumber
\end{align*}
which multiplying  by $ x_k^2 \overline{\phi}_\beta $, integrating on $[0,t]\times \mathbb{R}^d $, taking the imaginary part, summing   over $ |\beta| \leq s_1-\frac 32 $ and over $ k =1,\dots,  d $   show
\begin{align} \label{aqenx2}
		& \sum_{k=1}^d\sum_{|\beta| \leq s_1-\frac32}\left\|  x_k^2 \phi_{\beta}\left(t,\cdot\right)\right\|_{L^2(\mathbb{R}^d)}^2 +\varepsilon\sum_{i,j=1}^d\sum_{|\beta| \leq s_1-\frac32}\int_0^t \left\| \partial_{ij}\left( x_k^2
	\phi_\beta\right)\right\|_{L^2(\mathbb{R}^d)}\mathrm{d}s 
	\\
	\leq& \,C\varepsilon t^{\frac 12} \left\| \langle x_k\rangle \langle \nabla \rangle   \phi_{\beta}\right\|_{L^2(\mathbb{R}^d)}\Big(\sum_{|\beta| \leq s_1-\frac32}\int_0^t \left\| \partial_{jj}\left( x_k^2
	\phi_\beta\right)\right\|_{L^2(\mathbb{R}^d)}\mathrm{d}s \Big)^{\frac12}\nonumber\\
	&+C Y(0)+ \frac {1}{100} \sum_{k=1}^d \sum_{|\beta| \leq s_1-\frac32}  \left\|  \langle x_k\rangle^2 \phi_{\beta}\right\|_{L^2(\mathbb{R}^d)}^{2}   +Ct^4 \left\|  \phi\right\|_{H^{|\beta|+2}(\mathbb{R}^d)} ^2 \nonumber\\
	&+C tY(t)^{\frac 32}+Ct^2 \left\| \phi_{\beta}\right\|_{L^2(\mathbb{R}^d)}^2\nonumber \\
	\leq &\,C Y(0) +\frac {1}{50} \sum_{k=1}^d \sum_{|\beta| \leq s_1-\frac32}  \left\|  \langle x_k\rangle^2 \phi_{\beta}\right\|_{L^2(\mathbb{R}^d)}^{2}+\frac{1}4\sum_{|\beta| \leq s_1-\frac32}\int_0^t \left\| \partial_{jj}\left( x_k^2
	\phi_\beta\right)\right\|_{L^2(\mathbb{R}^d)}\mathrm{d}s\nonumber\\
	&\,+ CtY(t)^{\frac 32}+Ct^2 \left\| \phi_{\beta}\right\|_{L^2(\mathbb{R}^d)}^2+C(1+t^4)\left\|  \phi\right\|_{H^{|\beta|+2}(\mathbb{R}^d)} ^2  \nonumber\\
	\leq & \,	\frac {1}{50} \sum_{k=1}^d \sum_{|\beta| \leq s_1-\frac32}  \left\|  \langle x_k\rangle^2 \phi_{\beta}\right\|_{L^2(\mathbb{R}^d)}^{2}+\frac{1}4\sum_{|\beta| \leq s_1-\frac32}\int_0^t \left\| \partial_{jj}\left( x_k^2
	\phi_\beta\right)\right\|_{L^2(\mathbb{R}^d)}\mathrm{d}s\nonumber\\
	& (1+t^2+t^4) Y(0) +(t+t^3+t^5)  Y(t)^{\frac 32} +(t^2+t^4)  X(t)^{\frac 32}
\end{align}
analogous to \eqref{qenx2}, where we used \eqref{qenbkk} and
 $$ 	\frac {1}{50} \sum_{k=1}^d \sum_{|\beta| \leq s_1-\frac32}  \left\|  \langle x_k\rangle^2 \phi_{\beta}\right\|_{L^2(\mathbb{R}^d)}^{2}+\frac{1}4\sum_{|\beta| \leq s_1-\frac32}\int_0^t \left\| \partial_{jj}\left( x_k^2
\phi_\beta\right)\right\|_{L^2(\mathbb{R}^d)}\mathrm{d}s$$ can be absorbed by the left hand side of \eqref{aqenx2}. 
%Noticing that \eqref{aqener1} is better than \eqref{qener3},  the \eqref{qenerM0}-\eqref{qenx2_2} still hold for the  approximation equation	 \eqref{aqNLS}, ignoring  the  coefficient $ \varepsilon $. 

At the same time,  one can  utilize \eqref{leminter1} following the analysis of  \eqref{qener_mc4} to derive 
\begin{align}    \label{aqm1}
	&\sum_{i=1}^d\varepsilon\left|\int_0^t\int_{\mathbb{R}}     \frac{x_k}{\langle x_k\rangle}    \mathrm{Im}\int_{\mathbb{R}^{d-1}} \left(\partial_{i}\partial_{jj} \phi_\alpha\partial_i\partial_k \overline{\phi}_\alpha\right) \mathrm{d} \hat{x}_k\right|\\
%	=&\,\sum_{i=1}^d\varepsilon\left|\int_0^t  \int_{\mathbb{R}^{d-1}} \int_{\mathbb{R}} \frac{x_k}{\langle x_k\rangle}  \mathrm{Im}\left(\partial_i \phi_{\alpha } \partial_{ik} \overline{\phi}_\alpha\right) \mathrm{d} x_k \mathrm{d} \hat{x}_k  \mathrm{d} s\right| \nonumber\\
	\lesssim&\, \varepsilon\left|\mathrm{Im}\int_0^t\int_{\mathbb{R}^d}   D_k^{\frac 12}\left( \frac{x_k}{\langle x_k\rangle}    \partial_{kk} \overline{\phi}_\alpha \right)D_k^{-\frac 12}\left(\partial_k \partial_{jj}\phi_{\alpha } \right) \mathrm{d} x\mathrm{d}s\right|\nonumber\\
	&+\varepsilon\int_0^t\left\|   D_i^{\frac 12} \partial_{jj}\phi_{\alpha}\right\|_{L^2} \left\|   D_i^{\frac 12} \partial_{ik}\phi_{\alpha}\right\|_{L^2} \mathrm{d} s\nonumber\\
	%	=&\, \int_{\mathbb{R}^d} \mathrm{Im}\left(  \phi_{\alpha } \partial_k  \left( x_k\overline{\phi}_\alpha  \right)\right) \mathrm{d} x \\
%	\lesssim &\,\left\|  \frac{x_k}{\langle x_k\rangle}   D_k^{\frac 12}    \partial_i\phi_\alpha \right\|_{L^2} \left\|   D_k^{-1} \partial_k \left(D_k^{\frac 12}\partial_i\phi_{\alpha}\right)\right\|_{L^2} \\
%	&+ \left\|  \left[D_k^{\frac 12},\frac{x_k}{\langle x_k\rangle}   \right]  \partial_i \phi_\alpha \right\|_{L^2} \left\|   D_k^{-1} \partial_k \left(D_k^{\frac 12}\partial_i\phi_{\alpha}\right)\right\|_{L^2} \nonumber\\
%	\lesssim &\, \varepsilon\int_0^t\left\|   D_k^{\frac 12} \partial_{kk}\phi_{\alpha}\right\|_{L^2} \left\|   D_k^{\frac 12} \partial_{jj}\phi_{\alpha}\right\|_{L^2} \mathrm{d} s+ \varepsilon\int_0^t \left\|   \partial_{kk}\phi_{\alpha}\right\|_{L^2}\left\|   D_k^{\frac 12}\partial_{jj}\phi_{\alpha}\right\|_{L^2 } \mathrm{d} s+\varepsilon\int_0^t\left\|   D_i^{\frac 12} \partial_{jj}\phi_{\alpha}\right\|_{L^2} \left\|   D_i^{\frac 12} \partial_{ik}\phi_{\alpha}\right\|_{L^2} \mathrm{d} s \nonumber\\
	\lesssim &\,\varepsilon  \sum_{i,j=1}^d\int_0^t\left\| \partial_{ij}\phi\right\|_{H^{s_1+\frac 12}(\mathbb{R}^d)}^2 \mathrm{d} s.\nonumber
\end{align}
While other terms involving $\varepsilon$ in \eqref{aconsm} can be bounded similarly. Combining  \eqref{aconsm}, \eqref{qenerM} and \eqref{aqener1} and \eqref{aqm1}, we obtain the momentum type estimate
\begin{align} \label{aqenerM}
		X(t)&	\leq C_1\varepsilon  \sum_{i,j=1}^d\int_0^t\left\| \partial_{ij}\phi\right\|_{H^{s_1+\frac 12}(\mathbb{R}^d)}^2 \mathrm{d} s\\
		&+ C (1+t)Y(t)+Y(0) + C(1+t+t^{\frac 32}) Y(t)^{\frac 32} + X(t)^{\frac 32}
		.\nonumber
\end{align}

In summary, we can multiply \eqref{aqener1} by $ 2C_1 $, and sum it  with \eqref{aqener1} and \eqref{aqenx2} to establish
\begin{align*} 
	&X(t)+Y(t) +\varepsilon  \sum_{i,j=1}^d\int_0^t\left\| \partial_{ij}\phi\right\|_{H^{s_1+\frac 12}(\mathbb{R}^d)}^2 \mathrm{d} s+\varepsilon\sum_{i,j=1}^d\sum_{|\beta| \leq s_1-\frac32}\int_0^t \left\| \partial_{ij}\left( x_k^2
	\phi_\beta\right)\right\|_{L^2(\mathbb{R}^d)}\mathrm{d}s \\
	\lesssim 	&\,\left(1+ t^4\right)	Y(0)  +	\left(1+ t^5 \right)\left(X(t)^{\frac 32} +Y(t)^{\frac 32}\right),  \nonumber
\end{align*}
which asserts that if  $  \frac d2 +3<s_1+\frac 12=s $ and 
$ \phi_0 \in H^{s} (\mathbb{R}^d) \cap L^2(\langle x\rangle^N \mathrm{d}x) $  small enough for some integer $ N $, then the bootstrap argument shows
\begin{align} \label{aqbound1}
&X(t)+Y(t) +\varepsilon  \sum_{i,j=1}^d\int_0^t\left\| \partial_{ij}\phi\right\|_{H^{s_1+\frac 12}(\mathbb{R}^d)}^2 \mathrm{d} s\nonumber\\
&+\varepsilon\sum_{i,j=1}^d\sum_{|\beta| \leq s_1-\frac32}\int_0^t \left\| \partial_{ij}\left( x_k^2
\phi_\beta\right)\right\|_{L^2(\mathbb{R}^d)}\mathrm{d}s  \leq C\left(	Y(0)\right)
\end{align}
for $ t\in[0,1] $.
 
 	By  the same procedure in \cite{kenig2004cauchy}, we know that when  $0\leq t\leq 1,  \frac{d}{2}+3<s $ and $ \phi_0 \in H^{s} (\mathbb{R}^d) \cap L^2(\langle x\rangle^N \mathrm{d}x) $  small enough,  the \eqref{aqNLS} has a solution  in $ H^{s} (\mathbb{R}^d) \cap L^2(\langle x\rangle^N \mathrm{d}x)$  for $ t \in [0,1] $ and then by \eqref{aqbound1} we  can let $ \varepsilon\rightarrow 0 $ in \eqref{aqNLS} to obtain the solution of \eqref{qNLS}. 
 	
 		Now we prove uniqueness and continuity. Consider the following two IVP with different initial data
 	\begin{align}	 \label{qN1}
 		\begin{cases}
 			\sqrt{-1}\partial_t \phi_1+\partial_i\left( g^{ij}\left(\phi_1,\overline{\phi}_1\right) \partial_j\phi_1\right)=F\left(\phi_1,\overline{\phi}_1,\nabla\phi_1, \nabla \overline{\phi}_1\right),\\
 			\phi_1(0,x)=\phi_{10}(x),
 		\end{cases}\\
 		\label{qN2}
 		\begin{cases}
 			\sqrt{-1}\partial_t \phi_2+\partial_i\left( g^{ij}\left(\phi_2,\overline{\phi}_2\right) \partial_j\phi_2\right)=F\left(\phi_2,\overline{\phi}_2,\nabla\phi_2, \nabla \overline{\phi}_2\right),\\
 			\phi_2(0,x)=\phi_{20}(x)
 		\end{cases}
 	\end{align}
 	with solutions $ \phi_1, \phi_2 $ satisfying
 	\begin{align} \label{sol1-2}
 		% 	&	\left\|\langle x_j\rangle^2  F_z   \right\|_{L^\infty(\mathbb{R}^d)}+\left\| \langle x_j\rangle^2 F_{\overline{z}}   \right\|_{L^\infty(\mathbb{R}^d)}+\left\| \langle x_j\rangle^2\partial_k h^{ij}   \right\|_{L^\infty(\mathbb{R}^d)} \nonumber\\
 		&\max\left\{ \| \langle x_i \rangle^2  \partial_j \phi_l\|_{L^\infty},\| \langle x_i \rangle  \partial_{ij} \phi_l\|_{L^\infty} \right \}\lesssim  Y (\phi_l)(t)\lesssim Y (\phi_l)(0)\ll 1,~~l=1,2.
 	\end{align}
 Subtracting \eqref{qN2} from \eqref{qN1} gives
 	\begin{align} \label{qNLS1-2}
 		\left\{\begin{array}{l}
 			\sqrt{-1}\partial_t (\phi_1-\phi_2)+\partial_i\left( g^{ij}\left(\phi_1,\overline{\phi}_1\right) \partial_j(\phi_1-\phi_2)\right)=  N_{uni},\\
 			%	\partial_i\left( (g^{ij}\left(\phi_1,\overline{\phi}_1\right)-g^{ij}\left(\phi_2,\overline{\phi}_2\right)) \partial_j\phi_2\right)\\
 			%\qquad\qquad\qquad\qquad	=F\left(\phi_1,\overline{\phi}_1,\nabla\phi_1, \nabla \overline{\phi}_1\right)-F\left(\phi_2,\overline{\phi}_2,\nabla\phi_2, \nabla \overline{\phi}_2\right),\\
 			\phi_1(0,x)-	\phi_2(0,x)=\phi_{10}(x)-\phi_{20}(x),
 		\end{array}\right.
 	\end{align}
 	where
 	$$ N_{uni}:= F\left(\phi_1,\overline{\phi}_1,\nabla\phi_1, \nabla \overline{\phi}_1\right)-F\left(\phi_2,\overline{\phi}_2,\nabla\phi_2, \nabla \overline{\phi}_2\right)-\partial_i\left( (g^{ij}\left(\phi_1,\overline{\phi}_1\right)-g^{ij}\left(\phi_2,\overline{\phi}_2\right)) \partial_j\phi_2\right). $$
 	Let $ v=\phi_1-\phi_2, \phi_0=\phi_{10}(x)-\phi_{20}(x) $. Following the procedure in Subsection \ref{s2.3}, we multiply \eqref{qNLS1-2} by $\partial_k \overline{v}$, integrate over $\mathbb{R}^{d-1}$, and obtain
 	\begin{align} \label{consm1-2}
 		&	-\partial_t \int_{\mathbb{R}^{d-1}}  \mathrm{Im}\left(v \partial_k \overline{v} \right) \mathrm{d} \hat{x}_k-\partial_{kk} \int_{\mathbb{R}^{d-1}} \mathrm{Re}\left( v   g^{kj} \left(\phi_1,\overline{\phi}_1\right) \overline{v}_{ j}\right) \mathrm{d} \hat{x}\nonumber\\
 		&+2 \partial_k \int_{\mathbb{R}^{d-1}} \mathrm{Re}\left(\partial_k v  g^{kj} \left(\phi_1,\overline{\phi}_1\right) \overline{v}_{  j}\right) \mathrm{d} \hat{x} + \int_{\mathbb{R}^{d-1}}  \mathrm{Re}\left(\partial_{i} v \partial_k h^{ij} \left(\phi_1,\overline{\phi}_1\right) \partial_j \overline{v} \right)  \mathrm{d} \hat{x} \\
 		=&\,- \partial_k\int_{\mathbb{R}^{d-1}}   \mathrm{Re}\left( \overline{v}  N_{uni} \right)\mathrm{d} \hat{x}+2\int_{\mathbb{R}^{d-1}} \mathrm{Re} \left( N_{uni} \partial_k \overline{ v} \right)   \mathrm{d} \hat{x}. \nonumber
 	\end{align}
 	Multiplying \eqref{consm1-2} by $ \dfrac{x_k}{\langle x_k\rangle} $ and  integrating on $ [0,t]\times\mathbb{R} $ gives
 	\begin{align} \label{consm1-2_1}
 		\int_0^t\int_{ \mathbb{R}}   \frac{x_k}{\langle x_k\rangle}  \cdot\eqref{consm1-2}   \mathrm{d} x_k \mathrm{d} s.
 	\end{align}
 	Taking the good sign and summing over $ k $, we can get the good term
 	$ 	X_v(t):= \sum_{k=1}^d   \int_0^t     \int_{\mathbb{R}^{d}} \left|\frac{\partial_k v  } {\langle x_k\rangle}\right|^2  \mathrm{d} x \mathrm{d} s  $.
 	Most  terms in \eqref{consm1-2_1} handled as in Subsection \ref{s3.1}, except those involving $N_{\text{uni}}$. The mean value theorem gives
 	\begin{align*}
 		&g^{ij}\left(\phi_1,\overline{\phi}_1\right)-g^{ij}\left(\phi_2,\overline{\phi}_2\right)=g^{ij}\left(\phi_1,\overline{\phi}_1\right)-g^{ij}\left(\phi_2,\overline{\phi}_1\right)+g^{ij}\left(\phi_2,\overline{\phi}_1\right)-g^{ij}\left(\phi_2,\overline{\phi}_2\right)\\
 		%	=&\, g^{ij}_y\left(\theta_1\phi_1+(1-\theta_1)\phi_2,\overline{\phi}_1\right) (\phi_1- \phi_2)
 		%	+ g^{ij}_{\overline{y}}\left(\phi_2,\theta_2\overline{\phi}_1+(1-\theta_2)\overline{\phi}_2\right) (\overline{\phi}_1-\overline{\phi}_2)\\
 		=&\,h^{ij}_y (\phi_1- \phi_2)
 		+ h^{ij}_{\overline{y}} (\overline{\phi}_1-\overline{\phi}_2),
 	\end{align*}
 	where 
 	\begin{align} \label{mean_h}
 		h^{ij}_y= h^{ij}_y\left(\theta_1\phi_1+(1-\theta_1)\phi_2,\overline{\phi}_1\right),&\quad h^{ij}_{\overline{y}}=h^{ij}_{\overline{y}}\left(\phi_2,\theta_2\overline{\phi}_1+(1-\theta_2)\overline{\phi}_2\right)\\ 
 		\theta_1, \theta_2 \in (0,1),\qquad&\| h^{ij}_y\|_{L^\infty} +\| h^{ij}_{\overline{y}}\|_{L^\infty}  \leq C. \nonumber
 	\end{align}
 	Similarly, for $ \theta_i \in (0,1),i=3,\dots,6, $ there is
 	\begin{align*}
 		&F\left(\phi_1,\overline{\phi}_1,\nabla\phi_1, \nabla \overline{\phi}_1\right)-F\left(\phi_2,\overline{\phi}_2,\nabla\phi_2, \nabla \overline{\phi}_2\right)\\
 		=&\,F_y \left(\theta_3\phi_1+(1-\theta_3)\phi_2,\overline{\phi}_1,\nabla\phi_1, \nabla \overline{\phi}_1\right) (\phi_1-\phi_2)\\
 		&+\dots+ F_{\overline{z}} \left(\phi_2,\overline{\phi}_2,\nabla\phi_2,\theta_6 \nabla\phi_1+(1-\theta_6)\nabla \overline{\phi}_2\right) (\nabla\phi_1-\nabla\phi_2).
 	\end{align*}
 	Thus,
 	\begin{align} \label{momterm1}
 		&\int_0^t \int_{\mathbb{R}}\frac{x_k}{\langle x_k\rangle} 	\int_{\mathbb{R}^{d-1}} \mathrm{Re} \left(\partial_i\left( (g^{ij}\left(\phi_1,\overline{\phi}_1\right)-g^{ij}\left(\phi_2,\overline{\phi}_2\right)) \partial_j\phi_2\right)\partial_k \overline{v} \right)   \mathrm{d} \hat{x}  \mathrm{d} x_k \mathrm{d}s \\
 		=&\, \mathrm{Re}  \int_0^t \int_{\mathbb{R}}\frac{x_k}{\langle x_k\rangle} 	\int_{\mathbb{R}^{d-1}} \partial_i\left( h^{ij}_y  \left(\phi_1 -  \phi_2 \right)+h^{ij}_{\overline{y}} (\overline{\phi}_1-\overline{\phi}_2)\right) \partial_j\phi_2\partial_k \overline{v}  \mathrm{d} \hat{x}  \mathrm{d} x_k \mathrm{d}s \nonumber\\
 		&+ \mathrm{Re}\int_0^t \int_{\mathbb{R}}\frac{x_k}{\langle x_k\rangle} 	\int_{\mathbb{R}^{d-1}}   \left( h^{ij}_y  \left(\phi_1 -  \phi_2 \right)+h^{ij}_{\overline{y}} (\overline{\phi}_1-\overline{\phi}_2)\right)\partial_{ij}\phi_2\partial_k \overline{v}   \mathrm{d} \hat{x}  \mathrm{d} x_k \mathrm{d}s\nonumber \\
 		=&\,\mathrm{II}_1+\mathrm{II}_2\nonumber
 	\end{align}
 	and  it follows from H\"older and Sobolev that
 	\begin{align*}
 		%&	\mathrm{II}_1+ \lesssim  \sum_j \| \langle x_i \rangle^2  \partial_j \phi_2\|_{L^\infty} \cdot \left(\sum_{k=1}^d   \int_0^t     \int_{\mathbb{R}^{d}} \left|\frac{\partial_k v  } {\langle x_k\rangle}\right|^2  \mathrm{d} x \mathrm{d} s+ t^{\frac12} \|v\|_{L^2}\cdot \left(\sum_{k=1}^d   \int_0^t     \int_{\mathbb{R}^{d}} \left|\frac{\partial_k v  } {\langle x_k\rangle}\right|^2  \mathrm{d} x \mathrm{d} s\right)^{\frac 12}\right),\\
 		&	\mathrm{II}_1 \lesssim  \sum_j \| \langle x_i \rangle^2  \partial_j \phi_2\|_{L^\infty} \cdot (	X_v(t)+ t^{\frac12} \|v\|_{L^2}\cdot	X_v^{\frac12} (t) ),\\
 		& \mathrm{II}_2 \lesssim \sum_j \| \langle x_i \rangle  \partial_{ij} \phi_2\|_{L^\infty} \cdot t^{\frac12} \|v\|_{L^2}\cdot	X_v^{\frac12}(t) .
 	\end{align*} 
 	Similarly,
 	\begin{align*}
 		&\int_0^t \int_{\mathbb{R}}\frac{x_k}{\langle x_k\rangle} 	\int_{\mathbb{R}^{d-1}} \mathrm{Re} \left( \left(F\left(\phi_1,\overline{\phi}_1,\nabla\phi_1, \nabla \overline{\phi}_1\right)-F\left(\phi_2,\overline{\phi}_2,\nabla\phi_2, \nabla \overline{\phi}_2\right)\right)\partial_k \overline{v} \right)   \mathrm{d} \hat{x}  \mathrm{d} x_k \mathrm{d}s \nonumber\\
 		\lesssim &\,  \sum_j \| \langle x_i \rangle^2  \partial_j \phi_2\|_{L^\infty} \cdot (	X_v(t)+ t^{\frac12} \|v\|_{L^2}\cdot	X_v^{\frac12} (t) ).
 	\end{align*}
 The first term on the right-hand side of \eqref{consm1-2} is easier to bound than the second. 
 Combining these estimates, we obtain
 	\begin{align} \label{qunimon1}
 		X_v(t) \lesssim \|\phi_0\|_{H^{\frac12}}+(1+t) \cdot\|v(t)\|_{H^{\frac12}}+ \sum_j \| \langle x_i \rangle^2  \partial_j \phi_2\|_{L^\infty} \cdot (	X_v(t)+ t^{\frac12} \|v\|_{L^2}\cdot	X_v^{\frac12} (t) ).
 	\end{align}

 	For energy estimates, multiplying \eqref{qNLS1-2} by $ \overline{\phi}_1-\overline{\phi}_2, $ $ D_k(\overline{\phi}_1-\overline{\phi}_2)  $ respectively,  taking the imaginary part and  integrating on $[0,t]\times \mathbb{R}^d $ gives
 	\begin{align} \label{quniener}
 		&\max \left\{\left\|   v\right\|_{L^2(\mathbb{R}^d)}, \left\|   D^{\frac 12} v\right\|_{L^{ 2}(\mathbb{R}^d)}\right\}  	\lesssim	 \|v\|_{H^{ \frac 12}(\mathbb{R}^{d})}\\
 		\lesssim &\,  \|\phi_0\|_{H^{\frac12}}+ \sum_j \| \langle x_i \rangle^2  \partial_j \phi_2\|_{L^\infty} \cdot (	X_v(t)+ t^{\frac12} \|v\|_{L^2}\cdot	X_v^{\frac12} (t) ).\nonumber
 	\end{align}
 	While the analysis of other terms is  almost the same as Subsection  \ref{s3.2}, the term involving
 	\begin{align*} 
% 		\label{enerterm1}
 		\partial_i\left( (g^{ij}\left(\phi_1,\overline{\phi}_1\right)-g^{ij}\left(\phi_2,\overline{\phi}_2\right)) \partial_j\phi_2\right)
 	\end{align*}
 	can be bounded in the same way as \eqref{momterm1}. 
 	%\begin{align}
 	%& \int_0^t \int_{\mathbb{R}^d}   \partial_i\left( (g^{ij}\left(\phi_1,\overline{\phi}_1\right)-g^{ij}\left(\phi_2,\overline{\phi}_2\right)) \partial_j\phi_2\right)\cdot D_k( \overline{\phi}_1-\overline{\phi}_2)\mathrm{d} x\mathrm{d} s\nonumber\\
 	%	&\lesssim.
 	%\end{align}
 	%
 	%T 
 	Combining  \eqref{qunimon1}, \eqref{quniener} and using the smallness in \eqref{sol1-2}, we can obtain estimates analogous to \eqref{qbound1} for  $ v $ and subsequently establish the uniqueness and continuity.  	
\end{proof}

\section{Large data and	quadratic interactions}  \label{s_big}
Before presenting the proof, we first introduce some necessary notations and preliminary lemmas. We decompose the function \( \phi \) as
\[
\phi = \sum_{m=1}^K \sum_{\lambda} P_{\lambda,\varpi_m} \phi, \quad    P_{\lambda, \varpi_m} \phi = P_\lambda \chi_m(D) \phi,
\]
where \( P_\lambda \) is the standard Littlewood-Paley projection operator,   \( \varpi_m \in \mathbb{S}^{d-1} \), and \( K = C \epsilon^{-(d-1)} \) with \( \epsilon > 0 \) being a small constant to be determined later. In general, for any \( \epsilon > 0 \), the sphere \( \mathbb{S}^{d-1} \) can be covered by \( K \) balls of radius \( \epsilon \) centered at \( \varpi_m \).
\( \chi_m(D) \) is defined via the Fourier transform by \( \mathcal{F}(\chi_m(D)\phi)(\xi) = \chi_m(\xi) \widehat{\phi}(\xi) \), where the support of \( \chi_m(\xi) \) satisfies
\[
\mathrm{Supp}(\chi_m) \subset \left\{ \xi \;\big|\; \epsilon |\varpi_m \cdot \xi| \gtrsim |\varpi_m \times \xi| \right\}
\]
and the cross product \( \varpi_m \times \xi \) is defined component-wise as
\[
(\varpi_m \times \xi)_{ij} = \varpi_{mi} \xi_j - \varpi_{mj} \xi_i.
\]
We can observe that the following approximate orthogonality relation holds for the \( L^2 \)-norms of \( P_\lambda \phi \):  
\begin{align}
\left\| P_\lambda \phi \right\|_{L^{2}(\mathbb{R}^d)}^2 \sim \sum_{m=1}^K \left\| P_{\lambda, \varpi_m} \phi \right\|_{L^{2}(\mathbb{R}^d)}^2. \label{orth_relat}
\end{align}

Set $ s_2 >\frac{d}{2}+3$   and write the nonlinear term $ F(\phi, \overline{\phi}, \nabla \phi, \nabla \overline{\phi}) $ in \eqref{qNLS} into
\begin{align*}
	F(\phi, \overline{\phi}, \nabla \phi, \nabla \overline{\phi})=  &\,F(\phi, \overline{\phi}, P_{\leq 1} \nabla \phi,  P_{\leq 1} \nabla \overline{\phi})\\
	&+\sum_{ \mu}  \left(F(\phi, \overline{\phi}, P_{\leq \mu} \nabla \phi,  P_{\leq \mu}  \nabla \overline{\phi})-F(\phi, \overline{\phi}, P_{\leq  \mu/2} \nabla \phi,  P_{\leq  \mu/2}  \nabla \overline{\phi})\right)\\
	=&\, F(\phi, \overline{\phi}, P_{\leq 1} \nabla \phi,  P_{\leq 1} \nabla \overline{\phi})+\sum_{\mu} a \cdot P_\mu \nabla \phi + \sum_{\mu} b\cdot P_\mu \nabla\overline{ \phi },
\end{align*}
where $ \mu=2^n, n\geq 2, n \in \mathbb{N} $. By Hadamard's formula, we have
\begin{align} 
	a:= \int_0^1   F_z (\phi, \overline{ \phi }, s  P_{\leq  \mu}  \nabla \phi +(1-s) P_{\leq  \mu/2}  \nabla \phi,  s  P_{\leq  \mu}  \nabla \overline{\phi}+(1-s) P_{\leq  \mu/2}  \nabla \overline{\phi}) \mathrm{d} s,\label{a}\\
	b:= \int_0^1  F_{\overline{z}} (\phi, \overline{ \phi }, s  P_{\leq  \mu}  \nabla \phi +(1-s) P_{\leq  \mu/2}  \nabla \phi,  s  P_{\leq  \mu}  \nabla \overline{\phi}+(1-s) P_{\leq  \mu/2}  \nabla \overline{\phi}) \mathrm{d} s. \label{b}
\end{align}
Applying $ P_{\lambda, \varpi_m}   $ to $ F $, we obtain
\begin{align*}
	P_{\lambda, \varpi_m}   F(\phi, \overline{\phi}, \nabla \phi, \nabla \overline{\phi})=  &  P_{\lambda, \varpi_m}  F(\phi, \overline{\phi}, P_{\leq 1} \nabla \phi,  P_{\leq 1} \nabla \overline{\phi}) \\
	&+ \sum_{\mu` } [  P_{\lambda, \varpi_m},a  ]  P_\mu \nabla \phi + \sum_{\mu }  [P_{\lambda, \varpi_m},b ]  P_\mu \nabla\overline{ \phi }\\
	&+\sum_{\mu\sim \lambda} a \cdot  P_{\lambda, \varpi_m}   P_\mu \nabla \phi + \sum_{\mu\sim \lambda} b\cdot  P_{\lambda, \varpi_m}   P_\mu \nabla\overline{ \phi }.
\end{align*}
Then,
applying $  P_{\lambda, \varpi_m}   $ to \eqref{qNLS}  yields
\begin{align} \label{PqNLS}
	\sqrt{-1}  P_{\lambda, \varpi_m} \phi_t+\partial_i\left( g^{ij}\left(\phi,\overline{\phi}\right)   \partial_j (P_{\lambda, \varpi_m} \phi)\right)=\mathcal{N}_1,
\end{align}
where 
\begin{align}
	\mathcal{N}_1:=&\sum_{\mu\sim \lambda} a \cdot  P_{\lambda, \varpi_m}   P_\mu \nabla \phi + \sum_{\mu\sim \lambda} b\cdot  P_{\lambda, \varpi_m}   P_\mu \nabla\overline{ \phi }+\mathcal{R}, \\
	\mathcal{R}:=& \,P_{\lambda, \varpi_m}  F(\phi, \overline{\phi}, P_{\leq 1} \nabla \phi,  P_{\leq 1} \nabla \overline{\phi}) + \sum_{\mu\sim \lambda} [  P_{\lambda, \varpi_m},a  ]  P_\mu \nabla \phi + \sum_{\mu}  [P_{\lambda, \varpi_m},b ]  P_\mu \nabla\overline{ \phi } \nonumber\\
	&-\partial_i\left(\left[P_{\lambda, \varpi_m}, h^{ij}\left(\phi,\overline{\phi}\right) \right]  \partial_j  \phi \right).  \label{remainderterm}
\end{align}

Now let $ \phi=v+w $, where $ v $ satisfies   
\begin{align} \label{eq_v}
	\begin{cases}	\sqrt{-1}v_t+\Delta  v=0,\\
		v(0,x)=\phi_0(x).
	\end{cases}
\end{align}
The function  $v$ is then given explicitly by
\begin{align} \label{sol_v}
	v(t,x)=& e^{\sqrt{-1} t \Delta }    \phi_0=\left(e^{- 	\sqrt{-1} t |\xi|^2 } \widehat{   \phi_0}\right)^{\vee}(x).
\end{align}
As will be shown in Subsection \ref{s_big.3}, the remainder term  $w$ remains small in a suitable sense.

As outlined in the Introduction \ref{s1}, we will prove the main result via bootstrap arguments. We first introduce some notations and settings related to the initial data and time $t$.
\begin{align}
	&  \left\| \phi_0\right\|_{H^{s_2}} +1 = M_0, \label{ini1}\\
	&   \max\{\left\| \langle x\rangle^{N}\phi_0  \right\|_{L^{2}(\mathbb{R}^d)}, \left\| \langle x\rangle^2 \phi_0  \right\|_{H^{s_2-\frac 32}(\mathbb{R}^d)},\left\| \phi_0\right\|_{H^{s_2+\frac 12}} \}+1= M_1, \label{ini2}
\end{align}
where $ s_2< N $, $ 1\ll M_0 \leq M_1  $ and $ s_2 >\frac{d}{2}+3 $.
We fix a time $t > 0$ and constants $M$, $\epsilon$ such that
\begin{align}
	0<	t\ll \epsilon \ll M^{-1} \ll M_1^{-1}. \label{def_M}
\end{align}
Specifically, we can set
\begin{align}
	 t \leq C \min \{\varepsilon^{4(d-1)}, 
 e^{-48M}\},\label{def_t1}\\
 M^2\leq  Ce^{\frac M2}, \epsilon \leq Ce^{-6M},\label{def_t2}
\end{align}
where  the $ C $ is a positive constant  related to the estimates below, but is independent of the initial data. By \eqref{def_M} and \eqref{def_t2}, we have
\begin{align*}
     M_1^4\leq C e^M, \epsilon \leq CM_1^{-12}.
\end{align*}

We now state the bootstrap assumptions:
\begin{align}
	\left\| \phi (t,\cdot) \right\|_{H^{s_2}} &\leq M_0^2, \label{bstrap_ass1}\\
	\max \left\{\left\|\langle x\rangle^2 w(t,\cdot) \right\|_{H^{s_2-3}},	\left\| w(t,\cdot) \right\|_{H^{s_2-1}} \right\}&\leq t^{\frac 18}, \label{bstrap_ass2}\\
	\sum_{\lambda} \lambda^{2s_2}\sup_{\tau}\left\|  e^{M \frac{\varpi_m\cdot x}{2\langle \varpi_m\cdot x\rangle} }  \frac{\varpi_m \cdot\nabla P_{\lambda, \varpi_m}  \phi^{\tau}}{\langle \varpi_m\cdot x\rangle} \right\|_{L^2_{t,x}}^2 &\leq e^{4M}, \text{  for  } m =1,\dots ,K, \label{bstrap_ass3}\\
	\sum_{\lambda} \lambda^{2s_2+1}\sup_{\tau} \left\| P_{\lambda, \varpi_m}  \phi^\tau\right\|_{L^2(\mathbb{R}^d)}^2 &\leq M_1^4,\label{bstrap_ass4}\\
	\left\|\phi\right\|_{H^{s_2+\frac 12}}^2  &\leq  \epsilon^{-(d+1)}.\label{bstrap_ass5}
\end{align}
Our goal is to improve these to the following bounds:
\begin{align}
	\left\| \phi (t,\cdot) \right\|_{H^{s_2}} &\lesssim M_0, \label{bstrap_re1}\\
	\max \left\{\left\|\langle x\rangle^2 w(t,\cdot) \right\|_{H^{s_2-3}},	\left\| w(t,\cdot) \right\|_{H^{s_2-1}} \right\} &\lesssim t^{\frac 14}, \label{bstrap_re2}\\
	\sum_{\lambda} \lambda^{2s_2}\sup_{\tau}\left\|  e^{M \frac{\varpi_m\cdot x}{2\langle \varpi_m\cdot x\rangle} }  \frac{\varpi_m \cdot\nabla P_{\lambda, \varpi_m}  \phi^{\tau}}{\langle \varpi_m\cdot x\rangle} \right\|_{L^2_{t,x}}^2 &\leq e^{2M}, \text{  for  } m =1,\dots ,K,\label{bstrap_re3}\\
	\sum_{\lambda} \lambda^{2s_2+1} \sup_{\tau}\left\| P_{\lambda, \varpi_m}  \phi^\tau\right\|_{L^2(\mathbb{R}^d)}^2 &\lesssim M_1^2,\label{bstrap_re4}\\
	\left\|\phi\right\|_{H^{s_2+\frac 12}}^2  &\lesssim  \epsilon^{-d} \label{bstrap_re5}.
\end{align}
We also establish the following weighted estimates for $\phi$.
\begin{lem}
	Suppose we have \eqref{bstrap_ass1}. Then for the solution of \eqref{qNLS} with initial data satisfying \eqref{ini1} and \eqref{ini2}, the following estimates hold for sufficiently small $t$: 
	\begin{align}
		\left\|\langle x\rangle^{2}\phi\right\|_{H^{s_2-\frac 32}(\mathbb{R}^d)} \lesssim &\, M_1^2,\label{big_w_est1}\\
		\left\|  \langle x\rangle^2\phi (t,\cdot) \right\|_{L^\infty}\lesssim &\, M_1^2,\label{big_w_est2}\\
		\left\|  \langle x\rangle^2 \partial_k\phi (t,\cdot) \right\|_{L^\infty}\lesssim &\, M_1^2, \label{big_w_est3}
	\end{align}
	where $ k=1,2,\dots, d; s_2>\frac d2+3. $
 
 Moreover, for $a,h$ defined as \eqref{a}, \eqref{cgF} respectively, we have
 \begin{align}
     \left\|  \langle x\rangle^2 h  (t,\cdot) \right\|_{L^\infty}\lesssim &\, M_1^2, \label{big_w_est_h}\\
\left\|  \langle x\rangle^2 a  (t,\cdot) \right\|_{L^\infty}\lesssim &\, M_1^2. \label{big_w_est_a}
 \end{align}
\end{lem}
\begin{proof} Observe that
	\begin{align*}
		\left\|  \langle x\rangle^2 \partial_k\phi (t,\cdot) \right\|_{L^\infty} \lesssim  	\left\|\partial_k(  \langle x\rangle^2 \phi (t,\cdot) )\right\|_{L^\infty} +	\left\|  \langle x\rangle \phi (t,\cdot) \right\|_{L^\infty}.
	\end{align*}
	Thus, \eqref{big_w_est2} and \eqref{big_w_est3} follow directly from \eqref{big_w_est1} and the Sobolev embedding.  Meanwhile, it is evident that the \eqref{big_w_est_h} and \eqref{big_w_est_a} follows from \eqref{big_w_est1} and Lemma \ref{Lem_phi_infty}.

		We will prove
	\begin{align}
		\sum_{\lambda \geq 1} \left\| \langle x\rangle^{2} \lambda^{s_2 - 2} P_\lambda \phi \right\|_{L^{2}(\mathbb{R}^d)} 
		\lesssim M_1^2. \label{big_w_est4}
	\end{align}
	 in Subsection \ref{s_big.2}. Assuming \eqref{big_w_est4} holds, we deduce
	\begin{align*}
		\eqref{big_w_est1} \lesssim & \sum_{\lambda\geq 1} \left\| \lambda^{s_2-2}P_\lambda (\langle x\rangle^{2}\phi)\right\|_{L^{2}(\mathbb{R}^d)}\\
		\lesssim&\, \sum_{\lambda\geq 1} \left\| \langle x\rangle^{2}  \lambda^{s_2-2}P_\lambda\phi\right\|_{L^{2}(\mathbb{R}^d)}+ \sum_{\lambda\geq 1} \left\| \lambda^{s_2-2}[P_\lambda,\langle x\rangle^{2}]\phi)\right\|_{L^{2}(\mathbb{R}^d)}\\
		\lesssim &\,M_1^2 +\sum_{\lambda\geq 1} \left\|\langle x\rangle \lambda^{s_2-3}P_\lambda^\prime \phi\right\|_{L^{2}(\mathbb{R}^d)}\\
		\lesssim&\,M_1^2+\left(\sum_{\lambda\geq 1} \left\|\langle x\rangle^2 \lambda^{s_2-2}P_\lambda^\prime \phi\right\|_{L^{2}(\mathbb{R}^d)}^2\right)^{\frac 12} \left(\sum_{\lambda\geq 1} \left\|  \lambda^{s_2-4}P_\lambda^\prime \phi\right\|_{L^{2}(\mathbb{R}^d)}^2\right)^{\frac 12}\\
		\lesssim & \,M_1^2+M_1 \left(\sum_{\lambda\geq 1} \left\|  \lambda^{s_2}P_\lambda^\prime \phi\right\|_{L^{2}(\mathbb{R}^d)}^2\right)^{\frac 12}\\
		\lesssim &\,M_1^2.
	\end{align*}	
\end{proof}

Without loss of generality, we restrict our attention to the case \( \varpi_m = e_1=(1, 0, \dots, 0) \), so that the operator under consideration becomes \( P_{\lambda, e_1} \phi \). For the general $ \varpi_m $, we can apply a rotation \(O_m\) satisfying \(O_m \varpi_m = (1, 0, \dots, 0)=e_1\) and work in the new coordinates $ O_m x$. With a little abuse of notation, we sometimes write  $  P_{\lambda, e_1}  \phi$ as $ \psi $  in the subsequent pages for notational convenience.

We now present several preliminary lemmas. The first is a weighted Bernstein-type estimate for \(P_{\lambda, e_1} \phi\), where the operator \(\chi_1(D)\) enables us to convert derivatives of \(P_{\lambda, e_1} \phi\) in the norm from \(\partial_j\) for \(j \neq 1\) to \(\partial_1\). The proof of Lemma \ref{lem_transf_k} is deferred to Appendix \ref{appa} for the reader’s convenience, 
\begin{lem}  \label{lem_transf_k}
	Let $   \psi =  P_{\lambda, e_1}  \phi$, $ \lambda \in \mathbb{N}$, $ j=2,\dots ,d  $,   and  $ h(x), \psi(x), x\in \mathbb{R}^d $ be smooth function, where $ \chi_m(D) $ is defined as before. Then we have
	\begin{align}
		\sup_{\tau} \int_{\mathbb{R}^{d}} \left| h(x) \right|^2 \left| \partial_j \psi^\tau    \right|^2   \mathrm{d} x  	\leq c_1 \epsilon^2	\sup_{\tau}   \int_{\mathbb{R}^{d}} \left| h(x) \right|^2 \left| \partial_1 \psi^\tau    \right|^2   \mathrm{d} x,   \label{est_transf_k}
	\end{align}
	where $ c_1 $ is a positive constant  independent of \(\epsilon\).

	Moreover, the following estimates hold:
	\begin{align}
		\sup_{\tau} \int_{\mathbb{R}^{d}} \left| h(x) \right|^2 \left| q(x,D) \psi^\tau    \right|^2   \mathrm{d} x  	\lesssim	\sup_{\tau}   \int_{\mathbb{R}^{d}} \left| h(x) \right|^2 \left|  \psi^\tau    \right|^2   \mathrm{d} x,  \label{est_transf_0_order_P}\\
		\sup_{\tau} \int_{\mathbb{R}^{d}} \left| h(x) \right|^2 \left| \lambda \psi^\tau    \right|^2   \mathrm{d} x  	\sim	\sup_{\tau}   \int_{\mathbb{R}^{d}} \left| h(x) \right|^2 \left|  \partial_k \psi^\tau    \right|^2   \mathrm{d} x,  \label{est_transf_k2}
	\end{align}
	where   $ q(x,D) $ is a zeroth order pseudo-differential operator.
\end{lem}

The next two lemmas are useful for handling compositions and commutators respectively.
\begin{lem} \label{Lem_phi_infty} \cite{taylor1996partial}
	Let $F$ be smooth, and assume $F(0)=0$. Then, for $u \in H^k \cap L^{\infty}$,
	\begin{align} \label{big_a_inf}
		\|F(u)\|_{H^k} \leq C \left(\|u\|_{L^{\infty}}\right)\left(1+\|u\|_{H^k}\right).
	\end{align}
\end{lem}
\begin{lem} \cite{lai2025global}\label{lem_com}
	Suppose  $ G,f,g $ are smooth functions. Then,
	\begin{align}
		\sup_{\tau, \tau^\prime} \Big\| G\big[P_\rho,f^\tau\big] g^{\tau^\prime}\Big\|_{L^p} \lesssim \rho^{-1} 	\sup_{\tau, \tau^\prime} \Big\| G\big|\nabla f^\tau\big| g^{\tau^\prime}\Big\|_{L^p}= \rho^{-1} 	\sup_{\tau, \tau^\prime} \Big\| G^\tau\big|\nabla f\big| g^{\tau^\prime}\Big\|_{L^p}. 
	\end{align}
\end{lem}

\subsection{Momentum type estimates} \label{s_big.1}
In this subsection, we prove \eqref{bstrap_re3}. As previously mentioned, we will only consider the case $\varpi_m=e_1=(1, 0, \dots, 0) $ in detail. This restriction is justified by the fact that  for any $m=1,\dots, K,$ the case for  $\varpi_m$ can be established in the same manner via a rotation.

As \eqref{consm2}, we have following momentum balance equalities  for \eqref{PqNLS} with $\varpi_m=e_1$:
\begin{align} \label{big_consm2}
	&	-\partial_t \int_{\mathbb{R}^{d-1}}  \mathrm{Im}\left( \psi \partial_k \overline{\psi}\right) \mathrm{d} \hat{x}_k-\partial_{kk} \int_{\mathbb{R}^{d-1}} \mathrm{Re}\left( \psi   g^{kj}(\phi,\overline{\phi}) \partial_j \overline{\psi}\right) \mathrm{d} \hat{x}_k\nonumber\\
	&+2 \partial_k \int_{\mathbb{R}^{d-1}} \mathrm{Re}\left(\partial_k \psi  g^{kj}(\phi,\overline{\phi}) \partial_j \overline{\psi}\right) \mathrm{d} \hat{x}_k + \int_{\mathbb{R}^{d-1}}  \mathrm{Re}\left(\partial_{i} \psi  \partial_k h^{ij}(\phi,\overline{\phi}) \partial_j \overline{\psi}\right)  \mathrm{d} \hat{x}_k \\
	&=- \partial_k\int_{\mathbb{R}^{d-1}}   \mathrm{Re}\left(  \overline{\psi} \mathcal{N}_1\right)\mathrm{d} \hat{x}_k+2\int_{\mathbb{R}^{d-1}} \mathrm{Re} \left(\mathcal{N}_1 \partial_k  \overline{\psi}\right)   \mathrm{d} \hat{x}_k \nonumber,
\end{align}
where we write $P_{\lambda,e_1} \phi $ as $  \psi $ for convenience.

For the large initial data problem, we choose the weight $e^{M \frac{x_1}{\langle x_1\rangle}}$, with $M$ defined in \eqref{def_M}. 
 Multiplying  \eqref{big_consm2} with  $ e^{M \frac{x_1}{\langle x_1\rangle} }$, integrating it on $ [0,t]\times\mathbb{R}_{x_1} $, then multiplying it by $ \lambda^{2 s_2} $ and summing it over $ \lambda $,  we have
\begin{align} 
	&- \sum_{\lambda} \lambda^{ 2s_2}	\left(  \int_{\mathbb{R}^{d}} e^{M \frac{x_1}{\langle x_1\rangle} }   \mathrm{Im}\left( \psi \partial_1 \overline{\psi}\right)(t)   \mathrm{d} x -\int_{\mathbb{R}^d}   e^{M \frac{x_1}{\langle x_1\rangle} }   \mathrm{Im}\left( \psi_0 \partial_1 \overline{\psi}_0\right)(t)   \mathrm{d} x  \right)\label{big_me_t1}\\
	&- \sum_{\lambda} \lambda^{ 2s_2}\int_0^t \int_{\mathbb{R}}   	e^{M \frac{x_1}{\langle x_1\rangle} }   \partial_{11} \int_{\mathbb{R}^{d-1}} \mathrm{Re}\left( \psi  g^{1j} (\phi,\overline{\phi})\partial_j \overline{\psi}\right) \mathrm{d} \hat{x}_1   \mathrm{d} x_1 \mathrm{d} s\label{big_me_t2}\\
	&+2 \sum_{\lambda} \lambda^{ 2s_2}  \int_0^t\int_{\mathbb{R}}    	e^{M \frac{x_1}{\langle x_1\rangle} }   \partial_1 \int_{\mathbb{R}^{d-1}} \mathrm{Re}\left(\partial_1 \psi g^{1j} (\phi,\overline{\phi})\partial_j \overline{\psi}\right) \mathrm{d} \hat{x}_1  \mathrm{d} x_1 \mathrm{d} s \label{big_me_t3} \\
	&+\sum_{\lambda} \lambda^{ 2s_2} \int_0^t	\int_{\mathbb{R}}e^{M \frac{x_1}{\langle x_1\rangle} }   \int_{\mathbb{R}^{d-1}}  \mathrm{Re}\left(\partial_{i} \psi  \partial_1 h^{1j}(\phi,\overline{\phi}) \partial_j \overline{\psi }\right) \mathrm{d} \hat{x}_1  \mathrm{d} x_1 \mathrm{d} s \label{big_me_t4}\\
	=&- \sum_{\lambda} \lambda^{ 2s_2}\int_0^t	\int_{\mathbb{R}}e^{M \frac{x_1}{\langle x_1\rangle} }   \partial_1\int_{\mathbb{R}^{d-1}}   \mathrm{Re}\left(  \overline{\psi} \mathcal{N}_1\right) \mathrm{d} \hat{x}_1  \mathrm{d} x_1 \mathrm{d} s\label{big_me_t5}\\
	&+2\sum_{\lambda} \lambda^{ 2s_2}\int_0^t	\int_{\mathbb{R}}e^{M \frac{x_1}{\langle x_1\rangle} }  \int_{\mathbb{R}^{d-1}}  \mathrm{Re} \left(\mathcal{N}_1 \partial_1  \overline{\psi}\right)   \mathrm{d} \hat{x}_1  \mathrm{d} x_1 \mathrm{d} s\label{big_me_t6},
\end{align}
where $  \psi=P_{\lambda,e_1}  \phi  $ and $ \psi_0 =P_{\lambda, e_1}  \phi_0. $

Through integration by parts, we can derive the good terms from  \eqref{big_me_t3}:
\begin{align}
	-  \eqref{big_me_t3}	 =&\,  2 \sum_{\lambda} \lambda^{ 2s_2}\int_0^t\int_{\mathbb{R}}    \frac{ M	e^{M \frac{x_1}{\langle x_1\rangle} } }{\langle x_1\rangle^2}   \int_{\mathbb{R}^{d-1}} \mathrm{Re}\left(\partial_1 \psi g^{1j} (\phi,\overline{\phi})\partial_j \overline{\psi}\right) \mathrm{d} \hat{x}_1  \mathrm{d} x_1 \mathrm{d} s\nonumber \\
	=&  \,2 M\sum_{\lambda} \lambda^{ 2s_2}\int_0^t    \int_{\mathbb{R}^{d}} \left|	e^{M \frac{x_1}{2\langle x_1\rangle} }\cdot \frac{\partial_1 \psi }{\langle x_1\rangle}\right|^2   \mathrm{d} x  \mathrm{d} s  \label{big_me_t3_t1}\\
	&+2 \sum_{\lambda} \lambda^{ 2s_2}\int_0^t\int_{\mathbb{R}}    \frac{ M	e^{M \frac{x_1}{\langle x_1\rangle} } }{\langle x_1\rangle^2}    \int_{\mathbb{R}^{d-1}} \mathrm{Re}\left(\partial_1 \psi h^{1j} (\phi,\overline{\phi})\partial_j \overline{\psi}\right) \mathrm{d} \hat{x}_1  \mathrm{d} x_1 \mathrm{d} s.  \label{big_me_t3_t2}
\end{align}
Both \eqref{big_me_t3_t2} and \eqref{big_me_t4} can be estimated by the non-trapping condition;  we will therefore only analyze \eqref{big_me_t4}.
The triangle inequality gives 
\begin{align}
	\eqref{big_me_t4} \leq & \sum_{\lambda} \lambda^{ 2s_2} \int_0^t\int_{\mathbb{R}^d}  e^{M \frac{x_1}{\langle x_1\rangle} }   |\partial_i \psi| \left|\partial_1 h^{ij} (v+w,\overline{v}+\overline{w})-\partial_1  h^{ij} (v,\overline{v})\right| |\partial_j\overline{\psi } |  \mathrm{d} x  \mathrm{d} s \label{big_me_t4_1}\\
	&+\sum_{\lambda} \lambda^{ 2s_2}\int_0^t\int_{\mathbb{R}^d}     e^{M \frac{x_1}{\langle x_1\rangle} }  |\partial_i \psi| \left|\partial_1 h^{ij} (v ,\overline{v} )-\partial_1 h^{ij} (\phi_0,\overline{\phi}_0)\right| |\partial_j\overline{\psi}|  \mathrm{d} x  \mathrm{d} s  \label{big_me_t4_2}\\
	&+\sum_{\lambda} \lambda^{ 2s_2} \int_0^t\int_{\mathbb{R}^d}   e^{M \frac{x_1}{\langle x_1\rangle} }  |\partial_i \psi| \left| \partial_1 h^{ij} (\phi_0,\overline{\phi}_0)\right| |\partial_j\overline{\psi}|  \mathrm{d} x  \mathrm{d} s.  \label{big_me_t4_3}
\end{align}
Let $h(x) = \frac{ e^{M \frac{x_1}{2\langle x_1\rangle} } }{\langle x_1\rangle}$ in Lemma \ref{lem_transf_k}. Using \eqref{est_transf_k} and the mean value theorem, we estimate
\begin{align*}
	\eqref{big_me_t4_1} \leq & \sum_{\lambda} \lambda^{ 2s_2}\int_0^t\int_{\mathbb{R}^d}   e^{M \frac{x_1}{\langle x_1\rangle} }    |\partial_i \psi| \left|\partial_1 h^{ij}_y w+\partial_1 h^{ij}_{\overline{y}}\overline{w}\right| |\partial_j \overline{\psi} |  \mathrm{d} x  \mathrm{d} s\\
	\leq &\sum_{\lambda} \lambda^{ 2s_2} \left(\left\|\partial_1 h^{ij}_y\right\|_{L^\infty}+\left\|\partial_1 h^{ij}_{\overline{y}}\right\|_{L^\infty}\right)     \left\|\langle x_1\rangle^2 w\right\|_{L^\infty}\left\| e^{M \frac{x_1}{2\langle x_1\rangle} } \frac{\partial_i \psi  }{\langle x_1\rangle }  \right\|_{L^2_{t,x}} \left\| e^{M \frac{x_1}{2\langle x_1\rangle} } \frac{\partial_j \psi  }{\langle x_1\rangle }  \right\|_{L^2_{t,x}}\\
	\leq &\, C\left\|\langle x_1\rangle^2 w\right\|_{H^{s_2-3}}\sum_{\lambda} \lambda^{ 2s_2}  \sup_{\tau} \left\| e^{M \frac{x_1}{2\langle x_1\rangle} } \frac{\partial_1 \psi^\tau }{\langle x_1\rangle }  \right\|_{L^2_{t,x}}^2\\
	\leq & \, Ct^{\frac 18}  \sum_{\lambda} \lambda^{ 2s_2}\sup_{\tau}\left\| e^{M \frac{x_1}{2\langle x_1\rangle} } \frac{\partial_1 \psi^\tau }{\langle x_1\rangle }  \right\|_{L^2_{t,x}}^2,
\end{align*}
where $ C $ is concerning with $ c_1 $ in Lemma \ref{lem_transf_k}, $h^{ij}_y, h^{ij}_{\overline{y}}$ 
 are defined similarly as \eqref{mean_h}, and
 $$ \left(\left\|\partial_1 h^{ij}_y\right\|_{L^\infty}+\left\|\partial_1 h^{ij}_{\overline{y}}\right\|_{L^\infty}\right) \leq C, $$ since  $ h^{ij} $ is smooth function.

Similarly, applying Lemma \ref{lem_transf_k}, the mean value theorem, and \eqref{sol_v}, we obtain
\begin{align*}
	\eqref{big_me_t4_2} \leq &\sum_{\lambda} \lambda^{ 2s_2}\int_0^t\int_{\mathbb{R}^d}   e^{M \frac{x_1}{\langle x_1\rangle} }       |\partial_i \psi| \left| \partial_1 h^{ij}_y (v-\phi_0)+\partial_1 h^{ij}_{\overline{y}}(\overline{v}-\overline{\phi}_0)\right| | \partial_j\overline{w} |  \mathrm{d} x  \mathrm{d} s\\
	\leq & \sum_{\lambda} \lambda^{ 2s_2}\int_0^t\int_{\mathbb{R}^d} e^{M \frac{x_1}{\langle x_1\rangle} }     |\partial_i \psi| \left|\partial_1h^{ij}_y ( e^{\sqrt{-1} t\Delta } \phi_0-\phi_0)+\partial_1 h^{ij}_{\overline{y}}( e^{-\sqrt{-1} t\Delta }\overline{\phi}_0 -\overline{\phi}_0)\right| |\partial_j\overline{w} |  \mathrm{d} x  \mathrm{d} s\\
	\leq & \,   \sum_{\lambda} \lambda^{ 2s_2} \left\|(e^{\sqrt{-1} t\Delta } \phi_0-\phi_0)\langle x_1\rangle^2\right\|_{L^{\infty}}\left\| e^{M \frac{x_1}{2\langle x_1\rangle} } \frac{\partial_i \psi  }{\langle x_1\rangle }  \right\|_{L^2_{t,x}} \left\| e^{M \frac{x_1}{2\langle x_1\rangle} } \frac{\partial_j \psi  }{\langle x_1\rangle }  \right\|_{L^2_{t,x}}\\
	\leq &\, C     \left\|(e^{\sqrt{-1} t\Delta } \phi_0-\phi_0)\langle x_1\rangle^2\right\|_{L^{\infty}}   \sum_{\lambda} \lambda^{ 2s_2}\sup_{\tau} \left\| e^{M \frac{x_1}{2\langle x_1\rangle} } \frac{\partial_1 \psi^\tau }{\langle x_1\rangle }  \right\|_{L^2_{t,x}}^2.
\end{align*}

Again, by the mean value theorem 
\begin{align*}
	\left\| x_1^2 (e^{\sqrt{-1} t\Delta } \phi_0 -\phi_0 ) \right\|_{L^{\infty}} \lesssim &\,\left\|x_1^2 (e^{\sqrt{-1} t\Delta } \phi_0 -e^{0 \cdot\sqrt{-1} \Delta }\phi_0) \right\|_{H^{s_2-3}} \\
	\lesssim &\,  \left\|\langle \xi \rangle^{s_2-3} \partial_{\xi_1}^2\left((e^{-\sqrt{-1} t|\xi|^2 } -e^{0\sqrt{-1} |\xi|^2 })\widehat{\phi_0  }  \right) \right\|_{L^{2}} \\
	\lesssim &\,  \left\|\langle \xi \rangle^{s_2-3} \left( \partial_{\xi_1}^2(e^{-\sqrt{-1} t|\xi|^2 } )\cdot\widehat{\phi_0  }  \right) \right\|_{L^{2}}+  \left\|\langle \xi \rangle^{s_2-3}  \left( \partial_{\xi_1}(e^{-\sqrt{-1} t|\xi|^2 } )\cdot  \partial_{\xi_1}\widehat{\phi_0  }  \right)\right\|_{L^{2}} \\ 
	&\,+\left\|\langle \xi \rangle^{s_2-3} \left((e^{-\sqrt{-1} t|\xi|^2 } -e^{0\sqrt{-1} |\xi|^2 })\partial_{\xi_1}^2\widehat{\phi_0  }  \right) \right\|_{L^{2}}  \\
	\lesssim & \,(t+t^2) \left\|\langle \xi \rangle^{s_2-3}  \langle \xi_1 \rangle^2   e^{-\sqrt{-1} \vartheta_1 t|\xi|^2 }   \widehat{\phi_0  } \right\|_{L^2}+t \left\|\langle \xi \rangle^{s_2}   \xi_1 e^{-\sqrt{-1}   t|\xi|^2 }   \partial_{\xi_1} \widehat{\phi_0  } \right\|_{L^2} \\
	&\,+t \left\|\langle \xi \rangle^{s_2-3}  |\xi|^2 e^{-\sqrt{-1} \vartheta_1 t|\xi|^2 }   \partial_{\xi_1}^2\widehat{\phi_0  } \right\|_{L^2}\\
	\lesssim&\, (t+t^2) \left\| \langle x_1 \rangle^2  \phi_0\right\|_{H^{s_2-1}},
\end{align*}
where $0< \vartheta_1<1, s_2>d/2+3 $.
Similarly, one has
\begin{align*}
	\left\|e^{\sqrt{-1} t\Delta } \phi_0-\phi_0\right\|_{L^{\infty}} \lesssim  t \left\|\phi_0\right\|_{H^{s_2-1}}.
\end{align*}
Combining the above two estimates yields
\begin{align}
	\left\| \langle x_1\rangle^2  (e^{\sqrt{-1} t\Delta } \phi_0 -\phi_0 ) \right\|_{L^{\infty}} \lesssim & \,  (t+t^2) \left\| \langle x_1 \rangle^2  \phi_0\right\|_{H^{s_2-1}}.
\end{align}

Now we consider \eqref{big_me_t4_3}. The H\"older's inequality and Lemma \ref{lem_transf_k} show
\begin{align*}
	\eqref{big_me_t4_3} \leq &   \left\|\langle x_1\rangle^2  \partial_1 h^{ij} (\phi_0,\overline{\phi}_0)\right\|_{L^\infty} \sum_{\lambda} \lambda^{ 2s_2}\left\|e^{M \frac{x_1}{2\langle x_1\rangle} }  \frac{\partial_i \psi }{\langle x_1\rangle }  \right\|_{L^2_{t,x}}^2\\
	\leq & C\left\|\langle x_1\rangle^2  \partial_1 h^{ij} (\phi_0,\overline{\phi}_0)\right\|_{L^\infty} \sum_{\lambda} \lambda^{ 2s_2} \sup_\tau\left\|e^{M \frac{x_1}{2\langle x_1\rangle} }  \frac{\partial_1 \psi^\tau }{\langle x_1\rangle }  \right\|_{L^2_{t,x}}^2
\end{align*}
Recall the definition of \ref{def_nontrap}. 
For the part of outside  $ B/2 $, we have 
\begin{align*}
	\left\| \langle x_1\rangle^2 \partial_1  h^{ij} (\phi_0,\overline{\phi}_0)\right\|_{L^{\infty}((B/2)_{out})} \lesssim & \, \left( \left\|\partial_y h^{ij}\right\|_{L^\infty}+ \left\|\partial_{\overline{y}} h^{ij}\right\|_{L^\infty}\right) \left\| \langle x_1\rangle^2 \partial_1  \phi_0\right\|_{L^{\infty}((B/2)_{out})} 	\\
 \lesssim&\,\left\| \langle x_1\rangle^2\phi_0\right\|_{H^{s_2-3}((B/2)_{out})}\\
	\lesssim  &\left\| \langle x_1\rangle^2 \phi_0\right\|_{L^{2}((B/2)_{out})}^{\vartheta_2} \left\| \langle x_1\rangle^2 \phi_0\right\|_{H^{s_2-2}((B/2)_{out})}^{1-\vartheta_2}\\
	\lesssim  & \,\frac{1}{R^{ \vartheta_2}}\left\| \langle x \rangle^3\phi_0\right\|_{L^{2}((B/2)_{out})}^{\vartheta_2}  \left\| \langle x_1\rangle^2\phi_0\right\|_{H^{s_2-2}((B/2)_{out})}^{1-\vartheta_2}\\
	\lesssim & \,\frac{M_1}{R^{\vartheta_2}}\\
	\lesssim & \,\eta,
\end{align*}
where $ \eta  $  is small enough and we set $ R= \big(\frac{M_1}{\eta}\big)^{\frac{1}{  \vartheta_2}} $, $ 1<\vartheta_2<1$. Since bicharacteristics for $\Delta_{g\left(\psi_0, \overline{\psi}_0\right)}$  cannot  re-enter $B / 2$, we can freely assume that
\begin{align*}
	\left\| \chi_{R/2}  \partial_1 h^{ij} (\phi_0,\overline{\phi}_0)  \right\|_{L^\infty (\mathbb{R}^d)} \lesssim \eta .
\end{align*}
Then,
\begin{align*}
	\left\| \chi_{R/2} \langle x_1\rangle^2   \partial_1 h^{ij} (\phi_0,\overline{\phi}_0)  \right\|_{L^\infty (\mathbb{R}^d)} \lesssim& \,  	\left\| \chi_{R/2} \langle x_1\rangle^3   \partial_1 h^{ij} (\phi_0,\overline{\phi}_0)  \right\|_{L^\infty (\mathbb{R}^d)}^{\frac 23}\left\| \chi_{R/2}   \partial_1 h^{ij} (\phi_0,\overline{\phi}_0)  \right\|_{L^\infty (\mathbb{R}^d)}^{\frac 13} \\
	\lesssim&  \,	\left\| \langle x_1\rangle^3   \partial_1 h^{ij} (\phi_0,\overline{\phi}_0)  \right\|_{L^\infty (\mathbb{R}^d)}^{\frac 23 } \left\| \chi_{R/2}   \partial_1 h^{ij} (\phi_0,\overline{\phi}_0)  \right\|_{L^\infty (\mathbb{R}^d)}^{\frac 13} \\
	\lesssim&  \,\eta^{\frac 13} 	\left\| \langle x_1\rangle^3   \partial_1 h^{ij} (\phi_0,\overline{\phi}_0)  \right\|_{H^{s_2-3} (\mathbb{R}^d)}^{\frac 23 }\\
	\lesssim&  \,\eta^{\frac 13} \left( \left\|\partial_y h^{ij}\right\|_{L^\infty}+ \left\|\partial_{\overline{y}} h^{ij}\right\|_{L^\infty}\right) \left\| \langle x_1\rangle^N   \phi_0  \right\|_{L^{2} (\mathbb{R}^d)}^{\mathscr{S}_1}	\left\|  \phi_0  \right\|_{H^{s_2-2} (\mathbb{R}^d)}^{\mathscr{S}_2}\\
	\lesssim & \, \eta^{\frac 13}  M_1^{\frac 23 },
\end{align*}
where $ \mathscr{S}_1+\mathscr{S}_2=\frac 23 $.

Combining the estimates of \eqref{big_me_t4_1}-\eqref{big_me_t4_3}, we have
\begin{align*}
	\eqref{big_me_t4} \lesssim &\,  \left(\eta+  \eta^{\frac 13}  M_1^{\frac 23 }\right)   \sum_{\lambda} \lambda^{ 2s_2} \sup_{\tau} \left\| e^{M \frac{x_1}{2\langle x_1\rangle} } \frac{\partial_1 \psi^\tau }{\langle x_1\rangle }  \right\|_{L^2_{t,x}}^2\nonumber\\
	&\,+  \left( t^{\frac 18}+(t+t^2) \left\| \langle x_1 \rangle^2  \phi_0\right\|_{H^{s_2+2}}\right)   \sum_{\lambda} \lambda^{ 2s_2} \sup_{\tau} \left\| e^{M \frac{x_1}{2\langle x_1\rangle} } \frac{\partial_1 \psi^\tau }{\langle x_1\rangle }  \right\|_{L^2_{t,x}}^2\\
	\lesssim &  \left( t^{\frac 18}+ \eta+  \eta^{\frac 13}  M_1^{\frac 23 }\right) \sum_{\lambda} \lambda^{ 2s_2}\sup_{\tau} \left\| e^{M \frac{x_1}{2\langle x_1\rangle} } \frac{\partial_1 \psi^\tau }{\langle x_k\rangle }  \right\|_{L^2_{t,x}}^2. 
\end{align*}

Since 
\begin{align*}
	\partial_1^2 (e^{M \frac{x_1}{\langle x_1\rangle} } )=& \,	\partial_1 \left( \frac{ M	e^{M \frac{x_1}{\langle x_1\rangle} } }{\langle x_1\rangle^2}  \right) \\
	=&  \, \frac{ M^2	e^{M \frac{x_1}{\langle x_1\rangle} } }{\langle x_1\rangle^4} +e^{M \frac{x_1}{\langle x_1\rangle} } \cdot M \frac{\mathrm{sgn}(x_1)}{\langle x_1\rangle^3}, 
\end{align*}
we can use integration by parts, Lemma \ref{lem_transf_k}, \eqref{bstrap_ass1}-\eqref{bstrap_ass5} to obtain
\begin{align*}  
	\eqref{big_me_t2}	=&\, 2 \sum_{\lambda} \lambda^{ 2s_2}\int_0^t \int_{\mathbb{R}}   	\left(\frac{ M^2	e^{M \frac{x_1}{\langle x_1\rangle} } }{\langle x_1\rangle^4} +e^{M \frac{x_1}{\langle x_1\rangle} } \cdot M \frac{\mathrm{sgn}(x_1)}{\langle x_1\rangle^3} \right) \int_{\mathbb{R}^{d-1}} \mathrm{Re}\left( \psi  \overline{\psi}_{\alpha k}\right) \mathrm{d} \hat{x}_1   \mathrm{d} x_1 \mathrm{d} s \\
	&+2\int_0^t\sum_{\lambda} \lambda^{ 2s_2} \int_{\mathbb{R}}    \left(\frac{ M^2	e^{M \frac{x_1}{\langle x_1\rangle} } }{\langle x_1\rangle^4} +e^{M \frac{x_1}{\langle x_1\rangle} } \cdot M\frac{\mathrm{sgn}(x_1)}{\langle x_1\rangle^3} \right) \int_{\mathbb{R}^{d-1}} \mathrm{Re}\left( \psi   h^{1j} (\phi,\overline{\phi}) \partial_j \overline{\psi}\right) \mathrm{d} \hat{x}_1   \mathrm{d} x_1 \mathrm{d} s\\
	\lesssim&\,  M^2 e^{\frac {M}{2}}\left(1+\left\|h^{kj} (\phi,\overline{\phi}) \right\|_{L^{\infty}(\mathbb{R}^d)}\right)\left(  \sum_{\lambda} \lambda^{ 2s_2}\int_0^t     \int_{\mathbb{R}^{d}} \left|   \psi   \right|^2  \mathrm{d} x \mathrm{d} s\right)^{\frac 12}  \\
	&\,\cdot \left(  \sum_{\lambda} \lambda^{ 2s_2} \int_0^t     \int_{\mathbb{R}^{d}} \left|e^{M \frac{x_1}{2\langle x_1\rangle} } \frac{\partial_j \psi  } {\langle x_1\rangle}\right|^2  \mathrm{d} x \mathrm{d} s\right)^{\frac 12}\\
	\lesssim&\, t^{\frac12} M^2  e^{\frac {M}{2}} \cdot  \left(1+\left\|\phi \right\|_{L^{\infty}(\mathbb{R}^d)}\right)\left\|    \phi\right\|_{H^{s_2}(\mathbb{R}^d)} \cdot  \left( \sup_{\tau}   \sum_{\lambda} \lambda^{ 2s_2} \left\| e^{M \frac{x_1}{2\langle x_1\rangle} } \frac{\partial_1 \psi^\tau }{\langle x_1\rangle }  \right\|_{L^2_{t,x}}^2\right)^{\frac 12}\nonumber \\
	\lesssim&\, t^{\frac12} M^2  e^{\frac {M}{2}} \cdot  \left(1+\left\|    \phi\right\|_{H^{s_2-\frac 52}(\mathbb{R}^d)}\right)\left\|    \phi\right\|_{H^{s_2}(\mathbb{R}^d)} \cdot  \left( \sup_{\tau}   \sum_{\lambda} \lambda^{ 2s_2} \left\| e^{M \frac{x_1}{2\langle x_1\rangle} } \frac{\partial_1 \psi^\tau }{\langle x_1\rangle }  \right\|_{L^2_{t,x}}^2\right)^{\frac 12}\nonumber \\
	\lesssim&\,  t^{\frac12} e^{M}  M_0^4  e^{2M}\\
	\lesssim&\, e^{2M},
\end{align*}
where $  \psi=P_{\lambda,e_1}  \phi  $ and we used  \eqref{def_t1}, \eqref{def_t2}.

The  \eqref{big_me_t1} can be estimated directly by H\"older's inequality
\begin{align*}
	\eqref{big_me_t1}
	\lesssim &  \,e^M   \left\| P_{\lambda, e_1 } \phi   \right\|_{L^{2}(\mathbb{R}^d)} \left\|\partial_1  P_{\lambda, e_1 } \phi   \right\|_{L^{2}(\mathbb{R}^d)}\\
	\lesssim &\, e^M   (\lambda^{\frac 12} \left\| \psi\right\|_{L^{2}(\mathbb{R}^d)})^2,
\end{align*}
where $ \psi = P_{\lambda, e_1} \phi 
$. Thus by \eqref{bstrap_ass4} and \eqref{def_t2}, we have 
\begin{align*}
	\sum_{\lambda} \lambda^{2s_4} e^M   (\lambda^{\frac 12} \left\| \psi\right\|_{L^{2}(\mathbb{R}^d)})^2=  &	\, e^M    \sum_{\lambda} \lambda^{2s_4+1}   \left\|  P_{\lambda,e_1   }\phi \right\|_{L^{2}(\mathbb{R}^d)}^2\\
	\lesssim &\, e^{M } M_1^4 \\
	\lesssim & \,e^{2M}.
\end{align*}

Now we turn to the nonlinear terms \eqref{big_me_t5} and \eqref{big_me_t6}.
Integration by parts gives
\begin{align*}
	\eqref{big_me_t5}=	& \sum_{\lambda} \lambda^{ 2s_2} \int_0^t	\int_{\mathbb{R}}e^{M \frac{x_1}{\langle x_1\rangle} }   M \frac{1}{\langle x_1 \rangle^2}\int_{\mathbb{R}^{d-1}}   \mathrm{Re}\left(  \overline{\psi} \mathcal{N}_1\right) \mathrm{d} \hat{x}_1  \mathrm{d} x_1 \mathrm{d} s\\
	\lesssim & M e^{\frac M2}	t^{\frac12} \left\|\phi\right\|_{H^{s_2}(\mathbb{R}^d)} \left( \sum_{\lambda} \lambda^{ 2s_2} \int_0^t \int_{\mathbb{R}^{d}}( e^{M \frac{x_1}{2\langle x_1\rangle} }    \frac{1}{\langle x_1 \rangle}\mathcal{N}_1)^2 \mathrm{d} x \mathrm{d} s \right)^{\frac 12},
\end{align*}
where  $\left( \sum_{\lambda} \lambda^{ 2s_2} \int_0^t \int_{\mathbb{R}^{d}}( e^{M \frac{x_1}{2\langle x_1\rangle} }    \frac{1}{\langle x_1 \rangle}\mathcal{N}_1)^2 \mathrm{d} x \mathrm{d} s \right)^{\frac 12} $ can be bounded similarly as   \eqref{big_me_t6}. Therefore, it suffices to analyze \eqref{big_me_t6}. Recalling the definition of $ \mathcal{N}_1 $ in \eqref{PqNLS}, we have
\begin{align}
	\eqref{big_me_t6}=&\, 2\sum_{\lambda} \lambda^{ 2s_2}\int_0^t	\int_{\mathbb{R}}e^{M \frac{x_1}{\langle x_1\rangle} }    \int_{\mathbb{R}^{d-1}}  \mathrm{Re} \left( \sum_{\mu\sim \lambda} a \cdot  P_{\lambda, e_1}   P_\mu \nabla \phi  \cdot\partial_1  \overline{\psi}\right)   \mathrm{d} \hat{x}_1  \mathrm{d} x_1 \mathrm{d} s \label{big_me_t6_t1}\\
	&+ 2\sum_{\lambda} \lambda^{ 2s_2}\int_0^t	\int_{\mathbb{R}}e^{M \frac{x_1}{\langle x_1\rangle} }    \int_{\mathbb{R}^{d-1}}  \mathrm{Re} \left(  \sum_{\mu\sim \lambda} b\cdot  P_{\lambda, e_1}   P_\mu \nabla\overline{ \phi }\cdot\partial_1  \overline{\psi}\right)   \mathrm{d} \hat{x}_1  \mathrm{d} x_1 \mathrm{d} s \label{big_me_t6_t1_1}\\
	&+2\sum_{\lambda} \lambda^{ 2s_2}\int_0^t	\int_{\mathbb{R}}e^{M \frac{x_1}{\langle x_1\rangle} }   \int_{\mathbb{R}^{d-1}}  \mathrm{Re} \left( \mathcal{R}\cdot\partial_1  \overline{\psi}\right)   \mathrm{d} \hat{x}_1  \mathrm{d} x_1 \mathrm{d} s. \label{big_me_t6_t3}
\end{align}
\eqref{big_me_t6_t1} and \eqref{big_me_t6_t1_1} can be estimated in the same manner and  we will only estimate \eqref{big_me_t6_t1}. 
The Lemma \ref{lem_transf_k}, \eqref{big_w_est_a} type inequality and H\"older's inequality show 
\begin{align*}
	\eqref{big_me_t6_t1} \lesssim &\sum_{\mu\sim \lambda} \sum_{\lambda} \lambda^{ 2s_2}  \left\|a \langle x_1 \rangle^2\right\|_{L^{\infty}(\mathbb{R}^d)} \left\|e^{M \frac{x_1}{2\langle x_1\rangle} }  \frac{  P_{\lambda, e_1}  P_\mu \nabla \phi }{\langle x_1 \rangle } \right\|_{L^2_{t,x}}\left\|e^{M \frac{x_1}{2\langle x_1\rangle} }  \frac{\partial_1 \psi }{\langle x_1\rangle }  \right\|_{L^2_{t,x}}\\
	\lesssim &   \,M_1^2   \sup_{\tau} \sum_{\lambda} \lambda^{ 2s_2}\left\|e^{M \frac{x_1}{2\langle x_1\rangle} }  \frac{\partial_1 \psi^\tau  }{\langle x_1\rangle }  \right\|_{L^2_{t,x}}^2\\
	\lesssim &  \,M_1^2    \sup_{\tau}\sum_{\lambda} \lambda^{ 2s_2}\left\|e^{M \frac{x_1}{2\langle x_1\rangle} }  \frac{\partial_1 \psi^\tau  }{\langle x_1\rangle }  \right\|_{L^2_{t,x}}^2 \\
	\lesssim & \,M_1^2   \sup_{\tau}\sum_{\lambda} \lambda^{ 2s_2}\left\|e^{M \frac{x_1}{2\langle x_1\rangle} }  \frac{\partial_1 \psi^\tau  }{\langle x_1\rangle }  \right\|_{L^2_{t,x}}^2,
\end{align*}
where the  finite sum $\sum_{\mu\sim \lambda}  $ would not make a difference.

As for \eqref{big_me_t6_t3}, we use \eqref{remainderterm} to write it	 into
\begin{align}
	\eqref{big_me_t6_t3} = 
	&\,2\sum_{\lambda} \lambda^{ 2s_2}\int_0^t	   \int_{\mathbb{R}^{d}}  
 e^{M \frac{x_1}{\langle x_1\rangle} } \mathrm{Re} \left(  \left(\sum_{\mu\sim \lambda} [  P_{\lambda, e_1},a  ]  P_\mu \nabla \phi\right)\partial_1  \overline{\psi}\right)      \mathrm{d} x  \mathrm{d} s\label{Rem_t2}\\
	&+2\sum_{\lambda} \lambda^{ 2s_2}\int_0^t	\int_{\mathbb{R}^{d}}    e^{M \frac{x_1}{\langle x_1\rangle} }\mathrm{Re} \left(  \left( \sum_{\mu}  [P_{\lambda, e_1},b ]  P_\mu \nabla\overline{ \phi }\right)\partial_1  \overline{\psi}\right)      \mathrm{d} x \mathrm{d} s \label{Rem_t3}\\
	&+2\sum_{\lambda} \lambda^{ 2s_2}\int_0^t	\int_{\mathbb{R}^{d}}    e^{M \frac{x_1}{\langle x_1\rangle} }\mathrm{Re} \left(   \partial_i\left(\left[P_{\lambda, e_1}, h^{ij}\left(\phi,\overline{\phi}\right) \right]  \partial_j  \phi \right)\partial_1  \overline{\psi}\right)     \mathrm{d} x  \mathrm{d} s \label{Rem_t4}\\
	&+ 2\sum_{\lambda} \lambda^{ 2s_2}\int_0^t	\int_{\mathbb{R}^{d}}    e^{M \frac{x_1}{\langle x_1\rangle} }\mathrm{Re} \left( P_{\lambda, e_1}  F(\phi, \overline{\phi}, P_{\leq 1} \nabla \phi,  P_{\leq 1} \nabla \overline{\phi})\cdot\partial_1  \overline{\psi}\right)     \mathrm{d} x  \mathrm{d} s. \label{Rem_t1}
\end{align}
\eqref{Rem_t2} and \eqref{Rem_t3} can be handled similarly and we can restrict our attention to \eqref{Rem_t2}.
By commutator estimates in Lemma \ref{lem_com}, \eqref{def_t1}, \eqref{def_t2} and assumption \eqref{bstrap_ass1}-\eqref{bstrap_ass5}, we obtain
\begin{align*}
	\eqref{Rem_t2} \lesssim & \,t^{\frac 12} e^{\frac M2} \left\| \langle x_1\rangle\nabla a \right\|_{L^{\infty}(\mathbb{R}^d)} \left(\sum_{\lambda} \lambda^{ 2s_2} \sup_\tau \left\|  P_\lambda \phi^{\tau }\right\|_{L^{2}(\mathbb{R}^d)}^2\right )^{\frac 12} \left(\sum_{\lambda} \lambda^{ 2s_2}\sup_\tau \left\| e^{M \frac{x_1}{2\langle x_1\rangle} } \frac{\partial_k \psi^{\tau_1} }{\langle x_1\rangle }  \right\|_{L^2_{t,x}}^2\right)^{\frac 12}\\
 \lesssim & \,t^{\frac 12} e^{\frac M2} \left(\left\|   \partial_y a \right\|_{L^{\infty}}+\left\|   \partial_{\overline{y}} a \right\|_{L^{\infty} }+\left\|   \partial_z a \right\|_{L^{\infty}}+\left\|   \partial_{\overline{z}} a \right\|_{L^{\infty} }\right) \left\| \langle x_1\rangle \langle\nabla\rangle^2 \phi \right\|_{L^{\infty}(\mathbb{R}^d)}   \\
  &\cdot\left(\sup_\tau \sum_{\lambda} \lambda^{ 2s_2} \left\|  P_\lambda \phi^\tau\right\|_{L^{2}(\mathbb{R}^d)}^2\right)^{\frac 12} \left(\sum_{\lambda} \lambda^{ 2s_2}\sup_\tau \left\| e^{M \frac{x_1}{2\langle x_1\rangle} } \frac{\partial_k \psi^\tau }{\langle x_1\rangle }  \right\|_{L^2_{t,x}}^2\right)^{\frac 12}\\
   \lesssim & \,t^{\frac 12} e^{\frac M2}  \left\| \langle x_1\rangle  \phi \right\|_{H^{s_2-1}(\mathbb{R}^d)}   \\
  &\cdot\left(\sup_\tau \sum_{\lambda} \lambda^{ 2s_2} \left\|  P_\lambda \phi^\tau\right\|_{L^{2}(\mathbb{R}^d)}^2\right)^{\frac 12} \left(\sum_{\lambda} \lambda^{ 2s_2}\sup_\tau \left\| e^{M \frac{x_1}{2\langle x_1\rangle} } \frac{\partial_k \psi^\tau }{\langle x_1\rangle }  \right\|_{L^2_{t,x}}^2\right)^{\frac 12}\\
	\lesssim& \,t^{\frac 12} e^{\frac M2}  \left\| \langle x_1\rangle^2 \phi  \right\|_{H^{s_2 -2}(\mathbb{R}^d)}^{\frac 12} \left\|  \phi \right\|_{H^{s_2  } (\mathbb{R}^d)}^{\frac 12}\\
 &\cdot\left(\sup_\tau \sum_{\lambda} \lambda^{ 2s_2} \left\|  P_\lambda \phi^\tau\right\|_{L^{2}(\mathbb{R}^d)}^2\right)^{\frac 12} \left(\sum_{\lambda} \lambda^{ 2s_2}\sup_\tau \left\| e^{M \frac{x_1}{2\langle x_1\rangle} } \frac{\partial_k \psi^\tau }{\langle x_1\rangle }  \right\|_{L^2_{t,x}}^2\right)^{\frac 12}\\
	\lesssim&\, t^{\frac 12} e^{\frac M2}  M_1^4 e^{2M}\\
	\lesssim&\, e^{2M},
\end{align*}
where by \eqref{a}, we know $a$ is smooth.
As regards \eqref{Rem_t4}, we have
\begin{align}
	\eqref{Rem_t4} =&\,2\sum_{\lambda} \lambda^{ 2s_2}\int_0^t	\int_{\mathbb{R}}e^{M \frac{x_1}{\langle x_1\rangle} }   \int_{\mathbb{R}^{d-1}}  \mathrm{Re} \left(   \left(\left[P_{\lambda, e_1},\partial_i h^{ij}\left(\phi,\overline{\phi}\right) \right]  \partial_j  \phi \right)\partial_1  \overline{\psi}\right)   \mathrm{d} \hat{x}_1  \mathrm{d} x_1 \mathrm{d} s\label{Rem_t4_t1}\\
	&+2\sum_{\lambda} \lambda^{ 2s_2}\int_0^t	\int_{\mathbb{R}}e^{M \frac{x_1}{\langle x_1\rangle} }   \int_{\mathbb{R}^{d-1}}  \mathrm{Re} \left(   \left(\left[P_{\lambda, e_1}, h^{ij}\left(\phi,\overline{\phi}\right) \right]  \partial_i\partial_j  \phi \right)\partial_1  \overline{\psi}\right)   \mathrm{d} \hat{x}_1  \mathrm{d} x_1 \mathrm{d} s, \label{Rem_t4_t2}
\end{align}
where calculations as those for \eqref{Rem_t2} give
\begin{align*}
	\eqref{Rem_t4_t1} \lesssim &\, t^{\frac 12} e^{\frac M2} \left\| \langle x_1 \rangle \nabla \partial_i h^{ij}\left(\phi,\overline{\phi}\right) \right\|_{L^\infty}\left(\sum_{\lambda} \lambda^{ 2s_2} \sup_\tau \left\|  P_\lambda \phi^{\tau }\right\|_{L^{2}(\mathbb{R}^d)}^2\right )^{\frac 12}\\
 &\cdot\left(\sum_{\lambda} \lambda^{ 2s_2}\sup_\tau \left\| e^{M \frac{x_1}{2\langle x_1\rangle} } \frac{\partial_k \psi^{\tau_1} }{\langle x_1\rangle }  \right\|_{L^2_{t,x}}^2\right)^{\frac 12}\\
	\lesssim & \,t^{\frac 12} e^{\frac M2} \left(\left\|   \partial_y h \right\|_{L^{\infty}}+\left\|   \partial_{\overline{y}} h \right\|_{L^{\infty} } \right) \left\| \langle x_1\rangle \langle\nabla\rangle^2 \phi \right\|_{L^{\infty}(\mathbb{R}^d)}   \\
  &\cdot\left(\sup_\tau \sum_{\lambda} \lambda^{ 2s_2} \left\|  P_\lambda \phi^\tau\right\|_{L^{2}(\mathbb{R}^d)}^2\right)^{\frac 12} \left(\sum_{\lambda} \lambda^{ 2s_2}\sup_\tau \left\| e^{M \frac{x_1}{2\langle x_1\rangle} } \frac{\partial_k \psi^\tau }{\langle x_1\rangle }  \right\|_{L^2_{t,x}}^2\right)^{\frac 12}\\
	\lesssim&\, t^{\frac 12} e^{\frac M2} M_1^4 e^{2M}\\
	\lesssim&\, e^{2M}
\end{align*}
and 
\begin{align*}
	\eqref{Rem_t4_t2} \lesssim &\left\| \langle x_1 \rangle^2 \nabla  h^{ij}\left(\phi,\overline{\phi}\right) \right\|_{L^\infty}  \left(\sum_{\lambda} \left\| e^{M \frac{x_1}{2\langle x_1\rangle} } \frac{ \lambda^{ s_2-1}\partial_{i} \partial_j  \phi }{\langle x_1\rangle }  \right\|_{L^2_{t,x}}^2\right)^{\frac 12} \left(\sum_{\lambda} \lambda^{ 2s_2}\left\| e^{M \frac{x_1}{2\langle x_1\rangle} } \frac{\partial_1 \psi }{\langle x_1\rangle }  \right\|_{L^2_{t,x}}^2\right)^{\frac 12} \\
	\lesssim &  \left\|   \langle x_1 \rangle^2 \phi \right\|_{H^{s_2-\frac 32}}   \sum_{\lambda} \lambda^{ 2s_2}\sup_{\tau}\left\|e^{M \frac{x_1}{2\langle x_1\rangle} }  \frac{\partial_1 \psi^\tau  }{\langle x_1\rangle }  \right\|_{L^2_{t,x}}^2 \\
	\lesssim & M_1^2   \sum_{\lambda} \lambda^{ 2s_2}\sup_{\tau}\left\|e^{M \frac{x_1}{2\langle x_1\rangle} }  \frac{\partial_1 \psi^\tau  }{\langle x_1\rangle }  \right\|_{L^2_{t,x}}^2.
\end{align*}

While for the low frequency terms \eqref{Rem_t1}, we can use \eqref{def_t1}, \eqref{def_t2} and assumption \eqref{bstrap_ass1}-\eqref{bstrap_ass5} to obtain
\begin{align*}
	\eqref{Rem_t1}\lesssim &\, t^{\frac 12} e^{\frac M2} \left\|\langle x_1\rangle \phi \right\|_{L^{\infty}(\mathbb{R}^d)} \left(\sum_{\lambda} \lambda^{ 2s_2}\left\|  P_{\lambda,e_1} P_{\leq 1}  \phi\right\|_{L^{2}(\mathbb{R}^d)}^2 \right)^{\frac 12}\left(\sum_{\lambda} \lambda^{ 2s_2}\left\| e^{M \frac{x_1}{2\langle x_1\rangle} } \frac{\partial_1 \psi }{\langle x_1\rangle }  \right\|_{L^2_{t,x}}^2\right)^{\frac 12} \\
	\lesssim &\, t^{\frac 12} e^{\frac M2}    \left\|  \phi \langle x_1 \rangle\right\|_{H^{s_2-\frac 52}}    \left(\sum_{\lambda} \lambda^{ 2s_2}\left\|  P_{\lambda,e_1}    \phi\right\|_{L^{2}(\mathbb{R}^d)}^2 \right)^{\frac 12} \left(\sum_{\lambda} \lambda^{ 2s_2}\left\| e^{M \frac{x_1}{2\langle x_1\rangle} } \frac{\partial_1 \psi }{\langle x_1\rangle }  \right\|_{L^2_{t,x}}^2\right)^{\frac 12}  \\
	\lesssim & \,t^{\frac 12} e^{\frac M2}   M_1^4 e^{2M}\\
	\lesssim&\, e^{2M}.
\end{align*}

Collecting  all the terms above, we can use \eqref{def_t1}  and   \eqref{def_t2}  to  obtain   \eqref{bstrap_re3} for   $ e_1 $:
\begin{align} \label{big_est_mom}
	&\sum_{\lambda} \lambda^{ 2s_2}   \sup_{\tau}\int_0^t   \int_{\mathbb{R}^{d}}  \left|	e^{M \frac{x_1}{2\langle x_1\rangle} }\cdot \frac{\partial_1 \psi^\tau }{\langle x_1\rangle}\right|^2   \mathrm{d} x  \mathrm{d} s  \nonumber\\
	\leq 	&\,\left( 2M-M_1^2- \left( t^{\frac 18}+ \eta+  \eta^{\frac 13}  M_1^{\frac 23 }\right)  \right)  \sum_{\lambda} \lambda^{ 2s_2}    \sup_{\tau}\int_0^t   \int_{\mathbb{R}^{d}}  \left|	e^{M \frac{x_1}{2\langle x_1\rangle} }\cdot \frac{\partial_1 \psi^\tau }{\langle x_1\rangle}\right|^2   \mathrm{d} x  \mathrm{d} s \nonumber \\
	\leq  &\,  Ce^{2M},  
\end{align}
where $ \psi=P_{\lambda, e_1}  \phi   $ and we set $ \eta,t $ such that $ 2M-M_1^2- \left( t^{\frac 18}+ \eta+  \eta^{\frac 13}  M_1^{\frac 23 }\right)  \geq 1.  $

\subsection{Energy estimates} \label{s_big.2}
We begin with the  estimate of \eqref{bstrap_re1}, which is straightforward.  
Multiplying \eqref{PqNLS} with $ P_{\lambda, \varpi_m}   \overline{\phi} $, taking the imaginary part,  integrating on $ [0,t]\times \mathbb{R}^d $, multiplying with $\lambda^{2s_2}$, summing over $\lambda \geq 1$ and $m=1\dots, K$, we have
\begin{align*}
\left\|\phi \right\|_{H^{s_2}}^2\lesssim &\,\sum_{m=1}^K\sum_{\lambda\geq 1}	\lambda^{2s_2} \left\|P_{\lambda, \varpi_m}  \phi \right\|_{L^{2}(\mathbb{R}^d)}^2\\
=& \sum_{m=1}^K\sum_{\lambda\geq 1}	\lambda^{2s_2} \left\|P_{\lambda, \varpi_m}  \phi_0 \right\|_{L^{2}(\mathbb{R}^d)}^2\\
&+\sum_{m=1}^K \sum_{\lambda\geq 1}	\lambda^{2s_2} \int_0^t \int_{\mathbb{R}^d}  \mathrm{Im} \left(   \mathcal{N}_1 P_{\lambda, \varpi_m}  \overline{\phi}\right)  \mathrm{d} x \mathrm{d} s\nonumber\\
	\lesssim &\,\left\|\phi_0 \right\|_{H^{s_2}}^2+  t^{\frac12}  \left( \sum_{m=1}^K\sum_{\lambda\geq 1}	\lambda^{2s_2} \left\|P_{\lambda, \varpi_m}  \phi \right\|_{L^{2}(\mathbb{R}^d)}^2\right)^{\frac 12} 
 \left(\sum_{m=1}^K\sum_{\lambda\geq 1}\left\|\mathcal{N}_1\right\|_{L^2_{t,x}}^2\right)^{\frac 12}\\
 \lesssim &\,M_0^2+  t^{\frac12}   \epsilon^{-(d-1)}  \left\|\phi \right\|_{H^{s_2}}
 \left( \sum_{\lambda\geq 1}\left\|\mathcal{N}_1\right\|_{L^2_{t,x}}^2\right)^{\frac 12} \\
 \lesssim &\,M_0^2 + t^{\frac12}  \epsilon^{-(d-1)}   M_0^2 M_1 e^{2M}\\
 \lesssim &\, M_0^2,
\end{align*}
where we used  $K = C \epsilon^{-(d-1)}$, \eqref{orth_relat}, \eqref{def_t1},  and $\left\|\mathcal{N}_1\right\|_{L^2_{t,x}} $ can be estimated  as the nonlinear  term \eqref{big_me_t5}.   We  omit the details.

Also, the \eqref{bstrap_re5} can be estimated directly by \eqref{bstrap_ass5} and \eqref{def_t2}:
\begin{align*}
	\left\|\phi\right\|_{H^{s_2+\frac 12}}^2\lesssim& \sum_{m=1}^K	\sum_{\lambda} \lambda^{2s_2+1} \sup_{\tau}\left\| P_{\lambda, \varpi_m}  \phi^\tau\right\|_{L^2(\mathbb{R}^d)}^2 \\
	\lesssim &\, K M_1^4\\
	\lesssim & \,\epsilon^{-(d-1)} M_1^4 \\
	\lesssim &\, \epsilon^{-d}.
\end{align*}

Next, we estimate \eqref{bstrap_re4}. Again, we only consider the case of \( \varpi_m=e_1 \), as the result for any \( \varpi_m \) with \( m = 1, \dots, K \) follows analogously by applying a rotation.
Recall  the \eqref{PqNLS}, i.e.
\begin{equation*}  
	\begin{cases}
		\sqrt{-1}  \partial_t \psi + \partial_i\left( g^{ij}\left(\phi,\overline{\phi}\right) \partial_j \psi \right)=\sum_{\mu\sim \lambda} a \cdot      P_\mu \nabla \psi + \sum_{\mu\sim \lambda} b\cdot     P_\mu \nabla\overline{ \psi }+\mathcal{R}\\
		\psi (0,x)= P_{\lambda, e_1 } \phi_0 ,
	\end{cases}
\end{equation*}
where $\psi=P_{\lambda, e_1}  \phi $ and  $\mathcal{R} $ is defined as \eqref{remainderterm}.
Inspired by Kenig-Ponce-Vega \cite{kenig2004cauchy}, we will eliminate the part of $ \sum_{\mu\sim \lambda} a_1 \cdot P_\mu \partial_1 \psi $  in $ \mathcal{N}_1 $  by pseudo-differential technique.
Denote $ U:= e^{q (x,D)} \psi $, where $ q (x,D)$ is zeroth order pseudo-differential operator.
Then we  can apply  $e^{q (x,D)} $ to \eqref{PqNLS} to obtain 
\begin{align} \label{eq_U}
	\begin{cases}
		\sqrt{-1}  \partial_t U + \partial_i\left( g^{ij}\left(\phi,\overline{\phi}\right) \partial_j U \right)=\mathcal{N}_2+\mathcal{N}_3\\
		U (0,x)=e^{q (x,D)} \psi_0 ,
	\end{cases}
\end{align}
where 
\begin{align*}
	\mathcal{N}_2:= &-[e^{q (x,D)}, \partial_i  (g^{ij}\left(\phi,\overline{\phi}\right) \partial_j)] \psi + e^{q (x,D)} \left(\sum_{\mu\sim \lambda} a \cdot P_\mu \nabla \psi \right)+ \sum_{\mu\sim \lambda} b\cdot  P_\mu \nabla \overline{U},\\
	\mathcal{N}_3:=& \sum_{\mu\sim \lambda}  [e^{q (x,D)}, b  P_\mu \nabla] \overline{\psi}+ e^{q (x,D)} \mathcal{R}  .
\end{align*}

By pseudo-differential techniques in \cite{chen2006}, 
we take 
\begin{align*}
	\{q(x,\xi ), g^{ij}\left(\phi,\overline{\phi}\right) \xi_i\xi_j \}=   \sum_{\mu\sim \lambda} a \cdot\xi  P_\mu(\xi ).   
\end{align*}
Since  the $ x_1 $ is the dominate direction, we can only consider 
\begin{align*}
	\partial_{\xi_1 } q(x,\xi_1 )\cdot\xi_1^2 \partial_{x_1}  g^{11}\left(\phi,\overline{\phi}\right) -2  \partial_{x_1 } q(x,\xi_1 )\cdot g^{11}\left(\phi,\overline{\phi}\right)\xi_1
	=&\,\{q(x,\xi_1 ), g^{11}\left(\phi,\overline{\phi}\right) \xi_1^2 \}\\
	= &  \sum_{\mu\sim \lambda} a_1 \cdot\xi_1  P_\mu(\xi_1 ).  
\end{align*}
Therefore,
\begin{align*}
	\partial_{\xi_1 } q(x,\xi_1 )\cdot\xi_1 \partial_{x_1}  g^{11}\left(\phi,\overline{\phi}\right) -2  \partial_{x_1 } q(x,\xi_1 )\cdot g^{11}\left(\phi,\overline{\phi}\right) 
	=&  a_1  \sum_{\mu\sim \lambda}  P_\mu(\xi_1 ).  \end{align*}
We define the characteristic
\begin{align} \label{ode1}
	\frac{\mathrm{d}}{\mathrm{d} x_1} \xi_1 =-\frac{\xi_1 \partial_{x_1}  g^{11}\left(\phi,\overline{\phi}\right) }{2 g^{11}\left(\phi,\overline{\phi}\right) }.
\end{align}
Then 
\begin{align} \label{ode2}
	\frac{\mathrm{d}}{\mathrm{d} x_1} q(x_1)=-\frac{ a_1  \sum_{\mu\sim \lambda}  P_\mu(\xi_1 )   }{2 g^{11}\left(\phi,\overline{\phi}\right) }.
\end{align}
By the results in \cite{ta1985boundary}, the ordinary differential equations \eqref{ode1} and \eqref{ode2} admit solutions. This allows us to eliminate the term  $ \sum_{\mu\sim \lambda} a_1 \cdot P_\mu \partial_1 \psi  $, yielding
\begin{align*}
	\mathcal{N}_2= \mathcal{R}_1+ \sum_{\mu\sim \lambda} b\cdot  P_\mu \nabla \overline{U},
\end{align*}
where $ \mathcal{R}_1 $ is lower order terms can be estimated without using the momentum estimates:
\begin{align*}
	\mathcal{R}_1 =-\sum_{ i\neq 1\text{ or } j\neq 1}	[e^{q (x,D_1)}, \partial_i  (g^{ij}\left(\phi,\overline{\phi}\right) \partial_j)] \psi + \sum_{ i\neq 1} e^{q (x,D_1)} \left(\sum_{\mu\sim \lambda} a_i  P_\mu \partial_i \psi \right) +   q_0(x,D) U,
\end{align*}
and  $ q_0(x,D)  $ is zeroth order pseudo-differential operator.

%\begin{align}
%	\begin{cases}
%	\sqrt{-1}  \partial_t (U ) + \partial_i\left( g^{ij}\left(\phi,\overline{\phi}\right) \partial_j (U) \right)= \mathcal{R}_1+ \sum_{\mu\sim \lambda} b\cdot  P_\mu \nabla \overline{U}+\mathcal{N}_3\\
%	U (0,x)=0,
%\end{cases}
%\end{align}

Multiplying \eqref{eq_U} with $ \overline{U} $, frequency $  \lambda^{2s_2+1} $ and constant $ e^{M} $, taking the imaginary part and integrating on $ [0,t]\times \mathbb{R}^d $, we have  
\begin{align}
	e^{M} \lambda^{2s_2+1}	\|U\|_{L^{2}(\mathbb{R}^d)}^2= & \,	e^{M} \lambda^{2s_2+1}	\|U_0\|_{L^{2}(\mathbb{R}^d)}^2\nonumber\\
	&+e^{M}\lambda^{2s_2+1}\int_0^t \int_{\mathbb{R}^d}  \mathrm{Im} \left(   \mathcal{R}_1\overline{U}\right)  \mathrm{d} x \mathrm{d} s \label{Big_en_maint1}\\
	&+ e^{M}\lambda^{2s_2+1} \sum_{\mu\sim \lambda}\int_0^t \int_{\mathbb{R}^d}  \mathrm{Im} \left(   b\cdot  P_\mu \nabla \overline{U}~\overline{U}\right)  \mathrm{d} x \mathrm{d} s\label{Big_en_maint2}\\
	&+e^{M} \lambda^{2s_2+1}\int_0^t \int_{\mathbb{R}^d}  \mathrm{Im} \left(   \mathcal{N}_3\overline{U}\right)  \mathrm{d} x \mathrm{d} s.\label{Big_en_maint3}
\end{align}

For \eqref{Big_en_maint2}, we have
\begin{align}
	\eqref{Big_en_maint2}	=& \, e^{M} \lambda^{2s_2+1} \sum_{\mu\sim \lambda} \int_0^t \int_{\mathbb{R}^d}  \mathrm{Im} \left( (b\cdot \nabla  P_\mu^{\frac 12}\overline{U} )~ P_\mu^{\frac 12}\overline{U}  \right)  \mathrm{d} x \mathrm{d} s\nonumber \\
	&+e^{M} \lambda^{2s_2+1} \sum_{\mu\sim \lambda} \int_0^t \int_{\mathbb{R}^d}  \mathrm{Im} \left(    ( [b , \nabla P_\mu^{\frac 12} ] \overline{U} )P_\mu^{\frac 12}\overline{U}  \right)  \mathrm{d} x \mathrm{d} s\nonumber\\
	=&- e^{M}\lambda^{2s_2+1}  \sum_{\mu\sim \lambda} \frac 12 \int_0^t \int_{\mathbb{R}^d}  \mathrm{Im} \left(    \mathrm{div}b \cdot (P_\mu^{\frac 12}\overline{U}) ^2  \right)  \mathrm{d} x \mathrm{d} s\label{Big_en_maint2_1}\\
	&+  e^{M} \lambda^{2s_2+1}\sum_{\mu\sim \lambda} \int_0^t \int_{\mathbb{R}^d}  \mathrm{Im} \left(    ( [b , \nabla P_\mu^{\frac 12} ] \overline{U} )P_\mu^{\frac 12}\overline{U}  \right)  \mathrm{d} x \mathrm{d} s. \label{Big_en_maint2_2}
\end{align}
Observe that $ P_\mu^{\frac 12}  e^{q (x,D)} $ is a zeroth-order order pseudo-differential oprator. Thus, for \eqref{Big_en_maint2_1}, we apply \eqref{est_transf_0_order_P} and \eqref{est_transf_k2} to obtain
\begin{align*}
	\eqref{Big_en_maint2_1}\lesssim &  \sum_{\mu\sim \lambda}  t^{\frac12} e^{\frac 32M}\left\|\langle x_k\rangle \mathrm{div}b  \right\|_{L^{\infty}(\mathbb{R}^d)} \cdot \sup_{t }\left\| \lambda^{s_2}\overline{U}\right\|_{L^{2}(\mathbb{R}^d)}   \left\|e^{M \frac{x_k}{2\langle x_k\rangle} }  \frac{\lambda^{s_2+1}P_\mu^{\frac 12}  e^{q (x,D)} \psi}{\langle x_k\rangle }  \right\|_{L^2_{t,x}} \\
	%\lesssim &   \sum_{\mu\sim \lambda}  t^{\frac12} e^{\frac 32 M} \left\|\langle x_k\rangle \mathrm{div}b  \right\|_{L^{\infty}(\mathbb{R}^d)} \cdot \sup_{t }\left\| \lambda^{s_2}\overline{U}\right\|_{L^{2}(\mathbb{R}^d)}   \left\|P_\mu^{\frac 12}  e^{q (x,D)}\left(e^{M \frac{x_k}{2\langle x_k\rangle} }  \frac{\lambda^{s_2+1} \psi}{\langle x_k\rangle } \right) \right\|_{L^2_{t,x}} \\
	%&+ \sum_{\mu\sim \lambda} \lambda^{s_2+1} t^{\frac12} e^{\frac 32 M} \left\|\langle x_k\rangle \mathrm{div}b  \right\|_{L^{\infty}(\mathbb{R}^d)} \cdot \sup_{t }\left\| \lambda^{s_2}\overline{U}\right\|_{L^{2}(\mathbb{R}^d)}   \left\|\left[P_\mu^{\frac 12} e^{q (x,D)},\left(e^{M \frac{x_k}{2\langle x_k\rangle} }  \frac{1 }{\langle x_k\rangle } \right)\right] \psi\right\|_{L^2_{t,x}}\\
	\lesssim & \,t^{\frac12} e^{\frac 32M} M_1^2 \cdot \sup_{t }\left\| \lambda^{s_2} \psi\right\|_{L^{2}(\mathbb{R}^d)}  \sup_{\tau}   \left\| e^{M \frac{x_k}{2\langle x_1\rangle} }  \frac{\lambda^{s_2+1} \psi^{\tau }}{\langle x_1\rangle } \right\|_{L^2_{t,x}}\\
	%&+ \sum_{\mu\sim \lambda} \lambda^{s_2+1} t^{\frac12} e^{\frac 32 M} \left\|\langle x_k\rangle \mathrm{div}b  \right\|_{L^{\infty}(\mathbb{R}^d)} \cdot \sup_{t }\left\| \lambda^{s_2}\overline{U}\right\|_{L^{2}(\mathbb{R}^d)}   \left\|\left[P_\mu^{\frac 12} e^{q (x,D)},\left(e^{M \frac{x_k}{2\langle x_k\rangle} }  \frac{1 }{\langle x_k\rangle } \right)\right] \psi\right\|_{L^2_{t,x}}
	\lesssim & \,t^{\frac12} e^{\frac 32M} M_1^2  \cdot \sup_{t }\left\| \lambda^{s_2} \psi\right\|_{L^{2}(\mathbb{R}^d)}  \sup_{\tau}   \left\| e^{M \frac{x_k}{2\langle x_1\rangle} }  \frac{\lambda^{s_2+1} \psi^{\tau }}{\langle x_1\rangle } \right\|_{L^2_{t,x}} \\
	\lesssim & \,t^{\frac12} e^{\frac 32M} M_1^2 \cdot \sup_{t }\left\| \lambda^{s_2} \psi\right\|_{L^{2}(\mathbb{R}^d)}  \sup_{\tau}   \left\| e^{M \frac{x_1}{2\langle x_1\rangle} }  \frac{\lambda^{s_2}  \partial_1\psi^{\tau }}{\langle x_1\rangle } \right\|_{L^2_{t,x}},
\end{align*}
where    $ s_2>\frac d2 +3 $ and $\left\|\langle x_k\rangle \mathrm{div}b  \right\|_{L^{\infty}(\mathbb{R}^d)} $ can be estimated as $\left\|\langle x_1\rangle \nabla a  \right\|_{L^{\infty}(\mathbb{R}^d)} $ in  \eqref{Rem_t2}.
 Similarly for	\eqref{Big_en_maint2_2}, we can also use \eqref{est_transf_0_order_P} to obtain
\begin{align*}
	\eqref{Big_en_maint2_2} \lesssim &\sum_{\mu\sim \lambda}  t^{\frac12} e^{\frac 32M}\left\|\langle x_k\rangle  \nabla b  \right\|_{L^{\infty}(\mathbb{R}^d)} \cdot \sup_{t }\left\| \lambda^{s_2}\overline{U}\right\|_{L^{2}(\mathbb{R}^d)}   \left\|e^{M \frac{x_k}{2\langle x_k\rangle} }  \frac{\lambda^{s_2+1}P_\mu^{\frac 12}  e^{q (x,D)} \psi}{\langle x_k\rangle }  \right\|_{L^2_{t,x}} \\
	\lesssim &\, t^{\frac12} e^{\frac 32M} M_1^2  \cdot \sup_{t }\left\| \lambda^{s_2} \psi\right\|_{L^{2}(\mathbb{R}^d)}  \sup_{\tau}   \left\| e^{M \frac{x_k}{2\langle x_1\rangle} }  \frac{\lambda^{s_2} \partial_k \psi^{\tau }}{\langle x_1\rangle } \right\|_{L^2_{t,x}}.
\end{align*}

As regards  \eqref{Big_en_maint1}, we have
\begin{align}
	\eqref{Big_en_maint1}=&\,	e^{M}\lambda^{2s_2+1} \sum_{ i\neq 1\text{ or } j\neq 1}	\int_0^t \int_{\mathbb{R}^d}  \mathrm{Im} \left(  [e^{q (x,D_1)}, \partial_i  (g^{ij}\left(\phi,\overline{\phi}\right) \partial_j)] \psi \overline{U}\right)  \mathrm{d} x \mathrm{d} s \label{Big_en_maint1_1} \\
	&+e^{M}\lambda^{2s_2+1}\int_0^t \int_{\mathbb{R}^d}  \mathrm{Im} \left(   \sum_{ i\neq 1} e^{q (x,D_1)} \left(\sum_{\mu\sim \lambda} a_i  P_\mu \partial_i \psi \right) \overline{U}\right)  \mathrm{d} x \mathrm{d} s \label{Big_en_maint1_2} \\
	&+e^{M}\lambda^{2s_2+1}\int_0^t \int_{\mathbb{R}^d}  \mathrm{Im} \left(    q_0(x,D) U\overline{U}\right)  \mathrm{d} x \mathrm{d} s.  \label{Big_en_maint1_3}
\end{align}
Since \( q_0(x,D) \) in \eqref{Big_en_maint1_3} is a zeroth-order pseudo-differential operator, this term can be estimated similarly to \eqref{Big_en_maint1_1}. We now analyze \eqref{Big_en_maint1_1} in detail.
For   $ i\neq 1 $, we have
\begin{align*}
	[e^{q (x,D_1)}, \partial_i  (g^{ij}\left(\phi,\overline{\phi}\right) \partial_j)] \psi&=   \frac{1}{ (1+ |x_1| ^2)^{\frac 12} }q_1(x,D)  (\frac{1}{ (1+ |x_1| ^2)^{\frac 12} }e^{q (x,D_1)} \psi) +q_2(x,D)   \psi\\
	&=  \frac{1}{ (1+ |x_1| ^2)^{\frac 12} }q_1(x,D)( \frac{1}{ (1+ |x_1| ^2)^{\frac 12} }U) +q_2(x,D)   \psi,\\
\end{align*}
where  
\begin{align*}
	q_1(x,\xi)=( 1+ |x_1| ^2 )\cdot \{q(x,\xi_1),  g^{ij}\left(\phi,\overline{\phi}\right)\xi_i \xi_j  \}
\end{align*}
is first order pseudo-differential operator and $ q_2(x,D) $ is  zeroth order pseudo-differential operator. Thus,
\begin{align*}
	\eqref{Big_en_maint1_1}= &\,e^{M}\lambda^{2s_2+1}  	\int_0^t \int_{\mathbb{R}^d}  \mathrm{Im}   \left( \frac{1}{( 1+ |x_1| ^2 )^{\frac 12}} q_1(x,D)  ( \frac{1}{( 1+ |x_1| ^2 )^{\frac 12}}   U)\cdot \overline{U}\right)  \mathrm{d} x \mathrm{d} s \\
	&+e^{M}\lambda^{2s_2+1}  	\int_0^t \int_{\mathbb{R}^d}  \mathrm{Im} \left(  q_2(x,D)   \psi \cdot  \overline{U}\right)  \mathrm{d} x \mathrm{d} s \\
	\lesssim & \,e^{M}  \sum_{ i\neq 1 }	 \left\|   \frac{ \lambda^{s_2+1}  \psi }{\langle x_1\rangle } \right\|_{L^2_{t,x}}  \left\|  \frac{ \lambda^{s_2}  \partial_i  \psi }{\langle x_1\rangle } \right\|_{L^2_{t,x}}\\
	&\,+ t^{\frac 12} e^{\frac 32M}  \left\|\lambda^{s_2} q_2(x,D)   \psi\right\|_{L^2} \left\| e^{M \frac{x_k}{2\langle x_1\rangle} }  \frac{ \lambda^{s_2+1}  \psi }{\langle x_1\rangle } \right\|_{L^2_{t,x}}  \\
	\lesssim &\, e^{2M}  \sum_{ i\neq 1 }	 \left\| e^{M \frac{x_k}{2\langle x_1\rangle} }  \frac{ \lambda^{s_2+1}  \psi }{\langle x_1\rangle } \right\|_{L^2_{t,x}}  \left\| e^{M \frac{x_k}{2\langle x_1\rangle} }  \frac{ \lambda^{s_2}  \partial_i  \psi }{\langle x_1\rangle } \right\|_{L^2_{t,x}} \\
	&\,+ t^{\frac 12} e^{\frac 32M}  \left\|\lambda^{s_2}   \psi\right\|_{L^2} \left\| e^{M \frac{x_k}{2\langle x_1\rangle} }  \frac{ \lambda^{s_2+1}  \psi }{\langle x_1\rangle } \right\|_{L^2_{t,x}}  \\
	\lesssim & \, \epsilon e^{2M}  \sup_{\tau}	  \left\| e^{M \frac{x_k}{2\langle x_1\rangle} }  \frac{ \lambda^{s_2}  \partial_1  \psi^{\tau} }{\langle x_1\rangle } \right\|_{L^2_{t,x}}^2 \\
	&\,+ t^{\frac 12} e^{\frac 32M}  \left\|\lambda^{s_2}   \psi\right\|_{L^2} \left\| e^{M \frac{x_k}{2\langle x_1\rangle} }  \frac{ \lambda^{s_2+1}  \psi }{\langle x_1\rangle } \right\|_{L^2_{t,x}}.
\end{align*}

For \eqref{Big_en_maint1_2},  we can use integration by parts, \eqref{est_transf_0_order_P} , \eqref{est_transf_k} and \eqref{est_transf_k2} to obtain
\begin{align*}
	\eqref{Big_en_maint1_2} = &\,e^{M}\lambda^{2s_2+1}\int_0^t \int_{\mathbb{R}^d}  \mathrm{Im} \left(   \sum_{ i\neq 1} \sum_{\mu\sim \lambda} a_i  P_\mu \partial_i \psi \cdot  e^{q (x,D_1)} \overline{U}\right)  \mathrm{d} x \mathrm{d} s\\
	\lesssim& 	  \sum_{ i\neq 1}   e^{ 2M}\left\|\langle x_1\rangle^2  a  \right\|_{L^{\infty}(\mathbb{R}^d)} \cdot  \left\|e^{M \frac{x_k}{2\langle x_k\rangle} }  \frac{\lambda^{s_2}  \partial_i \psi}{\langle x_1\rangle }  \right\|_{L^2_{t,x}}   \left\|e^{M \frac{x_k}{2\langle x_k\rangle} }  \frac{\lambda^{s_2+1}P_\mu^{\frac 12}  e^{q (x,D_1)} \overline{U}}{\langle x_k\rangle }  \right\|_{L^2_{t,x}} \\
	\lesssim &\,\, \epsilon e^{ 2M} M_1^2 \sup_{\tau}   \left\| e^{M \frac{x_k}{2\langle x_1\rangle} }  \frac{\lambda^{s_2} \partial_k \psi^{\tau }}{\langle x_1\rangle } \right\|_{L^2_{t,x}}  \sup_{\tau^\prime}   \left\| e^{M \frac{x_k}{2\langle x_1\rangle} }  \frac{\lambda^{s_2+1} \psi^{\tau^\prime }}{\langle x_1\rangle } \right\|_{L^2_{t,x}}\\
	\lesssim &\,\, \epsilon e^{ 2M} M_1^2  \sup_{\tau}   \left\| e^{M \frac{x_k}{2\langle x_1\rangle} }  \frac{\lambda^{s_2} \partial_1 \psi^{\tau }}{\langle x_1\rangle } \right\|_{L^2_{t,x}}^2.  
\end{align*}
Recall that in \eqref{est_transf_k}, when \( i \neq k \),  we have  a  $ \epsilon  $ in the right hand side.

Next for  \eqref{Big_en_maint3}, we write it into
\begin{align}
	\eqref{Big_en_maint3}=&\sum_{\mu\sim \lambda}	e^{M} \lambda^{2s_2+1}\int_0^t \int_{\mathbb{R}^d}  \mathrm{Im} \left(     [e^{q (x,D)}, b  P_\mu \nabla] \overline{\psi}\overline{U}\right)  \mathrm{d} x \mathrm{d} s \label{Big_en_maint3_1}\\
	&	+e^{M} \lambda^{2s_2+1}\int_0^t \int_{\mathbb{R}^d}  \mathrm{Im} \left(   ( e^{q (x,D)}\mathcal{R}) \overline{U}\right)  \mathrm{d} x \mathrm{d} s.\label{Big_en_maint3_2}
\end{align}
By   Lemma \ref{lem_com} and similar analysis of $\left\|\langle x_1\rangle \nabla a  \right\|_{L^{\infty}(\mathbb{R}^d)} $ in  \eqref{Rem_t2}, we have
\begin{align*}
	\eqref{Big_en_maint3_1}\lesssim &\,  t^{\frac12} e^{\frac 32M}\left\|\langle x_1\rangle  \nabla  b \right\|_{L^{\infty}(\mathbb{R}^d)} \cdot \sup_{t }\left\| \lambda^{s_2} P_\lambda \psi \right\|_{L^{2}(\mathbb{R}^d)}   \left\|e^{M \frac{x_1}{2\langle x_1\rangle} }  \frac{\lambda^{s_2+1} \overline{U}}{\langle x_1\rangle }  \right\|_{L^2_{t,x}}\\
	\lesssim &\, t^{\frac12} e^{M}M_1^2  \sup_{t }\left\| \lambda^{s_2} \psi\right\|_{L^{2}(\mathbb{R}^d)}  \sup_{\tau}   \left\| e^{M \frac{x_1}{2\langle x_1\rangle} }  \frac{\lambda^{s_2}  \partial_1\psi^{\tau }}{\langle x_1\rangle } \right\|_{L^2_{t,x}}\\
	\lesssim &  \,t^{\frac12} e^{M} M_1^2  \sup_{t }\left\| \lambda^{s_2} \psi\right\|_{L^{2}(\mathbb{R}^d)}  \sup_{\tau}   \left\| e^{M \frac{x_1}{2\langle x_1\rangle} }  \frac{\lambda^{s_2}  \partial_1\psi^{\tau }}{\langle x_1\rangle } \right\|_{L^2_{t,x}}
\end{align*}
Similar as the analysis of \eqref{big_me_t6_t3}, we can obtain 
\begin{align*}
	\eqref{Big_en_maint3_2}\lesssim   t^{\frac 12} e^{\frac 32 M}   M_1^2 \left\|\lambda^{s_2}  P_\lambda   \phi\right\|_{L^{2}(\mathbb{R}^d)} \left\| e^{M \frac{x_1}{2\langle x_1\rangle} } \frac{\lambda^{s_2}\partial_1 \psi }{\langle x_1\rangle }  \right\|_{L^2_{t,x}}.
\end{align*}

By combining the estimates of  \eqref{Big_en_maint1}-\eqref{Big_en_maint3},  \eqref{bstrap_ass1}-\eqref{bstrap_ass5} and summing over $ \lambda $,   we can use  \eqref{def_t1}  and \eqref{def_t2} to obtain \eqref{bstrap_re4} for $ e_1$: 
\begin{align*}
	&	\sum_{\lambda} \lambda^{2s_2+1} \sup_{\tau}\left\| P_{\lambda, e_1}  \phi^\tau\right\|_{L^2(\mathbb{R}^d)}^2\\
	\leq  &     \sum_{\lambda} \lambda^{2s_2+1}		\|U\|_{L^{2}(\mathbb{R}^d)}\nonumber\\
	\lesssim   & \sum_{\lambda} \lambda^{2s_2+1}		\|U_0\|_{L^{2}(\mathbb{R}^d)} \nonumber\\
	&+t^{\frac12} e^{\frac 32M} M_1^2 \cdot \sup_{t }\left(\sum_{\lambda}\left\| \lambda^{s_2} \psi\right\|_{L^{2}(\mathbb{R}^d)}^2 \right)^{\frac 12} \left(\sup_{\tau}  \sum_{\lambda} \left\| e^{M \frac{x_1}{2\langle x_1\rangle} }  \frac{\lambda^{s_2} \partial_1 \psi^{\tau }}{\langle x_1\rangle } \right\|_{L^2_{t,x}}^2\right)^{\frac 12}\nonumber\\
	&+\epsilon e^{ 2M} M_1^2  \sup_{\tau}  \sum_{\lambda}   \left\| e^{M \frac{x_k}{2\langle x_1\rangle} }  \frac{\lambda^{s_2} \partial_1 \psi^{\tau }}{\langle x_1\rangle } \right\|_{L^2_{t,x}}^2 \nonumber\\
	\lesssim & \,   \|\phi_0\|_{H^{s_2+\frac 12}}^2+t^{\frac 12 }  e^{\frac 32 M}  M_1^4 e^{2M}  + \epsilon e^{ 2M} M_1^2  e^{4M}  \nonumber\\
	\lesssim &\, M_1^2,\nonumber
\end{align*}
where  $ U= e^{q (x,D)} \psi $ and $ q (x,D)$ is zeroth order pseudo-differential operator.

	Now we are going to prove \eqref{big_w_est4}. As \eqref{PqNLS}, applying $ P_\lambda, \lambda \geq 1 $  on \eqref{qNLS}, we have  
	\begin{align} \label{PqNLS_without}
		\sqrt{-1}  P_{\lambda} \phi_t+\partial_i\left( g^{ij}\left(\phi,\overline{\phi}\right)   \partial_j (P_{\lambda} \phi)\right)=\mathcal{N}_4,
	\end{align}
	where 
	\begin{align*}
		\mathcal{N}_4:=&\sum_{\mu\sim \lambda} a \cdot  P_{\lambda}   P_\mu \nabla \phi + \sum_{\mu\sim \lambda} b\cdot  P_{\lambda}   P_\mu \nabla\overline{ \phi }+\mathcal{R}_2, \\
		\mathcal{R}_2:=& P_{\lambda}  F(\phi, \overline{\phi}, P_{\leq 1} \nabla \phi,  P_{\leq 1} \nabla \overline{\phi}) + \sum_{\mu\sim \lambda} [  P_{\lambda},a  ]  P_\mu \nabla \phi + \sum_{\mu}  [P_{\lambda},b ]  P_\mu \nabla\overline{ \phi } \nonumber\\
		&-\partial_i\left(\left[P_{\lambda}, h^{ij}\left(\phi,\overline{\phi}\right) \right]  \partial_j  \phi \right).   \nonumber
	\end{align*}
	Multiplying \eqref{PqNLS_without} by \( x_k^2 \), we derive
	\begin{align}   \label{eq_big_x2}
		&	\sqrt{-1}x_k^2P_{\lambda}\phi_{ t}+\sum_{i,j}\partial_i\left( g^{ij}   \partial_j \left(x_k^2P_{\lambda}\phi\right)\right)\\
		&=4 x_k \partial_k  P_{\lambda}\phi +2 x_k h^{kj}  \partial_j P_{\lambda}\phi+2 x_k h^{ik} \partial_i P_{\lambda}\phi +2 g^{kk}P_{\lambda}\phi +2 x_k \partial_i h^{ik}P_{\lambda}\phi +x^2_k \cdot\mathcal{N}_4,\nonumber
	\end{align}
	where
	\begin{align*}
		&	\partial_i\left( g^{ij} \partial_j P_{\lambda}\phi\right) \cdot x^2_k= \partial_i\left(  x^2_k g^{ij} \partial_j P_{\lambda}\phi\right)  -2 g^{kk}_0x_k \partial_k  P_{\lambda}\phi -2 x_k h^{kj} \partial_j P_{\lambda}\phi\\ 
		%	&=\partial_i\left( g^{ij} \partial_j\left(P_{\lambda}\phix^2_k\right)\right) -2\partial_i\left(x_k g^{ik}P_{\lambda}\phi \right) -2g^{kk}_0 x_k \partial_k  P_{\lambda}\phi -2 x_k h^{kj} \partial_j P_{\lambda}\phi\\
		&=\partial_i\left( g^{ij} \partial_j\left(P_{\lambda}\phi x^2_k\right)\right)   -4 g^{kk}_0x_k \partial_k  P_{\lambda}\phi -2 x_k h^{kj} \partial_j P_{\lambda}\phi-2 x_k h^{ik} \partial_i P_{\lambda}\phi-2 g^{kk}P_{\lambda}\phi -2 x_k \partial_i h^{ik}P_{\lambda}\phi.
	\end{align*}
	Multiplying \eqref{eq_big_x2} with $x^2_k P_{\lambda} \overline{ \phi }   $, taking the imaginary part, integrating on $ [0,t]\times \mathbb{R}^d $, multiplying by $ \lambda^{2(s_2-2)} $, summing over $ \lambda\geq 1 $ and $k=1,\dots, d$, we have
	\begin{align}
		\sum_{k=1}^d\sum_{\lambda\geq 1} \lambda^{2(s_2-2)}&\, \int_{\mathbb{R}^{d}} \left| x_k^2P_{\lambda}\phi\right|^2 \mathrm{d} x  \lesssim   \sum_{k=1}^d\sum_{\lambda\geq 1} \lambda^{2(s_2-2)} \int_{\mathbb{R}^{d}} \left| x_k^2P_{\lambda}\phi_0\right|^2 \mathrm{d} x  \nonumber\\
		&+\sum_{k=1}^d\sum_{\lambda\geq 1} \lambda^{2(s_2-2)} \int_0^t \int_{\mathbb{R}^{d}} x^2_k \cdot\mathcal{N}_4 \cdot x^2_k P_{\lambda} \overline{ \phi }  \mathrm{d} x \mathrm{d} s \label{B_we_est_t1}\\
		&+\sum_{k=1}^d\sum_{\lambda\geq 1} \lambda^{2(s_2-2)} \int_0^t \int_{\mathbb{R}^{d}}  (4 x_k \partial_k  P_{\lambda}\phi +2 x_k h^{kj}  \partial_j P_{\lambda}\phi\nonumber\\
		&+2 x_k h^{ik} \partial_i P_{\lambda}\phi +2 g^{kk}P_{\lambda}\phi +2 x_k \partial_i h^{ik}P_{\lambda}\phi)\cdot x^2_k P_{\lambda} \overline{ \phi }  \mathrm{d} x \mathrm{d} s. \label{B_we_est_t2}
	\end{align}
	Since
	\begin{align*}
		\left\|( \sum_{\mu\sim \lambda} [  P_{\lambda},a  ]  P_\mu \nabla \phi) x^2_k P_{\lambda} \overline{ \phi } \right\|_{L^1} \lesssim & \left\|\nabla a\right\|_{L^\infty} \sup_\tau\left\| P_\lambda \phi^{\tau}\right\|_{L^2}  \left\|x^2_k P_{\lambda} \overline{ \phi }  \right\|_{L^2},\\
		\left\|(\partial_i\left(\left[P_{\lambda}, h^{ij}\left(\phi,\overline{\phi}\right) \right]  \partial_j  \phi \right)) \cdot (x^2_k P_{\lambda} \overline{ \phi } )\right\|_{L^1} \lesssim &   \left\|\left(\left[P_{\lambda}, \partial_ih^{ij}\left(\phi,\overline{\phi}\right) \right]  \partial_j  \phi \right) \cdot (x^2_k P_{\lambda} \overline{ \phi } )\right\|_{L^1}\\
		&+  \left\|\left(\left[P_{\lambda}, h^{ij}\left(\phi,\overline{\phi}\right) \right] \partial_i \partial_j  \phi \right) \cdot (x^2_k P_{\lambda} \overline{ \phi } )\right\|_{L^1}\\
		\lesssim &  \left\|\nabla \partial_ih^{ij}\ \right\|_{L^\infty}  \sup_{\tau} \left\|P_{\lambda}  \phi^\tau \right\|_{L^2} \left\|x^2_k P_{\lambda} \phi \right\|_{L^2}\\
		&+\left\|\nabla  h^{ij}\ \right\|_{L^\infty}  \sup_{\tau} \left\| \lambda P_{\lambda}  \phi^\tau \right\|_{L^2} \left\|x^2_k P_{\lambda} \phi \right\|_{L^2}
	\end{align*}
	and 
	\begin{align*}
		\left\|P_{\lambda}  F(\phi, \overline{\phi}, P_{\leq 1} \nabla \phi,  P_{\leq 1} \nabla \overline{\phi}) x^2_k P_{\lambda} \overline{ \phi } \right\|_{L^1}\lesssim \left\|\phi \right\|_{L^\infty} \left\|P_\lambda  \phi   \right\|_{L^2}
		\left\|x^2_k P_{\lambda} \phi \right\|_{L^2},
	\end{align*}
	we have
	\begin{align*} 
		\eqref{B_we_est_t1} \leq & \, \sum_{k=1}^d\sum_{\lambda\geq 1} \lambda^{2(s_2-2)}  C t    \left\| x_k^2 P_{\lambda}\phi\right\|_{L^{2}(\mathbb{R}^d)}  \left(\left\| x_k^2  a\cdot  P_{\lambda}  \nabla \phi \right\|_{L^{2}(\mathbb{R}^d)} + \left\|x_k^2  b\cdot  P_{\lambda}  \nabla \phi \right\|_{L^{2}(\mathbb{R}^d)}\right) \\
		&+ \sum_{k=1}^d\sum_{\lambda\geq 1} \lambda^{2(s_2-2)} C t\left\|x_k^2 P_{\lambda}\phi \cdot 
		x_k^2  \mathcal{R}_2\right\|_{L^{1}(\mathbb{R}^d)}\\
		\leq &  \, \sum_{k=1}^d\sum_{\lambda\geq 1} \lambda^{2(s_2-2)}  C t    \left\| x_k^2 P_{\lambda}\phi\right\|_{L^{2}(\mathbb{R}^d)} ( \left\| x_k^2 a \right\|_{L^\infty } + \left\| x_k^2 b \right\|_{L^\infty } )\left\| P_{\lambda}  \nabla \phi \right\|_{L^{2}(\mathbb{R}^d)}  \nonumber\\
		&+ \sum_{k=1}^d\sum_{\lambda\geq 1} \lambda^{2(s_2-2)} Ct   \left\|\nabla a\right\|_{L^\infty} \sup_\tau\left\| P_\lambda \phi^{\tau}\right\|_{L^2}  \left\|x^2_k P_{\lambda} \overline{ \phi }  \right\|_{L^2}\\
		&+ \sum_{k=1}^d\sum_{\lambda\geq 1} \lambda^{2(s_2-2)} Ct( \left\|\nabla \partial_ih^{ij}\ \right\|_{L^\infty}  \sup_{\tau} \left\|P_{\lambda}  \phi^\tau \right\|_{L^2} \left\|x^2_k P_{\lambda} \phi \right\|_{L^2}\\
		&+ \sum_{k=1}^d\sum_{\lambda\geq 1} \lambda^{2(s_2-2)} Ct\left\|\nabla  h^{ij}\ \right\|_{L^\infty}  \sup_{\tau} \left\| \lambda P_{\lambda}  \phi^\tau \right\|_{L^2} \left\|x^2_k P_{\lambda} \phi \right\|_{L^2})\\
		&+ \sum_{k=1}^d\sum_{\lambda\geq 1} \lambda^{2(s_2-2)} Ct \left\|\phi \right\|_{L^\infty} \left\|P_\lambda  \phi   \right\|_{L^2}
		\left\|x^2_k P_{\lambda} \phi \right\|_{L^2}\\
		\leq& \,  Ct\left( \sum_{k=1}^d\sum_{\lambda\geq 1} \lambda^{2(s_2-2)} \int_{\mathbb{R}^{d}} \left| x_k^2P_{\lambda}\phi\right|^2 \mathrm{d} x \right)^{\frac 12}  M_1^2\left(\sum_{k=1}^d\sum_{\lambda\geq 1} \lambda^{2(s_2-2)} \int_{\mathbb{R}^{d}} \left| P_{\lambda}\phi\right|^2 \mathrm{d} x \right)^{\frac 12}\\
		\leq &\, \frac{1}{100}   \sum_{k=1}^d\sum_{\lambda\geq 1} \lambda^{2(s_2-2)} \int_{\mathbb{R}^{d}} \left| x_k^2P_{\lambda}\phi\right|^2 \mathrm{d} x+ Ct  M_1^4    \left\|\phi\right\|_{H^{s_2}}^2\\
		\leq &\, \frac{1}{100}   \sum_{k=1}^d\sum_{\lambda\geq 1} \lambda^{2(s_2-2)} \int_{\mathbb{R}^{d}} \left| x_k^2P_{\lambda}\phi\right|^2 \mathrm{d} x+Ct M_1^8,
	\end{align*}
	where we used Lemma \ref{Lem_phi_infty}  and similar analysis of $\left\|\langle x_1\rangle \nabla a  \right\|_{L^{\infty}(\mathbb{R}^d)} $ in  \eqref{Rem_t2} to estimate terms containing compositions of $a,b, h$.
	Similarly for \eqref{B_we_est_t2}, we have
	\begin{align*}
		\eqref{B_we_est_t2} \leq& \, C t \cdot \sum_{k=1}^d\sum_{\lambda\geq 1} \lambda^{2(s_2-2)}  \sum_{i=1}^d \left\| x^2_k P_{\lambda}\phi  \right\|_{L^{2}(\mathbb{R}^d)} \left\|x_k \partial_i P_{\lambda}\phi \right\|_{L^{2}(\mathbb{R}^d)} M_1^2\\
		&+Ct \cdot \sum_{k=1}^d\sum_{\lambda\geq 1} \lambda^{2(s_2-2)}   \left\| x^2_k P_{\lambda}\phi  \right\|_{L^{2}(\mathbb{R}^d)} \left\| \langle x \rangle  P_{\lambda}\phi \right\|_{L^{2}(\mathbb{R}^d)} M_1^2\\
		\leq&\, Ct  \sum_{k=1}^d\sum_{\lambda\geq 1} \lambda^{2(s_2-2)}   \sum_{i=1}^d \left\| x^2_k P_{\lambda}\phi  \right\|_{L^{2}(\mathbb{R}^d)}^{\frac 32} \left\|\partial_{i} \partial_i P_{\lambda}\phi \right\|_{L^2 } M_1^2\\
		&+Ct \cdot  \sum_{k=1}^d\sum_{\lambda\geq 1} \lambda^{2(s_2-2)}    \left\| \langle x \rangle^2  P_{\lambda}\phi \right\|_{L^{2}(\mathbb{R}^d)}^{\frac 32}   \left\|   P_{\lambda}\phi \right\|_{L^{2}(\mathbb{R}^d)}^{\frac 12} M_1^2\\
		\leq & \,	\frac{1}{100}   \sum_{k=1}^d\sum_{\lambda\geq 1} \lambda^{2(s_2-2)} \int_{\mathbb{R}^{d}} \left| x_k^2P_{\lambda}\phi\right|^2 \mathrm{d} x+ Ct     \left\|\phi\right\|_{H^{s_2}}^2 M_1^4\\
		\leq &\, \frac{1}{100}   \sum_{k=1}^d\sum_{\lambda\geq 1} \lambda^{2(s_2-2)} \int_{\mathbb{R}^{d}} \left| x_k^2P_{\lambda}\phi\right|^2 \mathrm{d} x+Ct M_1^8.
	\end{align*}
	Combing the estimates of 	\eqref{B_we_est_t1} and 	\eqref{B_we_est_t2}, we can use \eqref{def_t1} and \eqref{def_t2} to obtain
	\begin{align}  \label{Big_wt_est}
		\sum_{k=1}^d	\sum_{\lambda\geq 1} \lambda^{2(s_2-2)} \int_{\mathbb{R}^{d}} \left| x_k^2P_{\lambda}\phi\right|^2 \mathrm{d} x  \lesssim &\, \sum_{\lambda\geq 1} \lambda^{2(s_2-2)} \int_{\mathbb{R}^{d}} \left| x_k^2P_{\lambda}\phi_0\right|^2 \mathrm{d} x  +t M_1^8 \\
		\lesssim& \,M_1^2, \nonumber
	\end{align} 
	which combing with \eqref{bstrap_re1} gives \eqref{big_w_est4}. Note that $\frac{1}{100}   \sum_{k=1}^d\sum_{\lambda\geq 1} \lambda^{2(s_2-2)} \int_{\mathbb{R}^{d}} \left| x_k^2P_{\lambda}\phi\right|^2 \mathrm{d} x$ in the above estimates has been absorbed  and
	\begin{align*}
		\sum_{\lambda\geq 1} \lambda^{2(s_2-2)} \int_{\mathbb{R}^{d}} \left| x_k^2P_{\lambda}\phi_0\right|^2 \mathrm{d} x  \leq& \,  \sum_{\lambda\geq 1} \lambda^{2(s_2-2)} \int_{\mathbb{R}^{d}} \left|P_{\lambda} (x_k^2\phi_0)\right|^2 \mathrm{d} x + \sum_{\lambda\geq 1} \lambda^{2(s_2-2)} \int_{\mathbb{R}^{d}} \left|  [P_{\lambda},x_k^2]\phi_0\right|^2 \mathrm{d} x  \\
		\lesssim &\,   \left\| \langle x\rangle ^2 \phi_0 \right\|_{H^{s_2-2}}+ \sum_{\lambda\geq 1} \left\|\langle x\rangle \lambda^{s_2-3}P_\lambda^\prime \phi_0\right\|_{L^{2}(\mathbb{R}^d)}^2 \\
		\lesssim &\, \left\| \langle x\rangle ^2 \phi_0 \right\|_{H^{s_2-2}}+ \sum_{\lambda\geq 1} \left\| \lambda^{s_2-3}P_\lambda^\prime(\langle x\rangle  \phi_0)\right\|_{L^{2}(\mathbb{R}^d)}^2\\
		&\,+  \sum_{\lambda\geq 1} \left\| \lambda^{s_2-3} [ P_\lambda^\prime,\langle x\rangle ] \phi_0\right\|_{L^{2}(\mathbb{R}^d)}^2 \\
		\lesssim& \,  \left\| \langle x\rangle ^2 \phi_0 \right\|_{H^{s_2-2}} + \sum_{\lambda\geq 1} \left\| \lambda^{s_2-4} P_\lambda^{\prime \prime}  \phi_0\right\|_{L^{2}(\mathbb{R}^d)}^2 \\
		\lesssim&\, M_1^2.
	\end{align*}

	\subsection{Smallness of \texorpdfstring{$w$}{w}} \label{s_big.3}
	
	We derive the equation for $w$ by subtracting the projected linear equation $P_{\lambda}$\eqref{eq_v} from \eqref{PqNLS_without}:
	\begin{equation}\label{eq_w0} 
		\begin{cases}
			\sum_{\mu} a \cdot  P_{\lambda}   P_\mu \nabla \phi + \sum_{\mu} b\cdot  P_{\lambda }   P_\mu \nabla\overline{ \phi }+\mathcal{R}_2\\
			~~~=\sqrt{-1}  P_{\lambda }   w_t+\partial_i\left( g^{ij}\left(\phi,\overline{\phi}\right)   \partial_j (P_{\lambda } \phi)\right)- \Delta  P_{\lambda }  v \\
			~~~ =\sqrt{-1}  P_{\lambda }  w_t+ \partial_i\left( g^{ij}\left(\phi,\overline{\phi}\right) \partial_j  ( P_{\lambda }  w)\right)+ \partial_i\left( (g^{ij}\left(\phi,\overline{\phi}\right) -g_0^{i j}) \partial_j  (P_{\lambda }  v)\right) \\
			w(0,x)=0.
		\end{cases}	
		%		 =& \sqrt{-1}w_t+ \partial_i\left( g^{ij}\left(\phi,\overline{\phi}\right) w_j\right)+ \partial_i\left( (g^{ij}\left(\phi,\overline{\phi}\right) -\delta_{i j})v_j\right)  \nonumber
	\end{equation}
	or equivalently
	\begin{equation} \label{eq_w}
		\begin{cases}
			\sqrt{-1}  \partial_t  P_{\lambda }  w+ \partial_i\left( g^{ij}\left(\phi,\overline{\phi}\right) \partial_j  P_{\lambda }  w\right)=\mathcal{N}_5\\
			w(0,x)=0,
		\end{cases}
	\end{equation}
	where
	\begin{align*}
		\mathcal{N}_5:=	&\sum_{\mu\sim \lambda} a \cdot  P_{\lambda }   P_\mu \nabla \phi + \sum_{\mu\sim \lambda} b\cdot  P_{\lambda }   P_\mu \nabla\overline{ \phi }+\mathcal{R}_2 - \partial_i\left( h^{ij}\left(\phi,\overline{\phi}\right)   \partial_j  P_{\lambda }  v).\right.
	\end{align*}
	
	To estimate $w$, we multiply \eqref{eq_w} by $P_{\lambda} \overline{w}$, take the imaginary part, integrate over $[0,t] \times \mathbb{R}^d$, multiply by $\lambda^{2(s_2-1)}$, and sum over $\lambda \geq 1$. This gives
	\begin{align}
		\left\| w\right\|^2_{H^{s_2-1}(\mathbb{R}^d)}\sim   \sum_{\lambda \geq 1} \lambda^{2 (s_2-1)}	\left\|	 P_{\lambda }  w\right\|_{L^2}^2 = &  \,\sum_{\lambda \geq 1} \lambda^{2(s_2-1)}\int_0^t \int_{\mathbb{R}^{d}}  \mathrm{Im} (\mathcal{N}_5P_{\lambda }  \overline{w} ) \mathrm{d} x\mathrm{d} s \nonumber\\
		\lesssim &\, t^{\frac 12} \left(\sum_{\lambda \geq 1} \left\|\lambda^{ 2 (s_2-1)}P_{\lambda }  w\right\|_{L^2(\mathbb{R}^d)}^2 \right )^{\frac 12} \left(\sum_{\lambda \geq 1} \lambda^{ 2 (s_2-1)}\left\|\mathcal{N}_5\right\|_{L^2_{t,x}}^2\right)^{\frac12 }\nonumber\\
		\lesssim &\, t^{\frac 58}
		\left(\sum_{\lambda \geq 1} \lambda^{ 2 (s_2-1)}\left\|\mathcal{N}_5\right\|_{L^2_{t,x}}^2\right)^{\frac12 }\nonumber,
	\end{align}
	where
	\begin{align}
		\sum_{\lambda \geq 1} \lambda^{ 2s_2-2}\left\|\mathcal{N}_5\right\|_{L^2_{t,x}}^2
		\lesssim &\,  \sum_{\lambda \geq 1} \lambda^{ 2 s_2-2}\sum_{\mu\sim \lambda}
		\left\|a \cdot    P_\mu \nabla \phi\right\|_{L^2_{t,x}}\label{w_t1}^2 \\
		&+ \sum_{\lambda \geq 1} \lambda^{ 2s_2-2} \sum_{\mu\sim \lambda}
		\left\| b\cdot     P_\mu \nabla\overline{ \phi }\phi\right\|_{L^2_{t,x}}^2\label{w_t2}\\
		&+ \sum_{\lambda \geq 1} \lambda^{ 2s_2-2}
		\left\|\mathcal{R}\right\|_{L^2_{t,x}}^2\label{w_t3}\\
		&+ \sum_{\lambda \geq 1} \lambda^{ 2s_2-2}
		\left\|\partial_i\left( h^{ij}\left(\phi,\overline{\phi}\right)   \partial_j  (P_{\lambda }  v)\right)\right\|_{L^2_{t,x}}^2. \label{w_t4}
	\end{align}
	Note that $ w(0,x)=0. $
	By \eqref{bstrap_ass1} and Lemma \ref{Lem_phi_infty}, we have
	\begin{align*}
		\eqref{w_t1}\lesssim &\, \left\|  a\right\|_{L^\infty}^2  \sum_{k=1}^d  \sum_{\lambda \geq 1} \lambda^{ 2s_2-2}  \left\| \partial_k  P_{\lambda } \phi   \right\|_{L^2_{t,x}}^2 \\
		\lesssim &\,  t \left\|  \langle \nabla \rangle  \phi \right\|_{H^{s_2-3}}^2   \left\|  \phi \right\|_{H^ {s_2}}^2 \\
		\lesssim &\,   t M_0^4
	\end{align*}
	and \eqref{w_t2} can be estimated similarly.  For \eqref{w_t3}, we decompose:
	\begin{align}
		\eqref{w_t3} = 	&\,2\sum_{\lambda \geq 1} \lambda^{ 2s_2-2}
		\left\| \sum_{\mu\sim \lambda} [  P_{\lambda, e_1},a  ]  P_\mu \nabla \phi\right\|_{L^2_{t,x}}^2\label{w_t3_t1}\\
		&+2\sum_{\lambda \geq 1} \lambda^{ 2s_2-2}	\left\|  \sum_{\mu}  [P_{\lambda, e_1},b ]  P_\mu \nabla\overline{ \phi } \right\|_{L^2_{t,x}}^2 \label{w_t3_t2}\\
		&+2\sum_{\lambda \geq 1} \lambda^{ 2s_2-2}	\left\|  \partial_i\left(\left[P_{\lambda, e_1}, h^{ij}\left(\phi,\overline{\phi}\right) \right]  \partial_j  \phi \right)  \right\|_{L^2_{t,x}}^2
		\label{w_t3_t3}\\
		&+ 2\sum_{\lambda \geq 1} \lambda^{ 2s_2-2} \left\| P_{\lambda, e_1}  F(\phi, \overline{\phi}, P_{\leq 1} \nabla \phi,  P_{\leq 1} \nabla \overline{\phi})\right\|_{L^2_{t,x}}^2.
		\label{w_t3_t4}
	\end{align}
	For \eqref{w_t3_t1}, we have
	\begin{align*}
		\eqref{w_t3_t1} \lesssim &  \left\|  \nabla a \right\|_{L^{\infty}(\mathbb{R}^d)}^2   \sum_{k=1}^d\sum_{\lambda \geq 1} \lambda^{ 2s_2-2} \left\|  \partial_k  P_\lambda \phi   \right\|_{L^2_{t,x}}^2 \\
		\lesssim &  \,  t \left\|  \langle \nabla \rangle  \phi \right\|_{H^{s_2-2}}^2   \left\|  \phi \right\|_{H^ {s_2}}^2 \\
		\lesssim &\,   t M_0^4,
	\end{align*}
	and 	\eqref{w_t3_t3} can be estimated similarly. Similarly, by \eqref{bstrap_ass1}, 
	\begin{align*}
		\eqref{w_t3_t3} \lesssim &    \sum_{\lambda \geq 1} \lambda^{ 2s_2-2}  \left\|  \left[P_{\lambda },\partial_i h^{ij}\left(\phi,\overline{\phi}\right) \right]  \partial_j  \phi  \right\|_{L^2_{t,x}}^2\\
		&+\sum_{\lambda \geq 1} \lambda^{ 2s_2-2}  \left\| \left[P_{\lambda },h^{ij}\left(\phi,\overline{\phi}\right) \right] \partial_i  \partial_j  \phi  \right\|_{L^2_{t,x}}^2\\
		\lesssim&\, t  \left\|   \nabla \partial_i h^{ij}\left(\phi,\overline{\phi}\right) \right\|_{L^\infty}^2  \sum_{\lambda \geq 1} \lambda^{ 2s_2-4}\left\|  P_\lambda  \partial_j \phi \right\|_{L^2 (\mathbb{R}^d)}^2 \\
		&+ t  \left\|   \nabla  h^{ij}\left(\phi,\overline{\phi}\right) \right\|_{L^\infty}^2  \sum_{\lambda \geq 1} \lambda^{ 2s_2-4}\left\|  P_\lambda \partial_i    \partial_j \phi \right\|_{L^2 (\mathbb{R}^d)}^2 \\
		\lesssim &\,	t  \left\| \langle \nabla \rangle^2 \phi \right\|_{H^{s_2-3}}^2 \left\|\phi \right\|_{H^{s_2}}^2  \\
		\lesssim&\,    t M_0^4,
	\end{align*}
	and
	\begin{align*}
		\eqref{w_t3_t4} \lesssim &\,t\left\|  \phi \right\|_{L^{\infty}(\mathbb{R}^d)}^2  \sum_{\lambda \geq 1} \lambda^{ 2s_2}\left\|  P_{\lambda,e_1} P_{\leq 1}  \phi\right\|_{L^{2}(\mathbb{R}^d)}^2  \\
		\lesssim &\,t\left\|  \phi \right\|_{H^{s_2-\frac 52}(\mathbb{R}^d)}^2   \left\|    \phi\right\|_{H^{s_2}(\mathbb{R}^d)}^2 \\
		\lesssim&\, t M_0^8.
	\end{align*}
	
	Lastly for \eqref{w_t4}, we split it as
	\begin{align*}
		\eqref{w_t4}	\lesssim &\sum_{\lambda \geq 1} \lambda^{ 2s_2-2}
		\left\|\partial_i h^{ij}\left(\phi,\overline{\phi}\right)   \partial_j  (P_{\lambda }  v) \right\|_{L^2_{t,x}}^2\\
		&+\sum_{\lambda \geq 1} \lambda^{ 2s_2-2}
		\left\| h^{ij}\left(\phi,\overline{\phi}\right)   \partial_i\partial_j  (P_{\lambda }  v) \right\|_{L^2_{t,x}}^2\\
		\lesssim &  \,t \cdot  \left\|   \partial_i h^{ij}\left(\phi,\overline{\phi}\right) \right\|_{L^\infty}   \sum_{\lambda \geq 1} \lambda^{ 2s_2-2}
		\left\|  \partial_j   P_{\lambda }  v   \right\|_{L^2 (\mathbb{R}^d)}^2\\
		&+  \sum_j  \left\| \langle x_j \rangle  \partial_i h^{ij}\left(\phi,\overline{\phi}\right) \right\|_{L^\infty}   \sum_{\lambda \geq 1} \lambda^{ 2s_2-2}
		\left\|  \frac { \partial_i\partial_j  (P_{\lambda }  v) }{\langle x_j \rangle }\right\|_{L^2_{t,x}}^2\\
		\lesssim & \, t \cdot  \left\|    \phi  \right\|_{H^{s_2-2}}   
		\left\|  \phi_0   \right\|_{H^{s_2} }^2 +\left\|  \langle x  \rangle \phi \right\|_{H^{s_2-2}}  \sum_j   \lambda^{ 2s_2} 	\left\|  \frac { \partial_j  (P_{\lambda }  v) }{\langle x_j \rangle }\right\|_{L^2_{t,x}}^2 \\
		\lesssim &  \,tM_0^4 +  M_1^4,
	\end{align*}
	where we can use the same way as Section \ref{s3} to obtain momentum estimates for the linear Schr\"odinger equation \eqref{eq_v} as follows:
	\begin{align*}
		\sum_j   \lambda^{ 2s_2} 	\left\|  \frac { \partial_j  (P_{\lambda }  v) }{\langle x_j \rangle }\right\|_{L^2_{t,x}}^2 \lesssim \left\| \phi_0\right\|_{H^{s_2+\frac 12}}^2.
	\end{align*}

	If $ t\lesssim     e^{-40M }$, then 
	\begin{align*}
		t^{\frac 58} M_1^2    e^{4M } \lesssim&\, t^{\frac12} e^{- 5M }  M_1^2    e^{4M }  \\
		\lesssim &\,  t^{\frac12} e^{- M }  M_1^{2}   \\
		\lesssim &\,  t^{\frac 12},
	\end{align*}
	where $ M_1\ll M. $
	Thus, by combing the estimates above, we obtain
	\begin{align*}
		\left\| w\right\|^2_{H^{s_2-1}(\mathbb{R}^d)} \lesssim& \,t^{\frac 58}
		\left(\sum_{\lambda \geq 1} \lambda^{ 2s_2-2}\left\|\mathcal{N}_5\right\|_{L^2_{t,x}}^2\right)^{\frac12 }\\
		\lesssim & \,t^{\frac 12}.
	\end{align*}

	We now turn to the weighted estimate $\left\| \langle x \rangle^2 w(t,\cdot) \right\|_{H^{s_2-3}} \lesssim t^{1/4}$ in \eqref{bstrap_re2}. Multiplying the weight $x_k^2$ to \eqref{eq_w} gives:
	\begin{align}   \label{w_x2}
		&	\sqrt{-1} \partial_t   ( x_k^2 P_{\lambda }  w )+\sum_{i,j}\partial_i\left( g^{ij}(\phi,\overline{\phi })   \partial_j \left(x_k^2  P_{\lambda }  w\right)\right)\\
		&=4 x_k \partial_k  P_{\lambda }  w +2 x_k h^{kj}  \partial_j P_{\lambda }  w+2 x_k h^{ik} \partial_i P_{\lambda }  w +2 g^{kk}P_{\lambda }  w +2 x_k \partial_i h^{ik}P_{\lambda }  w +x^2_k \cdot\mathcal{N}_5,\nonumber
	\end{align}
	where
	\begin{align*}
		&	\partial_i\left( g^{ij} \partial_j P_{\lambda }  w\right) \cdot x^2_k= \partial_i\left(  x^2_k g^{ij} \partial_j P_{\lambda }  w\right)  -2 g^{kk}_0x_k \partial_k  P_{\lambda }  w -2 x_k h^{kj} \partial_j P_{\lambda }  w\\ 
		%	&=\partial_i\left( g^{ij} \partial_j\left(P_{\lambda }  wx^2_k\right)\right) -2\partial_i\left(x_k g^{ik}P_{\lambda }  w \right) -2g^{kk}_0 x_k \partial_k  P_{\lambda }  w -2 x_k h^{kj} \partial_j P_{\lambda }  w\\
		=&\,\partial_i\left( g^{ij} \partial_j\left(P_{\lambda }  w x^2_k\right)\right)   -4 g^{kk}_0x_k \partial_k  P_{\lambda }  w -2 x_k h^{kj} \partial_j P_{\lambda }  w\\
		&-2 x_k h^{ik} \partial_i P_{\lambda }  w-2 g^{kk}P_{\lambda }  w -2 x_k \partial_i h^{ik}P_{\lambda }  w.
	\end{align*}
	
	Multiplying \eqref{w_x2} with $x^2_k P_{\lambda }  \overline{w}   $, taking the imaginary part and integrating on $ [0,t]\times \mathbb{R}^d $, summing over $k=1,\dots,d$ and $\lambda\geq 1$, we have
	\begin{align}
		& \sum_{k=1}^d \sum_{\lambda \geq 1} \lambda^{2 s_2-6 }\int_{\mathbb{R}^{d}} \left| x_k^2P_{\lambda }  w\right|^2 \mathrm{d} x  \lesssim   \sum_{k=1}^d   \sum_{\lambda \geq 1} \lambda^{2 s_2-6 }\int_0^t \int_{\mathbb{R}^{d}} x^2_k \cdot\mathcal{N}_5 \cdot x^2_k P_{\lambda }  \overline{w}  \mathrm{d} x \mathrm{d} s \label{w_we_est_t1}\\
		&+ \sum_{k=1}^d \sum_{\lambda \geq 1} \lambda^{2 s_2-6 }\int_0^t \int_{\mathbb{R}^{d}}  (4 x_k \partial_k  P_{\lambda }  w +2 x_k h^{kj}  \partial_j P_{\lambda }  w+2 x_k h^{ik} \partial_i P_{\lambda }  w +2 g^{kk}P_{\lambda }  w \nonumber\\
		&+2 x_k \partial_i h^{ik}P_{\lambda }  w)\cdot x^2_k P_{\lambda }  \overline{w}  \mathrm{d} x \mathrm{d} s. \label{w_we_est_t2}
	\end{align}
	For  \eqref{w_we_est_t1}, we have
	\begin{align*} 
		&\sum_{k=1}^d  \sum_{\lambda \geq 1} \lambda^{2 s_2-6 } \left\| x_k^2 P_{\lambda }  w\right\|_{L^{2}(\mathbb{R}^d)}^2\\
		\leq & \, C t  \left(\sum_{k=1}^d  \sum_{\lambda \geq 1} \lambda^{2 s_2-6 } \left\| x_k^2 P_{\lambda }  w\right\|_{L^{2}(\mathbb{R}^d)}^2\right)^{\frac 12}   \left( \sum_{k=1}^d \sum_{\lambda \geq 1} \lambda^{2 s_2-6 }    \left\|x^2_k \cdot\mathcal{N}_5 \right\|_{L^2(\mathbb{R}^2)}^2 \right)^{\frac 12}\\
		\leq & \frac{1}{100} \sum_{k=1}^d  \sum_{\lambda \geq 1} \lambda^{2 s_2-6 } \left\| x_k^2 P_{\lambda }  w\right\|_{L^{2}(\mathbb{R}^d)}^2 +Ct^2\sum_{k=1}^d \sum_{\lambda \geq 1} \lambda^{2 s_2-6 }    \left\|x^2_k \cdot\mathcal{N}_5 \right\|_{L^2(\mathbb{R}^2)}^2,
	\end{align*}
	where the second term in the right hand side can be estimated as \eqref{B_we_est_t1}. 
	
	As for \eqref{w_we_est_t2}, we have
	\begin{align*}
		\eqref{w_we_est_t2} \lesssim & \, t   \left(\sum_{k=1}^d  \sum_{\lambda \geq 1} \lambda^{2 s_2-6 } \left\| x_k^2 P_{\lambda }  w\right\|_{L^{2}(\mathbb{R}^d)}^2\right)^{\frac 12}  \cdot \left( \sum_{i=1}^d \sum_{k=1}^d  \sum_{\lambda \geq 1} \lambda^{2 s_2-6 } \left\|x_k \partial_i P_{\lambda }  w \right\|_{L^{2}(\mathbb{R}^d)}^2 \right)^{\frac 12}\left(1+ \left\|h\right\|_{L^\infty } \right)\\
		&+t \cdot   \left(\sum_{k=1}^d  \sum_{\lambda \geq 1} \lambda^{2 s_2-6 } \left\| x_k^2 P_{\lambda }  w\right\|_{L^{2}(\mathbb{R}^d)}^2\right)^{\frac 12}   \left(    \sum_{\lambda \geq 1} \lambda^{2 s_2-6 } \left\| \langle x \rangle  P_{\lambda }  w \right\|_{L^{2}(\mathbb{R}^d)}^2 \right)^{\frac 12} \left( 1+\left\| \langle \nabla \rangle  h  \right\|_{L^{\infty}(\mathbb{R}^d)}\right)\\
		\lesssim & \, t   \left(\sum_{k=1}^d  \sum_{\lambda \geq 1} \lambda^{2 s_2-6 } \left\| x_k^2 P_{\lambda }  w\right\|_{L^{2}(\mathbb{R}^d)}^2\right)^{\frac 34}  \cdot \left( \sum_{i=1}^d   \sum_{\lambda \geq 1} \lambda^{2 s_2-6 } \left\|  \partial_i P_{\lambda }  w \right\|_{H^{1}(\mathbb{R}^d)}^2 \right)^{\frac 12}\left(1+ \left\|\phi \right\|_{H^{s_2 -3} } \right)\\
		&+t \cdot   \left(\sum_{k=1}^d  \sum_{\lambda \geq 1} \lambda^{2 s_2-6 } \left\| x_k^2 P_{\lambda }  w\right\|_{L^{2}(\mathbb{R}^d)}^2\right)^{\frac 34}   \left(    \sum_{\lambda \geq 1} \lambda^{2 s_2-6 } \left\|   P_{\lambda }  w \right\|_{L^{2}(\mathbb{R}^d)}^2 \right)^{\frac 14} \left( 1+\left\|     \phi  \right\|_{H^{s_2-2}(\mathbb{R}^d)}\right)\\
		\lesssim	& \,   \frac{1}{100} \sum_{k=1}^d  \sum_{\lambda \geq 1} \lambda^{2 s_2-6 } \left\| x_k^2 P_{\lambda }  w\right\|_{L^{2}(\mathbb{R}^d)}^2 +Ct^4 M_0^{10}. 
	\end{align*}
	
	Combing the estimates of \eqref{w_we_est_t1} and \eqref{w_we_est_t2}, we have 
	\begin{align*}
		\sum_{k=1}^d  \sum_{\lambda \geq 1} \lambda^{2 s_2-6 } \left\| x_k^2 P_{\lambda }  w\right\|_{L^{2}(\mathbb{R}^d)}^2 \lesssim t^{\frac 12},
	\end{align*}
	where $\frac{1}{100} \sum_{k=1}^d  \sum_{\lambda \geq 1} \lambda^{2 s_2-6 } \left\| x_k^2 P_{\lambda }  w\right\|_{L^{2}(\mathbb{R}^d)}^2 $  in the right hand side of  \eqref{w_we_est_t1} and \eqref{w_we_est_t2} has been absorbed. Thus
	\begin{align*}
		\sum_{k=1}^d\left\| x_k^2   w\right\|_{H^{s_2-3}(\mathbb{R}^d)} \lesssim & 	\sum_{k=1}^d  \sum_{\lambda \geq 1} \lambda^{2 s_2-6 } \left\| x_k^2 P_{\lambda }  w\right\|_{L^{2}(\mathbb{R}^d)}^2  +	\sum_{k=1}^d  \sum_{\lambda \geq 1} \lambda^{2 s_2-6 } \left\| [P_{\lambda },x_k^2  ] w\right\|_{L^{2}(\mathbb{R}^d)}^2\\
		\lesssim& \,t^{\frac 12}+  	\sum_{k=1}^d  \sum_{\lambda \geq 1} \lambda^{2 s_2-8 } \left\|  P_{\lambda}^\prime  w\right\|_{L^{2}(\mathbb{R}^d)}^2\\
		\lesssim& \,t^{\frac 12} +\left\|w\right\|_{H^{s_2-1}}^2 \\
		\lesssim &\,t^{\frac 12},
	\end{align*}
	which combining with $\left\| w\right\|^2_{H^{s_2-1}(\mathbb{R}^d)} \lesssim t^{\frac 12}$ gives $\left\| \langle x \rangle^2   w\right\|_{H^{s_2-3}(\mathbb{R}^d)} \lesssim t^{\frac 12}.$ 
	
	By a standard bootstrap argument, we conclude that the improved bounds \eqref{bstrap_re1}--\eqref{bstrap_re5} hold. The proof of Theorem \ref{thm_bigdata} is now completed via the standard artificial viscosity method, which follows the small data case closely. We omit the  details.

\section{Small data and cubic interactions} \label{s4}

As the small data  quadratic interaction problem in Section~\ref{s3}, we first define the following space-time norm for the cubic nonlinearity in \eqref{qNLS}:
\begin{align*}
	&W(t)=W_{s_3}(t)
:=\sum_{k=1}^d\sum_{|\beta|\leq s_3-1} \sum_{|\alpha|=s_3}\int_0^t \int_{\mathbb{R}} \|\Lambda_k^{\frac 12}\phi_{\beta}(x_k,\cdot)\|_{L^2(\mathbb{R}^{d-1})}^2  \|\partial_k \phi_{\alpha}(x_k,\cdot)\|_{L^2(\mathbb{R}^{d-1})}^2 \mathrm{d} x_k \mathrm{d} s,\nonumber
\end{align*}
where $ s_3>\frac d2+1 $ and $ \Lambda_k= \left(1-\sum_{i=1,i\neq k}^d \partial_{ii}\right)^{\frac12} $.

For simplicity, we restrict to the elliptic case $  \left(g_0^{ij}\right) = \left(	\delta_{i j}\right)  =I_d  $, since the ultrahyperbolic case \eqref{g0} can be treated similarly as those in Section~\ref{s3}. Also,  if $ |\phi|+ |\nabla\phi| $ is small, we may neglect error terms, and deduce from Taylor expansion and definition  \eqref{cgF} of $ g, F $ that
\begin{align*}
	\left| F\left(\phi,\overline{\phi},\nabla\phi, \nabla \overline{\phi}\right)\right|& \sim |\phi|^3+|\nabla\phi|^3,\\
	\left| F_z\right|+\left| F_{\overline{z}}\right|& \sim |\phi|^2+|\nabla\phi|^2,\\
|	g^{ij}| &\sim  \delta_{ij}+|\phi|^2. 
\end{align*}
  
Note for $\mathbb{R}^{d-1}= \mathbb{R}^{d-1}_{\widehat{x}_k},$  where $ \hat{x}_k:= (x_1,\dots,x_{k-1},x_{k+1},\dots,x_d ) $ and for $ s_3>\frac d2+1 $, we have
\begin{align*}
	\| \phi(x_k,\cdot)\|_{H^{s_3-\frac 12}(\mathbb{R}^{d-1})}\leq	\sum_{|\beta|\leq s_3-1}  \| \Lambda_k^{\frac 12}\phi_{\beta}(x_k,\cdot)\|_{L^2(\mathbb{R}^{d-1})},
\end{align*}
where $ \phi_{\beta}=\partial_{\beta} \phi  $ and $ \partial_{\beta} $  may include $ \partial_k $.
It follows that
\begin{align} \label{cF}
	&\|F_z(x_k,\cdot)\|_{L^\infty(\mathbb{R}^{d-1})} +\|F_{\overline{z}}(x_k,\cdot)\|_{L^\infty(\mathbb{R}^{d-1})}\nonumber\\
	\lesssim &\,\|\phi(x_k,\cdot)\|_{H^{s_3-\frac 32}(\mathbb{R}^{d-1})}^2+\|\nabla\phi(x_k,\cdot)\|_{H^{s_3-\frac 32}(\mathbb{R}^{d-1})}^2 \\
	\lesssim &\,\sum_{|\beta|\leq s_3-1}  \| \Lambda_k^{\frac 12}\phi_{\beta}(x_k,\cdot)\|_{L^2(\mathbb{R}^{d-1})}^2 \nonumber
\end{align}
and
\begin{align} \label{cg}
	\|g^{ij}\|_{L^\infty(\mathbb{R}^d)}\leq&\, 1+\|h^{ij}\|_{L^\infty(\mathbb{R}^d)}\\
	\lesssim &\,1+\|\phi\|_{H^{s_3-1}(\mathbb{R}^d)}^2,\nonumber\\
	\left\|	\partial_k g^{ij} (\phi,\overline{\phi})(x_k,\cdot)\right\|_{L^\infty(\mathbb{R}^{d-1})}=&\,\left\|	\partial_k h^{ij} (\phi,\overline{\phi})(x_k,\cdot)\right\|_{L^\infty(\mathbb{R}^{d-1})}\label{cpg}\\
	=&\,\left\|  h^{ij}_z \partial_k \phi+  h^{ij}_{\overline{z}} \partial_k \overline{\phi} \right\|_{L^\infty(\mathbb{R}^{d-1})}\nonumber\\
	\lesssim &\,\sum_{|\beta|\leq s_3-1}  \| \Lambda_k^{\frac 12}\phi_{\beta}(x_k,\cdot)\|_{L^2(\mathbb{R}^{d-1})}^2,\nonumber
\end{align}
provided  $\frac{d}{2}+1<  s_3 $ and $\beta \in \mathbb{N}^d_0$. Similarly, for $ \frac 32 <l $ we have
\begin{align*} 
	\|F_z\|_{L^\infty(\mathbb{R}_{x_k})}+\|F_{\overline{z}}\|_{L^\infty(\mathbb{R}_{x_k})}+\|\partial h^{ij}\|_{L^\infty(\mathbb{R}_{x_k})} \lesssim   \sum_{i=1}^d\|\partial_i \phi \|_{H^{l-1}((\mathbb{R}_{x_k})}^2.
\end{align*}

\subsection{Momentum type estimates} \label{s4.1}

As Lemma \ref{leminter1} in Section \ref{s3}, a key technique in our analysis is to halve the derivative order for terms without time integration in momentum type estimates. To this end, we introduce the following lemma and present the proof  in Appendix~\ref{appa} for reader's convenience. 
\begin{lem} \label{lemhalve}
	If real function $ v(x_k) $ and complex function  $ u(x), w(x) $ are smooth, where $ x_k \in \mathbb{R}, x\in \mathbb{R}^d $,  then 
	\begin{align}  \label{lemhalve1}
		&\left|\int_{\mathbb{R}}  \int_{\mathbb{R}^{d-1}} w \partial_k u  \mathrm{d} \hat{x}_k \int_{-\infty}^{x_k} v(y_k) \mathrm{d} y_k  \mathrm{d} x_k\right|\\
		\lesssim &\, \left\|   D_k^{\frac 12}u\right\|_{L^2(\mathbb{R}^d)}\left\|   D_k^{\frac 12}w\right\|_{L^2(\mathbb{R}^d)} \int_{ \mathbb{R}} |v(x_k) |\mathrm{d} x_k +  \left\|   D_k^{\frac 12}u\right\|_{L^2(\mathbb{R}^d)} \left\|   w\right\|_{L^2(\mathbb{R}^d)} \left\| v \right\|_{L^2(\mathbb{R}_{x_k})}.\nonumber
	\end{align}
\end{lem}

We also examine  how to use \eqref{qNLS} to present the $ \int_{-\infty}^{x_k} \int_{\mathbb{R}^{d-1}}  \mathrm{Re}\left( \Lambda_k^{\frac 12}\phi_{\beta}\Lambda_k^{\frac 12}\partial_t\overline{\phi}_{\beta }\right)\mathrm{d} \hat{x}_k \mathrm{d} y_k,$ which arises in momentum-type estimates.  Recall  \eqref{qNLSb}:
\begin{align*}  
	&\sqrt{-1}\phi_{\beta t}+\partial_i\left( g^{ij} \phi_{\beta j}\right)=\mathcal{N}_\beta,\\
	&\mathcal{N}_\beta=F_z\nabla \phi_\beta+F_{\overline{z}}\nabla \overline{\phi}_\beta+\partial h^{ij}\partial_{ij } \phi_{\beta -1}+G_\beta.
\end{align*}
% and can let $ \Lambda_k^{\frac 12} $ act \eqref{qNLSb1} to get
%\begin{align}  \label{qNLSb}
%	&\sqrt{-1}\Lambda_k^{\frac 12}\phi_{\beta t}+\partial_i \Lambda_k^{\frac 12} \left(g^{ij} \phi_{\beta j}\right)=	\Lambda_k^{\frac 12}\mathcal{N}_\beta,
%\end{align}
Multiplying \eqref{qNLSb} by $ \Lambda_k \overline{\phi}_\beta $  and taking the imaginary part shows
\begin{align*} 
	&\mathrm{Re}\left(   \partial_t \phi_{\beta }  \Lambda_k\overline{\phi}_{\beta } \right)=\mathrm{Im}\left(  	\sqrt{-1} \partial_t  \phi_{\beta }  \Lambda_k\overline{\phi}_{\beta } \right)	\nonumber\\
	=&\,- \mathrm{Im} \left(  \partial_i  \left(g^{ij} \phi_{\beta j}\right)\Lambda_k\overline{\phi}_{\beta }\right)+\mathrm{Im} \left(   \mathcal{N}_\beta \Lambda_k \overline{\phi}_{\beta }\right) 	\\
	=&\,-  \mathrm{Im}\partial_i\left(  \left(g^{ij} \phi_{\beta j}\right)\Lambda_k\overline{\phi}_{\beta }\right)+\mathrm{Im} \left(  g^{ij} \phi_{\beta j}\Lambda_k\overline{\phi}_{\beta i}\right) +\mathrm{Im} \left(  \mathcal{N}_\beta \Lambda_k \overline{\phi}_{\beta }\right),	\nonumber
\end{align*}
which integrating on $ (-\infty,x_k)\times\mathbb{R}^{d-1} $ yields
\begin{align} \label{Hs0_bec}	
	&	\frac 12 \frac{\partial }{\partial t} \int_{-\infty}^{x_k} \int_{\mathbb{R}^{d-1}} \left| \Lambda_k^{\frac 12} \phi_{\beta } \right|^2\mathrm{d} \hat{x}_k \mathrm{d} y_k=\int_{-\infty}^{x_k} \int_{\mathbb{R}^{d-1}} \mathrm{Re}\left(   \partial_t \Lambda_k^{\frac 12} \phi_{\beta }  \Lambda_k^{\frac 12}\overline{\phi}_{\beta }  \right) \mathrm{d} \hat{x}_k \mathrm{d} y_k\nonumber\\
	=&\,\int_{-\infty}^{x_k} \int_{\mathbb{R}^{d-1}}  \mathrm{Re}\left(   \partial_t \phi_{\beta }  \Lambda_k\overline{\phi}_{\beta } \right) \mathrm{d} \hat{x}_k \mathrm{d} y_k\\
	=&\, -\mathrm{Im}	\int_{\mathbb{R}^{d-1}} \left(   g^{kj} \phi_{\beta j} \Lambda_k\overline{\phi}_{\beta }\right) (x_k,\cdot) \mathrm{d} \hat{x}_k+	\int_{-\infty}^{x_k} \int_{\mathbb{R}^{d-1}}   \mathrm{Im} \left(  g^{ij} \phi_{\beta j}\Lambda_k\overline{\phi}_{\beta i}\right)\mathrm{d} \hat{x}_k \mathrm{d} y_k\nonumber \\
	&+\int_{-\infty}^{x_k} \int_{\mathbb{R}^{d-1}}  \mathrm{Im} \left(  \mathcal{N}_\beta \Lambda_k\overline{\phi}_{\beta }\right)\mathrm{d} \hat{x}_k \mathrm{d} y_k\nonumber\\
	%		=&\, -\mathrm{Im}	\int_{\mathbb{R}^{d-1}} \left(   g^{kj} \Lambda_k^{\frac 12}\phi_{\beta j} \Lambda_k^{\frac 12}\overline{\phi}_{\beta }\right) (x_k,\cdot) \mathrm{d} \hat{x}_k-\mathrm{Im}	\int_{\mathbb{R}^{d-1}} \left( \left[\Lambda_k^{\frac12},h^{kj}\right]\phi_{\beta j}\Lambda_k^{\frac 12}\overline{\phi}_{\beta }\right) (x_k,\cdot) \mathrm{d} \hat{x}_k \nonumber \\
	%	&+	\int_{-\infty}^{x_k} \int_{\mathbb{R}^{d-1}}   \mathrm{Im} \left(  g^{ij} \phi_{\beta j}\Lambda_k\overline{\phi}_{\beta i}\right)\mathrm{d} \hat{x}_k \mathrm{d} y_k+\int_{-\infty}^{x_k} \int_{\mathbb{R}^{d-1}}  \mathrm{Im} \left(  \mathcal{N}_\beta \Lambda_k\overline{\phi}_{\beta }\right)\mathrm{d} \hat{x}_k \mathrm{d} y_k\nonumber\\
	=&\, -\mathrm{Im}	\int_{\mathbb{R}^{d-1}} \left(   g^{kj} \Lambda_k^{\frac 12}\phi_{\beta j} \Lambda_k^{\frac 12}\overline{\phi}_{\beta }\right) (x_k,\cdot) \mathrm{d} \hat{x}_k-\mathrm{Im}	\int_{\mathbb{R}^{d-1}} \left( \left[\Lambda_k^{\frac12},h^{kj}\right]\phi_{\beta j}\Lambda_k^{\frac 12}\overline{\phi}_{\beta }\right) (x_k,\cdot) \mathrm{d} \hat{x}_k \nonumber \\
	&+	\int_{-\infty}^{x_k} \int_{\mathbb{R}^{d-1}}   \mathrm{Im} \left(  h^{ij} \phi_{\beta j}\Lambda_k\overline{\phi}_{\beta i}\right)\mathrm{d} \hat{x}_k \mathrm{d} y_k+\int_{-\infty}^{x_k} \int_{\mathbb{R}^{d-1}}  \mathrm{Im} \left(  \mathcal{N}_\beta \Lambda_k\overline{\phi}_{\beta }\right)\mathrm{d} \hat{x}_k \mathrm{d} y_k.\nonumber
\end{align}
Note that the $g^{ij}=\delta_{ij}+h^{ij}  $  implies
\begin{align*}
	&\int_{-\infty}^{x_k} \int_{\mathbb{R}^{d-1}}   \mathrm{Im} \left(  g^{ij} \phi_{\beta j}\Lambda_k\overline{\phi}_{\beta i}\right)\mathrm{d} \hat{x}_k \mathrm{d} y_k\nonumber\\
	=&\, \sum_{i=1}^d \int_{-\infty}^{x_k} \int_{\mathbb{R}^{d-1}}   \mathrm{Im} \left(   \phi_{\beta i}\Lambda_k\overline{\phi}_{\beta i}\right)\mathrm{d} \hat{x}_k \mathrm{d} y_k+\int_{-\infty}^{x_k} \int_{\mathbb{R}^{d-1}}   \mathrm{Im} \left(  h^{ij} \phi_{\beta j}\Lambda_k\overline{\phi}_{\beta i}\right)\mathrm{d} \hat{x}_k \mathrm{d} y_k\nonumber\\
	=&\, \sum_{i=1}^d  \int_{-\infty}^{x_k} \int_{\mathbb{R}^{d-1}}   \mathrm{Im} \left(   \Lambda_k^{\frac12}\phi_{\beta i}\Lambda_k^{\frac12}\overline{\phi}_{\beta i}\right)\mathrm{d} \hat{x}_k \mathrm{d} y_k+\int_{-\infty}^{x_k} \int_{\mathbb{R}^{d-1}}   \mathrm{Im} \left(  h^{ij} \phi_{\beta j}\Lambda_k\overline{\phi}_{\beta i}\right)\mathrm{d} \hat{x}_k \mathrm{d} y_k\nonumber\\
	=&\,\int_{-\infty}^{x_k} \int_{\mathbb{R}^{d-1}}   \mathrm{Im} \left(  h^{ij} \phi_{\beta j}\Lambda_k\overline{\phi}_{\beta i}\right)\mathrm{d} \hat{x}_k \mathrm{d} y_k.
\end{align*}

We now derive the momentum-type estimate. As in the definition of $ W(t) $, we set $ s_3=|\alpha| $ in this subsection. Multiplying \eqref{consm2} by $ \int_{-\infty}^{x_k} \|\Lambda_k^{\frac 12}\phi_{\beta}(y_k,\cdot)\|^2_{L^2(\mathbb{R}^{d-1})} \mathrm{d} y_k $ and integrating it on  $ [0,t]\times\mathbb{R} $ indicates
\begin{align} \label{cener_c1_1}
	&\, -\int_0^t\int_{\mathbb{R}}  \partial_t\left(\int_{\mathbb{R}^{d-1}}  \mathrm{Im}\left( \phi_{\alpha } \partial_k \overline{\phi}_\alpha \right) \mathrm{d} \hat{x}_k\right) \int_{-\infty}^{x_k} \|\Lambda_k^{\frac 12}\phi_{\beta}(y_k,\cdot)\|^2_{L^2(\mathbb{R}^{d-1})}\mathrm{d} y_k  \mathrm{d} x_k \mathrm{d} s\nonumber\\
	&\,- \int_0^t \int_{\mathbb{R}}   	\int_{-\infty}^{x_k} \|\Lambda_k^{\frac 12}\phi_{\beta}(y_k,\cdot)\|^2_{L^2(\mathbb{R}^{d-1})}\mathrm{d} y_k \partial_{kk} \int_{\mathbb{R}^{d-1}} \mathrm{Re}\left( \phi_{\alpha }   g^{kj} \overline{\phi}_{\alpha j}\right) \mathrm{d} \hat{x}_k   \mathrm{d} x_k \mathrm{d} s\nonumber \\
	&+2  \int_0^t\int_{\mathbb{R}}    	\int_{-\infty}^{x_k} \|\Lambda_k^{\frac 12}\phi_{\beta}(y_k,\cdot)\|^2_{L^2(\mathbb{R}^{d-1})}\mathrm{d} y_k \partial_k \int_{\mathbb{R}^{d-1}} \mathrm{Re}\left(\partial_k \phi_{\alpha }  g^{kj} \overline{\phi}_{\alpha j}\right) \mathrm{d} \hat{x}_k  \mathrm{d} x_k \mathrm{d} s \nonumber\\
	&+ \int_0^t	\int_{\mathbb{R}} \int_{-\infty}^{x_k} \|\Lambda_k^{\frac 12}\phi_{\beta}(y_k,\cdot)\|^2_{L^2(\mathbb{R}^{d-1})}\mathrm{d} y_k\int_{\mathbb{R}^{d-1}}  \mathrm{Re}\left(\partial_{i} \phi_{\alpha }  \partial_k h^{ij} \partial_j \overline{\phi}_{\alpha }\right) \mathrm{d} \hat{x}_k  \mathrm{d} x_k \mathrm{d} s \\
	=&\,- \int_0^t	\int_{\mathbb{R}}\int_{-\infty}^{x_k} \|\Lambda_k^{\frac 12}\phi_{\beta}(y_k,\cdot)\|^2_{L^2(\mathbb{R}^{d-1})}\mathrm{d} y_k\nonumber\\
	&\cdot \partial_k\int_{\mathbb{R}^{d-1}}   \mathrm{Re}  \left\{ \phi_{\alpha } \left(\overline{F}_z\nabla \overline{\phi}_\alpha+\overline{F}_{\overline{z}}\nabla \phi_\alpha+\partial h^{ij}\partial_{ij } \overline{\phi}_{\alpha -1}+\overline{G}\right)  \right\} \mathrm{d} \hat{x}_k  \mathrm{d} x_k \mathrm{d} s\nonumber\\
	&+2\int_0^t	\int_{\mathbb{R}} \int_{-\infty}^{x_k} \|\Lambda_k^{\frac 12}\phi_{\beta}(y_k,\cdot)\|^2_{L^2(\mathbb{R}^{d-1})}\mathrm{d} y_k\nonumber\\
	&\cdot\int_{\mathbb{R}^{d-1}} \mathrm{Re}\left\{\left( F_z\nabla \phi_\alpha+F_{\overline{z}}\nabla \overline{\phi}_\alpha+\partial h^{ij}\partial_{ij } \phi_{\alpha -1}+G\right) \partial_k \overline{\phi}_\alpha \right\}  \mathrm{d} \hat{x}_k  \mathrm{d} x_k \mathrm{d} s\nonumber.
\end{align}
It follows from   $ g^{ij}=\delta_{ij}+h^{ij}$  that
\begin{align} \label{cener_c1_2}
	&\,-  \int_0^t\int_{\mathbb{R}}    	2 \partial_k \int_{\mathbb{R}^{d-1}} \mathrm{Re}\left(\partial_k \phi_{\alpha }  g^{kj} \overline{\phi}_{\alpha j}\right) \mathrm{d} \hat{x}_k\int_{-\infty}^{x_k} \|\Lambda_k^{\frac 12}\phi_{\beta}(y_k,\cdot)\|^2_{L^2(\mathbb{R}^{d-1})} \mathrm{d} y_k  \mathrm{d} x_k \mathrm{d} s \nonumber\\
	=&\,  \int_0^t\int_{\mathbb{R}}   	2  \int_{\mathbb{R}^{d-1}} \left|\partial_k \phi_{\alpha } \right|^2 \mathrm{d} \hat{x}_k\|\Lambda_k^{\frac 12}\phi_{\beta}(x_k,\cdot)\|^2_{L^2(\mathbb{R}^{d-1})}   \mathrm{d} x_k \mathrm{d} s\\
	&+ \int_0^t\int_{\mathbb{R}}   	2  \int_{\mathbb{R}^{d-1}} \mathrm{Re}\left(\partial_k \phi_{\alpha }  h^{kj} \overline{\phi}_{\alpha j}\right) \mathrm{d} \hat{x}_k\|\Lambda_k^{\frac 12}\phi_{\beta}(x_k,\cdot)\|^2_{L^2(\mathbb{R}^{d-1})}   \mathrm{d} x_k \mathrm{d} s,\nonumber
\end{align}
where the second term on the right-hand side is bounded by
\begin{align} \label{cener_c1_2_0}
	& \int_0^t\int_{\mathbb{R}}   	2  \int_{\mathbb{R}^{d-1}} \mathrm{Re}\left(\partial_k \phi_{\alpha }  h^{kj} \overline{\phi}_{\alpha j}\right) \mathrm{d} \hat{x}_k\|\Lambda_k^{\frac 12}\phi_{\beta}(x_k,\cdot)\|^2_{L^2(\mathbb{R}^{d-1})}   \mathrm{d} x_k \mathrm{d} s\nonumber\\ 
	\lesssim &\,	\|h^{kj}\|_{L^\infty} \sum_{j}^d \int_0^t\int_{\mathbb{R}} \left\| \partial_i\phi_{\alpha }\right\|_{L^2(\mathbb{R}^{d-1})}^2 \|\Lambda_k^{\frac 12}\phi_{\beta}(x_k,\cdot)\|^2_{L^2(\mathbb{R}^{d-1})}   \mathrm{d} x_k \mathrm{d} s\\
	\lesssim &\,\|\phi\|_{H^{|\alpha|}(\mathbb{R}^d)}^2\cdot W (t).\nonumber
\end{align}

Using \eqref{cg}, integration by parts, and H\"older's inequality, we have
\begin{align} \label{cener_c1_3}
	& \int_0^t \int_{\mathbb{R}}   	\partial_{kk} \int_{\mathbb{R}^{d-1}} \mathrm{Re}\left( \phi_{\alpha }   g^{kj} \overline{\phi}_{\alpha j}\right) \mathrm{d} \hat{x}_k \int_{-\infty}^{x_k} \|\Lambda_k^{\frac 12}\phi_{\beta}(y_k,\cdot)\|^2_{L^2(\mathbb{R}^{d-1})}\mathrm{d} y_k  \mathrm{d} x_k \mathrm{d} s \nonumber\\
	=&\,- \int_0^t\int_{\mathbb{R}}   	\partial_{k} \int_{\mathbb{R}^{d-1}} \mathrm{Re}\left( \phi_{\alpha }   g^{kj} \overline{\phi}_{\alpha j}\right) \mathrm{d} \hat{x}_k  \|\Lambda_k^{\frac 12}\phi_{\beta}(x_k,\cdot)\|^2_{L^2(\mathbb{R}^{d-1})}  \mathrm{d} x_k  \mathrm{d} s \nonumber\\
	=&\,2 \int_0^t\int_{\mathbb{R}}   \int_{\mathbb{R}^{d-1}} \mathrm{Re}\left( \phi_{\alpha }   g^{kj} \overline{\phi}_{\alpha j}\right) \mathrm{d} \hat{x}_k \int_{\mathbb{R}^{d-1}}\mathrm{Re}\left(\Lambda_k^{\frac 12}\phi_{\beta} \partial_k \Lambda_k^{\frac 12}\overline{\phi}_{\beta} \right) \mathrm{d} x_k \mathrm{d} s\nonumber\\
	%		\lesssim &\,\int_0^t \int_{\mathbb{R}} \left( \| \phi_{\alpha}\|_{L^2 (\mathbb{R}^{d-1})}^2+ \| \nabla\phi_{\alpha}\|_{L^2 (\mathbb{R}^{d-1})}^2\right)  \\
	%		&\cdot\left(\|\phi(x_k,\cdot)\|^2_{H^{s_0}(\mathbb{R}^{d-1})} +\|\phi_{x_k}(x_k,\cdot)\|_{H^{s_0}(\mathbb{R}^{d-1})}^2 \right)\mathrm{d} x_k \mathrm{d} s
	\lesssim &\,\|g^{kj}\|_{L^\infty}  \int_0^t\int_{\mathbb{R}}  \left\| \phi_{\alpha }\right\|_{L^2(\mathbb{R}^{d-1})}  \left\| \partial_j \phi_{\alpha }\right\|_{L^2(\mathbb{R}^{d-1})}\|\Lambda_k^{\frac 12}\partial_k\phi_{\beta}(x_k,\cdot)\|_{L^2(\mathbb{R}^{d-1})} \|\Lambda_k^{\frac 12}\phi_{\beta}(x_k,\cdot)\|_{L^2(\mathbb{R}^{d-1})}  \mathrm{d} x_k \mathrm{d} s\nonumber\\
	\leq&\,\delta\|g^{kj}\|_{L^\infty} \int_0^t\int_{\mathbb{R}}  \left\| \partial_j \phi_{\alpha }\right\|_{L^2(\mathbb{R}^{d-1})}^2 \|\Lambda_k^{\frac 12}\phi_{\beta}(x_k,\cdot)\|^2_{L^2(\mathbb{R}^{d-1})}   \mathrm{d} x_k \mathrm{d} s\\
	&+ C(\delta)\|g^{kj}\|_{L^\infty} \int_0^t\int_{\mathbb{R}}  \left\| \phi_{\alpha }\right\|_{L^2(\mathbb{R}^{d-1})}^2  \|\Lambda_k^{\frac 12}\partial_k \phi_{\beta}(x_k,\cdot)\|^2_{L^2(\mathbb{R}^{d-1})}  \mathrm{d} x_k \mathrm{d} s\nonumber\\
\leq&\, \delta \int_0^t\int_{\mathbb{R}}  \left\|\partial_j  \phi_{\alpha }\right\|_{L^2(\mathbb{R}^{d-1})}^2 \|\Lambda_k^{\frac 12}\phi_{\beta}(x_k,\cdot)\|^2_{L^2(\mathbb{R}^{d-1})}   \mathrm{d} x_k \mathrm{d} s\nonumber\\
	&+\delta C \|\phi\|_{H^{s_3}(\mathbb{R}^d)}^2\cdot  \int_0^t\int_{\mathbb{R}}  \left\| \partial_j \phi_{\alpha }\right\|_{L^2(\mathbb{R}^{d-1})}^2 \|\Lambda_k^{\frac 12}\phi_{\beta}(x_k,\cdot)\|^2_{L^2(\mathbb{R}^{d-1})}   \mathrm{d} x_k \mathrm{d} s\nonumber\\
	&+ C(\delta) \cdot t\cdot\|\phi\|_{H^{s_3+\frac 12}}^4+ C(\delta) \cdot t\cdot\|\phi\|_{H^{s_3+\frac 12}}^6 \nonumber\\
\leq&\,  \delta\cdot W (t)+\delta \cdot C\|\phi\|_{H^{s_3}(\mathbb{R}^d)}^2\cdot W (t)+ C(\delta) \cdot t\cdot\left(1+\|\phi\|_{H^{s_3+\frac 12}}^2\right)\|\phi\|_{H^{s_3+\frac 12}}^4 \nonumber
\end{align}
for $|\beta| \leq s_3-1 $ and $ \frac{d}{2}+1<s_3=|\alpha| $.

Also, one can use \eqref{cF}, \eqref{cpg} and H\"older's inequality to deduce
\begin{align}   \label{cener_c1_5}
	&\int_0^t	\int_{\mathbb{R}} \int_{-\infty}^{x_k} \|\Lambda_k^{\frac 12}\phi_{\beta}(y_k,\cdot)\|^2_{L^2(\mathbb{R}^{d-1})}\mathrm{d} y_k\nonumber\\
 &\cdot\int_{\mathbb{R}^{d-1}} \mathrm{Re}\left\{\left( F_z\nabla \phi_\alpha+F_{\overline{z}}\nabla \overline{\phi}_\alpha+\partial h^{ij}\partial_{ij } \phi_{\alpha -1}\right) \partial_k \overline{\phi}_\alpha \right\}  \mathrm{d} \hat{x}_k  \mathrm{d} x_k \mathrm{d} s\nonumber\\
	\lesssim &\,\|\phi\|_{H^{s_3}(\mathbb{R}^d)}^2 \int_0^t	\int_{\mathbb{R}}  \left( 	\|F_z\|_{L^\infty(\mathbb{R}^{d-1})}+\left\|	\partial h^{ij} \right\|_{L^\infty(\mathbb{R}^{d-1})}\right) (x_k)\|\phi_{\alpha +1} (x_k,\cdot)\|^2_{L^{2}(\mathbb{R}^{d-1})}\mathrm{d} x_k \mathrm{d} s\\
	\lesssim &\,\|\phi\|_{H^{s_3}(\mathbb{R}^d)}^2  \cdot \sum_{|\beta|\leq s_3-1}  \int_0^t	\int_{\mathbb{R}}   \| \Lambda_k^{\frac 12}\phi_{\beta}(x_k,\cdot)\|_{L^2(\mathbb{R}^{d-1})}^2  \|\phi_{\alpha +1} (x_k,\cdot)\|^2_{L^{2}(\mathbb{R}^{d-1})}\mathrm{d} x_k \mathrm{d} s\nonumber\\
	\lesssim &\,\|\phi\|_{H^{s_3}(\mathbb{R}^d)}^2\cdot\mathfrak{W} (t)\nonumber
\end{align}
and bound
\begin{align*}
	\int_0^t	\int_{\mathbb{R}} \int_{-\infty}^{x_k} \|\Lambda_k^{\frac 12}\phi_{\beta}(y_k,\cdot)\|^2_{L^2(\mathbb{R}^{d-1})}\mathrm{d} y_k\int_{\mathbb{R}^{d-1}}  \mathrm{Re}\left(\partial_{i} \phi_{\alpha }  \partial_k h^{ij} \partial_j \overline{\phi}_{\alpha }\right) \mathrm{d} \hat{x}_k  \mathrm{d} x_k \mathrm{d} s 
\end{align*}  
in the same manner.

Noticing  that the highest order derivatives for $ \phi $ of terms in  $ G $ is less than  $ |\alpha|$,  we can bound the remainder terms containing $ G $ by $ \|\phi\|_{H^{s_3 +\frac 12}} $ directly. Indeed, the Lemma \ref{lemhalve}  shows
\begin{align} \label{cener_c1_5_1}
	&\int_0^t\int_{\mathbb{R}}  \int_{\mathbb{R}^{d-1}}  \mathrm{Re}\left( G \partial_k \overline{\phi}_\alpha \right) \mathrm{d} \hat{x}_k \int_{-\infty}^{x_k} \|\Lambda_k^{\frac 12}\phi_{\beta}(y_k,\cdot)\|^2_{L^2(\mathbb{R}^{d-1})} \mathrm{d} y_k  \mathrm{d} x_k \mathrm{d} s\nonumber\\
	\lesssim &\,\int_0^t \left\|   D_k^{\frac 12}\phi_\alpha\right\|_{L^2(\mathbb{R}^d)} \left\|   D_k^{\frac 12}G\right\|_{L^2(\mathbb{R}^d)}  \int_{-\infty}^{\infty} \|\Lambda_k^{\frac 12}\phi_{\beta}(x_k,\cdot)\|^2_{L^2(\mathbb{R}^{d-1})}\mathrm{d} x_k \mathrm{d} s\nonumber\\
	&+\int_0^t \left\|   D_k^{\frac 12}\phi_{\alpha}\right\|_{L^2(\mathbb{R}^d)} \left\|   G\right\|_{L^2(\mathbb{R}^d)} \left\| \|\Lambda_k^{\frac 12}\phi_{\beta}(x_k,\cdot)\|^2_{L^2(\mathbb{R}^{d-1})} \right\|_{L^2(\mathbb{R}_{x_k})}\mathrm{d} s\nonumber\\
	\lesssim &\,t\cdot\left\|   D_k^{\frac 12}\phi_\alpha\right\|_{L^2(\mathbb{R}^d)} \left\|   D_k^{\frac 12}G\right\|_{L^2(\mathbb{R}^d)} \|\phi\|^2_{H^{|\beta|+\frac 12}(\mathbb{R}^{d})} \\
	&+t \cdot\left\|   D_k^{\frac 12}\phi_{\alpha}\right\|_{L^2(\mathbb{R}^d)} \left\|   G\right\|_{L^2(\mathbb{R}^d)}  \left\| \| \Lambda_k^{\frac 12}\phi_\beta(\cdot,\hat{x})\|_{L^4(\mathbb{R}_{x_k})} \right\|_{L^{2}(\mathbb{R}^{d-1})}^2\nonumber\\
	\lesssim &\,t\cdot\left\|   D_k^{\frac 12}\phi_\alpha\right\|_{L^2(\mathbb{R}^d)} \left\|   D_k^{\frac 12}G\right\|_{L^2(\mathbb{R}^d)} \|\phi\|^2_{H^{ |\beta|+\frac 12}(\mathbb{R}^{d})} \nonumber\\
	&+t \cdot\left\|   D_k^{\frac 12}\phi_{\alpha}\right\|_{L^2(\mathbb{R}^d)} \left\|   G\right\|_{L^2(\mathbb{R}^d)}\left\| \phi \right\|^2_{H^{|\beta|+\frac 34}(\mathbb{R}^{d})}\nonumber\\
	\lesssim &\,t\cdot \|\phi\|^3_{H^{s_3+\frac 12}(\mathbb{R}^{d})} \left(\left\|   D_k^{\frac 12}G\right\|_{L^2(\mathbb{R}^d)}+\left\|  G\right\|_{L^2(\mathbb{R}^d)}\right), \nonumber
\end{align}
where Minkowski's inequality is used in the second step and Sobolev embedding in the third. A similar analysis to \eqref{qG} gives
\begin{align}   \label{cG}
	%	\label{cener_c1_5_7}
&\max \left\{	\left\|     G \right\|_{L^2}, \left\|  D_k^{\frac 12}   G \right\|_{L^2} \right \}\\
\lesssim &\, \left\|  \phi\right\|_{W^{q_1,|\alpha^3|+\frac32}(\mathbb{R}^d)}  \left\|  \phi\right\|_{W^{q_2,|\alpha^4|+1}(\mathbb{R}^d)}  \left\|  \phi\right\|_{W^{q_3,|\alpha^5|+1}(\mathbb{R}^d)}\nonumber\\
	\lesssim &\,\|\phi\|_{H^{|\alpha|+\frac12}}^3,\nonumber
\end{align}
provided 
\begin{align*}  
	%	\label{cener_c1_5_7_1}
&	\left|\alpha^i\right|\leq \left|\alpha\right|-1,~~i=3,4,5, \qquad  \left|\alpha^3\right|+\left|\alpha^4\right|+\left|\alpha^5\right|= \left|\alpha\right|,\\
	&	\sum_{i=1}^3 \frac{1}{q_i}=\frac 12, i=3,4,5,\quad\frac{1}{q_3}\geq \frac 12-\frac{|\alpha|-|\alpha_3|-1}{d},\\
	&\frac{1}{q_j}\geq \frac 12-\frac{|\alpha|-|\alpha_j|-\frac12}{d},j=4,5\nonumber
\end{align*}
which implies  $s_3= |\alpha|\geq \frac{d}{2}+1  $. 
Combining \eqref{cG}  with \eqref{cener_c1_5_1}, we obtain 
\begin{align} \label{cener_c1_5_G}
	& \int_0^t\int_{\mathbb{R}}  \int_{\mathbb{R}^{d-1}}  \mathrm{Re}\left( G \partial_k \overline{\phi}_\alpha \right) \mathrm{d} \hat{x}_k \int_{-\infty}^{x_k} \|\Lambda_k^{\frac 12}\phi_{\beta}(y_k,\cdot)\|^2_{L^2(\mathbb{R}^{d-1})} \mathrm{d} y_k  \mathrm{d} x_k \mathrm{d} s\nonumber\\
	\lesssim &\, t\cdot\left\|   D_k^{\frac 12}\phi_\alpha\right\|_{L^2(\mathbb{R}^d)}\|\phi\|_{H^{|\alpha|+\frac12}(\mathbb{R}^{d})}^3 \|\phi\|^2_{H^{|\alpha|}(\mathbb{R}^{d})}\\
	\lesssim &\,t\cdot\|\phi\|_{H^{|\alpha|+\frac12}(\mathbb{R}^{d})}^6.\nonumber
\end{align}

For the term without time integration in \eqref{cener_c1_1}, we apply Lemma~\ref{lemhalve} to halve the derivative:
\begin{align} \label{cener_c1_4_1}
	&\int_{\mathbb{R}}  \int_{\mathbb{R}^{d-1}}  \mathrm{Im}\left( \phi_{\alpha } \partial_k \overline{\phi}_\alpha \right) \mathrm{d} \hat{x}_k \int_{-\infty}^{x_k} \|\Lambda_k^{\frac 12}\phi_{\beta}(x_k,\cdot)\|^2_{L^2(\mathbb{R}^{d-1})} \mathrm{d} y_k  \mathrm{d} x_k \nonumber\\
	\lesssim &\,\left\|   D_k^{\frac 12}\phi_\alpha\right\|_{L^2(\mathbb{R}^d)}^2 \int_{-\infty}^{\infty} \|\Lambda_k^{\frac 12}\phi_{\beta}(x_k,\cdot)\|^2_{L^2(\mathbb{R}^{d-1})}\mathrm{d} x_k\nonumber\\
	&+\left\|   D_k^{\frac 12}\phi_{\alpha}\right\|_{L^2(\mathbb{R}^d)} \left\|   \phi_{\alpha}\right\|_{L^2(\mathbb{R}^d)} \left\| \|\Lambda_k^{\frac 12}\phi_{\beta}(x_k,\cdot)\|^2_{L^2(\mathbb{R}^{d-1})}\right\|_{L^2(\mathbb{R}_{x_k})}\\
	\lesssim &\,\left\|   D_k^{\frac 12}\phi_\alpha\right\|_{L^2(\mathbb{R}^d)}^2 \|\phi\|_{H^{|\alpha|}(\mathbb{R}^d)}^2\nonumber\\
	&+\left\|   D_k^{\frac 12}\phi_{\alpha}\right\|_{L^2(\mathbb{R}^d)} \left\|   \phi_{\alpha}\right\|_{L^2(\mathbb{R}^d)}\left\| \|\Lambda_k^{\frac 12}\phi_{\beta}(\cdot,\hat{x}_k)\|_{L^4(\mathbb{R}_{x_k})} \right\|_{L^{2}(\mathbb{R}^{d-1})}^2\nonumber\\
	\lesssim &\,\left\|   D_k^{\frac 12}\phi_\alpha\right\|_{L^2(\mathbb{R}^d)}^2 \|\phi\|_{H^{|\alpha|}(\mathbb{R}^d)}^2\nonumber+\left\|   D_k^{\frac 12}\phi_{\alpha}\right\|_{L^2(\mathbb{R}^d)} \left\|   \phi_{\alpha}\right\|_{L^2(\mathbb{R}^d)}\left\| \|\Lambda_k^{\frac 12}\phi_{\beta}(\cdot,\hat{x}_k)\|_{H^{\frac14}(\mathbb{R}_{x_k})} \right\|_{L^{2}(\mathbb{R}^{d-1})}^2\nonumber\\
	\lesssim &\,\|\phi\|^4_{H^{s_3+\frac 12}(\mathbb{R}^{d})},\nonumber
\end{align}
where Minkowski's inequality is used in the second step and the Sobolev embedding $ W^{\frac14,2}(\mathbb{R}) \hookrightarrow L^4(\mathbb{R}) $  used in the third.

As \eqref{cener_c1_3} for $ \frac{d}{2}+1<s_3 $, we have
\begin{align} \label{cener_c1_4_2}
	&\int_0^t \int_{\mathbb{R}}  \mathrm{Im}	\int_{\mathbb{R}^{d-1}} \left(  g^{kj} \Lambda_k^{\frac 12}\phi_{\beta j} \Lambda_k^{\frac 12}\overline{\phi}_{\beta }\right) (x_k,\cdot) \mathrm{d} \hat{x}_k\int_{\mathbb{R}^{d-1}}  \mathrm{Im}\left( \phi_{\alpha } \partial_k \overline{\phi}_\alpha \right) \mathrm{d} \hat{x}_k  \mathrm{d} x_k \mathrm{d} s\nonumber\\
	\lesssim &\,  \|g^{kj}\|_{L^\infty(\mathbb{R}^d)}\int_0^t\int_{\mathbb{R}}  \left\| \Lambda_k^{\frac 12}\partial\phi_{\beta }\right\|_{L^2(\mathbb{R}^{d-1})} \|\Lambda_k^{\frac 12}\phi_{\beta}(x_k,\cdot)\|_{L^2(\mathbb{R}^{d-1})}  \left\| \partial_k \phi_{\alpha }\right\|_{L^2(\mathbb{R}^{d-1})} \|\phi_\alpha(x_k,\cdot)\|_{L^2(\mathbb{R}^{d-1})} \mathrm{d} x_k \mathrm{d} s\nonumber\\
\leq&\, \delta\cdot W (t)+\delta \cdot C\|\phi\|_{H^{s_3}(\mathbb{R}^d)}^2\cdot W (t)+ C(\delta) \cdot t\cdot\left(1+\|\phi\|_{H^{s_3+\frac 12}}^2\right)\|\phi\|_{H^{s_3+\frac 12}}^4
\end{align}
and can use \eqref{lemc3_2} to obtain
\begin{align} \label{cener_c1_4_3}
	&\int_0^t \int_{\mathbb{R}}  \mathrm{Im}	\int_{\mathbb{R}^{d-1}} \left( \left[\Lambda_k^{\frac12},h^{kj}\right]\phi_{\beta j}\Lambda_k^{\frac 12}\overline{\phi}_{\beta }\right) (x_k,\cdot) \mathrm{d} \hat{x}_k\int_{\mathbb{R}^{d-1}}  \mathrm{Im}\left( \phi_{\alpha } \partial_k \overline{\phi}_\alpha \right) \mathrm{d} \hat{x}_k  \mathrm{d} x_k \mathrm{d} s\nonumber\\
	\lesssim &\, \int_0^t\int_{\mathbb{R}}  \left\|  \left[\Lambda_k^{\frac12},h^{kj}\right]\phi_{\beta j}\right\|_{L^2(\mathbb{R}^{d-1})} \|\Lambda_k^{\frac 12}\phi_{\beta}(x_k,\cdot)\|_{L^2(\mathbb{R}^{d-1})}  \left\| \partial_k \phi_{\alpha }\right\|_{L^2(\mathbb{R}^{d-1})} \|\phi_\alpha(x_k,\cdot)\|_{L^2(\mathbb{R}^{d-1})} \mathrm{d} x_k \mathrm{d} s\nonumber\\
\leq&\,  \delta \int_0^t\int_{\mathbb{R}}  \|\Lambda_k^{\frac 12}\phi_{\beta}(x_k,\cdot)\|_{L^2(\mathbb{R}^{d-1})}^2  \left\| \partial_k \phi_{\alpha }\right\|_{L^2(\mathbb{R}^{d-1})}^2\mathrm{d} x_k \mathrm{d} s \\		
	&+C(\delta)	\cdot t\cdot \|\Lambda_k ^{\frac 12}h^{kj}\|_{L^\infty(\mathbb{R}^d)}^2 \left\| \partial\phi_{\beta }\right\|_{L^2(\mathbb{R}^{d})}^2  \|\phi_\alpha(x_k,\cdot)\|_{L^2(\mathbb{R}^{d})}^2\nonumber\\
\leq&\,\delta \cdot W(t)+C(\delta)	\cdot t\cdot \|\phi\|_{H^{s_3+\frac 12}}^6 . \nonumber
\end{align}

Similar to \eqref{cener_c1_4_1},  it follows from Lemma \ref{lemhalve}  and  Minkowski's inequality that
\begin{align}  \label{cener_c1_4_4}
	&\int_0^t \int_{\mathbb{R}}   \int_{-\infty}^{x_k} \int_{\mathbb{R}^{d-1}}   \mathrm{Im} \left(  h^{ij} \phi_{\beta j}\Lambda_k\overline{\phi}_{\beta i}\right)\mathrm{d} \hat{x}_k \mathrm{d} y_k\int_{\mathbb{R}^{d-1}}  \mathrm{Im}\left( \phi_{\alpha } \partial_k \overline{\phi}_\alpha \right) \mathrm{d} \hat{x}_k  \mathrm{d} x_k \mathrm{d} s\nonumber\\
	\lesssim &\,\left\|   D_k^{\frac 12}\phi_\alpha\right\|_{L^2(\mathbb{R}^d)}^2 \int_0^t \int_{\mathbb{R}_{x_k}}\left|    \int_{\mathbb{R}^{d-1}}   \mathrm{Im} \left(  h^{ij} \phi_{\beta j}\Lambda_k\overline{\phi}_{\beta i}\right)\mathrm{d} \hat{x}_k \right|\mathrm{d} x_k\mathrm{d} s\nonumber\\
	&+\left\|   D_k^{\frac 12}\phi_{\alpha}\right\|_{L^2(\mathbb{R}^d)} \left\|   \phi_{\alpha}\right\|_{L^2(\mathbb{R}^d)} \int_0^t\left\| \int_{\mathbb{R}^{d-1}}   \mathrm{Im} \left(  h^{ij} \phi_{\beta j}\Lambda_k\overline{\phi}_{\beta i}\right)\mathrm{d} \hat{x}_k  \right\|_{L^2(\mathbb{R}_{x_k})}\mathrm{d} s\nonumber\\
	\lesssim &\,\left\|   D_k^{\frac 12}\phi_\alpha\right\|_{L^2(\mathbb{R}^d)}^2 \int_0^t \int_{\mathbb{R}_{x_k}} \left\| h^{ij}\right\|_{L^\infty (\mathbb{R}^{d-1})} \left\| \phi_{\beta j}\right\|_{L^2(\mathbb{R}^{d-1})} \left\| \Lambda_k^{\frac 12} \phi_{\beta i} \right\|_{L^2 (\mathbb{R}^{d-1})} \mathrm{d} x_k\mathrm{d} s\\
	&+\left\|   D_k^{\frac 12}\phi_{\alpha}\right\|_{L^2(\mathbb{R}^d)} \left\|   \phi_{\alpha}\right\|_{L^2(\mathbb{R}^d)} \int_0^t \int_{\mathbb{R}^{d-1}}    \left\|  h^{ij} \phi_{\beta j}\Lambda_k\overline{\phi}_{\beta i} \right\|_{L^2(\mathbb{R}_{x_k})}\mathrm{d} \hat{x}_k \mathrm{d} s\nonumber\\
	\lesssim &\,\left\|   D_k^{\frac 12}\phi_\alpha\right\|_{L^2(\mathbb{R}^d)}^2 \cdot \sum_{|\beta|\leq s_3-1}  \int_0^t	\int_{\mathbb{R}}   \| \Lambda_k^{\frac 12}\phi_{\beta}(x_k,\cdot)\|_{L^2(\mathbb{R}^{d-1})}^2  \|\phi (x_k,\cdot)\|^2_{H^{s_3+1}(\mathbb{R}^{d-1})}\mathrm{d} x_k \mathrm{d} s\nonumber\\
	&+\left\|   D_k^{\frac 12}\phi_{\alpha}\right\|_{L^2(\mathbb{R}^d)} \left\|   \phi_{\alpha}\right\|_{L^2(\mathbb{R}^d)} \cdot W(t)\nonumber\\
	\lesssim &\,\|\phi\|_{H^{s_3+\frac 12}(\mathbb{R}^d)}^2\cdot W (t),\nonumber
\end{align}
where 
\begin{align*}
	&\int_0^t \int_{\mathbb{R}^{d-1}}    \left\|  h^{ij} \phi_{\beta j}\Lambda_k\overline{\phi}_{\beta i} \right\|_{L^2(\mathbb{R}_{x_k})}\mathrm{d} \hat{x}_k \mathrm{d} s\\
	\lesssim &\,\int_0^t \int_{\mathbb{R}^{d-1}}    \left\|  h^{ij} \right\|_{L^\infty(\mathbb{R}_{x_k})} \left\|  \phi_{\beta j} \right\|_{L^\infty(\mathbb{R}_{x_k})} \left\|  \Lambda_k\overline{\phi}_{\beta i}\right\|_{L^2(\mathbb{R}_{x_k})}\mathrm{d} \hat{x}_k \mathrm{d} s\\
	\lesssim&\, \int_0^t \int_{\mathbb{R}^{d-1}}    \left\| \phi \right\|_{H^
		{l_1}(\mathbb{R}_{x_k})}^2 \left\|  \phi_{\beta j} \right\|_{H^{\frac34}(\mathbb{R}_{x_k})} \left\|  \Lambda_k\overline{\phi}_{\beta i}\right\|_{L^2(\mathbb{R}_{x_k})}\mathrm{d} \hat{x}_k \mathrm{d} s\\
	\lesssim&\, \sum_{|\beta|\leq s_3-1}  \int_0^t	\int_{\mathbb{R}}   \| \Lambda_l^{\frac 12}\phi_{\beta}(x_l,\cdot)\|_{L^2(\mathbb{R}^{d-1})}^2  \|\phi (x_l,\cdot)\|^2_{H^{s_3+1}(\mathbb{R}^{d-1})}\mathrm{d} x_l \mathrm{d} s\\
	\lesssim &\,W(t)
\end{align*}
for $ l\neq k,  1/2<l_1. $
Noticing the highest order derivative in $ \mathcal{N}_\beta \Lambda_k\overline{\phi}_{\beta } $ is less than $ s_3 $, we can derive similarly as those in \eqref{cener_c1_4_4} and  \eqref{cG} to obtain
\begin{align} \label{cener_c1_4_5}
	&\int_0^t \int_{\mathbb{R}}  \int_{-\infty}^{x_k} \int_{\mathbb{R}^{d-1}}  \mathrm{Im} \left(  \mathcal{N}_\beta \Lambda_k\overline{\phi}_{\beta }\right)\mathrm{d} \hat{x}_k \mathrm{d} y_k \int_{\mathbb{R}^{d-1}}  \mathrm{Im}\left( \phi_{\alpha } \partial_k \overline{\phi}_\alpha \right) \mathrm{d} \hat{x}_k  \mathrm{d} x_k \mathrm{d} s \\
	\lesssim &\,\|\phi\|_{H^{s_3+\frac 12}(\mathbb{R}^d)}^2\cdot W (t)+t\cdot\|\phi\|_{H^{|\alpha|+\frac12}(\mathbb{R}^{d})}^6.\nonumber
\end{align}

Combining with \eqref{cener_c1_4_1}-\eqref{cener_c1_4_5}, and using integration by parts, Hölder's inequality, and \eqref{Hs0_bec}, we obtain
\begin{align} \label{cener_c1_4}
	&		 \int_0^t\int_{\mathbb{R}}  \partial_t\int_{\mathbb{R}^{d-1}}  \mathrm{Im}\left( \phi_{\alpha } \partial_k \overline{\phi}_\alpha \right) \mathrm{d} \hat{x}_k \int_{-\infty}^{x_k} \|\Lambda_k^{\frac 12}\phi_{\beta}(y_k,\cdot)\|^2_{L^2(\mathbb{R}^{d-1})}\mathrm{d} y_k  \mathrm{d} x_k \mathrm{d} s\nonumber\\
	=&\, \left. \int_{\mathbb{R}} \int_{\mathbb{R}^{d-1}}  \mathrm{Im}\left( \phi_{\alpha } \partial_k \overline{\phi}_\alpha \right) \mathrm{d} \hat{x}_k \int_{-\infty}^{x_k} \|\Lambda_k^{\frac 12}\phi_{\beta}(y_k,\cdot)\|^2_{L^2(\mathbb{R}^{d-1})}\mathrm{d} y_k  \mathrm{d} x_k \right|_{0}^t\nonumber\\
	&\,-2\int_0^t \int_{\mathbb{R}}  \int_{\mathbb{R}^{d-1}}  \mathrm{Im}\left( \phi_{\alpha } \partial_k \overline{\phi}_\alpha \right) \mathrm{d} \hat{x}_k \int_{-\infty}^{x_k} \int_{\mathbb{R}^{d-1}}  \mathrm{Re}\left( \Lambda_k^{\frac 12}\phi_{\beta}\Lambda_k^{\frac 12}\partial_t\overline{\phi}_{\beta }\right)\mathrm{d} \hat{x}_k \mathrm{d} y_k  \mathrm{d} x_k \mathrm{d} s\\
	\leq\,&  C\left. \|\phi\|^4_{H^{s_3+\frac 12}(\mathbb{R}^{d})}\right|_{0}^t+ 2\delta\cdot W (t)+\delta \cdot C\|\phi\|_{H^{s_3}(\mathbb{R}^d)}^2\cdot W (t)\nonumber\\
	& + C(\delta) \cdot\left(1+\|\phi\|_{H^{s_3+\frac 12}}^2\right)\|\phi\|_{H^{s_3+\frac 12}}^4.\nonumber
\end{align}

The remaining terms can be estimated similarly. From \eqref{cpg}, it is straightforward to see
\begin{align} \label{cener_c1_6}
	& \int_0^t	\int_{\mathbb{R}} \int_{-\infty}^{x_k} \|\Lambda_k^{\frac 12}\phi_{\beta}(y_k,\cdot)\|^2_{L^2(\mathbb{R}^{d-1})}\mathrm{d} y_k\int_{\mathbb{R}^{d-1}}  \mathrm{Re}\left(\partial_{i} \phi_{\alpha }  \partial_k h^{ij} \partial_j \overline{\phi}_{\alpha }\right) \mathrm{d} \hat{x}_k  \mathrm{d} x_k \mathrm{d} s\nonumber\\
	\lesssim &\,   \sup_t  \|\phi\|^2_{H^{|\alpha|}(\mathbb{R}^{d})}  \int_0^t	\int_{\mathbb{R}} \left\|  \partial_k h^{ij} \right\|_{L^\infty(\mathbb{R}^{d-1})}\int_{\mathbb{R}^{d-1}}  \left|\partial_{i} \phi_{\alpha }  \partial_j \overline{\phi}_{\alpha }\right| \mathrm{d} \hat{x}_k  \mathrm{d} x_k \mathrm{d} s\\
	\lesssim &\, \sup_t  \|\phi\|^2_{H^{|\alpha|}(\mathbb{R}^{d})}\int_0^t	\int_{\mathbb{R}} \left(  \|\phi(x_k,\cdot)\|^2_{H^{|\alpha|}(\mathbb{R}^{d-1})}+\|\partial_k \phi(x_k,\cdot)\|_{H^{|\alpha|-1}(\mathbb{R}^{d-1})}^2\right) \| \nabla\phi_{\alpha}\|^2_{L^2(\mathbb{R}^d)} \mathrm{d} x_k \mathrm{d} s\nonumber\\
	\lesssim &\,\left\|  \phi\right\|_{H^{s_3+\frac 12}(\mathbb{R}^d)}^2 \cdot W (t),\nonumber
\end{align}
while a calculation similar to \eqref{cG} and \eqref{cener_c1_5} shows
\begin{align} \label{cener_c1_7}
	&	- \int_0^t	\int_{\mathbb{R}}\int_{-\infty}^{x_k} \|\Lambda_k^{\frac 12}\phi_{\beta}(y_k,\cdot)\|^2_{L^2(\mathbb{R}^{d-1})}\mathrm{d} y_k\nonumber\\
	&\cdot \partial_k\int_{\mathbb{R}^{d-1}}   \mathrm{Re}  \left\{ \phi_{\alpha } \left(\overline{F}_z\nabla \overline{\phi}_\alpha+\overline{F}_{\overline{z}}\nabla \phi_\alpha+\partial h^{ij}\partial_{ij } \overline{\phi}_{\alpha -1}+\overline{G}\right)  \right\} \mathrm{d} \hat{x}_k  \mathrm{d} x_k \mathrm{d} s\\
	=&\,	\int_0^t	\int_{\mathbb{R}} \|\Lambda_k^{\frac 12}\phi_{\beta}(x_k,\cdot)\|^2_{L^2(\mathbb{R}^{d-1})} \int_{\mathbb{R}^{d-1}}   \mathrm{Re}  \left\{\phi_{\alpha }\overline{G}+ \phi_{\alpha } \left(\overline{F}_z\nabla \overline{\phi}_\alpha+\overline{F}_{\overline{z}}\nabla \phi_\alpha+\partial h^{ij}\partial_{ij } \overline{\phi}_{\alpha -1}\right)  \right\} \mathrm{d} \hat{x}_k  \mathrm{d} x_k \mathrm{d} s\nonumber\\
	\lesssim &\,t\cdot\|\phi\|_{H^{s_3+\frac12}(\mathbb{R}^{d})}^6+\left( \|F_z\|_{L^\infty}+\|\partial h^{ij}\|_{L^\infty}\right) \nonumber\\
	&\cdot\int_0^t	\int_{\mathbb{R}} 	\|\Lambda_k^{\frac 12}\phi_{\beta}(x_k,\cdot)\|^2_{L^2(\mathbb{R}^{d-1})} \|\phi_{\alpha +1} (x_k,\cdot)\|_{L^{2}(\mathbb{R}^{d-1})}\|\phi_{\alpha } (x_k,\cdot)\|_{L^{2}(\mathbb{R}^{d-1})}\mathrm{d} x_k \mathrm{d} s\nonumber\\
	\lesssim &\,t\cdot\|\phi\|_{H^{s_3+\frac12}(\mathbb{R}^{d})}^6+ \left\|  \phi\right\|_{H^{s_3+\frac 12}(\mathbb{R}^d)}^2 \cdot W (t)\nonumber
\end{align}
for $ s_3=|\alpha|\geq \frac{d}{2}+1  $.

%Recalling the definition 
%\begin{align*}
%	\mathfrak{W} (t)=\sum_{k=1}^d \sum_{|\beta|=0}^{|\alpha|}\int_0^t \int_{\mathbb{R}}  \|\Lambda_k^{\frac 12}\phi_{\beta}(x_k,\cdot)\|^2_{L^2(\mathbb{R}^{d-1})}  \left\|  \nabla\phi_{\alpha}(x_k,\cdot)\right\|_{L^2(\mathbb{R}^{d-1})}^2   \mathrm{d} x_k \mathrm{d} s,
%\end{align*}
Setting $ \delta=\frac{1}{20d} $,
 summing of \eqref{cener_c1_1} over $ k=1,\dots, d $ and over $ |\beta|\leq s_3-1 $, combining with \eqref{cener_c1_2}-\eqref{cener_c1_5}, \eqref{cener_c1_5_G}, \eqref{cener_c1_4}-\eqref{cener_c1_7}, we  arrive at
\begin{align}  \label{cenerM}
	W(t) 
	\lesssim &\,  \|\phi_0\|^4_{H^{s_3+\frac 12}(\mathbb{R}^{d})} +  (1+t)\cdot\left(1+\|\phi\|_{H^{s_3+\frac 12}}^2\right)\|\phi\|_{H^{s_3+\frac 12}}^4+\left\|  \phi\right\|_{H^{s_3+\frac 12}(\mathbb{R}^d)}^2 \cdot W (t),
\end{align}
where the $ \delta \cdot W(t) $ in \eqref{cener_c1_4_2}, \eqref{cener_c1_4_3} and \eqref{cener_c1_4} have been absorbed by left hand side of \eqref{cenerM}.

\subsection{Energy estimates and proof of main results}  \label{s4.2}

Just as \eqref{qener1}, \eqref{qener2}, we can 
 we derive energy estimates by  by multiplying  \eqref{qNLS1}
by
$ \overline{\phi}_{\alpha } $, taking the imaginary part,
integrating it on $  [0,t]\times\mathbb{R}^d  $ and  summing  over $ |\alpha| $ from 0 to $ s_3$.   This yields
\begin{align} \label{cener1}
	\| \phi\|_{H^{ s_3 }(\mathbb{R}^d)}^2\lesssim &\, 		\| \phi_0\|_{H^{s_3 }(\mathbb{R}^d)}^2\nonumber\\
	&+\sum_{|\alpha|=0}^{s_3}\int_0^t\int_{\mathbb{R}^d} \left| \mathrm{Im} \left[  \left(F_z\nabla \phi_\alpha+F_{\overline{z}}\nabla \overline{\phi}_\alpha+\partial h^{ij}\partial_{ij } \phi_{\alpha -1}+G\right)\overline{\phi}_{\alpha }\right]\right| \mathrm{d} x\mathrm{d} s\\
	\lesssim &\,\| \phi_0\|_{H^{s_3 }(\mathbb{R}^d)}^2+W(t)+t\cdot\|\phi\|_{H^{|\alpha|+\frac12}(\mathbb{R}^{d})}^4. \nonumber
\end{align}
In fact,  the estimate follows via an analysis similar to \eqref{cG}.  For $ |\alpha| =s_3 $, we use \eqref{cF} to compute directly
\begin{align}  \label{cener1_c1}
	&\int_0^t\int_{\mathbb{R}^d} \left| \mathrm{Im} \left[  \left(F_z\nabla \phi_\alpha+F_{\overline{z}}\nabla \overline{\phi}_\alpha+\partial h^{ij}\partial_{ij } \phi_{\alpha -1}\right)\overline{\phi}_{\alpha }\right]\right| \mathrm{d} x\mathrm{d} s\\
	\lesssim &\,\int_0^t	\int_{\mathbb{R}}  \left( 	\|F_z\|_{L^\infty(\mathbb{R}^{d-1})}+\left\|	\partial h^{ij} \right\|_{L^\infty(\mathbb{R}^{d-1})}\right) (x_k)\|\phi_{\alpha +1} (x_k,\cdot)\|_{L^{2}(\mathbb{R}^{d-1})} \|\phi_{\alpha } (x_k,\cdot)\|_{L^{2}(\mathbb{R}^{d-1})}\mathrm{d} x_k \mathrm{d} s\nonumber\\
	\lesssim &\,\int_0^t	\int_{\mathbb{R}}  \left(  \|\phi(x_k,\cdot)\|_{H^{s_3-\frac 12}(\mathbb{R}^{d-1})}^2+\|\partial_k \phi(x_k,\cdot)\|_{H^{s_3-\frac 32}(\mathbb{R}^{d-1})}^2\right) \nonumber\\
	&\cdot\|\phi_{\alpha +1} (x_k,\cdot)\|_{L^{2}(\mathbb{R}^{d-1})}\|\phi_{\alpha } (x_k,\cdot)\|_{L^{2}(\mathbb{R}^{d-1})}\mathrm{d} x_k \mathrm{d} s \nonumber\\
	\lesssim &\,W(t),\nonumber
\end{align}
while  \eqref{cG}  implies 
\begin{align} \label{cener1_c2}
	\left|\int_0^t\int_{\mathbb{R}^d}  G\overline{\phi}_{\alpha } \mathrm{d} x\mathrm{d} s\right|\lesssim t\cdot\|\phi\|_{H^{|\alpha|+\frac12}(\mathbb{R}^{d})}^4.
\end{align}

Next, we estimate $ 		\|D^{\frac 12}\phi_\alpha \|_{L^{2}(\mathbb{R}^{d})}  $ for $ |\alpha|=s_3. $
Multiplying  \eqref{qNLS1} by
$D_k \overline{\phi}_{\alpha } $ 	and taking the imaginary part gives
\begin{align}  \label{ccons3}
	&\mathrm{Im} \left(  \left(F_z\nabla \phi_\alpha+F_{\overline{z}}\nabla \overline{\phi}_\alpha+\partial h^{ij}\partial_{ij } \phi_{\alpha -1}+G\right)D_k\overline{\phi}_{\alpha } \right)\nonumber\\
	=&\,\mathrm{Re}\left( \partial_t\phi_{\alpha } D_k\overline{\phi}_{\alpha }  \right)+ \mathrm{Im} \left(  \partial_i\left(   g^{ij} \partial_j \phi_{\alpha }\right)D_k \overline{\phi}_{\alpha }\right)\nonumber\\
	%	=&\, \frac 12\frac{d}{d t}\left(  \left| \phi_{\alpha } \right|^2 \right)+   \partial_i \mathrm{Im}\left(  g^{ij} \partial_j \phi_{\alpha } \overline{\phi}_{\alpha }\right)\nonumber
	=&\, \mathrm{Re}\left( \partial_t\phi_{\alpha } D_k\overline{\phi}_{\alpha }  \right)+ \mathrm{Im} \partial_i\left(   g^{ij} \partial_j \phi_{\alpha }D_k \overline{\phi}_{\alpha }\right)- \mathrm{Im} \left(   \partial_j \phi_{\alpha } g^{ij} D_k \overline{\phi}_{\alpha i }\right)\nonumber\\
	=&\, \mathrm{Re}\left( \partial_t\phi_{\alpha } D_k \overline{\phi}_{\alpha }  \right)+ \mathrm{Im} \partial_i\left(  \left(   g^{ij} \partial_j \phi_{\alpha }\right) D_k \overline{\phi}_{\alpha }\right)\\
	&- \mathrm{Im} \left(   \partial_j \phi_{\alpha } \left[ h^{ij}  D_k^{\frac12}, D_k^{\frac12} \right] \overline{\phi}_{\alpha i }\right)-\mathrm{Im} \left(   \partial_j \phi_{\alpha } D_k^{\frac12}( g^{ij} D_k^{\frac12}  \overline{\phi}_{\alpha i })\right),\nonumber
\end{align}	
where 
\begin{align} \label{ccons3_c1}
	&	\mathrm{Im} \left(   \partial_j \phi_{\alpha }  g^{ij}  D_k  \overline{\phi}_{\alpha i }\right)\nonumber\\
	=&\, \mathrm{Im} \left(   \partial_j \phi_{\alpha } \left[ h^{ij}  D_k^{\frac12}, D_k^{\frac12} \right] \overline{\phi}_{\alpha i }\right)+\mathrm{Im} \left(   \partial_j \phi_{\alpha } D_k^{\frac12}( g^{ij} D_k^{\frac12}  \overline{\phi}_{\alpha i })\right).
\end{align}
It follows from \eqref{lemc1_3} that
\begin{align} \label{ccons3_c2}
	&\int_0^t \int_{\mathbb{R}^d}  \mathrm{Im} \left(   \partial_j \phi_{\alpha } \left[ h^{ij}  D_k^{\frac12}, D_k^{\frac12} \right] \overline{\phi}_{\alpha i }\right)\mathrm{d} x \mathrm{d} s\nonumber\\
	\lesssim &\,\int_0^t \int_{\mathbb{R}^{d-1}} \|\nabla\phi_{\alpha}\|_{L^{2}(\mathbb{R}_{x_k})}  \left\| \left[ h^{ij}  D_k^{\frac12}, D_k^{\frac12} \right] \overline{\phi}_{\alpha i }\right\|_{L^2(\mathbb{R}_{x_k})} \mathrm{d} \hat{x}_k \mathrm{d} s\\
	\lesssim &\,\int_0^t \int_{\mathbb{R}^{d-1}} \|\nabla\phi_{\alpha}\|_{L^{2}(\mathbb{R}_{x_k})}^2   \|\partial_k h^{ij}\|_{H^{ \mathscr{S}_1}(\mathbb{R}_{x_k})} \mathrm{d} \hat{x}_k\nonumber\\
	\lesssim &\,\int_0^t \int_{\mathbb{R}^{d-1}} \|\nabla\phi_{\alpha}\|_{L^{2}(\mathbb{R}_{x_k})}^2    \|\phi\|_{H^{ \mathscr{S}_1+1}(\mathbb{R}_{x_k})}^2 \mathrm{d} \hat{x}_k\nonumber\\
	\lesssim &\,W (t), \nonumber
\end{align}
where
\begin{align*}
	\left\|\partial_k h^{ij}\right\|_{H^{ \mathscr{S}_1}(\mathbb{R}_{x_k})}\lesssim &\, \left\|\phi\partial_k \phi\right\|_{H^{ \mathscr{S}_1}(\mathbb{R}_{x_k})}\\
	\lesssim &\,\left\|\phi \right\|_{L^\infty(\mathbb{R}_{x_k})} \ \left\|\langle D_k\rangle^{\mathscr{S}_1}\partial_k \phi\right\|_{L^{2 }(\mathbb{R}_{x_k})}+\left\|\langle D_k\rangle^{\mathscr{S}_1}\phi \right\|_{L^2(\mathbb{R}_{x_k})} \ \left\|\partial_k \phi\right\|_{L^{\infty }(\mathbb{R}_{x_k})}\\
	\lesssim &\, \|\phi\|_{H^{ \mathscr{S}_1+1}(\mathbb{R}_{x_k})}^2
\end{align*}
for $  \frac{1}{2}<\mathscr{S}_1$  and $ 1+\mathscr{S}_1\leq s_3-\frac 12 $.   Note that $ d\geq 2 $ implies $ \frac{d}{2}+1<s_3 $.

From \eqref{cF} and \eqref{cpg}, we have
\begin{align} \label{ccons3_c3}
	& \int_0^t \int_{\mathbb{R}^d} \left|  \left(F_z\nabla \phi_\alpha+F_{\overline{z}}\nabla \overline{\phi}_\alpha+\partial h^{ij}\partial_{ij } \phi_{\alpha -1}\right)D_k\overline{\phi}_{\alpha } \right| \mathrm{d} x \mathrm{d} s\\
	\lesssim &\,\int_0^t \int_{\mathbb{R}_{x_l}} \left(\|\Lambda_k^{\frac 12}\phi_{\beta}(x_k,\cdot)\|^2_{L^2(\mathbb{R}^{d-1}_{\hat{x}_l})}+\|\phi(x_k,\cdot)\|_{H^{s_0}(\mathbb{R}^{d-1}_{\hat{x}_l})}\right)\cdot\| \nabla \phi_{\alpha}\|_{L^2(\mathbb{R}^{d-1})}\| D_k \phi_{\alpha}\|_{L^2(\mathbb{R}^{d-1}_{\hat{x}_l})} \mathrm{d} x_l\mathrm{d}s  \nonumber\\
	\lesssim &\,W(t),\nonumber 
\end{align}
for  $  l\neq k  $,
using the fact that
\begin{align*}
	\| D_k \phi_{\alpha}\|_{L^2(\mathbb{R}^{d-1}_{\hat{x}_l})} \lesssim 	\| \partial_k \phi_{\alpha}\|_{L^2(\mathbb{R}^{d-1}_{\hat{x}_l})},\quad  \qquad\text{if}~~ l\neq k,
\end{align*}
while the  analysis  similar to \eqref{cG}  gives
\begin{align} \label{ccons3_c4}
	\left| \int_0^t \int_{\mathbb{R}^d}   G D_k\overline{\phi}_{\alpha }  \mathrm{d} x \mathrm{d} s \right|= 	\left| \int_0^t \int_{\mathbb{R}^d}   D_k^{\frac 12} G D_k^{\frac 12}\overline{\phi}_{\alpha }  \mathrm{d} x \mathrm{d} s \right|\lesssim t\cdot\|\phi\|_{H^{|\alpha|+\frac12}(\mathbb{R}^{d})}^4.
\end{align}
Noticing that the symmetry of $ g^{ij} $ shows
\begin{align*} 
	\sum_{i,j}\int_{\mathbb{R}^{d}}\mathrm{Im} \left(   \partial_j \phi_{\alpha } D_k^{\frac12}( g^{ij} D_k^{\frac12}  \overline{\phi}_{\alpha i })\right) \mathrm{d}x=\sum_{i,j}	\int_{\mathbb{R}^{d}}\mathrm{Im} \left( D_k^{\frac 12}\overline{\phi}_{\alpha i }  g^{ij} D_k^{\frac 12} \partial_j \phi_{\alpha }\right) \mathrm{d}x=0,
\end{align*}
we can integrate \eqref{ccons3} on  $  [0,t]\times\mathbb{R}^d  $, take the imaginary part and combine \eqref{ccons3_c1}, \eqref{ccons3_c2}, \eqref{ccons3_c3}, \eqref{ccons3_c4} to obtain
\begin{align*} 
	&	\| D_k^{\frac12}\phi_\alpha\|_{L^2}^2 \lesssim 	\| D_k^{\frac12}\partial^\alpha\phi_0\|_{L^2}^2+ W(t)
\end{align*}
which combining with \eqref{cener1} implies
\begin{align}\label{cener3}
	&	\|\phi\|_{H^{s_3+\frac12}(\mathbb{R}^{d})}^2\lesssim 	\|\phi_0\|_{H^{s_3+\frac12}(\mathbb{R}^{d})}^2+ W(t)+t\cdot\|\phi\|_{H^{|\alpha|+\frac12}(\mathbb{R}^{d})}^4.
\end{align}

Collecting \eqref{cenerM}, \eqref{cener1}, and \eqref{cener3}, we conclude
\begin{align*}
	\|\phi\|_{H^{s_3+\frac12}(\mathbb{R}^{d})}^2+W(t) \lesssim &\,\|\phi_0\|_{H^{s_3+\frac12}(\mathbb{R}^{d})}^2+  \|\phi_0\|^4_{H^{s_3+\frac 12}(\mathbb{R}^{d})}
	\\
	&+( 1+t)\cdot\left(1+\|\phi\|_{H^{s_3+\frac 12}}^2\right)\|\phi\|_{H^{s_3+\frac 12}}^4+\left\|  \phi\right\|_{H^{s_3+\frac 12}(\mathbb{R}^d)}^2 \cdot W (t).\nonumber
\end{align*}
A bootstrap argument then shows that for $ t\leq 1 $ and $ \|\phi_0\|_{H^{s_3+\frac12}(\mathbb{R}^{d})} $	 small enough, we have
\begin{align*} 
	\|\phi\|_{H^{s_3+\frac12}(\mathbb{R}^{d})}^2+	W(t) \leq C\left(\|\phi_0\|_{H^{s_3+\frac12}(\mathbb{R}^{d})}\right).
\end{align*}

%	\phi_t-\varepsilon\Delta \phi-	\sqrt{-1}\partial_i\left( g^{ij}\left(\phi,\overline{\phi}\right) \partial_j \phi\right)=	-\sqrt{-1}F\left(\phi,\overline{\phi},\nabla\phi, \nabla \overline{\phi}\right)

\begin{proof}[\textbf{Proof of Theorem \ref{thm3}}]
As in the proof of Theorem~\ref{thm2}, we employ the artificial viscosity method with the approximate equation \eqref{aqNLS}. Recall \eqref{aqNLS1}:
	\begin{align*}  
		\sqrt{-1}\phi_{\alpha t}+	\varepsilon\sqrt{-1}\Delta^2 \phi_\alpha+\partial_i\left( g^{ij} \partial_j \phi_{\alpha }\right)=F_z\nabla \phi_\alpha+F_{\overline{z}}\nabla \overline{\phi}_\alpha+\partial h^{ij}\partial_{ij } \phi_{\alpha -1}+G.
	\end{align*}
Analogous to to \eqref{cener1} and \eqref{cener3},
we derive the energy estimates
	\begin{align*}
		&	\| \phi\|_{H^{s_3 }(\mathbb{R}^d)}^2+\varepsilon  \sum_{i,j=1}^d\int_0^t\left\| \partial_{ij}\phi\right\|_{H^{s_3}(\mathbb{R}^d)}^2 \mathrm{d} s\\
		\lesssim &\, 		\| \phi_0\|_{H^{s_3 }(\mathbb{R}^d)}^2+\sum_{|\alpha|=0}^{s_3}\int_0^t\int_{\mathbb{R}^d} \left| \mathrm{Im} \left[  \left(F_z\nabla \phi_\alpha+F_{\overline{z}}\nabla \overline{\phi}_\alpha+\partial h^{ij}\partial_{ij } \phi_{\alpha -1}+G\right)\overline{\phi}_{\alpha }\right]\right| \mathrm{d} x\mathrm{d} s\\
		\lesssim &\,\| \phi_0\|_{H^{s_3 }(\mathbb{R}^d)}^2+W(t)+t\cdot\|\phi\|_{H^{s_3+\frac12}(\mathbb{R}^{d})}^6 \nonumber
	\end{align*}
	and 
	\begin{align} \label{acener3}
		\|\phi\|_{H^{s_3+\frac12}(\mathbb{R}^{d})}^2+\varepsilon  \sum_{i,j=1}^d\int_0^t\left\| \partial_{ij}\phi\right\|_{H^{s_3+\frac 12}(\mathbb{R}^d)}^2 \mathrm{d} s\lesssim 	\|\phi_0\|_{H^{s_3+\frac12}(\mathbb{R}^{d})}^2+ W(t)+t\cdot\|\phi\|_{H^{s_3+\frac12}(\mathbb{R}^{d})}^6.
	\end{align}

	Similar to \eqref{cener_c1_4_1}, we verify
	\begin{align*}
		&\varepsilon \left| \int_0^t	\int_{\mathbb{R}}\int_{-\infty}^{x_k} \|\Lambda_k^{\frac 12}\phi_{\beta}(y_k,\cdot)\|^2_{L^2(\mathbb{R}^{d-1})} \mathrm{d} y_k   \cdot\mathrm{Im}\int_{\mathbb{R}^{d-1}} \left(\partial_{kjj}\phi_{\alpha } \partial_{kk}\overline{\phi}_\alpha\right) \mathrm{d} \hat{x}_k\mathrm{d} x_k\mathrm{d} s\right|\\
		\lesssim &\,\varepsilon \int_0^t \left\|   D_k^{\frac 12}  \partial_{kk}\phi_\alpha\right\|_{L^2(\mathbb{R}^d)}  \left\|   D_k^{\frac 12} \partial_{jj}\phi_\alpha\right\|_{L^2(\mathbb{R}^d)} \|\Lambda_k^{\frac 12}\phi_{\beta}\|_{L^{2}(\mathbb{R}^d)}^2 \mathrm{d} s\nonumber\\
		&+\varepsilon\int_0^t \left\|    \partial_{kk}\phi_{\alpha}\right\|_{L^2(\mathbb{R}^d)}   \left\|   D_k^{\frac 12} \partial_{jj}\phi_\alpha\right\|_{L^2(\mathbb{R}^d)} \left\| \|\Lambda_k^{\frac 12}\phi_{\beta}(\cdot,\hat{x}_k)\|_{H^{\frac12}(\mathbb{R}_{x_k})} \right\|_{L^{2}(\mathbb{R}^{d-1})}^2 \mathrm{d} s\nonumber\\
		\lesssim &\,\varepsilon \left\| \Lambda_k^{\frac 12}\phi_{\beta}\right\|_{H^{\frac12}(\mathbb{R}^d)}^2  \sum_{i,j=1}^d\int_0^t\left\| \partial_{ij}\phi\right\|_{H^{s_3+\frac 12}(\mathbb{R}^d)}^2 \mathrm{d} s.
	\end{align*}
An analysis of \eqref{aconsm} analogous to \eqref{cenerM} then gives
	\begin{align*}
		&W(t) \lesssim   \|\phi_0\|^4_{H^{s_3+\frac 12}(\mathbb{R}^{d})} +  (1+t)\cdot\left(1+\|\phi\|_{H^{s_3+\frac 12}}^2\right)\|\phi\|_{H^{s_3+\frac 12}}^4+\left\|  \phi\right\|_{H^{s_3+\frac 12}(\mathbb{R}^d)}^2 \cdot W (t)\\
		&+\varepsilon \left\| \Lambda_k^{\frac 12}\phi_{\beta}\right\|_{H^{\frac12}(\mathbb{R}^d)}^2\int_0^t\left\| \nabla\phi\right\|_{H^{s_3+\frac 12}}^2 \mathrm{d} s. \nonumber
	\end{align*}
Combining this with \eqref{acener3}, we obtain
	\begin{align*}
		&W(t) +\|\phi\|_{H^{s_3+\frac12}(\mathbb{R}^{d})}^2+\varepsilon  \sum_{i,j=1}^d\int_0^t\left\| \partial_{ij}\phi\right\|_{H^{s_3+\frac 12}(\mathbb{R}^d)}^2 \mathrm{d} s\lesssim  \|\phi_0\|_{H^{s_3+\frac12}(\mathbb{R}^{d})}^2+\|\phi_0\|_{H^{s_3+\frac12}(\mathbb{R}^{d})}^4\nonumber\\
		&+(1+t)\cdot\left(1+\|\phi\|_{H^{s_3+\frac 12}}^2\right)\|\phi\|_{H^{s_3+\frac 12}}^4+\left\|  \phi\right\|_{H^{s_3+\frac 12}(\mathbb{R}^d)}^2 \cdot W (t)\nonumber\\
		&+\varepsilon  \left\|  \phi\right\|_{H^{s_3+\frac 12}(\mathbb{R}^d)}^2 \int_0^t\left\| \nabla\phi\right\|_{H^{s_3+\frac 12}}^2 \mathrm{d} s.
	\end{align*}
	For $ \|\phi_0\|_{H^{s_3+\frac12}(\mathbb{R}^{d})} $ small enough, we conclude that for \(t \in [0,1]\) and \(\frac{d+3}{2} < s_3 + \frac12\),
	\begin{align*} 
		W(t) +\|\phi\|_{H^{s_3+\frac12}(\mathbb{R}^{d})}^2+\varepsilon  \sum_{i,j=1}^d\int_0^t\left\| \partial_{ij}\phi\right\|_{H^{s_3+\frac 12}(\mathbb{R}^d)}^2 \mathrm{d} s \leq C\left(\|\phi_0\|_{H^{s_3+\frac12}(\mathbb{R}^{d})}\right).
	\end{align*}
	The existence of a solution follows by an argument similar to that in the proof of Theorem~\ref{thm2}.
	 
	 The proof of the uniqueness and continuity of the solution  is also similar with the quadratic interaction problem except the following points.  The weight to multiply by \eqref{consm1-2} is still
 $$ \int_{-\infty}^{x_k} \|\Lambda_k^{\frac 12}v_{\beta}(y_k,\cdot)\|^2_{L^2(\mathbb{R}^{d-1})} \mathrm{d} y_k  $$ and the space time norms for $ v(t) =\phi_1 -\phi_2$ is
	 \begin{align*}
	 	&W_v(t)	
	 	:=\sum_{k=1}^d\sum_{|\beta|\leq s_3-1} \int_0^t \int_{\mathbb{R}}  \| \Lambda_k^{\frac 12}v_{\beta}(x_k,\cdot)\|_{L^2(\mathbb{R}^{d-1})}^2  \|\partial_k v (x_k,\cdot)\|_{L^2(\mathbb{R}^{d-1})}^2 \mathrm{d} x_k \mathrm{d} s,
	 \end{align*}
	 and  
	 the solutions satisfy
	 \begin{align*}
	 	\label{csol1-2}
	 	&	\|\phi_l\|_{H^{s_3+\frac12}(\mathbb{R}^{d})}^2+	W(\phi_l)(t) \leq C\left(\|\phi_{l0}\|_{H^{s_3+\frac12}(\mathbb{R}^{d})}\right)\ll 1,~~l=1,2.
	 \end{align*}
	 Since 
	 \begin{align*}
	 	\max\{ \|h^{ij}_y\|_{L^{\infty}}, \|h^{ij}_{\overline{y}}\|_{L^{\infty}}\} \lesssim  \| \phi_1\|_{L^{\infty}}+\| \phi_2\|_{L^{\infty}},
	 \end{align*}
we have
	 \begin{align}
	 	&\int_0^t \int_{\mathbb{R}}\int_{-\infty}^{x_k} \|\Lambda_k^{\frac 12}v_{\beta}(y_k,\cdot)\|^2_{L^2(\mathbb{R}^{d-1})} \mathrm{d} y_k	\int_{\mathbb{R}^{d-1}} \mathrm{Re} \left(\partial_i\left( (g^{ij}\left(\phi_1,\overline{\phi}_1\right)-g^{ij}\left(\phi_2,\overline{\phi}_2\right)) \partial_j\phi_2\right)\partial_k \overline{v} \right)   \mathrm{d} \hat{x}  \mathrm{d} x_k \mathrm{d}s\nonumber\\
	 	=&\,\mathrm{Re}  \int_0^t \int_{\mathbb{R}} \int_{-\infty}^{x_k} \|\Lambda_k^{\frac 12}v_{\beta}(y_k,\cdot)\|^2_{L^2(\mathbb{R}^{d-1})} \mathrm{d} y_k	\int_{\mathbb{R}^{d-1}} \partial_i\left( h^{ij}_y  \left(\phi_1 -  \phi_2 \right)+h^{ij}_{\overline{y}} (\overline{\phi}_1-\overline{\phi}_2)\right) \partial_j\phi_2\partial_k \overline{v}  \mathrm{d} \hat{x}  \mathrm{d} x_k \mathrm{d}s \nonumber\\
	 	&+ \mathrm{Re}\int_0^t \int_{\mathbb{R}}  \int_{-\infty}^{x_k} \|\Lambda_k^{\frac 12}v_{\beta}(y_k,\cdot)\|^2_{L^2(\mathbb{R}^{d-1})} \mathrm{d} y_k	\int_{\mathbb{R}^{d-1}}   \left( h^{ij}_y  \left(\phi_1 -  \phi_2 \right)+h^{ij}_{\overline{y}} (\overline{\phi}_1-\overline{\phi}_2)\right)\partial_{ij}\phi_2\partial_k \overline{v}   \mathrm{d} \hat{x}  \mathrm{d} x_k \mathrm{d}s\nonumber \\
	 	\lesssim &\,\left( \left\|  \phi_1\right\|_{H^{s_3+\frac 12}(\mathbb{R}^d)}^2 + \left\|  \phi_2\right\|_{H^{s_3+\frac 12}(\mathbb{R}^d)}^2 \right) \cdot W_v(t)\\
	 	&+ \left( \left\|  \phi_1\right\|_{H^{s_3+\frac 12}(\mathbb{R}^d)}^2 + \left\|  \phi_2\right\|_{H^{s_3+\frac 12}(\mathbb{R}^d)}^2 \right) \cdot t^{\frac 12}\| \partial_{ij}\phi_2\|_{L^2(\mathbb{R}^d)} \|h^{ij}_y\|_{L^{\infty}(\mathbb{R}^d)} \nonumber\\
	 	&\cdot\left(\int_0^t \int_{\mathbb{R}}    \|\phi_1 -  \phi_2 (x_k,\cdot)\|_{L^\infty(\mathbb{R}^{d-1})}^2   \|\partial_k v (x_k,\cdot)\|_{L^2(\mathbb{R}^{d-1})}^2 \mathrm{d} x_k \mathrm{d} s  \right)\nonumber\\
	 	\lesssim &\,\left( \left\|  \phi_1\right\|_{H^{s_3+\frac 12}(\mathbb{R}^d)}^2 + \left\|  \phi_2\right\|_{H^{s_3+\frac 12}(\mathbb{R}^d)}^2 \right) \cdot W_v(t)+t^{\frac 12}\left( \left\|  \phi_1\right\|_{H^{s_3+\frac 12}(\mathbb{R}^d)}^2 + \left\|  \phi_2\right\|_{H^{s_3+\frac 12}(\mathbb{R}^d)}^2 \right)^2 \cdot W_v^{\frac 12}(t).\nonumber
	 \end{align}		
\end{proof}

	\begin{appendices} 
		\section{Proof of some preliminary lemmas}	\label{appa}
		\begin{proof}[\textbf{Proof of the Lemma \ref{lemc1}}]
			We begin by establishing the following elementary inequalities:	
			\begin{align}
				\left||\xi|^s-|\eta|^s\right| \lesssim &\,\frac{|\xi-\eta|}{\left(|\xi|+|\eta|\right)^{1-s}} \label{propc1_2}\\
				\lesssim &\, \frac{|\xi-\eta|}{|\eta|^{1-s}}  \label{propc1_3}
			\end{align}
			for $ s\leq 1. $ The \eqref{propc1_3}  is clear since  $ 1-s\geq 0. $ To prove  \eqref{propc1_2}, assume$ |\eta| \leq |\xi|  $. Using
			$$ \left(|\xi|+|\eta|\right)^{1-s} \sim |\xi|^{1-s}+|\eta|^{1-s},  $$
		we reduce \eqref{propc1_2} to showing
			\begin{align}
				|\xi|^s-|\eta|^s \lesssim &\,\frac{|\xi|-|\eta|}{|\xi|^{1-s}+|\eta|^{1-s}} \label{propc1_2.1}
			\end{align}
			which is equivalent to
			\begin{align} \label{propc1_2.2}
				|\xi|-|\eta| +|\xi|^s |\eta|^{1-s}-|\xi|^{1-s}|\eta|^s \lesssim |\xi|-|\eta|.
			\end{align}
			If $ s\leq \frac 12 $, then $ 1-2s\geq 0 $ and  
			\begin{align*}
				|\xi|^s |\eta|^{1-s}-|\xi|^{1-s}|\eta|^s= |\xi|^s |\eta|^s \left( |\eta|^{1-2s}-|\xi|^{1-2s}\right) \leq 0
			\end{align*}
			and it follows that \eqref{propc1_2.2}, \eqref{propc1_2.1}, \eqref{propc1_2} hold.  If $ \frac 12 <s\leq 1  $, setting $ s=2s_4 $, we have   $ s_4\leq \frac 12 $ and 
			\begin{align}
				|\xi|^s-|\eta|^s =&\,\left(	|\xi|^{s_4}+|\eta|^{s_4} \right)\left(	|\xi|^{s_4}-|\eta|^{s_4} \right)\nonumber\\
				\lesssim &\, \left(|\xi|+|\eta|\right)^{s_4}\frac{|\xi-\eta|}{\left(|\xi|+|\eta|\right)^{1-s_4}}\lesssim\frac{|\xi-\eta|}{\left(|\xi|+|\eta|\right)^{1-s}}. \nonumber
			\end{align}

	To verify  \eqref{lemc1_1},   we compute
			\begin{align} \label{propc1_4}
				&\left|\mathcal{F}\left\{ D^s\left(g \partial_j u \right)-g \partial_j \left(D^s u\right) \right\}\right|\\
				=&\,\left|	\left|\xi\right|^s \int_{\mathbb{R}^d}	\hat{g} (\xi-\eta) \eta_j  \hat{u}(\eta)\mathrm{d} \eta-\int_{\mathbb{R}^d} 	\hat{g} (\xi-\eta) \eta_j |\eta|^s \hat{u}(\eta)\mathrm{d} \eta\right| \nonumber\\
				=&\, \left|\int_{\mathbb{R}^d} \left( \left|\xi\right|^s-\left|\eta\right|^s\right)	\hat{g} (\xi-\eta) \eta_j  \hat{u}(\eta)\mathrm{d} \eta \right|\nonumber \\
				\lesssim &\,\int_{\mathbb{R}^d} \left|\xi-\eta\right|	\left|\hat{g} (\xi-\eta)\right| \left|\frac{\eta_j|\eta|^s}{|\eta|}  \hat{u}(\eta)\right|\mathrm{d} \eta, \nonumber 
			\end{align}
			where   \eqref{propc1_3} is  used in the last line.
		For $ \mathscr{S}>d/2 $,  the  Young's and H\"older's inequalities  indicate \eqref{lemc1_1}:
			\begin{align*}
				\left\| \eqref{propc1_4}\right\|_{L^2}\lesssim &\,\left\||\cdot|\hat{g}(\cdot)\right\|_{L^1}  \|u\|_{H^s}\\
				\lesssim &\,\left\|\frac{1}{\left(1+|\cdot|\right)^{\mathscr{S}}}\right\|_{L^2(\mathbb{R}^d)}  \|\nabla g\|_{H^{\mathscr{S}}}\|u\|_{H^s}\nonumber\\
				\leqslant&\, c \|\nabla g\|_{H^{\mathscr{S}}}\|u\|_{H^s}.\nonumber
			\end{align*}

		A similar argument gives \eqref{lemc1_2}. Indeed,	
			\begin{align*}
				&\left|\mathcal{F}\left\{ D^{s}\left(g D^{l} u \right)-g D^{s+l} u\right\}\right|\\
				=&\,\left|	\left|\xi\right|^{s} \int_{\mathbb{R}^d}	\hat{g} (\xi-\eta) |\eta|^{l} \hat{u}(\eta)\mathrm{d} \eta-\int_{\mathbb{R}^d} 	\hat{g} (\xi-\eta)  |\eta|^{s+l} \hat{u}(\eta)\mathrm{d} \eta\right| \nonumber\\
				=&\, \left|\int_{\mathbb{R}^d} \left( \left|\xi\right|^{s}-\left|\eta\right|^{s}\right)	\hat{g} (\xi-\eta) |\eta|^{l} \hat{u}(\eta)\mathrm{d} \eta \right|\nonumber \\
				\lesssim &\,\int_{\mathbb{R}^d} \left|\xi-\eta\right|	\left|\hat{g} (\xi-\eta)\right| \left|\frac{ \left|\eta\right|^{l}}{|\eta|^{1-s}}  \hat{u}(\eta)\right|\mathrm{d} \eta \nonumber \\
				\leqslant&\, c \|\nabla g\|_{H^{\mathscr{S}}}\|u\|_{L^2}\nonumber
			\end{align*}
due to $ s+l= 1. $
		\end{proof}
	
	\begin{proof} [\textbf{Proof of the Lemma \ref{leminter1}}]
		
	Using the well-known embeddings
		\begin{align*}
			&\dot{B}_{\infty, 1}^0\left(\mathbb{R}^d\right) \hookrightarrow \mathrm{BMO}(\mathbb{R}^d)\hookrightarrow \dot{B}_{\infty, \infty}^0\left(\mathbb{R}^d\right), 
		\end{align*}
		and the  interpolation inequality
		\begin{align*}
			\|u\|_{\dot{B}_{\infty, 1}^{\theta l_1+(1-\theta) l_2}} \lesssim  \|u\|_{\dot{B}_{\infty, \infty}^{l_1}}^\theta\|u\|_{\dot{B}_{\infty, \infty}^{l_2}}^{1-\theta}
		\end{align*}
		for  $ l_1<l_2 $ and $ \theta \in (0,1) $, 
		we obtain
		\begin{align*}
			\left\| D^{\frac12} f \right\|_{\mathrm{BMO}(\mathbb{R}^d)}\lesssim &\, 	\left\| D^{\frac12 } f \right\|_{ \dot{B}_{\infty, 1}^0\left(\mathbb{R}^d\right)} \\
			%		\lesssim &\, 	\left\|  f \right\|_{ \dot{B}_{\infty, \infty}^0\left(\mathbb{R}^d\right)}^{\frac12 }	\left\| D  f \right\|_{ \dot{B}_{\infty, \infty}^0\left(\mathbb{R}^d\right)}^{\frac12 }\\
			\lesssim &\, 	\left\|  f \right\|_{ \dot{B}_{\infty, \infty}^0\left(\mathbb{R}^d\right)}^{\frac12 }	\left\| \nabla  f \right\|_{ \dot{B}_{\infty, \infty}^0\left(\mathbb{R}^d\right)}^{\frac12 }\\
			\lesssim &\, 	\left\|  f \right\|_{\mathrm{BMO}(\mathbb{R}^d)}^{\frac12 }	\left\| \nabla  f \right\|_{ \mathrm{BMO}(\mathbb{R}^d)}^{\frac12 },
		\end{align*}
where $ \nabla f $ is equivalent to $ Df $ in Besov spaces.		
	\end{proof}

		\begin{proof} [\textbf{Proof of the Lemma \ref{lemx} }] Since \eqref{lemx_2} follows directly from \eqref{lemx_1}, it suffices to prove
			\begin{align}  \label{lem3_2}
				\int_{\mathbb{R}^d} |x_k|^{2N_1}|\phi_\gamma|^2 \mathrm{d} x\lesssim \left\|\phi\right\|_{H^{|\gamma|+1}(\mathbb{R}^d)}^{\mathscr{S}_1} \left\||x|^N \phi\right\|_{L^2(\mathbb{R}^d)}^{\mathscr{S}_2}
			\end{align}
		for $ k=1,\ldots,d, $ where $ \mathscr{S}_1+\mathscr{S}_2=2 $.
			Integration by parts gives
			\begin{align*}
				\int_{\mathbb{R}^d} |x_k|^{2N_1}|\phi_{x_k}|^2 \mathrm{d} x=&\,		\int_{\mathbb{R}^d} |x_k|^{2N_1}\phi_{x_k} \overline{\phi}_{x_k} \mathrm{d} x\\
				=&\,-	\int_{\mathbb{R}^d} 2N_1|x_k|^{2N_1-1}\phi \overline{\phi}_{x_k} \mathrm{d} x-	\int_{\mathbb{R}^d} |x_k|^{2N_1}\phi \overline{\phi}_{x_kx_k} \mathrm{d} x\\
				\lesssim &\,\left\| x_k^{2N_1}\phi \right\|_{L^2}^{\frac{2N_1-1}{2N_1}}\left\|\phi \right\|_{H^2}^{\frac{2N_1+1}{2N_1}}+ \left\| x_k^{2N_1}\phi \right\|_{L^2}\left\|\phi \right\|_{H^2}\\
				\lesssim&\max\left\{\left\| |x|^{2N_1}\phi \right\|_{L^2}^{\frac{2N_1-1}{2N_1}}\left\|\phi \right\|_{H^2}^{\frac{2N_1+1}{2N_1}}, ~\left\| |x|^{2N_1}\phi \right\|_{L^2}\left\|\phi \right\|_{H^2} \right\}
			\end{align*}
			where 
			\begin{align*}
				\left\| \int_{\mathbb{R}^d} 2N_1|x_k|^{2N_1-1}\phi \overline{\phi}_{x_k} \mathrm{d} x\right\|\lesssim &\,\left(\int_{\mathbb{R}^{d}} |x_k|^{4N_1}|\phi|^2\right)^{\frac{2N_1-1}{4N_1}} \left( \int_{\mathbb{R}^{d}} |\phi|^2\right)^{\frac{1}{4N_1}}\left( \int_{\mathbb{R}^{d}} |\phi_{x_k}|^2\right)^{\frac{1}{2}}\\
				\lesssim &\,\left\| x_k^{2N_1}\phi \right\|_{L^2}^{\frac{2N_1-1}{2N_1}}\left\|\phi \right\|_{H^2}^{\frac{2N_1+1}{2N_1}}.
			\end{align*}	
		For $ l\neq k $,
			\begin{align*}
				\int_{\mathbb{R}^d} |x_k|^{2N_1}|\phi_{x_l}|^2 \mathrm{d} x\lesssim &\, \left\| x_k^{2N_1}\phi \right\|_{L^2}\left\|\phi_{x_lx_l} \right\|_{L^2}
			\end{align*}
			is straightforward.  This establishes \eqref{lem3_2} for $ |\gamma|=1. $    Similarly, the case   \eqref{lem3_2} for $ |\gamma|>1 $ follows by induction. 
		\end{proof}

\begin{proof} [\textbf{Proof of Lemma \ref{lem_transf_k}}]
	The proofs of \eqref{est_transf_0_order_P} and \eqref{est_transf_k2} closely resemble that of \eqref{est_transf_k}, given the analogous behavior of the pseudo-differential operator  $ q(x,D) $ and that $ \lambda $ acts similarly with $ P_\lambda\partial_k  $ in frequency space. Thus, we confine ourselves to analyzing \eqref{est_transf_k}.	
	
	Suppose  that $ q(\xi ) \in \mathcal{S}(\mathbb{R}^d)  $ satisfies $ \mathrm{Supp}~ q(\xi )= \{ \xi |1/4\leq |\xi| \leq 2, \xi \in \mathbb{R}^d\}$ and $ q(\xi)=1  $ if $ 1/2\leq |\xi| \leq 1 $.   Then we have
	\begin{align*}
		\mathcal{F} ( \partial_j \psi^\tau ) =   \sqrt{-1}\xi_j 	\mathcal{F}   (P_{\lambda} \chi_1(D)  w^\tau ) =&\sqrt{-1} q\left(\frac{\xi}{\lambda} \right) \cdot  \xi_j  \chi_1 \left( \xi \right)\mathcal{F}  (P_{\lambda} w^\tau )\\
		=&\epsilon q\left(\frac{\xi}{\lambda} \right) \frac{\xi_j}{\epsilon\xi_ 1} \chi_1 \left( \xi \right) \cdot \sqrt{-1}   \xi_1  \chi_1 \left( \xi \right)\mathcal{F}  (P_{\lambda} w^\tau ).
	\end{align*}
	Let  $ Q(x) =\mathcal{F}^{-1}  \left(q\left( \xi\right)  \frac{\xi_j}{\epsilon \xi_ 1} \chi_1 \left( \xi \right)\right)$ and 	$ Q_\lambda= \lambda^d Q(\lambda x ) $. Then   $ \int_{\mathbb{R}^d} | Q(x)| \mathrm{d} x \leq C$. Since $ \chi_1 \left( \xi \right)= \chi_1 \left( \xi /\lambda\right)$, we have   $  Q_\lambda (x) = \mathcal{F}^{-1}  \left(q\left(\frac{\xi}{\lambda} \right) \frac{\xi_j}{\epsilon\xi_ 1} \chi_1 \left( \xi \right)\right)$
	and 
	\begin{align*}
		\partial_j \psi^{\tau}=& \epsilon  \partial_1 \psi^{\tau} *Q_\lambda\\
		=& \epsilon \int_{\mathbb{R}^{d}}\partial_1 \psi^{\tau+\tau^\prime} Q_\lambda(\tau^\prime)  \mathrm{d} \tau^\prime.
	\end{align*}
	
	By Minkowski's inequality, we have
	\begin{align*}
		\left\| h (x) \partial_j \psi_\alpha^\tau \right\|_{L^{2}(\mathbb{R}^d)} \leq& \, \epsilon\int_{\mathbb{R}^{d}}  \left\|  h (x)  \partial_1 \psi^{\tau+\tau^\prime} Q_\lambda(\tau^\prime) \right\|_{L^{2}(\mathbb{R}^d)}  \mathrm{d} \tau^\prime\\
		\leq & \,\epsilon\int_{\mathbb{R}^{d}}  \left| Q_\lambda(\tau^\prime)\right|   \left\|  h (x)  \partial_1 \psi^{\tau-\tau^\prime} \right\|_{L^{2}(\mathbb{R}^d)} \mathrm{d} \tau^\prime\\
		\leq &\, \epsilon\sup_{\tau^\prime}  \left\|  h (x)  \partial_1 \psi^{\tau+\tau^\prime} \right\|_{L^{2}(\mathbb{R}^d)} \cdot  \int_{\mathbb{R}^{d}}  \left| Q_\lambda(\tau^\prime)\right|    \mathrm{d} \tau^\prime\\
		=& \,\epsilon \left\|Q\right\|_{L^1(\mathbb{R}^d)} \cdot \sup_{\tau}  \left\|  h (x)  \partial_1 \psi^{\tau} \right\|_{L^{2}(\mathbb{R}^d)},
	\end{align*}
	which can imply \eqref{est_transf_k} if we set $ c_1=  \left\|Q\right\|_{L^1(\mathbb{R}^d)}^2.  $

\end{proof}

	\begin{proof} [\textbf{Proof of Lemma \ref{lemhalve}}]
	By fractional integration by parts, we compute
		\begin{align} \label{lemhalvec1}
			&\int_{\mathbb{R}}  \int_{\mathbb{R}^{d-1}} w\partial_k u  \mathrm{d} \hat{x}_k \int_{-\infty}^{x_k} v(y_k)\mathrm{d} y_k  \mathrm{d} x_k \nonumber\\
			=&\, \int_{\mathbb{R}^{d-1}}	\int_{\mathbb{R}}   w\partial_k u \int_{-\infty}^{x_k} v(y_k) \mathrm{d} y_k  \mathrm{d} x_k  \mathrm{d} \hat{x}_k\nonumber\\
			=&	\, \int_{\mathbb{R}^{d-1}}\int_{\mathbb{R}}    D_k^{\frac 12} \left( w\int_{-\infty}^{x_k} v(y_k) \mathrm{d} y_k  \right) D_k^{-\frac 12}\partial_k u  \mathrm{d} x_k  \mathrm{d} \hat{x}_k\\
			\lesssim &\,\int_{\mathbb{R}^{d-1}} \left\| \int_{-\infty}^{x_k} v(y_k) \mathrm{d} y_k   D_k^{\frac 12} w\right\|_{L^2(\mathbb{R}_{x_k})} \left\|  D_k^{-\frac 12}\partial_k u \right\|_{L^2(\mathbb{R}_{x_k})}\mathrm{d} \hat{x}_k\nonumber\\
			&+\int_{\mathbb{R}^{d-1}} \left\|\left[D_k^{\frac 12},\int_{-\infty}^{x_k} v(y_k) \mathrm{d} y_k \right]w \right\|_{L^2(\mathbb{R}_{x_k})} \left\|  D_k^{-\frac 12}\partial_k u \right\|_{L^2(\mathbb{R}_{x_k})}\mathrm{d} \hat{x}_k\nonumber\\
			\lesssim &\,\left\|   D_k^{\frac 12}w\right\|_{L^2(\mathbb{R}^d)}\left\|   D_k^{\frac 12}u\right\|_{L^2(\mathbb{R}^d)} \int_{-\infty}^{\infty} |v(x_k) |\mathrm{d} x_k \nonumber\\
			&+\int_{\mathbb{R}^{d-1}} \left\|   D_k^{\frac 12}u(\cdot)\right\|_{L^2(\mathbb{R}_{x_k})} \left\|\left[D_k^{\frac 12},\int_{-\infty}^{x_k} v(y_k) \mathrm{d} y_k \right]w(\cdot) \right\|_{L^2(\mathbb{R}_{x_k})}\mathrm{d} \hat{x}_k.\nonumber
		\end{align}
	Recall the properties of the Hilbert transform $ \mathcal{H} $: $ \mathcal{H}^2=-I, \mathcal{F}\left(\mathcal{H}(f)\right)=- i \mathrm{sgn} \xi \cdot\mathcal{F}(f)(\xi)$ and for $ 1<p<\infty $,  there exists a positive constant $C_p$ such that
		\begin{align}
			\|\mathcal{H}(f)\|_{L^p(\mathbb{R})} \leq C_p\|f\|_{L^p(\mathbb{R})}; \label{Hilt}
		\end{align}
		see  \cite{grafakos2008classical}.  Then for a smooth function $ f $,
		$ D_k^{\frac12} f=- D_k^{-\frac 12}\mathcal{H}^2D_k f=- D_k^{-\frac 12} \mathcal{H} \partial_k f$. Using \eqref{lemc3_1}, \eqref{lem4_1}, and \eqref{Hilt}, we obtain
		\begin{align} \label{lemhalvec2}
			&  \left\|\left[D_k^{\frac 12},\int_{-\infty}^{x_k} v(y_k) \mathrm{d} y_k \right]w(\cdot) \right\|_{L^2(\mathbb{R}_{x_k})}\nonumber\\
			\lesssim &\, \left\|w(\hat{x}_k)\right\|_{L^2(\mathbb{R}_{x_k})}  \left\|D_k^{\frac 12}\int_{-\infty}^{x_k} v(y_k) \mathrm{d} y_k  \right\|_{\mathrm{BMO}(\mathbb{R}_{x_k})}\nonumber\\
			\lesssim &\,\left\|w(\hat{x}_k) \right\|_{L^2(\mathbb{R}_{x_k})}  \left\|D_k^{-\frac 12} \mathcal{H}v(x_k) \right\|_{\mathrm{BMO}(\mathbb{R}_{x_k})}\\
			\lesssim &\,\left\|w(\hat{x}_k) \right\|_{L^2(\mathbb{R}_{x_k})}  \left\| v \right\|_{L^2(\mathbb{R}_{x_k})}. \nonumber
		\end{align}
		Combing \eqref{lemhalvec2} with \eqref{lemhalvec1}, we can use H\"older's inequality to yield  \eqref{lemhalve1}.	
	\end{proof}

	\end{appendices}

\section*{Acknowledgements}
Y. Zhou was supported in part by the NSFC
(Nos. 12171097, 12571231) and Shanghai Key Laboratory for Contemporary Applied Mathematics, Fudan University.
 J. Shao  was supported in part by the China Postdoctoral Science Foundation (No. 2023M730700).
The author J. Shao would like to  thank   Dr. Zhaonan Luo for the helpful discussions on Lemma  \ref{leminter1}. 
	
	\begin{center}
		\bibliographystyle{plain}

	\end{center}

\end{document}